\documentclass[12pt]{amsart}
\usepackage{amssymb}
\usepackage[noadjust]{cite}
\usepackage{booktabs}
\usepackage{url}
\usepackage{hyphenat}
\usepackage{mathtools}
\usepackage{mathrsfs}
\usepackage{ifpdf}
\usepackage[T1]{fontenc}
\usepackage[utf8]{inputenc}
\usepackage{enumitem}
\usepackage[all]{xy}
\usepackage{comment}

\ifpdf
  \usepackage[pdftex]{graphicx}
  \usepackage[pdftex,lmargin=1in,rmargin=1.5in,tmargin=1in,bmargin=1in]{geometry}
  \usepackage[bookmarks=true, bookmarksopen=true,%
    bookmarksdepth=3,bookmarksopenlevel=2,%
    colorlinks=true,%
    linkcolor=blue,%
    citecolor=blue,%
    filecolor=blue,%
    menucolor=blue,%
    urlcolor=blue]{hyperref}
  \hypersetup{pdftitle={Naturality and Mapping Class Groups in
      Heegaard Floer Homology}}
  \hypersetup{pdfauthor={András Juhász and Dylan P. Thurston}}
\else
  \usepackage[dvips]{graphicx}
  \usepackage[dvips,lmargin=1in,rmargin=1.5in,tmargin=1in,bmargin=1in]{geometry}
  \usepackage[dvipdfm]{hyperref}
\fi
\usepackage{color}

\newcommand{\RR}{\mathbb R}

\newcommand{\ZZ}{\mathbb Z}

\newcommand{\FF}{\mathbb F}
\newcommand{\NN}{\mathbb N}

\newcommand{\conn}{\mathbin{\#}}

\newcommand{\co}{\colon}

\newcommand{\eps}{\epsilon}
\newcommand{\abs}[1]{{\lvert #1 \rvert}}

\newcommand{\bdy}{\partial}

\newcommand{\spinc}{\mathfrak s}

\DeclareMathOperator{\Sym}{Sym}

\DeclareMathOperator{\id}{Id}

\DeclareMathOperator{\spin}{Spin}
\newcommand{\SpinC}{\spin^c}

\DeclareMathOperator{\ind}{ind}

\DeclareMathOperator{\Int}{Int}

\DeclareMathOperator{\Diff}{Diff}

\theoremstyle{plain}
\numberwithin{equation}{section}
\newtheorem{theorem}[equation]{Theorem}
\newtheorem{citethm}[equation]{Theorem}
\newtheorem{proposition}[equation]{Proposition}
\newtheorem{lemma}[equation]{Lemma}
\newtheorem{corollary}[equation]{Corollary}

\newtheorem*{thm}{Theorem}

\theoremstyle{definition}
\newtheorem{convention}[equation]{Convention}

\newtheorem{definition}[equation]{Definition}

\theoremstyle{remark}

\newtheorem{example}[equation]{Example}
\newtheorem{remark}[equation]{Remark}

\newcommand{\ol}[1]{\overline{#1}{}}
\newcommand{\HF}{\mathit{HF}}
\newcommand{\HFa}{\widehat {\HF}}
\newcommand{\HFm}{{\HF}^-}
\newcommand{\HFp}{{\HF}^+}
\newcommand{\HFinf}{{\HF}^\infty}
\newcommand{\CF}{\mathit{CF}}
\newcommand{\CFa}{\widehat {\CF}}

\newcommand{\HFL}{\mathit{HFL}}
\newcommand{\HFLa}{\widehat{\HFL}}
\newcommand{\HFLm}{\HFL^-}
\newcommand{\HFK}{\mathit{HFK}}

\newcommand{\HFKm}{\HFK^-}
\newcommand{\SFH}{\mathit{SFH}}
\newcommand{\x}{\mathbf x}
\newcommand{\y}{\mathbf y}
\newcommand{\z}{\mathbf z}

\newcommand{\MCG}{\mathit{MCG}}
\newcommand{\alphas}{{\boldsymbol{\alpha}}}
\newcommand{\betas}{{\boldsymbol{\beta}}}
\newcommand{\gammas}{{\boldsymbol{\gamma}}}
\newcommand{\deltas}{\boldsymbol{\delta}}
\newcommand{\etas}{\boldsymbol{\eta}}

\newcommand{\Torus}{\mathbb{T}}

\newcommand{\HD}{\mathcal{H}}

\newcommand{\Field}{\FF_2}
\newcommand{\bal}{\mathrm{bal}}
\newcommand{\man}{\mathrm{man}}
\newcommand{\link}{\mathrm{link}}
\newcommand{\Bl}{\mathrm{Bl}} 
\newcommand{\bp}{\mathbf{p}}
\newcommand{\bq}{\mathbf{q}}
\newcommand{\FV}{\mathcal{FV}}

\newcommand{\cW}{\mathcal{W}} 

\newcommand{\R}{\mathbb{R}}
\renewcommand{\S}{\Sigma} 
\newcommand{\G}{\mathcal{G}}
\renewcommand{\a}{\alpha} 
\renewcommand{\b}{\beta} 
\newcommand{\g}{\gamma}
\newcommand{\C}{\mathcal{C}}
\newcommand{\cS}{\mathcal{S}}
\newcommand{\s}{\mathrm{stab}}
\renewcommand{\d}{\mathrm{diff}} 

\newcommand{\om}{\ol{\mu}}
\newcommand{\I}{\mathcal{I}}
\newcommand{\D}{\mathcal{D}}

\newcommand{\Man}{\mathsf{Man}} 
\newcommand{\Manp}{\Man_*} 
\newcommand{\Manpv}{\Man_{*,V}} 
\newcommand{\Sut}{\mathsf{Sut}} 

\newcommand{\Link}{\mathsf{Link}}
\newcommand{\Linkp}{\Link_*}
\newcommand{\Linkpp}{\Link_{**}}
\newcommand{\Ab}{\mathsf{Ab}} 
\newcommand{\Mod}{\mathsf{Mod}} 
\newcommand{\Vect}{\mathsf{Vect}} 

\newlist{sing}{enumerate}{1}
\setlist[sing]{label=(S-\arabic*),resume}
\newlist{singab}{enumerate}{2}
\setlist[singab]{label=(\alph*),ref=(S-\theenumi\alph*)}

\renewcommand{\S}{\Sigma} 
\renewcommand{\a}{\alpha} 
\renewcommand{\b}{\beta} 
\renewcommand{\d}{\text{diff}} 
\newcommand{\grad}{\text{grad}}

\DeclareMathOperator{\Diag}{{Diag}}

\DeclareMathOperator{\Iden}{{Id}}
\DeclareMathOperator{\Sing}{{sing}}
\DeclareMathOperator{\codim}{{codim}}

\newread\testin

\graphicspath{{draws/}}
\makeatletter
\def\input@path{{}{draws/}}
\makeatother

\makeatletter
\newcommand\mi@kern[1]{%
  \settowidth\@tempdima{$\mi@obj^{#1}$}
  \kern-\@tempdima
  #1
  \settowidth\@tempdima{$\mi@obj$}
  \kern\@tempdima
}

\newtoks\mi@toksp
\newtoks\mi@toksb
\DeclareRobustCommand{\manyindices}[5]{%
  \def\mi@obj{#5}%
  \mi@toksp\expandafter{\mi@kern{#2}}%
  \mi@toksb\expandafter{\mi@kern{#1}}%
  \@mathmeasure4\textstyle{#5_{#1}^{#2}}%
  \@mathmeasure6\textstyle{#5_{#3}^{#4}}%
  \dimen0-\wd6 \advance\dimen0\wd4%
  \@mathmeasure8\textstyle{\hphantom{{}_{#1}^{#2}}#5^{\the\mi@toksp#4}_{\the\mi@toksb#3}}%
  \hbox to \dimen0{}{\kern-\dimen0\box8}%
}
\makeatother 

\usepackage{ifpdf}
\ifpdf
  
\else
  
\fi

\makeatletter
\newcommand{\labitem}[2]{%
  \item[#1]%
  \begingroup%
    \def\@currentlabel{#1}%
    \label{#2}%
  \endgroup%
}
\makeatother

\begin{document}
\title{Naturality and Mapping Class Groups in Heegaard Floer Homology}

\author[Juh\'asz]{Andr\'as Juh\'asz}
\thanks{AJ was supported by a Royal Society Research Fellowship and OTKA grant NK81203.
This project has received funding from the European Research Council (ERC)
under the European Union’s Horizon 2020 research and innovation programme
(grant agreement No 674978).}
\address{Mathematical Institute, University of Oxford,
  Woodstock Road, Oxford, OX2 6GG, UK}
\email{juhasza@maths.ox.ac.uk}

\author[Thurston]{Dylan~P.~Thurston}
\thanks {DPT was supported by NSF grants number DMS-1008049 and DMS-1507244.}
\address{Department of Mathematics\\
         Indiana University, Bloomington\\
         831 E. Third St.,
         Bloomington, Indiana 47405\\
         USA}
\email{dpthurst@indiana.edu}

\author[Zemke]{Ian Zemke}
\thanks{IZ was supported by NSF Postdoctoral Research Fellowship DMS-1703685.}
\address{Department of Mathematics\\Princeton University\\
Princeton, NJ 08544, USA}
\email{izemke@math.princeton.edu}

\begin{abstract}
    We show that all versions of Heegaard Floer homology,
    link Floer homology, and sutured Floer homology are natural.
    That is, they assign concrete groups to each based 3-manifold, based link, and balanced
    sutured manifold, respectively. Furthermore, we functorially assign isomorphisms to (based) diffeomorphisms,
    and show that this assignment is isotopy invariant. 

    The proof relies on finding a simple generating set for the fundamental group
    of the ``space of Heegaard diagrams,'' and then showing that Heegaard Floer homology
    has no monodromy around these generators.
    In fact, this allows us to give sufficient conditions for an arbitrary
    invariant of multi-pointed Heegaard diagrams to descend to
    a natural invariant of 3-manifolds, links, or sutured manifolds.
\end{abstract}

\maketitle

\newcounter{bean}

\tableofcontents
\section{Introduction}
\label{sec:introduction}

The Heegaard Floer homology groups, introduced by Ozsváth and
Szabó~\cite{OS04:HolomorphicDisks, OS04:HolDiskProperties},
are powerful invariants. They come in four different versions
-- plus, minus, infinity, and hat -- each of which associates a
graded abelian group to a closed oriented 3-manifold.
These were later extended to knots by Ozsv\'ath and Szab\'o~\cite{OS04:Knots}
and independently by Rasmussen~\cite{Rasmussen03:Knots}.
Furthermore, the hat and minus versions were extended to links by Ozsv\'ath and Szab\'o~\cite{OS08:HFL},
and the hat version to sutured manifolds by Juh\'asz~\cite{Juhasz06:Sutured}.
These groups are initially well-defined only up to
isomorphism, but in order to get more powerful invariants that admit diffeomorphism
actions, cobordism maps, and to be able to consider specific elements such as
contact invariants, we need
a naturally associated group, not just a group up to isomorphism. We address
that issue in this paper.

\subsection{Motivation}
To better understand what it means to have a naturally associated group,
we explain some of the naturality issues that
arise in topology. Even though the examples considered here are classical,
they have strong analogies with the case of Heegaard Floer homology.
The reader familiar with naturality issues should
skip to Section~\ref{sec:statement} for the statement of our results.

When defining algebraic invariants in topology, it is essential to
place them in a functorial setting.  For example, suppose we construct
an algebraic invariant of topological spaces that depends on various
choices, and hence only assigns an isomorphism class of, say, groups to
a space. We cannot talk about maps between isomorphism classes of
groups, or consider specific elements of an isomorphism class.

An early example of this phenomenon is provided by the fundamental
group, which depends on the choice of basepoint in an essential
way. Indeed, given a space~$X$ and basepoints~$p$, $q \in X$, there is
no ``canonical'' isomorphism between~$\pi_1(X,p)$ and~$\pi_1(X,q)$;
one has to specify a homotopy class of paths from $p$ to $q$
first. (The word ``canonical'' is often used in an imprecise way in
the literature, we will specify its precise meaning later in this
section.) It follows that~$\pi_1$ can only be defined functorially on
the category of pointed topological spaces.

The naturality/functoriality issues that might arise are perfectly
illustrated by simplicial homology. First, one has to restrict to the
category of triangulable spaces. Even settling invariance up to
isomorphism took several decades. The main question was the following:
Given triangulations~$T$ and~$T'$ of the space~$X$, how do we compare
the groups~$H_*(T)$ and~$H_*(T')$? The first attempts tried to proceed via the
Hauptvermutung: Do $T$ and $T'$ have \emph{isomorphic} subdivisions?
We now know this is false, but even if it were true, it would not
provide naturality as the choice of isomorphism is not unique. The
issue of invariance and naturality was settled by Alexander's method
of simplicial approximation, which provides for any pair of
triangulations~$T$ and~$T'$ of~$X$ an isomorphism $\beta(T,T') \colon
H_*(T) \to H_*(T')$. So how do we obtain the group~$H_*(X)$ from this
data?  First, let us recall a definition due to Eilenberg and Steenrod
\cite[Definition~6.1]{ES}.

\begin{definition} \label{def:transitive-system} A \emph{transitive
    system of groups} consists of
  \begin{itemize}
  \item a set $M$, and for every $\a \in M$, a group $G_\a$,
  \item for every pair $(\a,\b) \in M \times M$, an isomorphism
    $\pi^\a_\b \colon G_\a \to G_\b$ such that
    \begin{enumerate}
    \item \label{item:id} $\pi^\a_\a = \text{Id}_{G_\a}$ for every $\a
      \in M$,
    \item \label{item:trans} $\pi^\b_\g \circ \pi^\a_{\b} = \pi^\a_\g$
      for every $\a$, $\b$, $\g \in M$.
    \end{enumerate}
  \end{itemize}
  A transitive system of groups gives rise to a single group $G$ as
  follows: Let $G$ be the set of elements $g \in \prod_{\a \in M}
  G_\a$ for which $\pi^\a_\b(g(\a)) = g(\b)$ for every $\a$, $\b \in
  M$.
\end{definition}

\begin{remark} \label{rem:colimit}
  For every $\a \in M$, let $p_\a \colon \prod_{\a \in M} G_\a \to
  G_\a$ be the projection. Then $p_\a|_G \colon G \to G_\a$ is an
  isomorphism. In fact, $G$ is a universal object, obtained as a
  limit along the directed graph on~$M$ where there is a unique edge
  from~$\a$ to~$\b$ for every~$(\a,\b) \in M \times M$. The assignment
  $\a \mapsto G_\a$ is a functor from~$M$ to the category of groups,
  which is the diagram along which we take the limit.

  In this paper, instead of the limit, we work with the colimit, which is
  $\coprod_{\a \in M} G_\a/\mathord{\sim}$,
  where $g_\a \sim g_\b$ for $g_\a \in G_\a$ and $g_\b \in G_\b$ if
  and only if $\pi^\a_\b(g_\a) = g_\b$. The group structure on
  $\coprod_{\a \in M} G_\a/\mathord{\sim}$ is given by pointwise
  multiplication of equivalence classes. Each embedding of~$G_\a$
  into $\coprod_{\a \in M} G_\a/\mathord{\sim}$ is an isomorphism.
  It is easy to check that this satisfies the universal property for a colimit.
  The map $g \mapsto [g(\a)]$ for $g \in G$ gives a natural
  isomorphism between the limit and the colimit, where $[g(\a)]$ is the equivalence class
  of $g(\a)$ for an arbitrary $\a \in M$.
\end{remark}

We call the $\pi^\a_\b$ \emph{canonical isomorphisms}. So, if we are
constructing some algebraic invariant, and have isomorphisms for any
pair of choices, we only call these isomorphisms ``canonical'' if they
satisfy properties~\eqref{item:id} and~\eqref{item:trans} above. An
instance of a transitive system of groups is given by taking $M$ to be
the set of all triangulations of a triangulable space $X$, and for any
pair $(T,T') \in M \times M$, setting $\pi^T_{T'} = \b(T,T')$. Another
example of a transitive system is given in the case of Morse homology
by Schwarz~\cite[Section 4.1.3]{Sch}, where one needs to compare
homology groups defined using different Morse functions.

Classical homology was put in a functorial framework by the
Eilenberg-Steenrod axioms, whereas the gauge theoretic invariants of
3- and 4-manifolds are expected to satisfy properties similar to the topological quantum
field theory (TQFT) axioms of Atiyah, called a ``secondary TQFT.''
Our motivating question is whether Heegaard Floer homology fits into such a functorial picture.
Heegaard Floer homology is a package of invariants of 3- and
4-manifolds defined by Ozsv\'ath and Szab\'o~\cite{OS04:HolomorphicDisks, OS06:HolDiskFour}.
It follows from our work that each version of Heegaard Floer homology
individually satisfies the classical TQFT axioms, but it is important to note that
the closed 4-manifold invariant is obtained by mixing the $+$, $-$, and $\infty$ versions,
hence deviating from Atiyah's original description. Building on our work,
the third author~\cite{Zemke} has shown that the 4-manifold invariants are also well-defined.

In its simplest form, Heegaard Floer homology assigns an Abelian group $\HFa(Y)$ to a closed
oriented 3-manifold $Y$, well-defined up to isomorphism. The
construction depends on a choice of Heegaard diagram for $Y$.
Given two Heegaard diagrams $\HD$ and $\HD'$ for $Y$, our goal is to
construct a canonical isomorphism $\HFa(\HD) \to \HFa(\HD')$ such that
the set of diagrams, together with these isomorphisms form a
transitive system of groups, yielding a single group $\HFa(Y)$.
We want to do this in a way that every diffeomorphism $d \colon Y_0 \to
Y_1$ induces an isomorphism
\[
d_* \colon \HFa(Y_0) \to \HFa(Y_1).
\]
For this (and also to get the canonical isomorphisms), one has to consider
diagrams embedded in~$Y$, not only ``abstract'' ones.
That is, we consider triples $(\S,\alphas,\betas)$ where~$\S$ is a subsurface of~$Y$
that splits~$Y$ into two handlebodies, and~$\alphas$, $\betas \subset \S$ are attaching sets
for the two handlebodies.
Then the main question is:
How do we compare~$\HFa$ for diagrams that are embedded
in~$Y$ differently? Note that our construction of canonical isomorphisms
differs from that of Ozsv\'ath and Szab\'o~\cite[Theorem~2.1]{OS06:HolDiskFour}
in a subtle but essential way; see Remark~\ref{rem:difference}.

The Reidemeister-Singer theorem provides an analogue of the
Hauptvermutung in the case of Heegaard splittings: Any two Heegaard
splittings of $Y$ become isotopic after stabilizations. However, this
isotopy is far from being unique. In fact, the fundamental group
$\pi_1(\mathcal{S}(Y,\S))$ of the space of Heegaard splittings equivalent to
$(Y,\S)$ is highly non-trivial. So we could have a loop of Heegaard
diagrams $\{\, \HD_t \colon t \in [0,1] \,\}$ of~$Y$ along which
$\HFa(\HD_t)$ has monodromy. Indeed, let us recall the following
definition.

\begin{definition}
  Let $\S \subset Y$ be a Heegaard surface. Then the \emph{Goeritz
    group} of the Heegaard splitting $(Y,\S)$ is defined as
  \[
  G(Y,\S) = \ker \left(\text{MCG}(Y,\S) \to \text{MCG}(Y)\right).
  \]
  In other words, $G(Y,\S)$ consists of automorphisms $d$ of $(Y,\S)$
  (considered up to isotopy preserving the splitting) such that $d$ is
  isotopic to $\text{Id}_Y$ if we are allowed to move $\S$.
\end{definition}

According to Johnson and McCullough~\cite{JMc13:HeegaardSpace}, there
is a short exact sequence
\[
1 \to \pi_1(\text{Diff}(Y)) \to \pi_1(\mathcal{S}(Y,\S)) \to G(Y,\S) \to 1.
\]
Let $\HD = (\S, \alphas,\betas,z)$ be a Heegaard diagram of $Y$.
Ignoring basepoint issues, an element of $\pi_1(\mathcal{S}(Y,\S))$ coming
from $\pi_1(\Diff(Y))$ acts
trivially on the Heegaard Floer homology $\HFa(\HD)$ (as this is the action of~$\id_\S$,
the endpoint of the loop),
so this descends to an action of $G(Y,\S)$ on $\HFa(\HD)$.  The 3-sphere has a unique
genus~$g$ Heegaard splitting for every $g \ge 0$.  At the time of
writing of this paper, it is unknown whether $G(S^3,\S)$ is finitely
generated when the genus of $\S$ is greater than~$2$.  Understanding
the group~$G(Y,\S)$ for a general 3-manifold~$Y$ and
splitting~$\S$ seems even more difficult.
Heegaard Floer homology is invariant under
stabilization, and the ``fundamental group'' of the space of Heegaard
diagrams modulo stabilizations is easier to understand, as we shall
see in this paper.

\subsection{Statement of results}
\label{sec:statement}

We prove that Heegaard Floer homology is an invariant of based
3-manifolds in the following strong sense.  (We ignore gradings and
$\SpinC$ structures for the moment.)

\begin{definition}
  Let $\Man$ be the category whose class of objects $\abs{\Man}$ consists of closed, connected,
  oriented 3-manifolds, and whose morphisms are diffeomorphisms.  Let
  $\Manp$ be the category whose objects are pairs $(Y, p)$, where
  $Y\in \abs{\Man}$ and $p \in Y$ is a choice of basepoint, and whose
  morphisms are basepoint-preserving diffeomorphisms.

  Also, let $R$-$\Mod$ be the category of $R$-modules for any
  ring~$R$, and let $k$-$\Vect$ be the category of vector spaces
  over~$k$ for any field~$k$.
\end{definition}

Recall that Ozsváth and Szabó defined four different versions of
Heegaard Floer homology, named $\HFa$, $\HFm$, $\HFp$, and
$\HFinf$. We will write $\HF$ without decoration to mean any of these
four versions. Note that these are all modules over the polynomial
ring $\Field[U]$, where the $U$-action is trivial on~$\HFa$.

\begin{theorem}\label{thm:invt-hf}
  There are functors
  \[
  \HFa \text{, } \HFm \text{, } \HFp \text{, } \HFinf \colon \Manp \to
  \Field[U]\text{-}\Mod,
  \]
  such that for a based 3-manifold $(Y,p)$, the groups $\HF(Y,p)$ are
  isomorphic to the various versions of Heegaard Floer homology
  defined by
  Ozsv\'ath and Szab\'o~\cite{OS04:HolomorphicDisks,OS04:HolDiskProperties}.
  Furthermore, isotopic diffeomorphisms induce identical maps
  on~$\HF$.
\end{theorem}

As $\HF_{\text{red}}(Y,p)$ is defined
as $\HF^+(Y,p)/\text{Im}(U^k)$ for $k$ sufficiently large (see \cite[Definition~4.7]{OS04:HolomorphicDisks}),
it immediately follows from Theorem~\ref{thm:invt-hf} that $\HF_{\text{red}}$ is also functorial.

Ozsv\'ath and Szab\'o~\cite{OS04:HolomorphicDisks,OS04:HolDiskProperties}
showed that the \emph{isomorphism class} of $\HF(Y)$ is an invariant
of the 3-manifold $Y$.
The statement in Theorem~\ref{thm:invt-hf} is stronger, in that it
says that $\HF(Y,p)$ is actually a well-defined group, not just an
isomorphism class of groups.
The first step towards naturality was made by Ozsv\'ath and Szab\'o~\cite[Theorem~2.1]{OS06:HolDiskFour},
who constructed maps $\Psi$ between $\HF(\S,\alphas,\betas)$ and $\HF(\S,\alphas',\betas')$
for a fixed Heegaard surface $\S$ and equivalent ``abstract'' diagrams
$(\S,\alphas,\betas)$ and $(\S,\alphas',\betas')$.
The maps $\Psi$ are isomorphisms, and they satisfy conditions~\eqref{item:id}
and~\eqref{item:trans} of Definition~\ref{def:transitive-system}.
Furthermore, they also defined maps for stabilizations.

In the present paper, we explain how to canonically compare
invariants of diagrams with different embeddings in Y.
As it turns out, the additional checks are the following:
One has to prove that $\HF$ has no monodromy around the simple
handleswap loop of Figure~\ref{fig:handleswap}, and show that the map
on $\HF$ induced by a diffeomorphism $d \colon (\S,\alphas,\betas) \to
(\S,\alphas',\betas')$ isotopic to $\text{Id}_\S$ agrees with the
canonical isomorphism $\Psi$.

One surprise in Theorem~\ref{thm:invt-hf} is the appearance of the
basepoint. To make this precise, we look at the mapping class group.

\begin{definition}
  For a smooth manifold~$M$, its
  \emph{mapping class group} is
  \[
  \MCG(M) = \Diff(M) / \Diff_0(M) = \pi_0(\Diff(M)),
  \]
  where $\Diff(M)$ is the group of diffeomorphisms of~$X$, and
  $\Diff_0(M)$ is the subgroup of diffeomorphisms isotopic to the
  identity, which is also the connected component of the identity in
  $\Diff(M)$. Similarly, for a based smooth manifold $(M,p)$, its
  \emph{based mapping class group} is
  \[
  \MCG(M,p) = \Diff(M,p) / \Diff_0(M,p) = \pi_0(\Diff(M,p)),
  \]
  where we consider maps that preserve the basepoint.
\end{definition}

\begin{corollary}\label{cor:mcg-acts}
  For a based 3-manifold $(Y,p)$, the group
  $\MCG(Y,p)$ acts naturally on $\HF(Y,p)$ for any of the four
  versions of Heegaard Floer homology.
\end{corollary}

\begin{proof}
  This follows immediately from Theorem~\ref{thm:invt-hf} when
  restricted to automorphisms of $(Y,p)$.
\end{proof}

It is easy to construct examples where the action of the mapping class
group is non-trivial.  For instance, for a
$3$-manifold~$Y$, the evident diffeomorphism that exchanges the two factors of $Y
\conn Y$ (preserving a basepoint) will act via $x \otimes y \mapsto y \otimes x$ on
$\HFa(Y \conn Y) \cong \HFa(Y) \otimes \HFa(Y)$, which is non-trivial if~$Y$ is
sufficiently complicated.

Recall that, from the fibration $\Diff(Y,p) \to \Diff(Y) \to Y$, there
is a Birman exact sequence for based mapping class groups for any
connected manifold~$Y$:
\[
\pi_1(Y) \to \MCG(Y,p) \to \MCG(Y) \to 0.
\]
Thus, from an action of $\MCG(Y,p)$ on $\HFa$, we get an action of
$\pi_1(Y)$ on $\HFa$. This action of $\pi_1(Y)$ is trivial if the
action descends to an action of the unbased mapping class group
$\MCG(Y)$.  This action of $\pi_1(Y)$ is \emph{not} always
trivial. However, it factors through an action of $H_1(Y)$.
For the proof of a precise formula, originally conjectured by the first author,
see the work of the third author~\cite{Zemke1}.

For the other three versions, $\HFm(Y,p)$, $\HFp(Y,p)$, and
$\HFinf(Y,p)$, the action of $\pi_1(Y)$ on $\HF$
is always trivial, in analogy with the situation for monopole Floer
homology~\cite{KronheimerMrowka}. This has been shown by the third author~\cite{Zemke}.

There is also a version of Theorem~\ref{thm:invt-hf} for links.  Let
$\Link$ be the category of oriented links in $S^3$, whose morphisms
are orientation preserving diffeomorphisms $d \colon (S^3,L_1) \to
(S^3,L_2)$.  Let $\Link_*$ be the category whose objects
are \emph{based oriented links}: pairs $(L, \bp)$, where $L \subset
S^3$ is an oriented link and $\bp=\{\, p_1, \dots, p_n \,\} \subset L$
is a set of basepoints, exactly one on each component of~$L$.  The
morphisms are diffeomorphisms of $S^3$ preserving the based oriented
link.

\begin{theorem}\label{thm:invt-hfl}
  There are functors
  \begin{align*}
    \HFLa&\colon \Linkp \to \Field\text{-}\Vect,\\
    \HFL^-&\colon \Linkp \to \Field[U]\text{-}\Mod,
  \end{align*}
  agreeing up to isomorphism with the link invariants defined by
  Ozsv\'ath and Szab\'o~\cite{OS04:Knots, OS08:HFL} and
  Rasmussen~\cite{Rasmussen03:Knots}.  Isotopic
  diffeomorphisms induce identical maps on~$\HFLa$ and~$\HFL^-$.
\end{theorem}

As in Corollary~\ref{cor:mcg-acts}, Theorem~\ref{thm:invt-hfl} implies
that $\MCG(S^3,L,\bp)$ acts on $\HFLa(L,\bp)$ and $\HFL^-(L,\bp)$.  Again, one can ask
whether this action is non-trivial, and in particular, whether the
basepoint makes a difference.  For simplicity, consider
knots, in which case there is an exact sequence
\[
\pi_1(S^1) \to \MCG(S^3,K,p) \to \MCG(S^3,K) \to 0.
\]
In this context, Sarkar~\cite{Suc15:basepoints} has constructed
many examples where the action of $\pi_1(S^1)$ on $\HFLa(K,p)$ is
non-trivial. More concretely, let~$\sigma \in \MCG(S^3,K,p)$
be the positive finger move (or Dehn twist) along~$K$, defined
in~\cite[p.~4]{Suc15:basepoints}. Then it follows
from~\cite[Theorem~6.1]{Suc15:basepoints} that the action of~$\sigma$
on~$\HFLa(K,p)$ for prime knots up to~$9$ crossings is non-trivial
more often than not. He also proved a formula for the $\pi_1(S^1)$ action
for knots in $S^3$, which was then extended to links in arbitrary 3-manifolds
by the third author~\cite{Zemke2}.

There are several variants of Theorem~\ref{thm:invt-hfl}.  For
instance, Ozsv\'ath and Szab\'o~\cite{OS11:Rational} have
defined the group $\HFKm(Y,K,p)$ for a rationally
null-homologous knot~$K$ in an oriented 3-manifold~$Y$.
They also defined $\HFLm(Y,L,\bp)$ for a link~$L$ in an integer homology sphere~$Y$;
see~\cite[Theorem 4.7]{OS08:HFL}.
There is also more structure that can be put on the result.
In particular, there is a spectral sequence from $\HFKm(Y,K,p)$
converging to $\HFm(Y)$.  These invariants are again functorial.
In fact, Hendricks and Manolescu~\cite[Proposition~2.3]{HM17:Involutive} showed
that Theorem~\ref{thm:invt-hf} also holds on the chain level
in the homotopy category of chain complexes of $\FF_2[U]$-modules.
A knot induces a filtration on the Heegaard Floer chain complex, and
the canonical isomorphisms preserve the knot filtration; see
Juh\'asz and Marengon~\cite[Section~5.11]{JM16:Concordance}.
Finally, for the naturality of $\HFLa$ of links in arbitrary oriented 3-manifolds
and several basepoints on each link component,
see Juh\'asz~\cite[Proposition~4.10]{Juhasz16:Cobordism}.

We will unify the proofs of Theorems~\ref{thm:invt-hf}
and~\ref{thm:invt-hfl} in the more general setting of \emph{balanced sutured
manifolds}.  Let $\Sut$ be the category of sutured 3-manifolds and
diffeomorphisms, and let $\Sut_\bal$ be the full subcategory of
balanced sutured manifolds.  (For definitions and details, see
Definitions~\ref{def:sutured} and~\ref{def:balanced} below.)

\begin{theorem}\label{thm:invt-sfh}
  There is a functor
  \[
  \SFH\co \Sut_\bal \to \FF_2\text{-}\Vect,
  \]
  agreeing up to isomorphism with the sutured manifold invariant
  defined by the first author~\cite{Juhasz06:Sutured}.  Isotopic
  diffeomorphisms induce identical maps on~$\SFH$.
\end{theorem}

All Heegaard Floer homology groups discussed above decompose along
$\SpinC$ structures, for example,
\[
\SFH(M,\g) = \bigoplus_{\spinc \in \SpinC(M,\g)} \SFH(M,\g,\spinc).
\]
In addition, each summand $\SFH(M,\g,\spinc)$ carries a relative
homological $\ZZ_{\mathfrak{d}(\spinc)}$-grading, where
$\mathfrak{d}(\spinc)$ is the divisibility of the Chern class
$c_1(\spinc) \in H^2(M)$. So, for any $x$, $y \in \SFH(M,\g,\spinc)$,
the grading difference $\text{gr}(x,y)$ is an element of
$\ZZ_{\mathfrak{d}(\spinc)}$.

These gradings are natural in the following sense. Suppose that $d
\colon (M,\g) \to (N,\nu)$ is a diffeomorphism. Then the induced map
\[
d_* \colon \SFH(M,\g) \to \SFH(N,\nu)
\]
restricts to an isomorphism
\[
d_*|_{\SFH(M,\g,\spinc)} \colon \SFH(M,\g,\spinc) \to \SFH(N,\nu,
d(\spinc))
\]
that preserves the relative homological grading. Completely analogous
results hold for the other versions of Heegaard Floer homology.

Now we outline the main technical tools behind the above results; for
further details we refer the reader to Section~\ref{sec:invariants}.
To be able to treat the various versions of Heegaard Floer homology simultaneously,
we consider an arbitrary algebraic invariant $F$
of abstract (i.e., not necessarily embedded) diagrams of sutured manifolds in a given class~$\mathcal{S}$ of
sutured manifolds (e.g., knot complements in case $F$ is knot Floer homology).
An \emph{isotopy diagram} is a sutured diagram with attaching sets taken up to isotopy. We work
with these to avoid admissibility issues.
Let $\G(\mathcal{S})$ be the directed graph whose vertices are isotopy diagrams
of sutured manifolds in~$\mathcal{S}$, and the
vertices $H$ and $H'$ are connected by an edge if either the $\a$-curves or the $\b$-curves differ by a sequence
of isotopies and handleslides (called an $\a$- or $\b$-equivalence),
or if $H'$ is obtained from $H$ by a stabilization
or a destabilization, and there is an edge for every diffeomorphism $d \colon H \to H'$.
We say that $F$ is a \emph{weak Heegaard invariant} if, for every edge $e$ from $H$ to $H'$ in $\G(\mathcal{S})$,
there is an induced isomorphism
\[
F(e) \colon F(H) \to F(H').
\]
A weak Heegaard invariant then gives rise to an invariant of sutured manifolds in the class~$\mathcal{S}$,
well-defined up to isomorphism.

To assign a concrete algebraic object to each sutured manifold in the class~$\mathcal{S}$, we then define the
notion of a \emph{strong Heegaard invariant}. Such an $F$ has to commute along certain
distinguished loops in $\G(\mathcal{S})$. These loops include rectangles where opposite edges are of the
same type (i.e., they are both $\a$-equivalences, $\b$-equivalences, stabilizations, or diffeomorphisms,
see Definition~\ref{def:distinguished-rect}),
and the aforementioned simple handleswap triangles of Figure~\ref{fig:handleswap} (involving an $\a$-handleslide,
a $\b$-handleslide, and a diffeomorphism). Furthermore, a strong Heegaard invariant has
to satisfy the property that if $e \co H \to H$ is a diffeomorphism isotopic to the identity
of the Heegaard surface, then $F(e) = \text{Id}_{F(H)}$.

Given a strong Heegaard invariant~$F$ and a sutured manifold $(M,\g)$ in the class~$\mathcal{S}$,
we obtain the invariant $F(M,\g)$ as follows. We take the subgraph $\G_{(M,\g)}$
of $\G(\mathcal{S})$ whose vertices are isotopy diagrams \emph{embedded} in $(M,\g)$,
and where we only consider diffeomorphisms that are isotopic to
the identity in $M$. Then our main result is Theorem~\ref{thm:iso}, which states that
given any two paths in $\G_{(M,\g)}$ from $H$ to $H'$,
the composition of $F$ along these paths coincide. The proof of this occupies most of the paper,
and relies on a careful analysis of the bifurcations occurring in generic 2-parameter
families of gradient vector fields on 3-manifolds.
It easily follows that these compositions give a canonical isomorphism $F(H) \to F(H')$, and
we obtain $F(M,\g)$ via Definition~\ref{def:transitive-system}.
We show that the different versions of Heegaard Floer homology are strong Heegaard invariants
in Sections~\ref{sec:HFstrong}, \ref{subsec:Handleswaps}, and~\ref{sec:strong}.

One consequence of naturality is that we can now talk about specific elements of~$\HF(Y)$,
which is necessary for the definition of the contact class. Another consequence is that we can define maps on
Heegaard Floer homology induced by diffeomorphisms and cobordisms.
The paper might also be of interest to 3-manifold topologists, as it sheds more light on
the space of Heegaard splittings and diagrams, potentially telling more about the structure of the
Goeritz group.

\subsection{Outline}
In Section~\ref{sec:invariants}, we study Heegaard invariants.
In Subsection~\ref{sec:sutured-manifolds}, we recall the definition of \emph{sutured manifolds},
and explain how to assign sutured manifolds to based 3-manifolds and links.
This allows us to unify our treatment of naturality by focusing on sutured manifolds.
In Subsection~\ref{sec:sutured-diagrams}, we discuss \emph{abstract} and \emph{embedded Heegaard
diagrams} of sutured manifolds. This distinction plays an important role in our construction
of natural 3-manifold invariants.

In Subsection~\ref{sec:moves-diagrams}, we define
\emph{moves} on abstract sutured diagrams, namely, $\a$-equivalence, $\b$-equivalence, stabilization,
destabilization, and diffeomorphism. We show that any two diagrams defining the
same sutured manifold can be connected by a sequence of such moves.
We define \emph{weak Heegaard invariants} (Definition~\ref{def:weak-Heegaard}), and state that the different versions of Heegaard Floer
homology are weak Heegaard invariants (Theorems~\ref{thm:HF-weak}, \ref{thm:HFL-weak}, and~\ref{thm:SFH-weak}).
They assign an object in a category to every sutured diagram in a given class (e.g., diagrams of based 3-manifolds or link
complements), and a morphism to every move.
A weak Heegaard invariant gives rise to an invariant of based 3-manifolds
(or based links, or sutured manifolds) that is well-defined up to isomorphism.

A key new notion is that of a \emph{strong Heegaard invariant} (Definition~\ref{def:strong-Heegaard}),
which we introduce in Subsection~\ref{sec:strong-invar}. This is a weak Heegaard invariant~$F$ that
has to satisfy four axioms: functoriality (under compositions of moves of the same type),
commutativity (which implies that~$F$ remains the same if we swap the order of two moves),
continuity (isotopic diffeomorphisms induce the same morphism), and \emph{simple handleswap invariance} (this
requires that~$F$ commutes along a certain triangle of moves).

In Subsection~\ref{sec:main-theorems}, we explain how a strong Heegaard invariant~$F$ gives
rise to a natural invariant of a given class of sutured manifolds.
In Theorem~\ref{thm:iso}, which is one of the deepest results of this work,
we show that if we connect two \emph{embedded} diagrams of a sutured manifold $(M,\g)$ by
a sequence of moves such that we only allow diffeomorphisms that are \emph{isotopic to the
identity in~$M$} (Definition~\ref{def:isotopic}), and we compose the morphisms induced by~$F$, then the result does not
depend on the choice of moves. This gives rise to a transitive system (Definition~\ref{def:transitive-system})
that allows us to assign an invariant to every
sutured manifold in the given class that is functorial under diffeomorphisms (Definition~\ref{def:strong-functor}).
We then prove Theorems~\ref{thm:invt-hf}, \ref{thm:invt-hfl}, and \ref{thm:invt-sfh} stated in the introduction.

In Section~\ref{sec:examples}, we present some examples that highlight some of the issues that
arise when trying to define functorial Heegaard Floer invariants.

Most of the remainder of this paper is aimed at proving Theorem~\ref{thm:iso}.
In Section~\ref{sec:smooth}, we give a brief overview of the types of \emph{singularities} that
appear in generic 2-parameter families of smooth functions. A generic smooth function is Morse.
In a generic 1-parameter family, one can have~$A_2$ (birth-death) singularities at isolated parameter values.
In a generic 2-parameter family, $A_3^\pm$ (birth-death-birth) points appear at isolated parameter values.

As a Heegaard diagram arises from a generic gradient vector field, not just a Morse function, we
review the bifurcations of generic 1-parameter and 2-parameter families of \emph{gradient vector fields}
in Section~\ref{sec:gradients}. In Subsection~\ref{sec:invariant-manifold}, we recall some
classical results from the theory of dynamical systems, such as the center manifold
theorem (Theorem~\ref{thm:center-manifold}) and the reduction principle (Theorem~\ref{thm:reduction-principle}).
In Subsection~\ref{sec:bifurcations}, we state that a generic gradient vector field is Morse-Smale.
In a generic 1-parameter family, a non-hyperbolic singular point might appear, or a stable and an
unstable manifold might intersect non-transversely. The classification of bifurcations of generic
2-parameter families of gradients is already very complicated in dimension~3,
and was carried out by Vegter~\cite{Vegter85}. In Subsection~\ref{sec:sutured-functions},
we define the space~$\FV(M,\g)$ of gradient-like vector fields
on the sutured manifold~$(M,\g)$ that we are going to work with. This is weakly contractible
by Corollary~\ref{cor:FV-contractible}.

In Section~\ref{sec:transl-hd}, we convert Morse-Smale vector fields to (overcomplete) Heegaard diagrams,
codimension-1 bifurcations of gradients to Heegaard moves, and codimension-2 bifurcations
to loops of diagrams.

We introduce the notion of \emph{separable gradients} in
Subsection~\ref{sec:separability}. These have at most codimension-2 bifurcations,
and the most important criterion is that there is no gradient flow-line from an index~2
to an index~1 critical point. A surface \emph{separates} such a gradient vector field
essentially if it is transverse to it, and contains all index~0 and~1 critical points on
one side, and all index~2 and~3 critical points on the other side. The space of such
surfaces is contractible, and they divide the sutured
manifold into two sutured compression bodies (Proposition~\ref{prop:grad-like}).
Furthermore, we can continuously deform the separating surface if we deform
the gradient vector field without introducing a birth-death bifurcation (Proposition~\ref{prop:fibration}).

In Subsection~\ref{sec:codimension-0}, we show how to obtain an \emph{overcomplete} diagram
(Definition~\ref{def:overcomplete}) from a codimension-0 gradient vector field
by intersecting the separating surface~$\S$ with the unstable manifolds of index-1 critical points
and the stable manifolds of index-2 critical points. The diagram is overcomplete because
the attaching curves might not be linearly independent in~$H_1(\S)$.
We assign a certain graph to a separable gradient, and we obtain a sutured diagram
once we choose a spanning forest for this graph by
deleting attaching curves corresponding to edges that lie outside the spanning forest (Definition~\ref{def:diagram}).
In the opposite direction, given a diagram of a sutured manifold such that the $\a$- and
$\b$-curves are transverse, the space of \emph{simple}
Morse-Smale gradients (Definition~\ref{def:simple}) that induce it is non-empty and connected
(Propositions~\ref{prop:existence} and~\ref{prop:connected-MS}).

In Subsection~\ref{sec:transl-codim-1}, we translate codimension-1 bifurcations of gradients
to moves on overcomplete diagrams. We show that a 1-parameter family of
Morse-Smale gradients induces a diffeomorphism of the diagram isotopic to the identity in~$M$ (Lemma~\ref{lem:isotopy}).
In Proposition~\ref{prop:1-param}, we show that birth-death bifurcations correspond to
\emph{generalized (de)stabilizations} (Definition~\ref{def:gen-stab})
or the appearance of a null-homotopic $\a$- or $\b$-curve,
and a non-transverse intersection of stable and unstable manifolds to a generalized
handleslide (Definition~\ref{def:gen-handleslide}) or a tangency between an $\a$- and a $\b$-curve.
In Proposition~\ref{prop:rectangle}, we show that, given a generic 2-parameter family of gradients
and a rectangle in the parameter space that intersects a codimension-1 stratum in two points on opposite sides,
the diagrams in the corners form a distinguished rectangle appearing in the Commutativity Axiom
of strong Heegaard invariants.

In Subsection~\ref{sec:codimension-1-actual}, we show how to choose spanning trees
appropriately in Propositions~\ref{prop:1-param}
and~\ref{prop:2-param-cod-1} to pass from overcomplete to actual
Heegaard diagrams, without altering the relationship of the diagrams
before and after the bifurcation in an essential way.
In Subsection~\ref{sec:codim-1:hd-to-func}, we show that, given a move on sutured diagrams,
it can be converted to a path of gradients (Proposition~\ref{prop:lift-moves}).

In Subsection~\ref{sec:transl-codim-2}, we translate codimension-2 bifurcations to loops of diagrams.
Given a codimension-2 bifurcation of gradients at a parameter value~$p$, the co\-di\-men\-sion-1 strata divide a neighborhood
of~$p$ in the parameter space into chambers. Taking a point in each chamber, we obtain a loop of diagrams
such that neighboring diagrams are related by generalized Heegaard moves. We describe
these loops in Theorem~\ref{thm:2-param}. There are many more loops than what appear
in the definition of strong Heegaard invariants.

In Section~\ref{sec:simplify}, we break down generalized stabilizations, generalized
handleslides, and the loops of diagrams
appearing in Theorem~\ref{thm:2-param} into
the simpler moves and loops that feature in the definition of strong
Heegaard invariants. During the
simplification procedure, we work with overcomplete diagrams, and will
only later choose spanning trees to pass to actual sutured diagrams.
We pass from 2-parameter families of gradients to certain combinatorial
structures, namely, \emph{polyhedral decompositions} of the parameter space~$D^2$
(Definition~\ref{def:polyhedral}). These are certain regular CW decompositions
of~$D^2$, and we label the vertices with (overcomplete) diagrams. Neighboring
diagrams are related by generalized moves. From now on, we only modify
the polyhedral decomposition, and forget about the family of gradients.

In Subsection~\ref{sec:simplify-codim-1}, we show how to replace
a generalized stabilization with a simple stabilization, a number of $\a$-handleslides,
and a number of $\b$-handleslides. There are several choices one can make during this
resolution process. If we use different choices at the two ends of a codimension-1 generalized
stabilization stratum, then we can interpolate between them (Lemma~\ref{lem:ori-change}).
We can also replace a generalized handleslide by a simple handleslide and an isotopy.
In Subsection~\ref{sec:simplify-codim-2}, we apply this resolution process
to the loops described in Theorem~\ref{sec:transl-codim-2}. We then break them
down into loops of diagrams appearing in the definition of strong Heegaard invariants
and $(k,l)$-handleswaps (Definition~\ref{def:handleswap}). We write $(k,l)$
handleswaps in terms of simple handleswaps in Subsection~\ref{sec:simpl-handleswap}.
The above procedure can be thought of as finding a simple generating set for
the space of Heegaard diagrams of a given sutured manifold.

We prove Theorem~\ref{thm:iso} in Section~\ref{sec:proof}.
The idea is the following: Given a loop of diagrams, we lift them to a loop of gradients
such that the bifurcations induce the given moves. We then extend this loop from~$S^1$
to a generic 2-parameter family of gradients over~$D^2$. We choose a polyhedral decomposition
of~$D^2$ compatible with the bifurcation stratification given by this family.
Each vertex is labeled by a diagram. We then apply the resolution procedure
to obtain the simple loops appearing in the definition of strong Heegaard invariants.
Finally, a simple combinatorial argument gives that, since a strong Heegaard invariant
commutes along the simple loops, it also commutes along~$S^1$.

In Section~\ref{sec:HeegaardFloer}, we prove Theorem~\ref{thm:strong},
which states that Heegaard Floer homology~$\HF$ is a strong Heegaard invariant.
We define the isomorphisms associated to Heegaard moves in Subsection~\ref{sec:HeegaarFloerInvariant},
where we reprove that $\HF$ is a weak Heegaard invariant. For this, we only use
triangle maps that are more computable than continuation maps.
We verify that $\HF$ is a strong Heegaard invariant in Subsection~\ref{sec:HFstrong}.
We prove the Continuity Axiom in Proposition~\ref{prop:continuity},
and simple handleswap invariance in Subsection~\ref{subsec:Handleswaps}.
We put the above pieces together in Subsection~\ref{sec:strong} to prove Theorem~\ref{thm:strong}.

Finally, in Appendix~\ref{app:handleslide}, we sketch a description of strong Heegaard invariants
for classical (i.e., not sutured) single pointed Heegaard diagrams
that is equivalent to Definition~\ref{def:strong-Heegaard}, and
instead of $\a$-equivalences and $\b$-equivalences, uses more
elementary moves: $\a$-isotopies, $\b$-isotopies, $\a$-handleslides,
and $\b$-handleslides.

\subsection{Acknowledgements}

We are extremely grateful to Peter Ozsv\'ath for numerous helpful conversations.
We would also like to thank Valentin
Afraimovich, 
Ryan Budney, 
Boris Hasselblatt, 
Matthew Hedden, Michael Hutchings, Martin Hyland, 
Jesse Johnson, Robert
Lipshitz, 
Saul Schleimer,  
Zolt\'an Szab\'o, and the referee for their guidance and suggestions.

This project would not have been possible without the
hospitality of the Tambara Institute of Mathematical Sciences,
the Mathematical Sciences Research Institute, and the Isaac Newton Institute.
Most of the work was carried out while the first author was at the University of Cambridge and
the second author was at Barnard College, Columbia University.
\section{Heegaard invariants}
\label{sec:invariants}

\subsection{Sutured manifolds}
\label{sec:sutured-manifolds}
Sutured manifolds were introduced by Gabai~\cite{Gabai}. The following definition
is slightly less general, in that it excludes toroidal sutures.

\begin{definition}\label{def:sutured}
A \emph{sutured manifold} $(M,\g)$ is a compact oriented
3-manifold $M$ with boundary, together with a set $\g \subset
\partial M$ of pairwise disjoint annuli. Furthermore, the interior of each component of
$\g$ contains a \emph{suture}; i.e., a homologically
nontrivial oriented simple closed curve. We denote the union of the
sutures by $s(\g)$.
In addition, every component of $R(\g) = \partial M \setminus
\Int(\g)$ is oriented. Define $R_+(\g)$ (respectively
$R_-(\g)$) to be those components of $\partial M \setminus
\Int(\g)$ whose orientations agree (respectively disagree) with the
orientation of $\bdy M$, or equivalently, whose
normal vectors point out of (respectively into) $M$.
The orientation on $R(\g)$ must be coherent with respect to~$s(\g)$;
i.e., if $\delta$ is a component of $\partial
R(\g)$ and is given the boundary orientation, then $\delta$ must
represent the same homology class in $H_1(\g)$ as some suture.%

A sutured manifold $(M,\g)$ is called \emph{proper} if
the map $\pi_0(\g) \to \pi_0(\partial M)$ is surjective and
$M$ has no closed components (i.e., the map $\pi_0(\partial M) \to
  \pi_0(M)$ is surjective).%
\end{definition}

Note that, for a proper sutured manifold $(M,\g)$, the orientations of~$s(\g)$ and~$M$
completely determine the orientation on $R(\gamma)$, and hence the
decomposition $R(\g) = R_+(\g) \cup R_-(\g)$.

\begin{convention}
  In this paper, we will assume that all sutured manifolds are proper,
  in addition to not having any toroidal sutures.
\end{convention}

To see the connection between sutured manifolds and closed
3-manifolds, observe that, if $(M,\g)$ is a sutured manifold such that
$\bdy M$ is a sphere with a single suture (dividing $\bdy M$ into two
disks), then the quotient of $M$ where $\bdy M$ is
identified with a point is a closed 3-manifold with a distinguished
basepoint given by the equivalence class of $\partial M$.  For the other direction,
we introduce the following definitions.

\begin{definition}\label{def:blowup}
  Suppose that $M$ is a smooth manifold, and let $L \subset M$ be a properly embedded
  submanifold. For each point $p \in L$, let
  $N_p L = T_p M/T_p L$ be the fiber of the
  normal bundle of~$L$ over $p$, and let $U N_p L =
  (N_p L \setminus\{0\})/\RR_+$ be the fiber of the unit normal bundle of $L$ over $p$.
  Then the \emph{(spherical) blowup} of~$M$ along~$L$, denoted by $\Bl_L(M)$, is
  a manifold with boundary obtained from~$M$ by replacing each
  point $p\in L$ by $U N_p(L)$.  There is a natural projection
  $\Bl_L(M) \to M$.  For further details, see Arone and Kankaanrinta~\cite{AK10:FunctorialityBlowUp}.%
\end{definition}

For instance, if $L \subset M$ is a submanifold of codimension~$1$,
then $\Bl_L(M)$ is the usual operation of cutting $M$ open along~$L$.

\begin{definition}\label{def:Y-p}
  Let $Y$ be a closed, connected, oriented 3-manifold, together with a
  basepoint~$p$ and an oriented tangent 2-plane~$V < T_pM$.
  Then $Y(p,V) = (M,\g)$ is the sutured manifold with $M = \Bl_p(Y)$ and
  suture $s(\g) = (V \setminus \{0\})/\RR_+$ in the resulting $S^2$ boundary
  component of $M$. See the left-hand side of Figure~\ref{fig:blowup}.
  We orient $s(\g)$ such that if we lift it to $V$, then the lift
  goes around the origin in the positive direction.
\end{definition}

\begin{figure}
\centering
\includegraphics{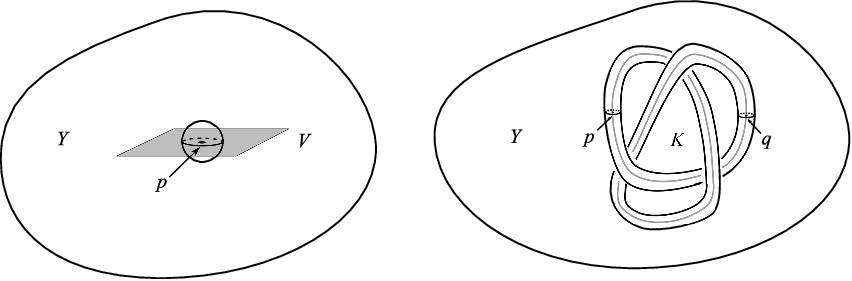}
\caption{The sutured manifolds $Y(p,V)$ and $Y(K,p,q)$.}
\label{fig:blowup}
\end{figure}

There is a similar construction for links, as well.

\begin{definition}\label{def:Y-K}
  Let $(Y,K,p,q)$ be an oriented knot with two basepoints.
  Then $Y(K,p,q) = (M,\g)$ is the sutured manifold
  with $M = \Bl_K(Y)$ and $s(\g) = UN_pK \cup UN_qK$, sitting inside the torus $\partial M$,
  as on the right-hand side of Figure~\ref{fig:blowup}.
  The orientation of $K$ induces an orientation of $NK$. We orient
  $UN_pK$ coherently with $N_pK$, while $UN_qK$ is oriented incoherently with $N_qK$.

  Similarly, if $(Y,L,\mathbf{p}, \mathbf{q})$ is a based oriented link with exactly one $p$ and one $q$
  basepoint on each component of $L$, then we define
  $Y(L,\mathbf{p}, \mathbf{q})$ to be the sutured manifold $(M,\g)$ with $M = \Bl_L(Y)$ and sutures
  obtained for each component of $L$ as above.
\end{definition}

\subsection{Sutured diagrams}
\label{sec:sutured-diagrams}
With these examples in mind, we turn to definitions for sutured
Heegaard diagrams. Since in this paper we need to be careful about
naturality of the constructions, we are careful in our definitions,
distinguishing, for instance, between attaching sets and isotopy
classes of attaching sets.

\begin{definition}
Let $\S$ be a compact oriented surface with boundary. An
\emph{attaching set in $\S$} is a one-dimensional smooth submanifold
$\deltas \subset \text{Int}(\S)$ such that each component of $\S \setminus \deltas$ contains at least one component
of $\partial \S$. We will denote the isotopy class of~$\deltas$ by~$[\deltas]$.
\end{definition}

\begin{definition} \label{def:scb}
The sutured manifold $(M,\g)$ is a \emph{sutured compression body} if either there is an attaching
set $\deltas \subset R_{+}(\g)$ such that if we compress
$R_{+}(\g)$ inside~$M$ along all the components of $\deltas$, we get
a surface that is isotopic to $R_{-}(\g)$ relative to $\g$,
or there is an attaching
set $\deltas \subset R_{-}(\g)$ such that if we compress
$R_{-}(\g)$ inside~$M$ along all the components of $\deltas$, we get
a surface that is isotopic to $R_{+}(\g)$ relative to $\g$. We call
$\deltas$ an \emph{attaching set} for $(M,\g)$.
\end{definition}

\begin{definition}
Given an attaching set $\deltas$ in $\S$, let $C(\deltas) = (M,\g)$ be the
sutured compression body obtained by taking $M$ to be $\S \times [0,1]$ and
attaching 3-dimensional 2-handles along $\deltas \times \{1\}$, while
$\g = \partial \S \times [0,1]$.
In addition, let $C_-(\deltas) = R_-(M,\g) = \S \times \{0\}$
and
\[
C_+(\deltas) = R_+(M,\g) = \partial C(\deltas) \setminus \text{Int} \left(C_-(\deltas) \cup  \g \right).
\]

If~$\deltas$ and~$\deltas'$ are both attaching sets in~$\S$, then we say they are
\emph{compression equivalent}, and we write
$\deltas \sim \deltas'$, if there is a diffeomorphism $d \colon C(\deltas) \to C(\deltas')$
such that $d|_{C_-(\deltas)}$ is the identity. This is an equivalence relation that
descends to isotopy classes of attaching sets.
So we will write $[\deltas] \sim [\deltas']$ if $\deltas \sim \deltas'$.
\end{definition}

Observe that $\chi(C_+(\deltas)) = \chi(C_-(\deltas)) + 2|\deltas|$. So $\deltas \sim \deltas'$ implies that
$|\deltas| = |\deltas'|$.

\begin{lemma} \label{lem:pi2}
Let $\deltas \subset \S$ be an attaching set in a compact oriented surface with boundary,
and let $C(\deltas) = (M,\g)$ be the corresponding sutured compression body. Then $\pi_2(M) = 0$.
\end{lemma}

\begin{proof}
Consider the Mayer-Vietoris sequence for the pair~$(\S \times I, H)$, where~$H$ is the union of the handles
attached to~$\S \times \{1\}$ along~$\deltas \times \{1\}$:
\[
0 = H_2(\S \times I) \oplus H_2(H) \to H_2(M) \to H_1((\S \times I) \cap H) \stackrel{i}{\to} H_1(\S \times I) \oplus H_1(H).
\]
Of course, $H_i(H) = 0$ for~$i \in \{1,2\}$, and $H_2(\S \times I) =
0$ as~$\S$ has no closed components.
Since $\deltas$ is an attaching set, the map $\pi_0(\partial \S) \to
\pi_0(\S \setminus \deltas)$ is surjective, so the components of~$\deltas$ are
linearly independent in~$H_1(\Sigma)$ and so the map~$i$ is injective.
It follows that~$H_2(M) = 0$. In particular, every smoothly embedded $2$-sphere~$S$ in~$M$ is null-homologous;
i.e., there is a compact submanifold-with-boundary~$N$ of~$M$ such that~$\partial N = S$.
If we attach $2$-handles to~$M$ along the components of~$\g$,
we obtain a compression body, which embeds into a handlebody, and hence also into~$\RR^3$.
In~$\RR^3$, the sphere~$S$ bounds a ball, hence~$N$ is diffeomorphic to~$D^3$, and~$S$ is
null-homotopic.
\end{proof}

\begin{definition}
  \label{def:Handleslide}
  Let $\delta_1$ and $\delta_2$ be two disjoint simple closed curves in $\S$, and fix an
  embedded arc~$a$ from $\delta_1$ to~$\delta_2$ whose interior is
  disjoint from $\delta_1 \cup \delta_2$ and from~$\partial \S$. Then
  a regular neighborhood of the graph $\delta_1\cup a \cup \delta_2$
  is a planar surface with three boundary components: one is isotopic
  to $\delta_1$, the other is isotopic to $\delta_2$, and the third is
  a new curve $\delta_1'$ that we call the curve obtained by
  \emph{handlesliding $\delta_1$ over $\delta_2$ along the arc $a$}, see Figure~\ref{fig:handleslide}.

  \begin{figure}
\centering
\includegraphics{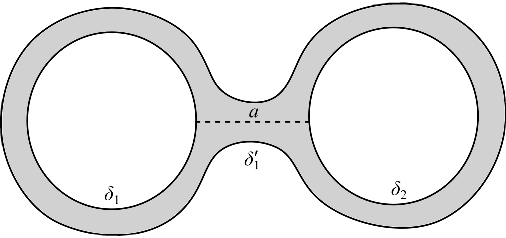}
\caption{Handlesliding the curve $\delta_1$ over $\delta_2$ along the arc $a$ gives $\delta_1'$.}
\label{fig:handleslide}
\end{figure}

  Suppose $\deltas$ and $\deltas'$ are two systems of attaching
  circles. We say that $\deltas$ and $\deltas'$ are \emph{related by a
  handleslide} if there are components $\delta_1$ and $\delta_2$ of $\deltas$ and
  a component $\delta_1'$ of $\deltas'$ such that $\delta_1'$ can be obtained by handle-sliding
  $\delta_1$ over $\delta_2$ along some arc whose interior is disjoint from~$\deltas$, and $\deltas' = (\deltas
  \setminus \delta_1) \cup \delta_1'$.
  If $D$ and $D'$ are isotopy classes of attaching sets, then they are related by a handleslide
  if they have representatives $\deltas$ and $\deltas'$, respectively, such that $\deltas$
  and $\deltas'$ are related by a handleslide.
\end{definition}

\begin{lemma} \label{lem:handleslide}
If $\deltas$ and $\deltas'$ are related by a handleslide, then
$\deltas \sim \deltas'$. Conversely, if $\deltas \sim \deltas'$, then
$[\deltas]$ and $[\deltas']$ are related by a sequence of handleslides.
\end{lemma}

\begin{proof}
  The first part is immediate.  For the second part,
  the proof of Bonahon~\cite[Proposition~B.1]{Bonahon83:CobordAutSurface}
  for ordinary compression bodies can be adapted to this context.
\end{proof}

\begin{definition}
A \emph{sutured diagram} is a triple $(\S, \alphas, \betas)$, where $\S$ is a compact oriented surface with boundary,
and $\alphas$ and $\betas$ are two attaching sets in~$\S$.
An \emph{isotopy diagram} is a triple $(\S, [\alphas], [\betas])$, where $(\S,\alphas,\betas)$ is a sutured diagram.
\end{definition}

\begin{definition}
Let $(M,\g)$ be a sutured manifold. Then we say that $(\S, \alphas,
\betas)$ is an \emph{(embedded) diagram of $(M,\g)$}
if
\begin{enumerate}
\item $\S \subset M$ is an oriented surface with $\partial \S =
  s(\g)$ as oriented 1-manifolds,%
\item the components of $\alphas$ bound disjoint disks to the negative side of $\S$, and the components of $\betas$
bound disjoint disks to the positive side of $\S$,
\item if we compress $\S$ along $\alphas$, we get a surface isotopic to
  $R_-(\g)$ relative to $\g$,
\item if we compress $\S$ along $\betas$, we get a surface isotopic to $R_+(\g)$ relative to $\g$.
\end{enumerate}
In other words, $\S$ cuts $(M,\g)$ into two sutured compression bodies, with attaching sets
$\alphas$ and $\betas$, respectively (see Definition~\ref{def:scb}).

Note that if $[\alphas'] = [\alphas]$ and $[\betas'] = [\betas]$, then $(\S, \alphas',\betas')$
is also a sutured diagram of $(M,\g)$. So we say that $(\S,A,B)$ is an
\emph{isotopy diagram of $(M,\g)$} if there is a sutured diagram $(\S,\alphas,\betas)$
of $(M,\g)$ such that $A = [\alphas]$ and $B = [\betas]$.
\end{definition}

\begin{lemma} \label{lem:existence}
Let $(M,\g)$ be a sutured manifold. Then there is a diagram of $(M,\g)$.
\end{lemma}

\begin{proof}
The proof of Juh\'asz \cite[Proposition 2.13]{Juhasz06:Sutured}
provides a sutured Heegaard diagram $(\S,\alphas,\betas)$ such
that $\S \subset M$.
\end{proof}

\begin{definition}
Let $(M,\g)$ be a sutured manifold. We say that the oriented surface $\S \subset M$ is a \emph{Heegaard surface of $(M,\g)$}
if $\partial \S = s(\g)$ and $\S$ divides $(M,\g)$ into two sutured compression bodies.
\end{definition}

\begin{definition}
A sutured diagram $(\S,\alphas,\betas)$ \emph{defines} a sutured manifold $(M,\g)$ as follows. To obtain $M$, take
$\S \times [-1,1]$ and attach 3-dimensional 2-handles to $\S \times \{-1\}$ along $\alphas \times \{-1\}$ and to
$\S \times \{1\}$ along $\betas \times \{1\}$. The annuli are taken to
be $\g = \partial \S \times [-1,1]$, with the sutures $s(\g) = \S \times
\{0\}$.
Then $(M,\g)$ is well-defined up to diffeomorphism relative to~$\S$.
(Note that, if we think of
$\Sigma$ as the middle level $\S \times \{0\} \subset M$, then
$(\S,\alphas,\betas)$ is a sutured diagram of~$M$.)

If $\alphas'$ is isotopic to $\alphas$ and $\betas'$ is isotopic to $\betas$, then the sutured manifold $(M',\g')$
defined by $(\S,\alphas',\betas')$ is diffeomorphic (relative to $\S$)
to the sutured manifold $(M,\g)$ defined by $(\S,\alphas,\betas)$.
So an isotopy diagram $H$ defines a diffeomorphism type of
sutured manifolds that we will denote by $S(H)$.
\end{definition}

\subsection{Moves on diagrams and weak Heegaard invariants}
\label{sec:moves-diagrams}

\begin{definition}
We say that the isotopy diagrams $(\S_1, A_1, B_1)$ and $(\S_2, A_2, B_2)$ are \emph{$\alpha$-equivalent}
if $\S_1 = \S_2$ and $B_1 = B_2$, while $A_1 \sim A_2$. Similarly, they are \emph{$\beta$-equivalent}
if $\S_1 = \S_2$ and $A_1 = A_2$, while $B_1 \sim B_2$.
\end{definition}

\begin{definition} \label{def:stab}
The sutured diagram $(\S_2, \alphas_2, \betas_2)$ is obtained from $(\S_1, \alphas_1, \betas_1)$ by a \emph{stabilization}
if
\begin{itemize}
\item there is a disk $D \subset \S_1$ and a punctured torus $T \subset \S_2$ such that $\S_1 \setminus D = \S_2 \setminus T$,
\item $\alphas_1  = \alphas_2 \cap (\S_2 \setminus T)$,
\item $\betas_1  = \betas_2 \cap (\S_2 \setminus T)$,
\item $\alphas_2 \cap T$ and $\betas_2 \cap T$ are simple closed curves that intersect each other transversely in a single point.
\end{itemize}
In this case, we also say that $(\S_1, \alphas_1, \betas_1)$ is obtained from $(\S_2, \alphas_2, \betas_2)$ by a \emph{destabilization}.
For an illustration, see Figure~\ref{fig:stabilization}.%

\begin{figure}
\centering
\includegraphics{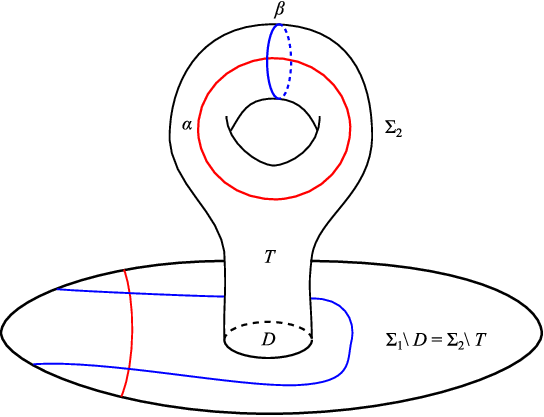}
\caption{The diagram $(\S_2,\alphas_2,\betas_2)$ is obtained from $(\S_1,\alphas_1,\betas_1)$ by a stabilization.}
\label{fig:stabilization}
\end{figure}

Let $H_1$ and $H_2$ be isotopy diagrams. Then $H_2$ is obtained from $H_1$ by a \emph{(de)sta\-bi\-li\-za\-tion} if
they have representatives $(\S_2, \alphas_2, \betas_2)$ and $(\S_1, \alphas_1, \betas_1)$, respectively, such that $(\S_2, \alphas_2, \betas_2)$ is obtained from $(\S_1, \alphas_1, \betas_1)$ by a (de)stabilization.
\end{definition}

If $d \colon \S \to \S'$ is a diffeomorphism of surfaces and $C$ is an isotopy class of attaching sets in $\S$,
then $d(C)$ is defined as $[d(\gammas)]$, where $\gammas$ is an arbitrary attaching set representing $C$.

\begin{definition}
Given isotopy diagrams $H_1 = (\S_1, A_1, B_1)$ and $H_2 = (\S_2, A_2, B_2)$, a diffeomorphism $d \colon H_1 \to H_2$
is an orientation preserving diffeomorphism $d \colon \S_1 \to \S_2$ such that $d(A_1) = A_2$ and $d(B_1) = B_2$.
\end{definition}

We now recall the notion of a graph, from a rather categorical viewpoint.

\begin{definition}
A \emph{graph} $G$ consists of
\begin{enumerate}
\item a class $|G|$ whose elements are called the objects (or
  vertices) of the graph,
\item for each pair $(A,B) \in |G| \times |G|$, a set $G(A,B)$ whose elements are called the morphisms (or arrows)
from $A$ to $B$.
\end{enumerate}
\end{definition}

\begin{definition}
A \emph{morphism of graphs} $F \colon G \to K$ between two graphs $G$ and $K$ consists of
\begin{enumerate}
\item a map $F \colon |G| \to |K|$,
\item for each pair $(A,B) \in |G| \times |G|$ of objects, a map
\[
F \colon G(A,B) \to K(F(A),F(B)).
\]
\end{enumerate}
\end{definition}

Notice that every category is a graph, and every functor between categories is a morphism of graphs.

\begin{definition} \label{def:big-graph}
Let $\G$ be the graph
whose class of vertices $|\G|$ consists of isotopy diagrams and, for
$H_1$,$H_2 \in |\G|$, the edges
$\G(H_1,H_2)$ can be written as a union of four sets
\[
\G(H_1,H_2) = \G_\a(H_1,H_2)
\cup \G_\b(H_1,H_2) \cup \G_\s(H_1, H_2) \cup \G_\d(H_1,H_2).
\]
The set $\G_\a(H_1,H_2)$
consists of a single arrow if $H_1$ and $H_2$ are $\a$-equivalent and
is empty otherwise. The set $\G_\b(H_1,H_2)$ is defined similarly using
$\b$-equivalence. The set $\G_\s(H_1,H_2)$ consists of a single arrow
if $H_2$ is obtained from $H_1$ by a stabilization or a
destabilization and is empty otherwise.
Finally, $\G_\d(H_1,H_2)$ consists of all diffeomorphisms from $H_1$ to $H_2$.
The graph $\G$ is thus the union of four subgraphs, namely $\G_\a$, $\G_\b$,
$\G_\s$, and $\G_\d$.%
\end{definition}

Note that the graphs $\G_\a$, $\G_\b$, and $\G_\d$ are in fact categories
when endowed with the obvious compositions. We have a version of the Reidemeister-Singer theorem.

\begin{proposition}\label{prop:diag-connected-weak}
The isotopy diagrams $H_1$, $H_2 \in |\G|$ can be connected by an oriented path if and only if they define diffeomorphic sutured
manifolds. Furthermore, the existence of an unoriented path from $H_1$ to $H_2$ implies the existence of an oriented one.
\end{proposition}

\begin{proof}
By Juh\'asz~\cite[Proposition 2.15]{Juhasz06:Sutured}, if two diagrams define diffeomorphic sutured manifolds,
then they become diffeomorphic after
a sequence of isotopies, handleslides, stabilizations, and
destabilizations. (Actually, the above mentioned result is stated for balanced diagrams,
but the same proof works for arbitrary ones.)
Lemma~\ref{lem:handleslide}
implies that every handleslide is an $\a$- or $\b$-equivalence, which concludes the proof of the first claim.

For the second claim, observe that if $* \in \{\, \a, \b, \s, \d \,\}$,
then $\G_*(H_1,H_2) \neq \emptyset$ if and only if $\G_*(H_2,H_1) \neq \emptyset$.
\end{proof}

\begin{definition} \label{def:weak-Heegaard}
Let $\cS$ be a set of diffeomorphism types of sutured manifolds, and let $\C$ be any category.
We denote by $\G(\cS)$ the full subgraph of $\G$ spanned by those
isotopy diagrams $H$ for which $S(H) \in \cS$.
A \emph{weak Heegaard invariant of $\cS$}
is a morphism of graphs
$F \colon \G(\cS) \to \C$ such that for every arrow $e$ of $\G(\cS)$ the image $F(e)$ is an isomorphism.
\end{definition}

Observe that, if $F \colon \G(\cS) \to \C$ is a weak Heegaard invariant
and $H_1$ and $H_2$ lie in the same path-component of $\G(\cS)$ (i.e., if
$S(H_1) = S(H_2)$), then $F(H_1)$ and $F(H_2)$ are isomorphic objects
of the category~$\C$.
In this language, we can state the main invariance results previously
known.  We first introduce some important sets
of diffeomorphism types of sutured manifolds.
We denote by $[(M,\g)]$ the diffeomorphism type of $(M,\g)$.

\begin{definition}\label{def:balanced}
  A \emph{balanced} sutured manifold is a proper sutured manifold $(M,\g)$ such that
  $\chi(R_+(\g)) = \chi(R_-(\g))$.  Equivalently, by Juh\'asz~\cite[Proposition 2.9]{Juhasz06:Sutured},
  it is a proper sutured manifold that has a diagram
  $(\Sigma, \alphas, \betas)$ with $|\alphas| = |\betas|$.
  Then define the following types of sutured manifolds.
\begin{enumerate}
\item Let $\cS_\man$ be the set of all $[Y(p,V)]$, where $Y$ is a closed,
  oriented, based 3-manifold, $p\in Y$, and $V < T_pM$ is an oriented tangent
  2-plane.
\item Let $\cS_\link$ be the set of all $[S^3(L,\bp,\bq)]$, where
  $(L,\bp,\bq)$ is a based oriented link in $S^3$ with exactly one $\bp$
  and one $\bq$ marking on each component of $L$.
\item Let $\cS_\bal$ be the set of all $[(M,\g)]$, where $(M,\g)$ is a
  balanced sutured manifold.
\end{enumerate}
\end{definition}

\begin{theorem}[\cite{OS04:HolomorphicDisks}] \label{thm:HF-weak}
 The morphisms
 \[
 \HFa \text{, }\HFm \text{, } \HFp \text{, }\HFinf \colon \G(\cS_\man) \to \Field[U]\text{-}\Mod
 \]
 are weak Heegaard invariants of $\cS_\man$ (where the $U$-action is trivial on $\HFa$).
\end{theorem}

\begin{theorem}[\cite{OS04:Knots,Rasmussen03:Knots,OS08:HFL}] \label{thm:HFL-weak}
The morphisms
\[
\HFLa \text{, } \HFLm \colon \G(\cS_\link) \to \Field[U]\text{-}\Mod
\]
are weak Heegaard invariants of $\cS_\link$.
\end{theorem}

\begin{theorem}[\cite{Juhasz06:Sutured}] \label{thm:SFH-weak}
The morphism
\[
\SFH \colon \G(\cS_\bal) \to \Field \text{-} \Vect
\]
is a weak Heegaard invariant of $\cS_\bal$.
\end{theorem}

However, these theorems are not enough to give invariants in the
stronger sense of Theorems~\ref{thm:invt-hf}, \ref{thm:invt-hfl}, and~\ref{thm:invt-sfh},
which assign to a based 3-manifold, a based link, or a balanced sutured manifold an
object of~$\Field[U]\text{-}\Mod$, rather than an isomorphism class of objects of~$\Field[U]\text{-}\Mod$.
For that, we look for further structure in the graph~$\G$.

\subsection{Strong Heegaard invariants}
\label{sec:strong-invar}

\begin{definition}\label{def:distinguished-rect}
Let $H_i = (\S_i,[\alphas_i],[\betas_i])$ be isotopy diagrams for $1 \le i \le 4$.
A \emph{distinguished rectangle} in $\G$ is a subgraph
\[
\xymatrix{H_1 \ar[r]^e \ar[d]^f & H_2 \ar[d]^g \\ H_3 \ar[r]^h & H_4}
\]
of $\G$ that satisfies one of the following properties:
\begin{enumerate}
\item\label{item:rect-alpha-beta}  Both $e$ and $h$ are $\a$-equivalences,
while $f$ and $g$ are $\b$-equivalences.
\item\label{item:rect-alpha-stab} Either both $e$ and $h$ are $\a$-equivalences or both $e$ and $h$ are $\b$-equivalences,
while $f$ and $g$ are both stabilizations.
\item\label{item:rect-alpha-diff}  Either both $e$ and $h$ are $\a$-equivalences or both $e$ and $h$ are $\b$-equivalences,
while $f$ and $g$ are both diffeomorphisms. In this case, we necessarily have
$\S_1 = \S_2$ and $\S_3 = \S_4$. We require, in addition, that the diffeomorphisms
$f \colon \S_1 \to \S_3$ and $g \colon \S_2 \to \S_4$ are the same.
\item\label{item:rect-stab-stab} The maps $e$, $f$, $g$, and~$h$ are all stabilizations, such that there are disjoint discs
$D_1$, $D_2 \subset \S_1$ and disjoint punctured tori $T_1$, $T_2 \subset \S_4$ satisfying
$\S_1 \setminus (D_1 \cup D_2) = \S_4 \setminus (T_1 \cup T_2)$, and such that
$\S_2 = (\S_1 \setminus D_1) \cup T_1$ and $\S_3 = (\S_1 \setminus D_2) \cup T_2$.
\item\label{item:rect-stab-diff} The maps $e$ and $h$ are stabilizations, while $f$ and $g$ are diffeomorphisms. Furthermore,
there are disks $D \subset \S_1$ and $D' \subset \S_3$ and punctured tori $T \subset \S_2$ and
$T' \subset \S_4$ such that $\S_1 \setminus D = \S_2 \setminus T$ and $\S_3 \setminus D' = \S_4 \setminus T'$,
and the diffeomorphisms $f$, $g$ satisfy $f(D) = D'$,  $g(T) = T'$, and $f|_{\S_1 \setminus D} = g|_{\S_2 \setminus T}$.
\end{enumerate}
\end{definition}

\begin{remark}\label{rem:distinguished-rect}
In case~\eqref{item:rect-alpha-beta}, we have $\S_i = \S_j$ for $i$, $j \in \{\,1, \dots,4\,\}$,
so a distinguished rectangle in this case is of the form
\[
\xymatrix{(\S,A,B) \ar[r] \ar[d] & (\S,A',B) \ar[d] \\ (\S,A,B') \ar[r] & (\S,A',B').}
\]

In case~\eqref{item:rect-alpha-stab}, we necessarily have $\S_1 = \S_2$ and $\S_3 = \S_4$. Without loss of generality, consider the situation
when both $e$ and $h$ are $\a$-equivalences. Then we have a rectangle
\[
\xymatrix{(\S,[\alphas],[\betas]) \ar[r] \ar[d] & (\S,[\overline{\alphas}],[\betas]) \ar[d] \\
(\S',[\alphas'],[\betas']) \ar[r] & (\S',[\overline{\alphas}'],[\betas'])}
\]
such that there is a disk $D \subset \S$ and a punctured torus $T \subset \S'$ with $\S \setminus D = \S' \setminus T$.
Furthermore, we can assume that $\alphas = \alphas' \cap (\S' \setminus T)$ and $\betas = \betas' \cap (\S' \setminus T)$,
while $\overline{\alphas} = \overline{\alphas}' \cap (\S' \setminus T)$. Since $\alphas' \sim \overline{\alphas}'$, the curves
$\alphas' \cap T$ and $\overline{\alphas}' \cap T$ are isotopic.

In case~\eqref{item:rect-stab-stab}, the fact that all four diagrams
contain
\[
S = \Sigma_1 \setminus (D_1 \cup D_2)
\]
implies that
$\alphas_i \cap S = \alphas_j \cap S$ and $\betas_i \cap S = \betas_j \cap S$
for every $i$, $j \in \{\, 1, \dots, 4 \,\}$.
\end{remark}

\begin{figure}
\centering
\includegraphics{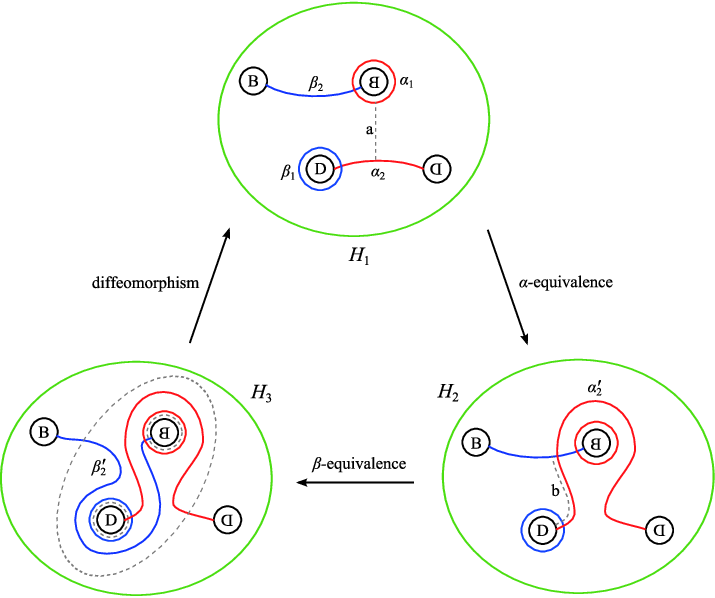}
\caption{A simple handleswap. The green curve is the boundary of the punctured genus two surface $P$ that is obtained
by identifying the circles marked with corresponding letters (namely, $B$ and $D$). We draw the $\a$ curves in red and the
$\b$ curves in blue.}
\label{fig:handleswap}
\end{figure}

\begin{definition} \label{def:simple-handleswap}
A \emph{simple handleswap} is a subgraph of $\G$ of the form
\[
\xymatrix{H_1 \ar[rd]^e & \\ H_3 \ar[u]^g & H_2 \ar[l]^f}
\]
such that
\begin{enumerate}
\item $H_i = (\S \# \S_0,[\alphas_i],[\betas_i])$ are isotopy diagrams for $i \in \{\,1,2,3\,\}$,
where $\S_0$ is a genus two surface,
\item $e$ is an $\a$-equivalence, $f$ is a $\b$-equivalence, and $g$
  is a diffeomorphism,
\item in the punctured genus two surface $P = (\S \# \S_0) \setminus \S$, the above triangle
is conjugate to the triangle in Figure~\ref{fig:handleswap}; i.e.,
there is a diffeomorphism that throws $P \cap H_i$ onto the pictures
in the green circles, sending the $\alpha$-circles in~$P$ to the two red circles,
and the $\beta$-circles in $P$ to the two blue circles,
\item \label{it:identical} in $\S$, the diagrams $H_1$, $H_2$, and $H_3$ are identical.
\end{enumerate}

So $P \cap \alphas_1$ consists of closed curves $\a_1$ and $\a_2$ and $P \cap \betas_1$ consists of closed curves
$\b_1$ and $\b_2$ such that $\a_i \cap \b_i$ consists of a single point for $i \in \{1,2\}$,
while both $\a_1 \cap \b_2$ and $\a_2 \cap \b_1$ are empty.
The arrow $e$ from $H_1$ to $H_2$ corresponds to handle-sliding $\a_1$ over $\a_2$ along the dashed arc $a$.
The arrow $f$ from $H_2$ to $H_3$ corresponds to handle-sliding $\b_2$ over $\b_1$ along the dashed arc $b$. Finally, the diffeomorphism $g$ maps
$H_3$ to $H_1$ by performing Dehn twists around the
dashed curves depicted in the lower left corner of
Figure~\ref{fig:handleswap}; namely a left-handed Dehn twist along
the large dashed curve and right-handed Dehn twists around the smaller ones.
\end{definition}

\begin{definition}\label{def:strong-Heegaard}
Let $\cS$ be a set of diffeomorphism types of sutured manifolds.
A \emph{strong Heegaard invariant of $\cS$} is a weak Heegaard
invariant $F \colon \G(\cS) \to \C$ that satisfies the following
axioms:
\begin{enumerate}
\item\label{item:strong-funct} \textbf{Functoriality}: The restriction
  of $F$ to $\G_\a(\cS)$,
  $\G_\b(\cS)$, and~$\G_\d(\cS)$ are functors to~$\C$. Furthermore, if
  $e \colon H_1 \to H_2$
  is a stabilization and $e' \colon H_2 \to H_1$ is the corresponding destabilization,
  then $F(e') = F(e)^{-1}$.
\item\label{item:strong-commute} \textbf{Commutativity}: For every
  distinguished rectangle
\[
\xymatrix{H_1 \ar[r]^e \ar[d]^f & H_2 \ar[d]^g \\ H_3 \ar[r]^h & H_4}
\]
in $\G(\cS)$, we have $F(g) \circ F(e) = F(h) \circ F(f)$.
\item\label{item:strong-cont} \textbf{Continuity}: If $H \in |\G(\cS)|$ and $e \in \G_\d(H,H)$ is a diffeomorphism isotopic to $\text{Id}_\S$, then $F(e) = \text{Id}_{F(H)}$.
\item\label{item:strong-handleswap} \textbf{Handleswap invariance}: For every simple handleswap
\[
\xymatrix{H_1 \ar[rd]^e & \\ H_3 \ar[u]^g & H_2 \ar[l]^f}
\]
in $\G(\cS)$, we have $F(g) \circ F(f) \circ F(e) = \text{Id}_{F(H_1)}$.
\end{enumerate}
\end{definition}

Note that in axiom~\eqref{item:strong-cont}, if $H = (\S,\alphas,\betas)$ and $e_t \colon \S \to \S$ for $t \in [0,1]$ is an
isotopy from $e$ to $\text{Id}_\S$, then $(\S,e_t(\alphas),e_t(\betas))$ represents the same isotopy diagram as $H$.
Hence $e_t \in \G_\d(H,H)$ for every $t \in [0,1]$.

Axiom~\eqref{item:strong-commute} of Definition~\ref{def:strong-Heegaard}
and the second part of axiom~\eqref{item:strong-funct}
imply commutativity for any distinguished rectangle involving destabilizations.
That is why we only considered stabilizations in Definition~\ref{def:distinguished-rect}.

\begin{theorem} \label{thm:strong}
The following are strong Heegaard invariants:
\begin{enumerate}
\item Sutured Floer homology, $\SFH$, is a strong Heegaard invariant
  of~$\cS_\bal$.
\item The Heegaard Floer homology invariants $\HFa$, $\HFp$, $\HFm$,
  and $\HFinf$ are strong
  Heegaard invariants of~$\cS_\man$.
\item The link Floer homology groups $\HFLa$ and $\HFLm$ are strong
  Heegaard invariants of~$\cS_\link$.
\end{enumerate}
\end{theorem}

We will prove Theorem~\ref{thm:strong} in Section~\ref{sec:HeegaardFloer}.

\subsection{Construction of the Heegaard Floer functors}
\label{sec:main-theorems}
We next explain how Theorem~\ref{thm:strong} lets us associate, for
instance, a group
$\SFH(M,\g)$ to a balanced sutured manifold $(M,\g)$.

\begin{definition} \label{def:isotopic}
  Suppose that $H_1$ and~$H_2$ are two isotopy diagrams of
  $(M,\gamma)$ with $H_i = (\S_i,A_i,B_i)$, and let $\iota_i \co
  \Sigma_i \to M$ be the inclusion for $i \in \{1,2\}$.  Then a
  diffeomorphism $d \co H_1 \to H_2$ is \emph{isotopic to the identity
    in~$M$} if $\iota_2 \circ d \colon \S_1 \to M$ is isotopic to
  $\iota_1 \colon \S_1 \to M$ relative to $s(\gamma)$.
\end{definition}

\begin{definition} \label{def:subgraph} Let $(M,\g)$ be a sutured
  manifold. Then $\G_{(M,\g)}$ is the subgraph of~$\G$ whose vertex
  set%
  \footnote{Observe that $|\G_{(M,\gamma)}|$ is a set, not a proper
    class, as we defined a sutured diagram for $(M,\gamma)$ to be a
    submanifold of~$M$.}  $|\G_{(M,\g)}|$ consists of all isotopy
  diagrams of $(M,\g)$.  The set of edges between $H_1$, $H_2 \in
  |\G_{(M,\g)}|$ is defined to be
  \[
  \G_{(M,\g)}(H_1,H_2) = \G_\a(H_1,H_2) \cup \G_\b(H_1,H_2) \cup
  \G_\s(H_1,H_2) \cup \G^0_\d(H_1,H_2),
  \]
  where $\G_\a$, $\G_\b$, and $\G_\s$ are as before, and
  $\G^0_\d(H_1,H_2)$ is the set of diffeomorphisms from $H_1$ to $H_2$
  isotopic to the identity in~$M$.
\end{definition}

We will prove the following stronger version of
Proposition~\ref{prop:diag-connected-weak} in
Section~\ref{sec:simplify-codim-1}.   

\begin{proposition}\label{prop:diag-connected-strong}
  Let $(M,\g)$ be sutured manifold. In the graph
  $\G_{(M,\g)}$, any two vertices can be connected by an oriented
  path.
\end{proposition}

\begin{definition} \label{def:iso}
  Given a weak Heegaard invariant $F \colon \G(\cS) \to \C$ and an
  oriented path $\eta$ in $\G(\cS)$ of the form
  \[
  H_0 \stackrel{e_1}{\longrightarrow} H_1
  \stackrel{e_2}{\longrightarrow} \dots
  \stackrel{e_n}{\longrightarrow} H_n,
  \]
  define $F(\eta)$ to be the isomorphism
  \[
  F(e_n) \circ \dots \circ F(e_1) \,\colon\, F(H_0) \to F(H_n).
  \]
\end{definition}

For a weak Heegaard invariant, the isomorphism $F(\eta)$ from $F(H_0)$
to $F(H_n)$ might depend on the choice of the path $\eta$. However,
according to the following theorem, this ambiguity disappears if we
require $F$ to be a strong Heegaard invariant and we restrict our
attention to the subgraph $\G_{(M,\g)}$.

\begin{theorem} \label{thm:iso} Let $\cS$ be a set of diffeomorphism
  types of sutured manifolds containing $[(M,\g)]$. Furthermore, let
  $F \colon \G(\cS) \to \C$ be a strong Heegaard invariant. Given
  isotopy diagrams $H$, $H' \in |\G_{(M,\g)}|$ and any two oriented
  paths $\eta$ and $\nu$ in $\G_{(M,\g)}$ from $H$ to $H'$, we
  have
  \[
  F(\eta) = F(\nu).
  \]
\end{theorem}

\begin{remark} \label{rem:iso} Another interpretation of
  Theorem~\ref{thm:iso} is that if we extend $\G_{(M,\gamma)}$ to a
  2-complex with 2-cells corresponding to the various polygons in
  Definition~\ref{def:strong-Heegaard}, the result is
  simply-connected.
\end{remark}

Theorem~\ref{thm:iso} is one of the most important and deepest results
of this paper.  We will prove it in Section~\ref{sec:proof}, and
develop the necessary technical tools in
Sections~\ref{sec:smooth}--\ref{sec:simplify}.

\begin{definition}
  Let $\cS$ be a set of diffeomorphism types of balanced sutured
  manifolds containing $[(M,\g)]$, and let $F \colon \G(\cS) \to \C$
  be a strong Heegaard invariant. If $H$ and $H'$ are isotopy diagrams
  of $(M,\g)$, then let
  \[
  F_{H,H'} = F(\eta),
  \]
  where $\eta$ is an arbitrary oriented path connecting $H$ to~$H'$.
  By Theorem~\ref{thm:iso}, the map $F_{H,H'}$ does not depend on the
  choice of~$\eta$.
\end{definition}

\begin{corollary} \label{cor:comp} Suppose that $H$, $H'$, $H'' \in
  |\G_{(M,\g)}|$. Then
  \[
  F_{H,H''} = F_{H',H''} \circ F_{H,H'}.
  \]
\end{corollary}

Motivated by Definition~\ref{def:transitive-system} and Remark~\ref{rem:colimit},
we obtain a natural invariant of sutured manifolds from a strong Heegaard
invariant as follows.  As usual, we denote the category of abelian
groups by $\Ab$.

\begin{definition}\label{def:strong-functor}
  Let $\cS$ be a set of diffeomorphism types of balanced sutured
  manifolds, and let $F \colon \G(\cS) \to \Ab$ be a strong Heegaard
  invariant. Fix a balanced sutured manifold $(M,\g)$ with $[(M,\g)]
  \in \cS$, and suppose that $H$ and $H'$ are isotopy diagrams of
  $(M,\g)$. We say that the elements $x \in F(H)$ and $y \in F(H')$
  are \emph{equivalent}, in short $x \sim y$, if $y = F_{H,H'}(x)$.
  By Theorem~\ref{thm:iso}, 
  this is an equivalence relation on the disjoint union $\coprod_{H
    \in |\G_{(M,\g)}|} F(H)$. The equivalence class of an element $x
  \in F(H)$ is denoted by $[x]$. Under the natural addition operation,
  these equivalence classes form an abelian group that we call
  $F(M,\g)$. Furthermore, let $I_H \colon F(H) \to F(M,\g)$ be the
  isomorphism that maps $x$ to $[x]$.

  If $\phi \colon (M,\g) \to (M',\g')$ is a diffeomorphism, then we
  define
  \[
  F(\phi) \colon F(M,\g) \to F(M',\g')
  \]
  as follows. Pick an isotopy diagram $H = (\S,A,B)$ of $(M,\g)$, and
  let $d = \phi|_\S$. Then $H' = d(H)$ is an isotopy diagram of
  $(M',\g')$, and $d$ is a diffeomorphism from~$H$ to~$H'$, so it
  induces a map $F(d) \colon F(H) \to F(H')$. We define the
  isomorphism $F(\phi)$ as $I_{H'} \circ F(d) \circ (I_H)^{-1}$. So we
  have a commutative diagram
  \[
  \xymatrix{F(H) \ar[r]^{F(d)} \ar[d]^{I_H} & F(H') \ar[d]^{I_{H'}} \\
    F(M,\g) \ar[r]^{F(\phi)} & F(M',\g').}
  \]
\end{definition}

\begin{proposition}\label{prop:F-phi-invt}
  In the above definition, the isomorphism $F(\phi)$ does not depend
  on the choice of isotopy diagram $H$ of $(M,\g)$.
\end{proposition}

\begin{proof}
  Suppose that $H_1 = (\S_1,A_1,B_1)$ and $H_2 = (\S_2,A_2,B_2)$ are
  isotopy diagrams of $(M,\g)$.  Let $d_1 = d|_{\S_1}$ and $d_2 =
  d|_{\S_2}$, and write $H_1' = d_1(H_1)$ and $H_2' = d_2(H_2)$.  Then
  we have to show that
  \[
  I_{H_1'} \circ F(d_1) \circ (I_{H_1})^{-1} = I_{H_2'} \circ F(d_2)
  \circ (I_{H_2})^{-1}.
  \]
  Since $(I_{H_2})^{-1} \circ I_{H_1} = F_{H_1,H_2}$ and
  $(I_{H_2'})^{-1} \circ I_{H_1'} = F_{H_1',H_2'}$, this amounts to
  proving that
  \begin{equation} \label{eqn:diff-canonical-commute} F_{H_1',H_2'}
    \circ F(d_1) = F(d_2) \circ F_{H_1,H_2}.
  \end{equation}
  Pick a path $\eta$ in $\G_{(M,\g)}$ of the form
  \[
  D_0 \stackrel{e_1}{\longrightarrow} D_1
  \stackrel{e_2}{\longrightarrow} \dots
  \stackrel{e_n}{\longrightarrow} D_n,
  \]
  such that $D_0 = H_1$ and $D_n = H_2$. There is a corresponding path
  $\eta'$ in $\G_{(M,\g)}$ from $H_1'$ to $H_2'$ of the form
  \[
  D_0' \stackrel{e_1'}{\longrightarrow} D_1'
  \stackrel{e_2'}{\longrightarrow} \dots
  \stackrel{e_n'}{\longrightarrow} D_n',
  \]
  obtained as follows. For every $i \in \{\, 1,\dots,n \,\}$, let
  $D_i' = \phi(D_i)$, and let $h_i \colon D_i \to D_i'$ be the
  restriction of $\phi$ to $D_i$.  If $e_i$ is an $\a$-equivalence,
  $\b$-equivalence, or stabilization, then we denote by $e_i'$ the
  corresponding $\a$-equivalence, $\b$-equivalence, or stabilization
  from $D_{i-1}'$ to $D_i'$. Furthermore, if $e_i$ is a diffeomorphism
  isotopic to the identity, then we take
  \[
  e_i' = h_i \circ e_i \circ h_{i-1}^{-1};
  \]
  this is also isotopic to the identity. Consider the following
  subgraph of $\G(\cS)$:
  \[
  \xymatrix{
    D_0 \ar[r]^{e_1} \ar[d]^{h_0 = d_1} & D_1 \ar[r]^{e_2} \ar[d]^{h_2} & \dots \ar[r]^{e_n} & D_n  \ar[d]^{h_n = d_2} \\
    D_0' \ar[r]^{e_1'} & D_1' \ar[r]^{e_2'} & \dots \ar[r]^{e_n'} &
    D_n'.  }
  \]
  By construction, each small square is either a distinguished
  rectangle, or a commuting rectangle of diffeomorphisms.  The functor
  $F$ commutes along the former by the Commutativity Axiom of strong
  Heegaard invariants, and along the latter by the Functoriality Axiom
  for $\G_\d(\cS)$. Hence $F$ also commutes along the large
  rectangle, giving equation~\eqref{eqn:diff-canonical-commute}.
\end{proof}

Let $\Sut_\bal$, $\Sut_\man$, and $\Sut_\link$ denote the full
subcategories of $\Sut$ whose objects have diffeomorphism types lying
in $\cS_\bal$, $\cS_\man$, and $\cS_\link$, respectively.

\begin{proof}[Proof of Theorem~\ref{thm:invt-sfh}]
  By Theorem~\ref{thm:strong}, the morphism $F = \SFH$ is a strong
  Heegaard invariant of $\cS_\bal$. Given isotopy diagrams $H$ and
  $H'$ of the balanced sutured manifold $(M,\g)$,
  Theorem~\ref{thm:iso} gives an isomorphism $F_{H,H'} \colon F(H) \to
  F(H')$.  These isomorphisms are canonical according to
  Corollary~\ref{cor:comp}.  Hence, the groups~$F(H)$ and the
  isomorphisms $F_{H,H'}$ form a transitive system, and we obtain the
  limit
  \[
  \SFH(M,\g) = F(M,\g)
  \]
  as in Definition~\ref{def:strong-functor}.  A diffeomorphism $\phi$
  between balanced sutured manifolds induces an isomorphism $F(\phi)$
  as in Definition~\ref{def:strong-functor}, and these are
  well-defined according to Proposition~\ref{prop:F-phi-invt}.  So we
  have all the ingredients for a functor $\SFH\co \Sut_\bal \to
  \Field\text{-}\Vect$.

  What remains to show is that isotopic diffeomorphisms induce
  identical maps on~$\SFH$, or equivalently, that for any
  diffeomorphism $\phi \colon (M,\g) \to (M,\g)$ isotopic
  to~$\text{Id}_{(M,\g)}$, we have $F(\phi) =
  \text{Id}_{\SFH(M,\g)}$. Let $H$ be an isotopy diagram of $(M,\g)$,
  and write $d = \phi|_H$ and $H' = \phi(H)$.  By definition,
  $F(\phi) = I_{H'} \circ F(d) \circ (I_H)^{-1}$.  So this is the
  identity if and only if
  \[
  F(d) = (I_{H'})^{-1} \circ I_H = F_{H,H'}.
  \]
  This is true since $d$ is isotopic to the identity, hence it
  corresponds to an edge of $\G_{(M,\g)}$ between $H$ and $H'$, and so
  if we take the path $\eta$ from $H$ to $H'$ to be the single
  edge~$d$, then $F(d) = F(\eta) = F_{H,H'}$.
\end{proof}

\begin{lemma} \label{lem:2-planes} Let $(Y,p)$ be a based 3-manifold,
  and let $V_0$ and $V_1$ be oriented $2$-planes in $T_pY$.  Suppose
  that $\phi$, $\psi \colon (Y,p) \to (Y,p)$ are diffeomorphisms such
  that $d\phi(V_0) = V_1$ and $d\psi(V_0) = V_1$ in an oriented sense.
  Suppose furthermore that both $\phi$ and $\psi$ are isotopic to $\text{Id}_Y$
  through diffeomorphisms fixing $p$.  Then~$\phi$ and~$\psi$ are
  isotopic to each other through diffeomorphisms fixing $p$ and mapping $V_0$ to
  $V_1$.
\end{lemma}

\begin{proof}
  This follows from the fact that the Grassmannian $M$ of oriented
  $2$-planes in~$T_pY$ is homeomorphic to $S^2$ and is hence
  simply-connected, together with an isotopy extension argument as
  follows.

  Let $\{\, \phi_t \colon t \in I \,\}$ and $\{\, \psi_t \colon t \in
  I \,\}$ be isotopies from $\text{Id}_Y$ to $\phi$ and $\psi$,
  respectively, through diffeomorphisms fixing $p$.  Since the
  Grassmannian $M$ is
  simply-connected, there is a $2$-parameter family of $2$-planes $V_{t,u}
  < T_pY$ for $(t,u) \in I \times I$ such that $V_{t,0} =
  d\phi_t(V_0)$ and $V_{t,1} = d\psi_t(V_0)$ for every $t \in I$,
  while $V_{0,u} = V_0$ and $V_{1,u} = V_1$ for every $u \in I$. The
  2-planes $V_{t,u}$ form a vector bundle $\nu$ over $I \times I$.
  Since $\nu$ is trivial, there is a family of isomorphisms $i_{t,u}
  \colon V_0 \to V_{t,u}$ for $(t,u) \in I \times I$ such that
  $i_{0,u} = \text{Id}_{V_0}$ for every $u \in I$, and $i_{t,0} =
  (d\phi_t)|_{V_0}$ and $i_{t,1} = (d\psi_t)|_{V_0}$ for every $t \in
  I$.  We can extend this to a $2$-parameter family of isomorphisms
  $j_{t,u} \colon T_pY \to T_pY$ such that $j_{t,u}|_{V_0} = i_{t,u}$
  for every $(t,u) \in I \times I$, while $j_{0,u} = \text{Id}_{T_pY}$
  for every $u \in I$, and $j_{t,0} = d\phi_t$ and $j_{t,1} = d\psi_t$
  for every~$t \in I$.  By Gromov's parametric $h$-principle for open manifolds~\cite[Theorem~7.2.3]{EM:h-principle},
  there is a neighborhood~$U$ of~$p$ and a family of diffeomorphisms $h_{t,u} \colon (U,p) \to
  (Y,p)$ such that $dh_{t,u} = j_{t,u}$ for every $(t,u) \in I \times
  I$, and $h_{t,0} = \phi_t|_U$ and $h_{t,1} = \psi_t|_U$ for every $t
  \in I$.  Using the relative isotopy extension theorem, we obtain a
  $2$-parameter family of diffeomorphism $g_{t,u} \colon (Y,p) \to
  (Y,p)$ such that $g_{0,u} = \text{Id}_Y$ for every $u \in I$, while
  $g_{t,0} = \phi_t$ and $g_{t,1} = \psi_t$ for every $t \in I$;
  furthermore, $g_{t,u}|_U = h_{t,u}$ for every $(t,u) \in I \times
  I$.  Then the family $\{\, g_{1,u} \colon u \in I \,\}$ provides the desired
  isotopy from $g_{1,0} = \phi$ to $g_{1,1} = \psi$.
\end{proof}

\begin{proof}[Proof of Theorem~\ref{thm:invt-hf}]
  Let $\HF$ be one of the four versions of Heegaard Floer homology,
  and let $\Manpv$ be the category of based 3-manifolds with a choice
  of oriented tangent 2-plane at the basepoint. A morphism from the
  object $(Y,p,V)$ to $(Y',p',V')$ is a diffeomorphism $\phi \colon
  (Y,p) \to (Y',p')$ such that $d\phi(V) = V'$ in an oriented sense.
  As in the proof of Theorem~\ref{thm:invt-sfh}, by Theorem~\ref{thm:strong},
  $\HF$ induces a functor $\HF_1 \co \Sut_\man \to \Ab$.  Composing with the
  functor $(Y,p,V) \mapsto Y(p,V)$ from Definition~\ref{def:Y-p} gives
  a functor $\HF_2\co\Manpv\to\Ab$. As in the proof of
  Theorem~\ref{thm:invt-sfh}, we obtain that isotopic morphisms induce
  identical maps, where we say that two morphisms from $(Y,p,V)$ to
  $(Y,p,V')$ are isotopic if they can be connected by a path of
  morphisms from $(Y,p,V)$ to $(Y,p,V')$.

  Each fiber of the forgetful functor $\Manpv \to \Manp$ is a sphere,
  which is simply-connected, so $\HF_2(Y,p,V)$ has no monodromy along
  any loop of oriented 2-planes in $T_pY$.  More precisely, fix a
  based manifold $(Y,p) \in \Manp$, and let $M$ be the Grassmannian of
  oriented 2-planes in $T_pY$.  Our goal is to construct a canonical
  isomorphism from $\HF_2(Y,p,V_0)$ to $\HF_2(Y,p,V_1)$ for any pair
  $(V_0, V_1) \in M \times M$.  Take an arbitrary morphism $\phi$ from
  $(Y,p,V_0)$ to $(Y,p,V_1)$, and such that $\phi$ is isotopic to
  $\text{Id}_Y$ through diffeomorphisms fixing~$p$.  Then we claim
  that the isomorphism
  \[
  \HF_2(\phi) \colon \HF_2(Y,p,V_0) \to \HF_2(Y,p,V_1)
  \]
  is independent of the choice of $\phi$. Indeed, by
  Lemma~\ref{lem:2-planes}, if $\psi$ is another choice, then $\phi$
  and $\psi$ are isotopic through diffeomorphisms fixing $p$ and
  mapping $V_0$ to $V_1$, and hence $\HF_2(\phi) = \HF_2(\psi)$. We
  denote this isomorphism by $i_{V_0,V_1}$.  So the groups
  $\HF_2(Y,p,V)$ for $V \in M$ and the isomorphisms $I_{V_0,V_1}$ for
  $(V_0,V_1) \in M \times M$ form a transitive system, and hence we
  can take the limit $\HF(Y,p)$. We have shown that~$\HF_2$ factors
  through a functor $\HF \co \Manp \to \Ab$.

  Since the forgetful morphism $\Field[U]\text{-}\Mod \to \Ab$ is faithful,
  for each of $\HFa$, $\HFm$, $\HFp$, and $\HFinf$,
  Theorem~\ref{thm:strong} gives a functor with target
  category $\Field[U]\text{-}\Mod$, as in the statement of the theorem.
\end{proof}

\begin{proof}[Proof of Theorem~\ref{thm:invt-hfl}]
  By Theorem~\ref{thm:strong}, the link Floer homology groups $\HFLa$ and $\HFLm$ are strong
  Heegaard invariants of~$\cS_\link$. Let $\HFL$ be one of $\HFLa$ and $\HFLm$.
  As in the proof of Theorem~\ref{thm:invt-sfh}, $\HFL$ induces a functor
  \[
  \HFL_1\co \Sut_\link \to \Ab
  \]
  for both versions of link Floer
  homology. Composing with the map
  \[
  (S^3,L,\bp,\bq) \mapsto S^3(L,\bp,\bq)
  \]
  introduced in Definition~\ref{def:Y-K} gives a
  functor $\HFL_2 \co\Linkpp \to \Ab$, where~$\Linkpp$ is the category
  of oriented links with two (distinguished) basepoints on each component. The fiber of
  the forgetful map $\Linkpp \to \Linkp$ over a based link $(L,\bp)$
  is homeomorphic to $\RR^{|L|}$ and hence contractible,
  so -- as in the proof of Theorem~\ref{thm:invt-hf} -- the morphism $\HFL_2$
  factors through a functor $\HFL \co \Linkp \to \Ab$.
  Again, the invariant takes values in a somewhat richer category than
  $\Ab$.
\end{proof}

Finally, we indicate how to obtain $\SpinC$-refined versions of the
above results. Let~$F$ be a strong Heegaard invariant defined
on a set~$\cS$ of diffeomorphism types of balanced sutured manifolds.
Fix a sutured manifold $(M,\g)$ such that $[(M,\g)] \in \cS$.
Suppose that, for every isotopy diagram $H$ of $(M,\g)$ and every
$\spinc \in \SpinC(M,\g)$, we are given
an abelian group $F(H,\spinc)$ such that
\[
F(H) = \bigoplus_{\spinc \in \SpinC(M,\g)} F(H,\spinc).
\]
In addition, we assume that if $e$ is an edge of $\G_{(M,\g)}$ from
$H$ to $H'$, then $F(e)|_{F(H,\spinc)}$ is an isomorphism between
$F(H,\spinc)$ and $F(H',\spinc)$. Then the limit $F(M,\g)$ will split
as a direct sum $\bigoplus_{\spinc \in \SpinC(M,\g)} F(M,\g,\spinc)$.
Relative homological gradings on the summands
$F(M,\g,\spinc)$ can be treated in a similar manner.

\begin{remark} \label{rem:difference}
Suppose that $\mathcal{H}$ and $\mathcal{H}'$ are admissible diagrams of the same $\SpinC$ 3-manifold~$(Y,\spinc)$.
Ozsv\'ath and Szab\'o~\cite[Theorem~2.1]{OS06:HolDiskFour} constructed an isomorphism $\Psi^\circ_\spinc \colon \HF^\circ(\mathcal{H}) \to \HF^\circ(\mathcal{H}')$ by composing maps associated to $\a$-equivalences, $\b$-equivalences, and stabilizations.
As the Reidemeister-Singer theorem only implies two diagrams become isotopic after a sequence of such moves, implicit in their construction is a non-unique diffeomorphism map (see~\cite[Lemma~2.10]{OS06:HolDiskFour}
and~\cite[Proposition~2.2]{OS04:HolomorphicDisks}), and hence $\Psi^\circ_\spinc$ is not well-defined.
The isomorphism we construct in Definition~\ref{def:iso} is different in that it involves a new move, namely
a diffeomorphism isotopic to the identity. As we shall see in Section~\ref{sec:HeegaarFloerInvariant}, the isomorphisms we associate
to $\a$-equivalences, $\b$-equivalences, and stabilizations agree with the isomorphisms defined by Ozsv\'ath and Szab\'o~\cite[Section~2.5]{OS06:HolDiskFour}, but they are defined in a more computable manner; see Remark~\ref{rem:equivalence}.
\end{remark}
\section{Examples} \label{sec:examples}

In this section, we give several examples that illustrate some of the
issues that arise when one tries to define Heegaard Floer homology in
a functorial manner.

\begin{example}
  \begin{figure}
    \centering
    \includegraphics{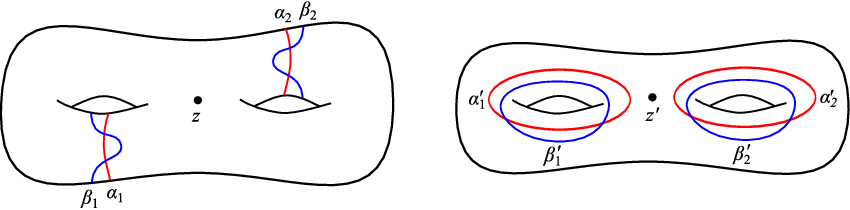}
    \caption{Two diffeomorphic diagrams, both defining manifolds
      diffeomorphic to $(S^1 \times S^2) \conn (S^1 \times S^2)$, for
      which different identifications induce different maps on
      $\HFa$.}
    \label{fig:example-abstract}
  \end{figure}
  This example shows why it does not suffice to work with abstract
  (i.e., non-embedded) Heegaard diagrams to obtain canonical
  isomorphisms, and hence a functorial invariant of 3-manifolds. See
  the diagrams
  \[
  \HD = \left(\S,\{\,\a_1,\a_2\,\},\{\,\b_1,\b_2\,\},z \right) \quad \text{and} \quad
  \HD' = \left( \S',\{\,\a_1',\a_2'\,\},\{\,\b_1',\b_2'\,\},z' \right)
  \]
  in Figure~\ref{fig:example-abstract}.  Both define sutured manifolds
  diffeomorphic to $(S^1 \times S^2) \conn (S^1 \times S^2)$ . The
  diagrams $\HD$ and $\HD'$ are clearly diffeomorphic. Choose a
  diffeomorphism $d \colon \HD \to \HD'$. Observe that there is an
  involution $f \colon \HD \to \HD$ such that $f(\a_1) = \a_2$, $f(\b_1) =
  \b_2$, and $f(z) = z$, obtained by $\pi$ rotation about the axis perpendicular
  to the surface and passing through $z$. Then $d \circ
  f$ is also an identification between~$\HD$ and~$\HD'$.  However,
  the diffeomorphisms~$d$ and~$d \circ f$
  induce different isomorphisms between~$\HFa(\HD)$ and~$\HFa(\HD')$.
  Indeed, $\HFa(\HD) \cong (\ZZ_2)^4$, and $f_*$ swaps the
  two $\ZZ_2$ terms lying in the ``middle'' homological grading.  This
  is why in the graph~$\G_{(M,\g)}$ we only consider diagrams embedded
  in~$(M,\g)$, and edges corresponding only to diffeomorphisms
  isotopic to the identity in~$M$. Otherwise, Theorem~\ref{thm:iso}
  would not hold.
\end{example}

\begin{example}
  Consider the diagram $\HD = (\S,\alphas,\betas,z)$ of $S^1 \times S^2$
  shown in Figure~\ref{fig:example-ori}.  Here, $S^1 \times S^2$ is
  represented by the region bounded by the two concentric spheres with
  common center~$O$, and we identify the points of the outer and inner
  spheres that lie on a ray through~$O$. The Heegaard surface $\S$ is
  represented by the horizontal annulus; after gluing the outer and
  inner boundary circles we get a torus. There is a single $\a$-circle
  and a single $\b$-circle; they intersect in two points~$a$ and~$b$.
  In the diagram, the dashed line represents an axis~$A$ passing
  through the basepoint $z$. If we rotate $\S$ about~$A$ by an angle~$\pi
  t$ for some $t \in [0,1]$, we get an automorphism~$d_t$ of~$S^1
  \times S^2$.  Notice that $d_1(\S,\alphas,\betas,z) =
  (\S,\alphas,\betas,z)$ and $d_1(a) = b$ and $d_1(b) = a$;
  furthermore, $d_t$ fixes the basepoint $z$ for every $t \in
  [0,1]$. Since $\HFa(\S,\alphas,\betas,z)$ is generated by~$a$ and~$b$,
  it appears that~$\HFa$ has non-trivial monodromy around the
  loop of diagrams~$d_t(\HD)$. However, $d_1|_\S$ is orientation
  reversing. This shows that we need to consider oriented Heegaard
  surfaces in~$\G_{(M,\g)}$ to obtain naturality.

\begin{figure}
  \centering
  \includegraphics{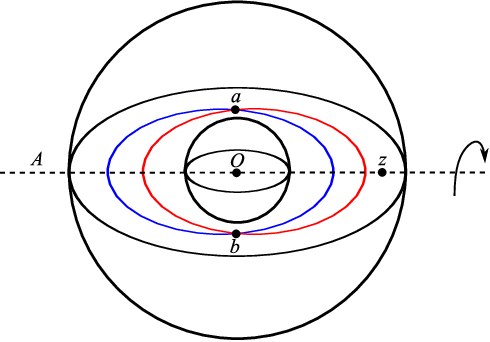}
  \caption{A diagram of $S^1 \times S^2$ for which an orientation
    reversing isotopy swaps the two generators of $\HFa$.}
  \label{fig:example-ori}
\end{figure}

\end{example}

\begin{example}
  Next, consider the diagram $\HD = (\S,\alphas,\betas)$ of the
  lens-space~$L(p,1)$ shown in
  Figure~\ref{fig:example-basepoint} for~$p =3$.  In particular, $\S$
  is the torus obtained by identifying the opposite edges of the
  rectangle $[0,1] \times [0,1]$, the curve~$\a$ is a line of slope~$0$
  and~$\b$ is a line of slope~$p$. Then~$\a \cap \b$ consists of~$p$ points
  $a_0, \dots, a_{p-1}$ that generate $\HFa(\HD)$. For $t \in
  [0,1]$, let $\HD_t$ be the diagram of $L(p,1)$ obtained by translating~$\b$
  horizontally by~$t/p$.  Then $\HD_0 = \HD_1$, so we obtain a loop
  of diagrams for~$L(p,1)$. Notice that $\HFa(\HD_t)$ has non-trivial
  monodromy, as it maps~$a_i$ to~$a_{i+1}$ for $0 \le i \le p-1$, where
  $a_p = a_0$. Non-trivial monodromy makes it impossible to assign a
  Heegaard Floer group to~$L(p,1)$ independent of the choice of diagram.
  This example is ruled out by requiring that there is at
  least one basepoint, and isotopies of the~$\alphas$ and~$\betas$
  curves cannot pass through the basepoints. Note that a choice of basepoint
  is necessary to assign a $\text{Spin}^c$ structure to each generator.

\begin{figure}
  \centering
  \includegraphics{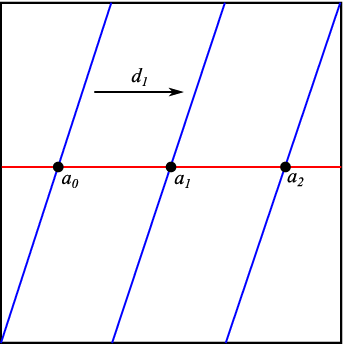}
  \caption{A diagram of $L(3,1)$ with no basepoint, together with an
    isotopy that permutes the 3 generators of $\HFa$.}
  \label{fig:example-basepoint}
\end{figure}

\end{example}

\begin{example}
  Here, we also consider a genus one Heegaard diagram $\HD =
  (\S,\alphas,\betas,\boldsymbol{z})$ of~$S^1 \times S^2$, but with
  two basepoints $\boldsymbol{z} = \{z_1,z_2\}$, see
  Figure~\ref{fig:example-bp-swap}. Heegaard Floer homology for
  multi-pointed Heegaard diagrams was introduced in~\cite[Section~4]{OS08:HFL}.
  Again, we draw the diagram on
  $[0,1] \times [0,1]$. We have two $\a$-curves: $\a_1 =
  \{1/4\} \times [0,1]$ and $\a_2 = \{3/4\} \times
  [0,1]$. The curve $\b_i$ is a small Hamiltonian translate of
  $\a_i$ such that $\a_1 \cap \b_1$ consists of two points that we
  label $a$, $b$, and $\a_2 \cap \b_2$ consists of two points $x$, $y$. We
  also arrange that $\b_2$ is a translate of $\b_1$ by the vector
  $(1/2,0)$. We choose two basepoints, namely $z_1 = (0,1/2)$ and $z_2
  = (1/2,1/2)$. For $t \in [0,1]$, let $d_t$ be the diffeomorphism of
  $\S$ given by $d_t(u,v) = (u+t/2,v)$. (This extends to~$S^1 \times
  S^2$ as rotation by $\pi t$ in the $S^1$-direction.)  Let $\HD_t =
  d_t(\HD)$ for $t \in [0,1]$, then $\HD_1$ = $\HD_0$, so we have a loop of
  doubly pointed diagrams of $S^1 \times S^2$.  Notice that
  $\HFa(\S,\alphas,\betas,z_1,z_2)$ is generated by the pairs
  $\{a,x\}$, $\{a,y\}$, $\{b,x\}$, and $\{b,y\}$.  The diffeomorphism
  $d_1$ swaps the generators $\{a,y\}$ and $\{b,x\}$, and swaps the
  basepoints $z_1$ and $z_2$.  Hence, to have naturality for $\HFa$,
  we need to work with based 3-manifolds and based diffeomorphisms.
  However, if $\spinc_0$ denotes the torsion $\SpinC$-structure
  on~$S^1 \times S^2$, a straightforward computation shows that
  \[
  \HFm(\HD,\spinc_0) \cong \ZZ[U_1, U_2]/(U_1 - U_2) \langle\, \{a,x\},
  \{a,y\} + \{b,x\} \,\rangle
  \]
  as a $\ZZ[U_1,U_2]$-module, and $d_1$ acts trivially on it.
  Compare this with our discussion in the introduction that
  the basepoint moving map can be non-trivial on~$\HFa$ but
  is trivial on~$\HFm$.

\begin{figure}
  \centering
  \includegraphics{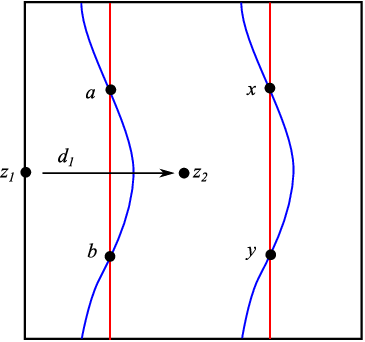}
  \caption{A doubly pointed diagram of $S^1\times S^2$. A
    $\pi$-rotation in the $S^1$-direction swaps the basepoints and
    induces a non-trivial automorphism of $\HFa$, while being trivial
    on $\HFp$ and $\HFm$ in the torsion $\SpinC$-structure.}
  \label{fig:example-bp-swap}
\end{figure}

\end{example}

\begin{example}
  Even if we isotope the $\a$- and $\b$-curves in the complement of
  the basepoint, one might obtain a loop of diagrams such that
  $\CFa$ has non-trivial holonomy around it. However, as we
  shall prove, there is no holonomy if we pass to homology.

  We describe a diagram $\HD$ of $S^1 \times S^2$ as follows;
  see Figure~\ref{fig:example-twist}.  Let $\S$ be the torus
  represented by $[0,1] \times [0,1]$, take $\a$ to be $\{1/2\} \times
  [0,1]$, and let $\b$ be a Hamiltonian translate of $\a$ such that
  $\a \cap \b$ consists of four points $a_0,\dots,a_3$, and $\b$ is
  invariant under translation by $(0,1/2)$. Let $d_t$ be translation
  by $(0, t/2)$ for $t \in [0,1]$, and let $\HD_t = (\S,\a,
  d_t(\b),z)$. By construction, $\HD_0 = \HD_1$. Since $d_1(a_i) =
  a_{i+2}$ (where $i+2$ is to be considered modulo 4), we see that
  $d_1$ acts non-trivially on $\CFa(\HD)$. However, as~$\HFa(\HD)$
  is generated by $a_0 + a_2$ and $a_1 + a_3$, the induced action on
  homology is trivial.

  More generally, suppose that $(\S,\alphas,\betas,z)$ is a Heegaard
  diagram, $\a \in \alphas$ and $\b \in \betas$.  Furthermore, suppose
  that there is a regular neighborhood $N \approx \a \times [-1,1]$ of
  $\a$ such that $\b \subset N$ and $\b$ is transverse to the fibers
  $\{p\} \times [-1,1]$ for every $p \in \a$. Then we can apply a
  ``finger move'' inside~$N$ that is the identity outside~$N$ and
  preserves~$\a \cup \b$
  setwise, and hence permutes the points of $\a \cap \b$.
  Even though this isotopy acts non-trivially on the chain level,
  it is trivial on the homology level.

\begin{figure}
  \centering
  \includegraphics{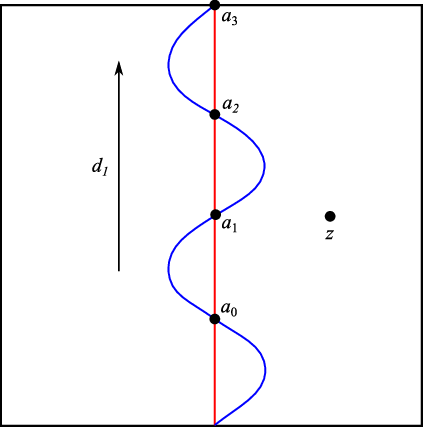}
  \caption{A Heegaard diagram of $S^1 \times S^2$. If we translate the
    $\b$-curve in the vertical direction by $1/2$ we get a non-trivial
    automorphism of the chain complex that is trivial on the
    homology.}
  \label{fig:example-twist}
\end{figure}

\end{example}
\section{Singularities of smooth functions}
\label{sec:smooth}

In this section, we recall some classical results about singularities
of smooth real valued functions following Arnold, Goryunov, Lyashko, and Vasil'ev~\cite{AGLV98:sing}.
The reader familiar with singularity theory can safely skip to Section~\ref{sec:gradients}.
This part is the beginning of the proof of Theorem~\ref{thm:iso} on strong
Heegaard invariants that culminates in Section~\ref{sec:proof}.
The reader interested in the proof of Theorem~\ref{thm:invt-hf},
the application of Theorem~\ref{thm:iso} to Heegaard Floer homology,
should skip to Section~\ref{sec:HeegaardFloer}.

\begin{definition}
  Let $f$ be a smooth function on the manifold $M$. A point $p \in M$
  is said to be a \emph{critical point} of $f$ if $df_p = 0$.
\end{definition}

\begin{definition}
  Let $\mathcal{E}_n$ be the set of germs at~$0$ of smooth functions $f
  \colon \R^n \to \R$.  Let
  $\mathcal{D}_n$ be the group of germs of diffeomorphisms $g \colon
  (\R^n,0) \to (\R^n,0)$.  The group $\mathcal{D}_n$ acts on
  $\mathcal{E}_n$ by the rule $g(f) = f \circ g^{-1}$.  Two
  function-germs $f$, $f' \in \mathcal{E}_n$ are said to be
  \emph{equivalent} if they lie in the same $\mathcal{D}_n$-orbit.

  The equivalence class of a function germ at a critical point is
  called a \emph{singularity}.  A \emph{class of singularities} is any
  subset of the space $\mathcal{E}_n$ that is invariant under the
  action of the group $\mathcal{D}_n$.
\end{definition}

\begin{definition}
  A critical point is said to be \emph{nondegenerate} (or a
  \emph{Morse critical point}) if the second differential (or Hessian)
  of the function at that point is a nondegenerate quadratic form. The
  class of non-degenerate critical points is called $A_1$.
\end{definition}

\begin{theorem}[Morse Lemma]
  In a neighborhood of a nondegenerate critical point $a \in M^n$ of
  the function $f \colon M^n \to \R$, there exists a coordinate system
  in which $f$ has the form
  \[
  f(x) = -x_1^2 - \dots - x_k^2 + x_{k+1}^2 + \dots + x_n^2 + f(a).
  \]
\end{theorem}

In the above theorem, $k$ is called the \emph{index} of the
nondegenerate critical point~$a$, and will be denoted by $\I(a)$. If
two elements of $\mathcal{E}_n$ have nondegenerate critical points at
zero, then they are equivalent if and only if they have the same
value and the same index. More generally, for an arbitrary critical point, $\I(a)$ is
the index of the second differential of the function at $a$.

The most important characteristic of a class of singularities is its
codimension $c$ in the space $\mathcal{E}_n$ of function-germs.
A generic function has only nondegenerate critical points.
So, the codimension $c$ of a nondegenerate critical point is zero,
while any class of degenerate critical points has positive codimension.
Degenerate critical points occur in an irremovable manner only in families of functions depending on
parameters. In a family of functions depending on $l$ parameters,
there may occur (in such a manner that it cannot be removed through a
small perturbation of the family) only a class of singularities for
which $c \le l$.

\begin{definition}
  A \emph{deformation} with base $\Lambda = \R^l$ of the germ $f \in
  \mathcal{E}_n$ is the germ at zero of a smooth map $F \colon \R^n
  \times \R^l \to \R$ such that $F(x,0) \equiv f(x)$.

  A deformation $F' \colon \R^n \times \R^l \to \R$ is \emph{equivalent} to $F$ if
  \[
  F'(x,\lambda) = F(g(x,\lambda),\lambda),
  \]
  where $g$ is the germ at zero of a smooth map $(\R^n \times \R^l,0) \to
  (\R^n,0)$ with $g(x,0) \equiv x$, representing a
  family of diffeomorphisms depending on $\lambda \in \R^l$.

  The deformation $F' \colon \R^n \times \R^{l'} \to \R$ is \emph{induced} from $F$ if
  \[
  F'(x,\lambda') = F(x, \theta(\lambda')),
  \]
  where $\theta \colon (\R^{l'},0) \to (\R^l,0)$ is a smooth germ of a
  mapping of the bases.
\end{definition}

\begin{definition} \label{def:versal}
  A deformation $F(x,\lambda)$ of the germ $f(x)$ is said to be
  \emph{versal} if every deformation $F' \colon \R^n \times \R^{l'} \to \R$ of $f(x)$ can be represented
  in the form
  \[
  F'(x,\lambda') = F(g(x,\lambda'),\theta(\lambda')),
  \]
  where $g$ is a germ at zero of an $l'$-parameter family of diffeomorphisms
  $\R^n \times \R^{l'} \to \R^n$ such that
  $g(x,0) \equiv x$, and $\theta \colon (\R^{l'},0) \to (\R^l,0)$
  is such that $\theta(0) = 0$; i.e., every deformation
  of $f(x)$ is equivalent to a deformation induced from $F$.

  A versal deformation for which the base $\Lambda$ has the smallest
  possible dimension is called \emph{miniversal}.
\end{definition}

\begin{proposition}{\cite[p.~18]{AGLV98:sing}} \label{prop:miniversal} Let $f(x) \in
  \mathcal{E}_n$ be a germ of a critical point. We denote by
  $I_{\nabla f}$ the ideal of $\mathcal{E}_n$ generated by all partial
  derivatives $f_i = \partial f / \partial x_i$ of $f$, and let $Q_f =
  \mathcal{E}_n / I_{\nabla f}$. If $\varphi_1, \dots, \varphi_l \in \mathcal{E}_n$
  project to a basis of $Q_f$, then
  \[
  F(x, \lambda) = f(x) + \sum_{j=1}^l \lambda_j \varphi_j
  \]
  is a miniversal deformation of the germ $f(x)$.
\end{proposition}

A versal deformation is unique in the following sense.

\begin{theorem}{\cite[p.~18]{AGLV98:sing}}
  Every $\ell$-parameter versal deformation of a germ $f$ is equivalent
  to a deformation induced from any other $\ell$-parameter versal
  deformation by a suitable diffeomorphism of their bases.
\end{theorem}

Let~$K$ be a subset of~$\mathcal{E}_n$ invariant under the action of~$\mathcal{D}_n$; i.e., a
class of singularities. Furthermore, let~$P \subset \mathcal{E}_n$
be the set of germs at~$0$ of all elements of the polynomial ring $\RR[x_1,\dots,x_n]$.
A \emph{normal form} for the class~$K$ is a map $\Phi \colon B \to P$
from a finite dimensional linear ``parameter space'' $B$ to the space of polynomial germs
satisfying three conditions:
\begin{itemize}
\item $\Phi(B)$ intersects all $\mathcal{D}_n$-orbits in~$K$,
\item the preimage in~$B$ of every $\mathcal{D}_n$-orbit is finite, and
\item $\Phi^{-1}(\mathcal{E}_n \setminus K)$ lies in some proper algebraic hypersurface in~$B$.
\end{itemize}
For more detail, see \cite[p.~22]{AGLV98:sing}.

\begin{theorem}{(\cite[p.~33]{AGLV98:sing} and \cite[p.~78]{Cerf})}
  In a generic 1-parameter family of smooth functions, the only
  degenerate critical points that appear have normal form
  \[
  f(x) = -x_1^2 - \dots - x_k^2 + x_{k+1}^2 + \dots + x_{n-1}^2 +
  x_n^3 + f(a).
  \]
  The class of such singularities is called~$A_2$.

  In addition, in 2-parameter families, singularities of the form
  \[
  f(x) = -x_1^2 - \dots - x_k^2 + x_{k+1}^2 + \dots + x_{n-1}^2 \pm
  x_n^4 + f(a)
  \]
  might also appear. The class of such singularities is called
  $A_3^{\pm}$.
\end{theorem}

As a corollary of Proposition~\ref{prop:miniversal}, a miniversal
deformation of a singularity of type~$A_1$ is given by
\[
F(x,\lambda) = -x_1^2 - \dots -x_k^2 + x_{k+1}^2 + \dots x_n^2 +
\lambda,
\]
where $\lambda \in \R$.

Miniversal deformations of type~$A_2$ singularities are given by the
formula
\[
F(x,\lambda) = -x_1^2 - \dots - x_k^2 + x_{k+1}^2 + \dots + x_{n-1}^2
+ x_n^3 + \lambda_1 x_n + \lambda_2,
\]
where the parameter $\lambda = (\lambda_1,\lambda_2) \in \R^2$. By Definition~\ref{def:versal},
every generic 1-parameter deformation of a type~$A_2$ singularity (with parameter~$t$) is
equivalent to one induced from this via a suitable map $t \mapsto (\lambda_1(t),\lambda_2(t))$,
and so has normal form
\[
-x_1^2 - \dots - x_k^2 + x_{k+1}^2 + \dots + x_{n-1}^2 + x_n^3 +
\lambda_1(t) x_n + \lambda_2(t).
\]
The concrete value of the constant term $\lambda_2(t)$ does not affect
the types of singularities appearing in the family, so from a
qualitative point of view, we can assume that $\lambda_2(t) \equiv
0$. Then, in this family, for $\lambda_1 < 0$ we have two
nondegenerate critical points of indices $k$ and $k+1$ that cancel
each other at $\lambda_1 = 0$, and the germs have no critical points
for $\lambda_1 > 0$. Hence, we will sometimes refer to a type~$A_2$
singularity of index~$k$ as an index $k$--$(k+1)$ birth-death
singularity (death if $\lambda_1(t)$ is decreasing, and birth if
$\lambda_1(t)$ is increasing).  In 2-parameter families, type~$A_2$
singularities appear along curves in the parameter space $\Lambda$ (in
the above normal form given by the formula $\lambda_1 = 0$).

Finally, type $A_3^{\pm}$ singularities have miniversal deformation
\[
F(x,\lambda) = -x_1^2 - \dots - x_k^2 + x_{k+1}^2 + \dots + x_{n-1}^2
\pm x_n^4 + \lambda_1 x_n^2 + \lambda_2 x_n + \lambda_3
\]
with parameter $\lambda \in \R^3$. These first appear in a
non-removable manner in 2-parameter families.
By versality, every generic 2-parameter deformation can be induced from this
by a generic map $\lambda \colon \RR^2 \to \RR^3$ of their bases.
To study the bifurcation set in $\Lambda = \R^2$ for a generic 2-parameter deformation $\Lambda
\to \mathcal{E}_n$, we again disregard the constant term, and consider
\[
-x_1^2 - \dots - x_k^2 + x_{k+1}^2 + \dots + x_{n-1}^2 \pm x_n^4 +
\lambda_1(t_1,t_2) x_n^2 + \lambda_2(t_1,t_2) x_n,
\]
where $(t_1,t_2) \mapsto (\lambda_1(t_1,t_2),\lambda_2(t_1,t_2))$ has non-vanishing
differential at the origin.
For generic values of~$(\lambda_1,\lambda_2)$, this may have
\begin{itemize}
\item three nondegenerate critical points, of indices $k$, $k+1$, and~$k$
  for~$A_3^+$ and $k+1$, $k$, and~$k+1$ for~$A_3^-$; or
\item one nondegenerate critical point, of index~$k$ for~$A_3^+$
  and~$k+1$ for~$A_3^-$.
\end{itemize}
When the polynomial $\pm 4x_n^3 + 2\lambda_1 x_n + \lambda_2$ has
multiple roots, the behaviour is different. The discriminant
is the cuspidal curve
\[
\Delta = \{\, \lambda \in \Lambda \,\colon\, 8\lambda_1^3 \pm
27\lambda_2^2 = 0\,\}.
\]
For $\lambda \in \Delta \setminus \{0\}$, the germ $F(x,\lambda)$ has
an~$A_1$ and an~$A_2$ singularity, while for $\lambda = 0$, it has
an~$A_3^{\pm}$ singularity. Sometimes, we will refer to an~$A_3^+$
singularity of index~$k$ as an index $k$--$(k+1)$--$k$\/
birth-death-birth, while an $A_3^-$ singularity of index~$k$ as an
index $(k+1)$--$k$--$(k+1)$ birth-death-birth. For the bifurcation
diagrams of the singularities~$A_3^+$ and~$A_3^-$, see Figure~\ref{fig:A3}.
Note that, in case of~$A_3^+$, for~$\lambda_1 < 0$ and~$\lambda_2 = 0$,
the values of two critical points coincide, which is a type of bifurcation
that we disregard in this paper as we will only work with the
gradient flows of functions and not with their values.
There is an analogous bifurcation for~$A_3^-$ singularities
in case~$\lambda_1 > 0$ and~$\lambda_2 = 0$.

\begin{figure}
  \centering
  \includegraphics{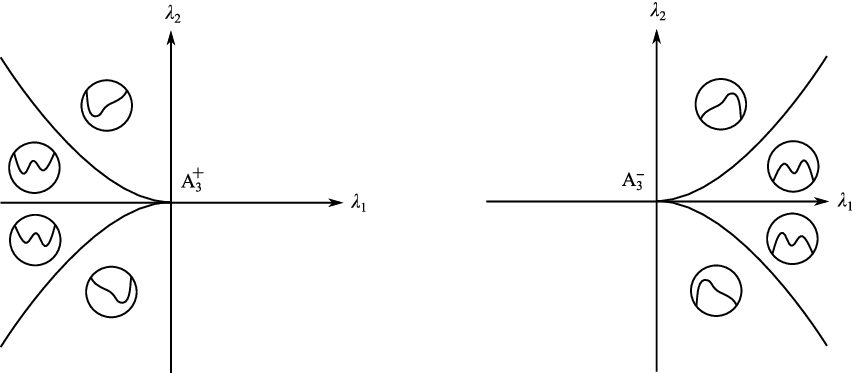}
  \caption{Bifurcation diagrams of the singularities~$A_3^+$ and~$A_3^-$ for~$n=1$.}
  \label{fig:A3}
\end{figure}

We now apply the above discussion to global 1- and 2-parameter
families of smooth real valued functions on manifolds; see Cerf~\cite[p.~78]{Cerf}.
For a generic 1-parameter family of smooth functions $\{\, f_\lambda \co \lambda \in
\Lambda \,\}$, there is a discrete subset $D \subset \Lambda$ such that
for every $\lambda \in \Lambda \setminus D$ the function $f_{\lambda}$
has only nondegenerate critical points, while for $\lambda \in D$ it
has a single degenerate critical point of type $A_2$, where two
nondegenerate critical points of neighboring indices collide.

For a generic 2-parameter family $\{\, f_\lambda \co \lambda \in \Lambda \,\}$,
there is a subset $D \subset \Lambda$ such that $f_{\lambda}$ has only
nondegenerate critical points for $\lambda \in \Lambda \setminus
D$. In addition, $D$ is a union of embedded curves that have only cusp
singularities and intersect each other in transverse double points. At
a regular point $\lambda \in D$, the function $f_t$ has a single
degenerate critical point of type~$A_2$. If $\lambda$ is a double
point of $D$, then $f_t$ has two degenerate critical points of type~$A_2$.
Finally, at each cusp of $D$, the function $f_t$ has a single
degenerate critical point of type~$A_3^{\pm}$.
\section{Generic 1- and 2-parameter families of gradients}
\label{sec:gradients}

Next, we summarize the results of Palis and Takens~\cite{PT83} and
Vegter~\cite{Vegter85} on the classification of global bifurcations
that appear in generic 1- and 2-parameter families of gradient vector
fields on 3-manifolds. In Section~\ref{sec:transl-hd}, we will
translate the codimension-1 bifurcations to moves on Heegaard
diagrams, while codimension-2 bifurcations translate to loops of
Heegaard diagrams. Note that the bifurcation theory of gradients is
much richer than the corresponding theory for smooth functions, due to
the tangencies that can appear between invariant manifolds of singular
points.

\subsection{Invariant manifolds} \label{sec:invariant-manifold}

First, we review some classical definitions and results
from Anosov, Aranson, Arnold, Bronshtein, Grines, and Il’yashenko~\cite[Section~4]{dyn97}.

\begin{definition}
  An \emph{invariant manifold} of a vector field is a submanifold that
  is tangent to the vector field at each of its points.
\end{definition}

If $v$ is a smooth vector field on a manifold $M$ with a singularity
at $p$ (i.e., $v$ is zero at~$p$), then the linear part $L_pv$ of $v$
at $p$ is an endomorphism
of $T_pM$. In local coordinates $x = (x_1,\dots,x_n)$ around $p$,
the linear part of $v$ is $Ax$, where $A = \left.\frac{\partial
    v}{\partial x}\right|_{x = 0}$ and $\frac{\partial v}{\partial x}$
is the Jacobian matrix whose $(i,k)$ entry is $\frac{\partial
  v_i}{\partial x_k}$.

The space $T_pM$ can be written as a direct sum of three
$L_pv$-invariant subspaces, namely $T^s$, $T^u$, and $T^c$, such that
every eigenvalue of $L_pv|_{T^s}$ has negative real part, every
eigenvalue of $L_pv|_{T^u}$ has positive real part, and every
eigenvalue of $L_pv|_{T^c}$ has real part zero.  Indeed, $T^s$, $T^u$,
and $T^c$ are spanned by the generalized eigenvectors of $L_pv$
corresponding to eigenvalues with negative, positive, and zero real
parts, respectively.  Here, the superscripts $s$, $u$, and~$c$ correspond
to ``stable,'' ``unstable,'' and ``center.''

\begin{definition}
We say that $v$ has a
\emph{hyperbolic singularity} at $p$ if none of the eigenvalues of
$L_p v$ are purely imaginary; i.e., if $T^c = 0$.
\end{definition}

\begin{theorem}[Center manifold theorem {\cite[Subsection~4.2]{dyn97}}] \label{thm:center-manifold}
  Let $v$ be a $C^{r+1}$ vector field on $M$ with a singular point at
  $p$. Let $T^s$, $T^u$, and $T^c$ be the $L_pv$-invariant subspaces of $T_pM$
  defined above.

  Then the vector field~$v$ has invariant
  manifolds $\cW^s$, $\cW^u$, and $\cW^c$ of class $C^{r+1}$,
  $C^{r+1}$, and $C^r$, respectively, that go through $p$ and are
  tangent to $T^s$, $T^u$, and $T^c$, respectively, at $p$.  Integral curves of~$v$
  with initial point on $\cW^s$ (respectively, $\cW^u$) tend
  exponentially to~$p$ as $t \to + \infty$ (respectively, $t \to
  -\infty$).
\end{theorem}

Here, $\cW^s$ is called the (strong) stable manifold and $\cW^u$ the
(strong) unstable manifold of the singular point $p$. The behavior of
the integral curves on the center manifold $\cW^c$ is determined by the
nonlinear terms. If $v$ is $C^{\infty}$, then $\cW^s$ and $\cW^u$ can
be chosen to be $C^{\infty}$, whereas the center manifold is only
finitely smooth. In addition, while $\cW^s$ and $\cW^u$ are well-defined,
the choice of $\cW^c$ might not be unique.

We say that the flows~$\phi$ on~$M$ and $\psi$ on~$N$
are \emph{topologically equivalent} if there is a homeomorphism
$h \colon M \to N$ mapping orbits of $\psi$ to orbits of $\phi$
homeomorphically, and preserving the orientation of the orbits.

\begin{theorem}[Reduction Principle {\cite[Subsection~4.3]{dyn97}}] \label{thm:reduction-principle}
  Suppose that a $C^2$ vector field~$v$ has a
  singular point at~$p$. Let $T^s$, $T^u$, and $T^c$ be the invariant
  subspaces corresponding to the map~$L_pv$. Then, in a neighborhood
  of the singular point~$p$, the flow of the vector field~$v$ is
  topologically equivalent to the flow of the direct product of two vector fields: the
  restriction of~$v$ to the center manifold, and the
  ``standard saddle'' $(-a, b)$ for $a \in T^s$ and $b \in T^u$.
\end{theorem}

This theorem can be used to study both individual vector fields and
families of vector fields; by replacing a family $v(x,\varepsilon)$ on~$M$
with $(v(x,\varepsilon),0)$ on $M \times \RR^l$, where $x \in M$ and $\varepsilon \in \RR^l$.

The following discussion has been taken from Palis and Takens~\cite{PT83}.

\begin{definition}
  We say that a smooth vector field $v$ on $M$ has a
  \emph{saddle-node} at~$p$ (or a \emph{quasi-hyperbolic singularity of
  type~1}) if $\dim T^c = 1$ and $v|_{\cW^c}$ has the form $v = ax^2
  \frac{\partial}{\partial x} + O(|x|^3)$ with $a \neq 0$ for some
  center manifold $\cW^c$ through $p$, where $x$ is a local coordinate on $\cW^c$
  around~$p$.

  If $v^{\ol{\mu}}$, belonging to a one-parameter family $\{v^\mu\}$
  of vector fields, has a saddle-node at $p$, we say that it
  \emph{unfolds generically} if there is a center
  manifold for the family $\{v^\mu\}$ passing through $p$ (at $\mu
  =\ol{\mu}$) such that
  $v_\mu$, restricted to this center manifold, has the form
  \[
  v_\mu = \left(ax^2 + b(\mu - \ol{\mu})\right)
  \frac{\partial}{\partial x} + O\left(|x^3| + |x \cdot (\mu
    -\ol{\mu})| + |\mu - \ol{\mu})|^2\right),
  \]
  with $a$, $b \neq 0$.
\end{definition}

For example, if $f: M \to \RR$ has an $A_2$ singularity at~$p$, then
$\nabla f$ has a saddle-node at~$p$.

\begin{definition}
  A point $p \in M$ is called a \emph{quasi-hyperbolic singularity of
    type~2} of a vector field $v$ if $\dim T^c = 1$ and there is a
  center manifold $\cW^c$ of class $C^m$ such that on $\cW^c$, there is
  a local $C^m$-coordinate $x$ with $v|_{\cW^c} = x^3 \cdot v_1(x)
  \frac{\partial}{\partial x}$ with $v_1(0) \neq 0$.
\end{definition}

For example, if $f: M \to \RR$ has an $A_3^\pm$ singularity at~$p$,
then $\nabla f$ has a quasi-hyperbolic singularity of type~2 at~$p$.

\begin{definition} \label{def:stable-set} Let $p$ be a singular point
  of the vector field $v$ on $M$. Furthermore, let the maximal flow of
  $v$ be $\varphi \colon D \to M$, where $D \subset M \times \R$ is
  the flow domain.  Then the \emph{stable set} of $p$ is
  \[
  W^s(p) = \{\, x \in M \,\colon\, \lim_{t \to \infty} \varphi(x,t) =
  p \,\},
  \]
  and the \emph{unstable set} of $p$ is
  \[
  W^u(p) = \{\, x \in M \,\colon\, \lim_{t \to -\infty} \varphi(x,t) =
  p \,\}.
  \]
\end{definition}

If $p$ is a hyperbolic singular point of $v$, then both $W^s(p)$ and
$W^u(p)$ are injectively immersed submanifolds of $M$ with tangent
spaces $T_p W^s(p) = T^s$ and $T_p W^u(p) = T^u$, respectively.  So, in
this case, the stable and unstable sets $W^s(p)$ and $W^u(p)$ coincide
with the stable and unstable
manifolds $\cW^s$ and $\cW^u$ of the singular point $p$, respectively.

If $p$ is a saddle-node, $W^s(p)$ is an injectively immersed
submanifold with boundary. This boundary is the \emph{strong stable
  manifold} of $p$, and is denoted by $W^{ss}(p)$. Note that
$T_pW^{ss}(p) = T^s$. Similarly, $W^u(p)$ is an injectively immersed
submanifold with boundary the \emph{strong unstable manifold}
$W^{uu}(p)$, see Figure~\ref{fig:saddle-node}.  In the terminology of
Theorem~\ref{thm:center-manifold}, we have $W^{ss}(p) = \cW^s$ and
$W^{uu}(p) = \cW^u$.

\begin{figure}
  \centering
  \includegraphics{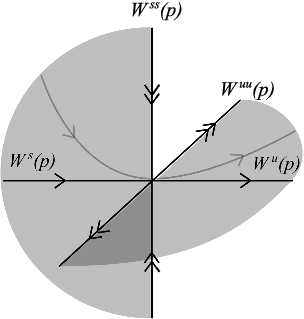}
  \caption{The stable and unstable sets of a saddle-node singularity
    are manifolds with boundary. This figure also illustrates the
    non-uniqueness of the center manifold; the black and the light grey curves
    with single arrows show two possible center manifolds.}
  \label{fig:saddle-node}
\end{figure}

\subsection{Bifurcations of gradient vector fields on 3-manifolds}
\label{sec:bifurcations}

Let~$M$ be compact 3-manifold. A gradient vector field $X =
\grad_g(f)$ on~$M$ is associated with a Riemannian metric~$g$ and a
smooth function $f \colon M \to \R$ by the relation $g(X,Y) = df(Y)$
for all smooth vector fields~$Y$ on~$M$.  Since~$f$ is strictly
increasing along regular orbits of~$X$, the vector field~$\grad_g(f)$
has no periodic orbits or other kinds of recurrence. The singular
points of~$X$ coincide with the critical points of~$f$. Since the
linear part~$L_pX$ of~$X$ at a singularity~$p$ is symmetric with
respect to~$g$, all eigenvalues of~$L_pX$ are real. (In suitable
coordinates~$L_pX$ is the Hessian matrix of~$f$ at the singular
point.)

\begin{definition} \label{defn:Morse-Smale} A gradient vector field
  $X$ on $M$ is \emph{Morse-Smale} if
  \begin{enumerate}
  \labitem{(H)}{cond:hyp}  all singular points of $X$ are
    hyperbolic, and
  \labitem{(T)}{cond:transv} all stable and unstable manifolds are
    transversal.
  \end{enumerate}
\end{definition}

The Morse-Smale vector fields constitute an open and dense subset
of the set~$X^g(M)$ of all gradient vector fields on a closed manifold~$M$.
In addition, a gradient vector field is structurally stable if and only
if it is Morse-Smale. Recall that a $C^1$ vector field~$v$ on~$M$ is \emph{structurally stable}
if the following holds: For any $C^0$ neighborhood~$U$ of $\id_M$ in $\text{Homeo}(M)$,
there is a $C^1$ neighborhood~$V$ of $v$
such that for every $v' \in V$ there is a homeomorphism $h \in U$
that maps the oriented trajectories of $v$ to the oriented trajectories of $v'$
(in particular, the flows of $v$ and $v'$ are topologically conjugate).

\begin{definition} \label{def:family-of-gradients}
  A \emph{$k$-parameter family of gradients} on a compact
  manifold $M$ is a family of pairs $(g^\mu,f^\mu)$ for $\mu \in
  \R^k$, where $\{g^\mu\}$ is a $k$-parameter family of Riemannian
  metrics and $\{f^\mu\}$ is a $k$-parameter family of smooth real-valued
  functions on $M$. Let $X^\mu = \grad_{g^\mu}(f^\mu)$ for $\mu \in \R^k$.
  By a slight abuse of notation, we will only write $\{X^\mu\}$
  to denote a $k$-parameter family of gradients.

  We assume that both $g^\mu$ and $f^\mu$, and hence $X^\mu$, depend
  smoothly on $(x,\mu) \in M \times \R^k$.  The set of such pairs is
  endowed with the strong Whitney topology; i.e., the topology of
  uniform convergence of $g^\mu$, $f^\mu$ and all their derivatives on
  compact sets. The resulting topological space of $k$-parameter
  families is denoted by $X^g_k(M)$.
\end{definition}

\begin{definition} \label{def:bifurcation} A parameter value $\ol{\mu}
  \in \R^k$ is called a \emph{bifurcation value} for the family~$\{X^\mu\}$
  in~$X^g_k(M)$ if~$X^{\ol{\mu}}$ fails to be a Morse-Smale system.
  Hence $X^{\ol{\mu}}$ has at least one \emph{orbit of
  tangency} between stable and unstable manifolds (i.e., an integral curve along which
  a stable and an unstable manifold are tangent),
  or at least one non-hyperbolic singular point.

  The \emph{bifurcation set} (or \emph{bifurcation diagram})
  of a family $\{X^\mu\}$ in $X^g_k(M)$ is
  the subset of $\R^k$ consisting of all bifurcation values of
  $\{X^\mu\}$.
\end{definition}

Note that, in dimension three, if a stable manifold and an unstable
manifold do not intersect transversely at a point~$x$, then they are
actually tangent at~$x$. Indeed, stable and unstable manifolds always
contain the flow direction, and any two linear subspaces of~$\RR^3$
containing a fixed vector are either transverse or tangent (in the
sense that one is contained inside the other).

\subsubsection{Generic 1-parameter families of gradients in dimension
  3} \label{sec:1-param}
We now turn to the results of Palis and
Takens~\cite{PT83}, following the notation of Vegter~\cite{Vegter85}.
We will translate the bifurcations occurring in generic 1-parameter families of
gradients to moves on Heegaard diagrams in Proposition~\ref{prop:1-param}.
For the following definition, see Vegter~\cite[p.~122]{Vegter85} and Palis--Takens~\cite[Section~2.a]{PT83}.

\begin{definition} \label{def:quasi-transversal} Let $v$ be a vector
  field on a 3-manifold~$M$. Two invariant submanifolds~$A$ and~$B$ of~$v$
  have a \emph{quasi-transversal tangency} if their intersection has a
  connected integral curve~$\g$ that is not a single point, and at some (and hence
  every) point $r \in \g$, the following two conditions hold:
  \begin{enumerate}[label=(QT-\arabic*)]
  \item \label{item:QT1} $\dim(T_r A + T_r B) = 2$; so we have three
    cases:
    \begin{enumerate}[ref=(QT-\arabic{enumi}\alph*)]
    \item \label{item:dim-2-2} $\dim A = 2$ and $\dim B = 2$,
    \item \label{item:dim-1-2} $\dim A = 1$ and $\dim B = 2$,
    \item \label{item:dim-2-1} $\dim A = 2$ and $\dim B = 1$.
    \end{enumerate}
  \item \label{item:QT2} In case~\ref{item:dim-2-2}, we impose the
    condition that the tangency between~$A$ and~$B$ is as generic as
    possible in the following sense. Let~$S$ be a smooth 2-dimensional
    cross-section for $v$, containing~$r$. Take coordinates $(x_1, x_2)$
    on a neighborhood of~$r$ in~$S$, in which $A \cap S = \{x_2 =
    0\}$, while $B \cap S$ is of the form $\{x_2 = F(x_1)\}$ for some
    smooth function $F$. Condition~\ref{item:QT1} amounts to $F(0) =
    0$ and $\frac{dF}{dx_1}(0) = 0$. In addition, we require that
    \[
    \frac{d^2F}{dx_1^2}(0) \neq 0.
    \]
  \end{enumerate}
\end{definition}

By Palis and Takens~\cite[Main Theorem]{PT83} and Vegter~\cite[p.~124]{Vegter85},
for an open and dense set of 1-parameter families of gradients, it is
easy to describe the bifurcation diagram: It consists of isolated
points in the parameter space~$\R$ at which one of the conditions
appearing in the characterization of Morse-Smale gradients is violated
``in the mildest possible manner.''  For a generic family $\{X^\mu\}
\in X_1^g(M)$, at each bifurcation value $\ol{\mu} \in \R$ exactly one
of the following two possibilities holds:

\begin{enumerate}
\labitem{(NH)}{item:failure-H} Failure of condition~\ref{cond:hyp}. The vector
  field $X^{\ol{\mu}}$ has exactly one generically unfolding
  saddle-node~$p$, while all other singular points are hyperbolic, and all
  stable and unstable manifolds are transversal. By convention (here
  and later), at saddle-nodes, the set of stable and unstable
  manifolds required to be transverse includes
  $W^{ss}(p)$ and $W^{uu}(p)$ as well as $W^s(p)$ and $W^u(p)$.
\labitem{(NT)}{item:failure-T} Failure of condition~\ref{cond:transv}. All singular
  points of $X^{\ol{\mu}}$ are hyperbolic, and there is a single
  non-transversal orbit of intersection~$\g$ of the unstable manifold
  $W^u(p_1^{\ol{\mu}})$ of $p_1^{\ol{\mu}}$ and the stable manifold
  $W^s(p_2^{\ol{\mu}})$ of $p_2^{\ol{\mu}}$ that is quasi-transversal
  and satisfies the additional non-degeneracy conditions below.

\begin{enumerate}[label=(ND-\arabic*)]
\item \label{item:unfolding} This condition expresses the ``crossing
  at non-zero speed'' of $W^u(p_1^{\mu})$ and $W^s(p_2^{\mu})$ as the
  parameter passes the value $\om$, where $p_1^{\mu}$ and $p_2^{\mu}$
  are the saddle points of $X^\mu$ near $p_1^{\om}$ and $p_2^{\om}$,
  respectively. For this, we choose paths $\sigma^u$, $\sigma^s \colon \R \to
  M$ with $\sigma^u(\mu) \in W^u(p_1^{\mu})$, $\sigma^s(\mu) \in
  W^s(p_2^{\mu})$, and $\sigma^u(\om) = \sigma^s(\om) = r \in \g$. We
  require that
  \[
  \dot{\sigma}^u(\om) - \dot{\sigma}^s(\om) \not\in
    T_rW^u(p_1^{\om}) + T_rW^s(p_2^{\om}).
  \]
\end{enumerate}

When $\dim W^u(p_1^{\om}) = 1$ and $\dim W^s(p_2^{\om}) = 2$, we
impose the following additional conditions.  The case when $\dim
W^u(p_1^{\om}) = 2$ and $\dim W^s(p_2^{\om}) = 1$ have the same
conditions, but with the sign of $X^\mu$ reversed.
\begin{enumerate}[resume,label=(ND-\arabic*)]
\item \label{item:eigenvalues} The contracting (i.e., negative) eigenvalues of the
  linear part of $X^{\om}$ at $p_1^{\om}$ are distinct
  (for gradients only real eigenvalues occur).  This
  implies that there is a unique 1-dimensional invariant manifold
  $W^{ss}(p_1^{\om}) \subset W^s(p_1^{\om})$ such that
  $T_{p_1^{\om}}W^{ss}(p_1^{\om})$ is the eigenspace of
  $L_{p_1^{\om}}X^{\om}$ corresponding to the strongest contracting (i.e., smallest negative)
  eigenvalue.  We call $W^{ss}(p_1^{\om})$ the strong stable manifold
  of $p_1^{\om}$.
\item \label{item:limit} For some $r \in \g$, let $E_r \subset
  T_rW^s(p_2^{\om})$ be a 1-dimensional subspace complementary to
  $X^{\om}(r)$. Let $\phi_t^{\om}$ for $t \in \R$ be the flow of
  $X^{\om}$. Then we require that
  \[
  \lim_{t \to -\infty} (d\phi_t^{\om})_r(E_r) = T_{p_1^{\om}}
  W^{ss}(p_1^{\om}).
  \]
\item \label{item:transversality} The stable and unstable manifolds of
  any singularity $p^* \not\in \left\{\, p_1^{\om}, p_2^{\om} \,\right\}$ are transversal to
  $W^{ss}(p_1^{\om})$.
\end{enumerate}
\end{enumerate}

The possibilities occurring in cases~\ref{item:failure-H}
and~\ref{item:failure-T} are shown schematically in
Figure~\ref{fig:codim1}. Note that, in case~\ref{item:failure-T}, we
have $\I(p_1^{\om})$, $\I(p_2^{\om}) \in \{1,2\}$ and $\I(p_1^{\om})
\le \I(p_2^{\om})$ by condition~\ref{item:QT1} of
quasi-transversality.

The above conditions are labeled (i)--(vii) in Vegter~\cite[pp.~121--123]{Vegter85};
see also conditions (i)--(v) in Palis and Takens~\cite[pp.~384--385]{PT83}.

\begin{figure}
  \centering
  \includegraphics{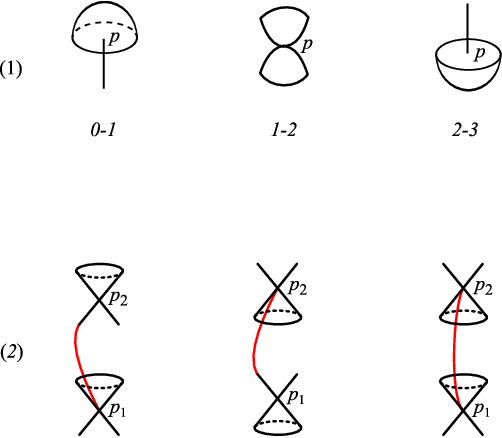}
  \caption{Generic codimension-1 singularities of gradient vector
    fields.  The vertical direction in these pictures is the direction
    of the Morse function, so the gradient flow always goes upwards.
    The critical points are shown schematically, indicating only the
    stable and unstable manifolds.  The top row shows the
    possibilities in case~\ref{item:failure-H}: an index~0-1, an
    index~1-2, and an index~2-3 saddle-node. Here the stable and
    unstable manifolds are manifolds with boundary. In the bottom
    row, we see a quasi-transversal orbit of tangency, shown in red,
    between $W^u(p_1)$ and $W^s(p_2)$. There are three cases, from left to
    right: $\I(p_1) = \I(p_2) = 1$, or $\I(p_1) = \I(p_2) = 2$, or
    $\I(p_1) = 1$ and $\I(p_2) = 2$.}
  \label{fig:codim1}
\end{figure}

\subsubsection{Generic 2-parameter families of gradients in dimension
  3} \label{sec:2-param}
This section summarizes results of Vegter~\cite{Vegter85}; also see
Carneiro and Palis~\cite{CP89} for the classification of generic
2-parameter families of gradients in arbitrary dimensions.

The instabilities~\ref{item:failure-H} and~\ref{item:failure-T} of
Section~\ref{sec:1-param} may also occur in an open and dense class
of 2-parameter gradient families on~$M$.  The corresponding parameter
values form smooth curves in the parameter space~$\R^2$. Moreover, by the work of
Vegter~\cite[p.~108]{Vegter85}, for
a generic family $\{X^\mu\} \in X^g_2(M)$, at isolated values~$\om$
of the parameter, exactly one of the following situations may
occur. (These cases are described in more detail shortly.)
\begin{enumerate}[label=(\Alph*)]
\item \label{item:A} The vector field $X^{\om}$ has exactly two quasi-transversal
  orbits of tangency between stable and unstable manifolds,
  satisfying analogues of
  conditions~\ref{item:unfolding}--\ref{item:transversality}, while
  all singularities are hyperbolic.
\item \label{item:B} The vector field $X^{\om}$ has exactly one non-hyperbolic
  singularity, which is a saddle-node, and exactly one
  quasi-transversal orbit of tangency between stable and unstable
  manifolds that satisfies analogues of
  conditions~\ref{item:unfolding}--\ref{item:transversality}.
\item \label{item:C} The vector field $X^{\om}$ has exactly two non-hyperbolic
  singularities, which are saddle-nodes, while all stable and unstable
  manifolds intersect transversely.
\item \label{item:D} The vector field $X^{\om}$ has exactly one non-hyperbolic
  singularity, which is quasi-hyperbolic of type~2, while all stable
  and unstable manifolds are transversal.
\item \label{item:E} All singular points of $X^{\om}$ are hyperbolic, and a single
  degenerate orbit of tangency occurs between $W^u(p_1^{\om})$ and
  $W^s(p_2^{\om})$ that violates exactly one of the conditions
  \ref{item:QT1}, \ref{item:QT2},
  \ref{item:limit}, or \ref{item:transversality} in the
  ``mildest possible manner'' (for more detail, see \ref{item:E1}--\ref{item:E4}).
  Observe that it does not make sense to consider
  violation of condition~\ref{item:unfolding}; it can be replaced by a similar
  condition for 2-parameter families. Condition~\ref{item:eigenvalues} also holds for
  generic 2-parameter families, since the set of linear
  \emph{gradients} on $\RR^2$ having two equal eigenvalues has
  codimension~2. Hence, generically, a pair of equal contracting
  eigenvalues at $p_1^{\om}$ does not occur together with an orbit of
  tangency.
\end{enumerate}
If $X^{\om}$ has a non-hyperbolic singular point $p \in M$, as in
cases~\ref{item:B}--\ref{item:D}, the set of stable and unstable manifolds also includes
$W^{ss}(p)$ and $W^{uu}(p)$, respectively.
Conditions~\ref{item:A}--\ref{item:D} correspond to conditions I--IV of Vegter \cite[pp.~111--112]{Vegter85}.
For condition~\ref{item:E}, see \cite[p.~125]{Vegter85}.

Next, we consider the bifurcation sets in $\R^2$ near parameter values
$\om$ for which we have one of the situations described above. Such a
parameter value is in the closure of smooth curves in $\R^2$ that
correspond to the occurrence of bifurcations that may also occur in
1-parameter families.  The curves limiting on $\om$
corresponding to tangencies with invariant manifolds of far-away
singularities (i.e., singular points of~$X^{\om}$ not directly contributing
to the failure of conditions \ref{item:failure-H} or \ref{item:failure-T}
in cases \ref{item:A}--\ref{item:E})
are called \emph{secondary bifurcations}.
We do not list the cases that arise from the ones below by reversing
the sign of $X^\mu$, which simply amounts
to swapping superscripts~``$u$'' and~``$s$.''

\begin{enumerate}[label=(A\arabic*)]
\item \label{item:A1} There are four hyperbolic singular points $p_1^{\om},
  \dots, p_4^{\om}$ of $X^{\om}$ such that the orbits of tangency are
  contained in $W^u(p_1^{\om}) \cap W^s(p_2^{\om})$ and
  $W^u(p_3^{\om}) \cap W^s(p_4^{\om})$, respectively. We allow $p_1^{\om} = p_3^{\om}$ or $p_2^{\om} = p_4^{\om}$, or both.
  Note that $\I(p_i^{\om}) \in \{1,2\}$ for every $i \in \{\,
  1,\dots,4 \,\}$, and necessarily $\I(p_1^{\om}) \le \I(p_2^{\om})$
  and $\I(p_3^{\om}) \le \I(p_4^{\om})$. See the top row of
  Figure~\ref{fig:flow-pairs} for schematic drawings of the
  possibilities when each~$p_i^{\om}$ has index~1.
  Generically, the bifurcation set consists of two smooth curves that intersect
  transversely at $\om$; see Figure~\ref{fig:link-A1}.
  \begin{figure}
    \centering
    \includegraphics{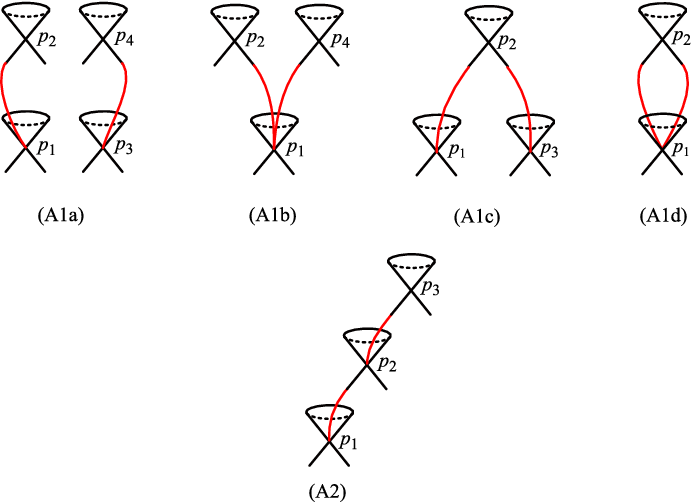}
    \caption{\textbf{The possibilities for two nondegenerate orbits of
        tangency between hyperbolic singular points from case~\ref{item:A}.}
      The top row shows the possibilities in case~\ref{item:A1}, while the
      bottom row illustrates case~\ref{item:A2}.  In this figure, each critical
      points is index~1, with stable manifold consisting of a curve
      and unstable manifold consisting of a disk.  The orbits of
      tangencies are shown in red.}
    \label{fig:flow-pairs}
  \end{figure}

\item \label{item:A2} There are three hyperbolic singular points $p_1^{\om}$,
  $p_2^{\om}$, and $p_3^{\om}$ of $X^{\om}$ such that the orbits of
  tangency are contained in $W^u(p_1^{\om}) \cap W^s(p_2^{\om})$ and
  $W^u(p_2^{\om}) \cap W^s(p_3^{\om})$, respectively. Again, each
  $p_i^{\om}$ has index 1 or~2, and
  \[
  \I(p_1^{\om}) \le \I(p_2^{\om}) \le \I(p_3^{\om}).
  \]
  See the bottom row of Figure~\ref{fig:flow-pairs} for an
  illustration.
  The bifurcation set consists of five codimension-1 strata meeting at~$\om$;
  see Figure~\ref{fig:link-chain-flows}.
\end{enumerate}

\begin{enumerate}[label = (B\arabic*)]
\item \label{item:B1} The vector field $X^{\om}$ has one saddle-node $p^{\om}$
  and one quasi-transverse orbit of tangency between $W^u(p_1^{\om})$
  and $W^s(p_2^{\om})$, where $p_1^{\om}$ and $p_2^{\om}$ are
  hyperbolic saddle-points of $X^{\om}$; see
  Figure~\ref{fig:codim-two-cases} for one case. The bifurcation set
  consists of two curves that intersect transversely at~$\om$;
  see Figure~\ref{fig:link-6a}.

\item \label{item:B2} The vector field $X^{\om}$ has a saddle-node~$p^{\om}$ and
  a hyperbolic saddle-point~$\ol{p}^{\om}$ whose stable manifold has
  one quasi-transverse orbit of tangency with the unstable manifold
  of~$p^{\om}$.
  Secondary bifurcations are due to the occurrence of
  tangencies between $W^u(p_*^{\mu})$ and $W^s(\ol{p}^{\mu})$ for each
  saddle point $p_*^{\mu} \neq \ol{p}^{\mu}$ such that $W^u(p_*^{\om})
  \cap W^s(p^{\om}) \neq \emptyset$; suppose there are $s$ of these.
  For an illustration of one case, see Figure~\ref{fig:codim-two-cases}.
  Then the bifurcation diagram consists of $s+3$ codimension-1 strata
  meeting at $\om$; see Figure~\ref{fig:link-7a}.

\item \label{item:B3} The vector field $X^{\om}$ has a saddle-node~$p^{\om}$ and
  a quasi-transverse orbit of tangency between $W^{ss}(p^{\om})$ and
  $W^u(\ol{p}^{\om})$, where $\ol{p}^{\om}$ is a saddle-point of
  $X^{\om}$; see Figure~\ref{fig:codim-two-cases}. The bifurcation set
  consists of 3 codimension-1 strata meeting at~$\om$; see Figure~\ref{fig:link-9a}.
\end{enumerate}

\begin{enumerate}[label = (\Alph*)]
\addtocounter{enumi}{2}
\item For an open and dense class of 2-parameter families
  $\{X^\mu\}$ of $X_2^g(M)$, we have a pair $p_1$, $p_2$ of saddle-nodes
  occurring at isolated values $\om$ of the parameter.
  For an illustration, see Figure~\ref{fig:codim-two-cases}. There are two
  curves $\Gamma_1$ and $\Gamma_2$ in the parameter plane
  corresponding to the occurrence of exactly one saddle-node of
  $X^\mu$ near $p_1$ and $p_2$, respectively. Generically, these
  curves are transversal; see Figure~\ref{fig:link-8}.

\begin{figure}
  \centering
  \includegraphics{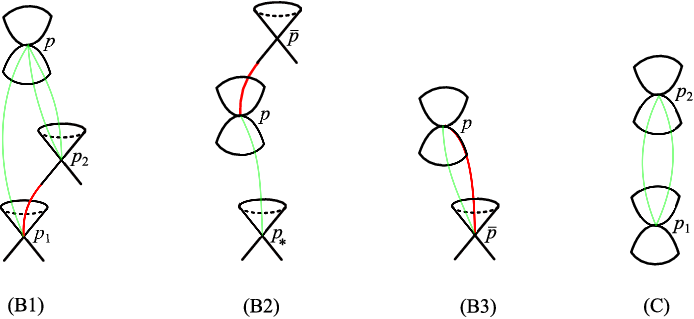}
  \caption{Codimension-2 bifurcations that involve a saddle-node.  As
    in Figure~\ref{fig:flow-pairs}, we have drawn schematic
    possibilities showing only the stable and unstable manifolds.  In
    the case of an index~1-2 saddle-node, the stable and unstable
    manifolds are each half-planes. The red lines show
    quasi-transverse orbits of tangency between different
    singularities.  The green lines are flows that are transverse
    intersections between stable and unstable manifolds.}
  \label{fig:codim-two-cases}
\end{figure}
\end{enumerate}

\begin{enumerate}[label = (\Alph*)]
\addtocounter{enumi}{3}
\item In a neighborhood of~$\om$,
  the bifurcation diagram consists of parameter values~$\mu$ for which~$X^\mu$,
  and hence~$f^\mu$, has a degenerate singular point near~$p$.
  For an open and dense class of 2-parameter families~$\{f^\mu\}$
  for which~$\grad(f^\mu)$ has a quasi-hyperbolic singularity of type~2,
  there are $\mu$-dependent local coordinates $(x,y,z)$ in which~$f^\mu$
  can be written as
  \[
  \pm x^4 + \mu_1 x^2 + \mu_2 x \pm y^2 \pm z^2,
  \]
  having a singularity of type~$A_3^\pm$.  So the bifurcation diagram
  near~$\om$ is the well-known cusp; see Figure~\ref{fig:link-10a}.
  The pair of curves having~$\om$
  in their closure corresponds to the occurrence of exactly one
  saddle-node near~$p$. For an illustration, see
  Figure~\ref{fig:codim-two-local}.

\begin{figure}
  \centering
  \includegraphics{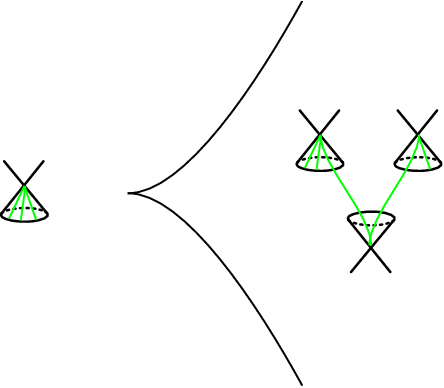}
  \caption{
    Local codimension-2 bifurcation of type~\ref{item:D}. We have drawn the bifurcation diagram for an
    $A_3^-$ singularity, and indicated the dynamics on the two sides of the bifurcation set.
  }
  \label{fig:codim-two-local}
\end{figure}
\end{enumerate}

\begin{enumerate}[label = (E\arabic*)]
\begin{figure}
  \centering
  \includegraphics{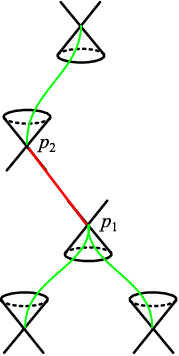}
  \caption{The dynamics at a bifurcation of type~\ref{item:E1}.}
  \label{fig:codim-two-E1}
\end{figure}

\item \label{item:E1} We have $\dim W^u(p_1^{\om}) = \dim W^s(p_2^{\om}) = 1$,
  violation of~\ref{item:QT1}.  Secondary bifurcations may be present
  due to occurrence of an orbit of tangency between $W^s(p_2^\mu)$ and
  an unstable manifold (of dimension~2) intersecting $W^s(p_1^\mu)$,
  or between~$W^u(p_1^\mu)$ and a stable manifold (of dimension~2)
  intersecting~$W^u(p_2^\mu)$; see Figure~\ref{fig:codim-two-E1}.  For~$\mu$
  close to~$\om$, let~$D_r^\mu$ be a continuous family of smooth
  discs contained in a level set of~$f^\mu$ such that $W^u(p_1^\mu)
  \cap D_r^\mu = \{r\}$.  Let $U_1^\mu, \dots, U_n^\mu$ be the
  intersections of~$D_r^\mu$ with unstable manifolds having non-empty
  intersections with~$W^s(p_1^\mu)$. Similarly, $S_1^\mu, \dots,
  S_m^\mu$ denote intersections of~$D_r^\mu$ and stable manifolds
  meeting~$W^u(p_2^\mu)$. The corresponding bifurcation diagram
  consists of~$n+m$ curves in the parameter plane, having $\om$ in
  their closure; see Figure~\ref{fig:E1-surface}.  For parameter
  values~$\mu$ on these curves, we have
  either $W^s(p_2^\mu) \cap D_r^\mu \in U_i^\mu$ for some $i \in \{\,1,\dots,
  n\,\}$, or $W^u(p_1^\mu) \cap D_r^\mu \in S_j^\mu$ for some $j \in \{\,1, \dots,
  m\,\}$.

\item \label{item:E2} We have $d^2 F^{\om}/dx^2 = 0$, violation
  of~\ref{item:QT2}, where~$F^{\om}$ is the function~$F$ defined in~\ref{item:QT2}
  for the vector field~$v^{\om}$. For an open and dense class of 2-parameter
  families, we have $(d^3F^{\om}/dx^3)(0) \neq 0$, while the family
  $\{F^\mu\}$ is a versal unfolding of $F^{\om}$. The latter condition
  implies the existence of local coordinates $(\mu, x)$ near $(\om,r)$
  in which $(\om, r)$ corresponds to $(0,0) \in \R^2 \times \R$, such
  that $F^\mu(x) = x^3 + \mu_1 x + \mu_2$. The bifurcation values form
  a cusp in the parameter space $\R^2$.
\item \label{item:E3} Situation of case~\ref{item:dim-1-2}, where $\lim_{t \to
    -\infty} d\phi_t^{\om}(E_r)$ is the eigenspace of the linear part at
  $p_1^{\om}$ corresponding to the \emph{weakest} contracting
  eigenvalues, which is violation of~\ref{item:limit}. Let~$D^\mu_r$ be
  as in case~\ref{item:E1}. Then~$D^\mu_r$ contains $U_1^\mu, \dots, U_n^\mu$
  that are intersections of~$D_r^\mu$ and unstable manifolds meeting
  $W^s(p_1^\mu)$. Secondary bifurcations occur for parameter values~$\mu$
  for which $W^s(p_2^\mu) \cap D_2^\mu$ is tangent to~$U_i^\mu$
  for some $i \in \{\, 1, \dots, n \,\}$.

\item \label{item:E4} The vector field $X^\mu$ has an orbit of tangency as in
  case~\ref{item:dim-1-2}, and exactly one hyperbolic saddle
  $p^{\om}_*$ different from $p_1^{\om}$ and $p_2^{\om}$ such that
  $W^u(p_*^{\om})$ and $W^{ss}(p_1^{\om})$ are not transversal, which
  is violation of~\ref{item:transversality}.  In this case, we have
  secondary bifurcations for parameter values~$\mu$ for which one of
  the following occur:
  \begin{enumerate}[label = (\alph*)]
  \item $W^u(p_1^\mu) \cap D_r^\mu \in W^s(p_2^\mu) \cap
    D_r^\mu$,
  \item $W^u(p_*^\mu) \cap D_r^\mu$ is tangent to $W^s(p_2^\mu)
    \cap D_r^\mu$.
  \end{enumerate}
\end{enumerate}

We remark that conditions \ref{item:A1} and \ref{item:A2} correspond to
conditions I.a and I.b of Vegter~\cite[p.~113]{Vegter85}.
Conditions \ref{item:B1}--\ref{item:B3} correspond to conditions II.a--II.c of Vegter~\cite[p.~114]{Vegter85},
while conditions \ref{item:E1}--\ref{item:E4} appear as conditions 1--4 in Vegter~\cite[pp.~125--128]{Vegter85}.

\subsection{Sutured functions and gradient-like vector fields}
\label{sec:sutured-functions}

In this section, we introduce sutured functions, which are smooth
functions on a sutured manifold with prescribed boundary
behavior. Then we define and study gradient-like vector fields for
sutured functions.

\begin{definition}
  Let $(M,\g)$ be a sutured manifold. A \emph{sutured function} on
  $(M,\g)$ is a smooth function $f \colon M \to [-1,1]$ such that
  \begin{enumerate}
  \item $f^{-1}(\pm 1) = R_\pm(\g)$ and $f^{-1}(0) \supset s(\g)$,
  \item \label{item:suturedfn-on-R}
  $f$ has no critical points along $R(\g)$,
  \item \label{item:suturedfn-on-gamma} $f|_\g$ has no critical points.
  \end{enumerate}
\end{definition}

The space of sutured functions on $(M,\g)$ is contractible. Indeed, the set of sutured functions
for a given $f|_\g$ is convex, while the space of possible $f|_\g$ is
contractible.
For a sutured function~$f$, we denote by~$C(f)$ the set of critical
points of~$f$; i.e.,
\[
C(f) = \{\, p \in M  \colon  df_p = 0\,\}.
\]
By conditions~\eqref{item:suturedfn-on-R}
and~\eqref{item:suturedfn-on-gamma}, the set~$C(f)$ lies in the
interior of~$M$.  The following definition was motivated by Milnor
\cite[Definition 3.1]{Milnor}.

\begin{definition} \label{def:grad-like} Let $f$ be a sutured function
  on $(M,\g)$. A vector field~$v$ on $M$ is a \emph{gradient-like
    vector field for~$f$} if
  \begin{enumerate}
  \item $v(f) > 0$ on $M \setminus C(f)$,
  \item \label{item:grad} $C(f)$ has a neighborhood $U$ such that
    $v|_U = \grad_g(f|_U)$ for some Riemannian metric $g$ on $U$,
  \item $v_p \in T_p\g$ for every $p \in \g$.
  \end{enumerate}
\end{definition}

\begin{remark}
  Note that Milnor \cite[Definition 3.1]{Milnor} defined
  gradient-like vector fields for a Morse function~$f$ on an
  $n$-manifold~$M$.  Instead of condition~\eqref{item:grad}, he
  required that, for any critical point~$p$ of~$f$, there are
  coordinates $(x_1,\dots,x_n)$ in a neighborhood $U$ of $p$ such
  that
  \[
  f = f(p) - x_1^2 -\dots-x_\lambda^2 + x_{\lambda+1}^2 + \dots
  + x_n^2
  \]
  and $v$ has coordinates
  $(-x_1,\dots,-x_\lambda,x_{\lambda+1},\dots,x_n)$ throughout $U$.
  When studying families of smooth functions, as we have seen, more
  complicated singularities can arise. We could require that around
  such a singularity, $v$ is the Euclidean gradient in a local
  coordinate system in which the singularity is in normal form.  But
  then it is unclear whether the space of gradient-like vector fields
  is contractible, as the space of such local coordinate systems is
  rather complicated.  Hence, we have chosen to work with
  condition~\eqref{item:grad}, as the space of metrics is clearly
  contractible. As a tradeoff, one has to resort to such results as
  Theorems~\ref{thm:center-manifold} and~\ref{thm:reduction-principle}
  to understand the invariant manifolds of~$v$ near a singular point.
\end{remark}

Let $\FV(M,\g)$ be the space of pairs $(f,v)$, where $f$ is a sutured
function on $(M,\g)$ and $v$ is a gradient-like vector field for
$f$. We endow $\FV(M,\g)$ with the $C^{\infty}$-topology.

\begin{definition}
  A \emph{Morse function} on $(M,\g)$ is a sutured function $f \colon
  M \to [-1,1]$ such that all critical points of~$f$ are
  non-degenerate.  For a Morse function~$f$ and $i \in
  \{\,0,1,2,3\,\}$, let $C_i(f)$ be the set of critical points of~$f$
  of index~$i$.
\end{definition}

By condition~\eqref{item:grad} of Definition~\ref{def:grad-like},
every gradient-like vector field~$v$ of a Morse
function has only hyperbolic singular points. In particular, we can
talk about the stable and unstable manifolds $W^s(p)$ and $W^u(p)$ of
a singular point $p$ of $v$. If we also want to refer to the vector
field~$v$, then we write $W^u(p,v)$ and $W^s(p,v)$. Note that the
Morse index~$\I(p)$ of the critical point $p \in C(f)$ agrees with
$\dim W^s(p)$. Indeed, in a suitable coordinate system around~$p$, the
linearization $L_pv$ coincides with the Hessian of~$f$ at~$p$.
Furthermore, notice that every point
\[
x \in M \setminus \bigcup_{p \in C(f)} (W^u(p) \cup W^s(p))
\]
lies on a compact flow-line connecting $R_-(\g)$ and $R_+(\gamma)$.

\begin{definition}
  We say that $(f,v) \in \FV(M,\g)$ satisfies the \emph{Morse-Smale
    condition} if $v$ is Morse-Smale in the sense of
  Definition~\ref{defn:Morse-Smale}. We denote the subspace of
  Morse-Smale pairs in $\FV(M,\g)$ by $\FV_0(M,\g)$.
\end{definition}

If $(f,v) \in \FV_0(M,\g)$, then $f$ is a Morse function on $(M,\g)$.
Furthermore, for every $p$, $q \in C(f)$, the intersection $W^u(p) \cap
W^s(q)$ is a manifold of dimension $\I(q) - \I(p)$ that we denote by
$W(p,q)$.  In particular, $W(p,q) = \emptyset$ if $\I(q) - \I(p) <
0$.

\begin{remark}
  Notice that, for $(f,v) \in \mathcal{FV}(M,\g)$, the condition that
  $W^u(p)$ and $W^s(q)$ intersect transversely is automatically
  satisfied if at least one of $p$ or~$q$ has index~0 or~3.  For a
  pair $(f,v) \in \FV(M,\g)$ where $f$ is a Morse function, the
  Morse-Smale condition can be violated by having flows between
  critical points of index~1, flows between critical points of
  index~2, flows from index~2 to index~1 critical points,
  or an orbit of tangency (see Definition~\ref{def:bifurcation})
  in~$W^u(p) \cap W^s(q)$ for $p \in C_1(f)$ and $q \in C_2(f)$.
  Flows from index~2 to index~1 critical points
  only appear generically in 2-parameter families, while the other possibilities
  already occur in generic 1-parameter families.
\end{remark}

\begin{definition}
  We say the pair $(f,v) \in \FV(M,\g)$ is \emph{codimension-1} if
  $(f,v) \not\in \FV_0(M,\g)$, but $v$ appears as $X^{\om}$ for some
  1-parameter family $\{ X^\mu \} \in X_1^g(M)$ that is generic in the
  sense of Section~\ref{sec:1-param}. We denote the space of
  codimension-1 pairs by $\FV_1(M,\g)$, and the union $\FV_0(M,\g)
  \cup \FV_1(M,\g)$ by $\mathcal{FV}_{\le 1}(M,\g)$.

  In an analogous manner, we say that $(f,v) \in \FV(M,\g)$ is
  \emph{codimension-2} if $(f,v) \not\in \mathcal{FV}_{\le 1}(M,\g)$,
  but $v$ appears as $X^{\om}$ for some 2-parameter family $\{ X^\mu
  \} \in X_2^g(M)$ that is generic in the sense of
  Section~\ref{sec:2-param}. We denote the space of codimension-2
  pairs by $\FV_2(M,\g)$. Finally, we set
  \[
  \FV_{\le 2}(M,\g) = \bigcup_{i \in \{\, 0,1,2 \,\}}
  \mathcal{FV}_i(M,\g).
  \]
\end{definition}

The following proposition implies that every gradient-like vector
field is actually a gradient for some Riemannian metric. The advantage
of gradient-like vector fields is that they are easier to manipulate
than metrics, which is useful in actual constructions.

\begin{proposition} \label{prop:grad-like-metric} Let $(f,v) \in
  \FV(M,\g)$.  Then the space $G(f,v)$ of Riemannian metrics~$g$ on
  $M$ for which $v = \grad_g(f)$ is non-empty and contractible.
\end{proposition}

\begin{proof}
  By definition, there is a metric~$g$ on a neighborhood~$U$ of $C(f)$
  such that $v|_U = \grad_g(f|_U)$.  Pick a smaller neighborhood~$V$
  of $C(f)$ such that $\ol{V} \subset U$.  We are going to extend
  $g|_V$ to the whole manifold~$M$ such that $v = \grad_g(f)$
  everywhere. Such a metric~$g$ on~$M$ satisfies $g(v_x,v_x) = v_x(f)
  > 0$ and $g(v_x,w_x) = 0$ for every $x \not \in C(f)$ and $w_x \in
  \ker(df_x)$. So the extension~$g$ on~$M$ is uniquely determined by a
  choice of metric on the 2-plane bundle $\ker(df)|_{M \setminus V}$ that
  smoothly extends the metric given on $\ker(df)|_{V \setminus
    C(f)}$. For this, pick an arbitrary metric on $\ker(df)|_{M
    \setminus V}$ and piece it together with $g|_{U \setminus C(f)}$
  using a partition of unity subordinate to the covering $\{U, M
  \setminus V\}$ of $M$.  Hence $G(f,v) \neq \emptyset$.

  The space $G(f,v)$ is contractible because it is convex. Indeed, if
  $g_0$, $g_1 \in G(f,v)$, then $g_i(v,w) = w(f)$ for every vector field~$w$
  on~$M$ and $i \in \{0,1\}$. Let $g_t = (1-t)g_0 + tg_1$ for $t \in I$
  be an arbitrary convex combination of $g_0$ and $g_1$.  Then
  $g_t(v,w) = w(f)$ for every~$w$ on~$M$; i.e., $v = \grad_{g_t}(f)$.
\end{proof}

\begin{corollary} \label{cor:FV-contractible} The space $\FV(M,\g)$ is
  weakly contractible.
\end{corollary}

\begin{proof}
  Let $\mathcal{F}(M,\g)$ be the space of sutured functions and
  $\mathcal{G}(M,\g)$ the space of Riemannian metrics on $(M,\g)$,
  respectively. Both $\mathcal{F}(M,\g)$ and $\mathcal{G}(M,\g)$ are
  contractible. Consider the projection
  \[
  \pi \colon \mathcal{F}(M,\g) \times \mathcal{G}(M,\g) \to \FV(M,\g)
  \]
  given by $\pi(f,g) = (f,\grad_g(f))$; this is a Serre fibration. For
  $(f,v) \in \FV(M,\g)$, the fiber is $\pi^{-1}(f,v) = G(f,v)$, which is
  contractible by Proposition~\ref{prop:grad-like-metric}. Hence the
  base space $\FV(M,\g)$ is weakly contractible.
\end{proof}
\section{Translating bifurcations of gradients to Heegaard diagrams}
\label{sec:transl-hd}

We will now translate the singularities of
Sections~\ref{sec:1-param} and~\ref{sec:2-param} in terms of
Heegaard diagrams.  Loosely speaking, each generic gradient gives a
Heegaard diagram, each codimension\hyp 1 singularity gives a move
between Heegaard diagrams, and each codimension\hyp 2 singularity
gives a contractible loop of Heegaard diagrams.  The codimension\hyp 1
and codimension\hyp 2 singularities give moves and loops of moves,
respectively, that are more complicated than the ones appearing in the
definition of weak and strong Heegaard invariants.
In Section~\ref{sec:simplify}, we will see how to simplify these families.

The overall idea is that to construct a Heegaard splitting from the
gradient of a generic Morse function (not necessarily self-indexing),
take one compression body to be a small neighborhood of the union of
all flows starting at~$R_-(\g)$ or index~0 critical points and ending at index~1
critical points.
The other compression body is then isotopic to a small neighborhood of
the flows starting at index~2 critical points and ending at~$R_+(\g)$
or at index~3
critical points.
To further construct the $\alpha$- and $\beta$-curves of a Heegaard diagram, we take
the intersection of the Heegaard surface with the unstable
manifolds of some of the index~1 critical points and with the
stable manifolds of some of the index~2 critical points.

This Heegaard splitting extends naturally across codimension\hyp 1 and
codimension\hyp 2 singularities, as long as there is no flow from
an index~2 to an index~1 critical point.  In each case, we will
analyze how the corresponding Heegaard diagrams change.

\subsection{Separability of gradients}
\label{sec:separability}

We now introduce \emph{separability}, our main technical tool for
obtaining Heegaard splittings compatible with gradient-like vector
fields that have at most codimension-2 degeneracies. In the sections
that follow, we explain how to enhance these Heegaard splittings to
Heegaard diagrams for generic gradients, to moves between diagrams for
codimension-1 gradients, and to loops of diagrams for codimension-2
gradients.

\begin{definition} \label{def:separable} We say that the pair $(f,v) \in
  \FV_{\le 2}(M,\g)$ is \emph{separable} if
  \begin{itemize}
  \item it is not codimension-2 of type \ref{item:E1}; i.e., if for
    every pair of non-degenerate critical points $p \in C_2(f)$ and $q
    \in C_1(f)$, we have $W^u(p) \cap W^s(q) = \emptyset$; and
  \item if it is codimension-2 of type~\ref{item:C} (i.e., it has two
    birth-death singularities at~$p$ and at~$q$), then $f(p) \neq f(q)$.
  \end{itemize}
  (This second condition is codimension-3, and hence generic for
  2-parameter families.)
\end{definition}

\begin{definition} \label{def:partition} Suppose that $(f,v) \in
  \FV_{\le 2}(M,\g)$. Then we partition~$C(f)$ into two subsets,
  namely $C_{01}(f,v)$ and~$C_{23}(f,v)$, as follows.
  We define $C_{01}(f,v)$ to be the set of those critical points $p \in C(f)$
  for which one of the following holds:
  \begin{enumerate}
  \item $p \in C_0(f) \cup C_1(f)$, or
  \item $p$ is an index~0-1 birth-death, or
  \item $p$ is an index~1-2 birth-death, $(f,v)$ is codimension-2 of
  type~\ref{item:B1}, and $\I(p_1) = \I(p_2) = 1$, or
  \item $p$ is an index~1-2 birth-death,
  $(f,v)$ is codimension-2 of type~\ref{item:B2} or~\ref{item:B3}, and $\I(\ol{p}) = 1$,
  \item \label{item:onlyone} $(f,v) \in \FV_2(M,\g)$ is type~\ref{item:C},
  and if~$q$ is the other birth-death critical point, then $f(p) < f(q)$, or
  \item $(f,v)$ is codimension-2 of type~\ref{item:D}, and $p$
  is an index 1-0-1, 0-1-0, or~1-2-1 birth-death-birth.
  \end{enumerate}
  Finally, if $p$ is an index~1-2 birth-death critical
  point and $(f,v)$ is codimension-1 of
  type~\ref{item:failure-H}, then we can put~$p$ in either
  $C_{01}(f,v)$ or $C_{23}(f,v)$. Having defined $C_{01}(f,v)$,
  we let $C_{23}(f,v) = C(f) \setminus C_{01}(f,v)$.
\end{definition}

It follows from condition~\eqref{item:onlyone} of Definition~\ref{def:partition} that
$C_{01}(f,v)$ contains at most one index 1-2 birth-death critical point.

\begin{definition} \label{defn:separates} 
  Suppose that $(f,v) \in \FV_{\le 2}(M,\g)$. 
  We say that a properly embedded surface $\S
  \subset M$ \emph{separates} $(f,v)$ if
  \begin{enumerate}
  \item $\S \pitchfork v$,
  \item $M = M_- \cup M_+$, such that $M_- \cap M_+ = \S$ and
    $R_\pm(\g) \subset M_\pm$,
  \item $C_{01}(f,v) \subset M_-$ and $C_{23}(f,v) \subset M_+$, and
  \item $\partial \S = s(\g)$.
  \end{enumerate}
  We denote the set of surfaces that separate $(f,v)$ by
  $\S(f,v)$. When $(f,v) \in \FV_1(M,\g)$ is
  type~\ref{item:failure-H} with an index~1-2 birth-death
  singularity~$p$, then there are two different choices for the
  partition $(C_{01}(f,v), C_{23}(f,v))$ of $C(f)$, depending on where
  we put~$p$, hence $\S(f,v)$ is
  not completely unique. If we put~$p$ into $C_{01}(f,v)$, then we
  denote the resulting set $\S_-(f,v)$, and we write $\S_+(f,v)$ when
  $p \in C_{23}(f,v)$.  Often, we suppress this choice in our
  notation, and simply write $\S(f,v)$ (which is then either
  $\S_-(f,v)$ or $\S_+(f,v)$).
\end{definition}

Notice that, if~$\S$ separates $(f,v)$, then~$\S$ is necessarily
orientable as it is transverse to~$v$ and~$M$ is orientable.  We
orient~$\S$ such that the normal orientation given by~$v$, followed by
the orientation of~$\S$, agrees with the orientation on~$M$; i.e.,
such that~$\Sigma$ is oriented as the boundary of~$M_-$.

If $(f,v) \in \mathcal{FV}_{\le 2}(M,\g)$, then, for every $p \in
C_{01}(f,v)$, the manifold $W^s(p) \setminus R_-(\g)$ is diffeomorphic to
\begin{itemize}
\item a single point if $p \in C_0(f)$, or $p$ is a birth-death-birth
  of index~0-1-0,
\item $\R$ if $p \in C_1(f)$, or $p$ is an index~1-0-1 or~1-2-1
  birth-death-birth,
\item $[0,\infty)$ if $p$ is an index 0-1 birth-death, or
\item $\R \times [0,\infty)$ if $p$ is an index 1-2 birth-death.
\end{itemize}
In addition, if $(f,v)$ is separable and $p \in C_{01}(f,v)$ is not an
index~1-2 birth-death, then $\partial W^s(p)
\subset C_{01}(f,v) \cup R_-(\g)$ (where $\partial W^s(p)$ is the
topological boundary).  If $p \in C_{01}(f,v)$ is an index~1-2
birth-death, then $\partial W^{ss}(p) \subset C_{01}(f,v) \cup
R_-(\g)$, while
\[
\partial W^s(p) \subset C_{01}(f,v) \cup R_-(\g) \cup W^{ss}(p) \cup
\bigcup \bigl\{\, W^s(p') \,\colon\, p'\in C_{01}(f,v) \setminus \{p\}
  \,\bigr\}.
\]
Analogous statements hold for $C_{23}(f,v)$.  The above discussion
justifies the following definition.

\begin{definition} \label{def:gamma} Suppose that $(f,v) \in \FV_{\le 2}(M,\g)$
  is separable. Then we define the relative CW complexes $(R_-(\g)
  \cup \Gamma_{01}(f,v), R_-(\g))$ and $(R_+(\g) \cup
  \Gamma_{23}(f,v), R_+(\g))$ by taking
  \begin{align*}
    \Gamma_{01}(f,v) &= \bigcup_{p \in C_{01}(f,v)} W^s(p),\\
    \Gamma_{23}(f,v) &= \bigcup_{p \in C_{23}(f,v)} W^u(p).
  \end{align*}
  The set of vertices of $\Gamma_{01}(f,v)$ is $C_{01}(f,v)$.
  The closed 1-cells of $\Gamma_{01}(f,v)$ are the closures of the
  components of the $W^{ss}(p) \setminus \{p\}$ for $p \in
  C_{01}(f,v)$ an index~1-2 birth-death, and the closures of the
  components of $W^s(p) \setminus \{p\}$ for every other $p \in
  C_{01}(f,v)$. Finally, $\Gamma_{01}(f,v)$ has at most one (closed)
  2-cell, namely $\ol{W^s(p)}$ if $p \in C_{01}(f,v)$ is an index~1-2
  birth-death.  We define the cell decomposition of $\Gamma_{23}(f,v)$
  in an analogous manner.  Finally, we set $\Gamma(f,v) =
  \Gamma_{01}(f,v) \cup \Gamma_{23}(f,v)$.
\end{definition}

Note that $\Gamma_{01}$ is either a graph (i.e., a 1-complex), or
obtained from a graph by an elementary expansion
involving the 1-cell $W^{ss}(p)$ and the 2-cell
$W^s(p)$  if~$C_{01}(f,v)$ contains an index~1-2 birth-death~$p$.

\begin{remark}
  In light of Definition~\ref{def:gamma}, we now motivate
  Definition~\ref{def:partition}.  We partitioned the set of critical
  points $C(f)$ into $C_{01}(f,v)$ and $C_{23}(f,v)$ precisely so that
  we can form the relative CW complexes $(R_-(\g) \cup
  \Gamma_{01}(f,v), R_-(\g))$ and $(R_+(\g) \cup
  \Gamma_{23}(f,v), R_+(\g))$.  We would like to have $C_1(f)
  \subset C_{01}(f,v)$ and $C_2(f) \subset C_{23}(f,v)$ because if
  $\S$ is a separating surface, then -- as we shall see in
  Section~\ref{sec:codimension-0} -- we can obtain a
  Heegaard diagram from it by taking $\a$-curves to be $W^u(p) \cap \S$ for
  some $p \in C_1(f)$ and $\b$-curves to be $W^s(p) \cap \S$ for some
  $p \in C_2(f)$.  This also explains our rule in case~\ref{item:B2}. For
  example, suppose that $(f,v)$ has an index~1-2 birth-death critical
  point at~$p$, and a non-degenerate critical point at $\ol{p}$ of index~1,
  such that there is a flow~$\varphi$ from~$p$ to~$\ol{p}$. Since we
  have to place~$\ol{p}$ in $C_{01}(f,v)$, we must also put~$p$ into
  $C_{01}(f,v)$, otherwise the 1-cell $\ol{\varphi} \subset
  W^s(\ol{p})$ would have one endpoint in $\Gamma_{23}(f,v)$.

  In case~\ref{item:E1}, we do not obtain a CW complex (whichever side
  we assign $p$ to) for a similar reason,
  explaining why those gradients are not separable.  Our choices for
  placing the index~1-2 birth-death critical points of a pair $(f,v)
  \in \FV_2(M,\g)$ in every case other than~\ref{item:B2} are purely
  conventional to make the construction more canonical, and most
  proofs would also work for the other choices. However, we do adhere
  to these conventions in Theorem~\ref{thm:2-param}.

  When $(f,v) \in \FV_1(M,\g)$ has an index~1-2 birth-death critical
  point~$p$, there is no canonical way to decide where to put~$p$, and,
  in fact, the rule in case~\ref{item:B2} forces us to allow both
  possibilities: If $\{\,(f_\lambda,v_\lambda) \,\colon\, \lambda \in
  \RR^2\,\}$ is a generic 2-parameter family such that $(f_0,v_0)$ has
  a type~\ref{item:B2} bifurcation, where we have to put the~$A_2$ point in
  $C_{01}(f,v)$, then we have to do the same for
  $(f_\lambda,v_\lambda)$ when $\lambda$ lies in the stratum of the
  bifurcation set corresponding to the~$A_2$ singularity.
\end{remark}

\begin{lemma} \label{lem:flow} Suppose that $(f,v) \in \FV_{\le
    2}(M,\g)$ is separable and $\S \in \S(f,v)$. Then the surface~$\S$
  intersects every
  flow line of~$v$ in~$M \setminus \Gamma(f,v)$ in exactly one point.
\end{lemma}

\begin{proof}
  Note that $M \setminus \Gamma(f,v)$ is a saturated subset of~$M$
  (i.e., it is a union of complete flow lines).
  The closure of a non-constant flow line~$\tau$ of~$v$ is
  diffeomorphic to~$I$, and has both endpoints in $R(\g) \cup
  C(f)$.
  If the maximal open interval on which~$\tau$ is defined is~$(a,b)$
  (where~$a$ might be~$-\infty$ and~$b$ might be~$+\infty$), then let
  these endpoints be $\tau(a) = \lim_{t \to a+} \tau(t)$ and $\tau(b)
  = \lim_{t \to b-} \tau(t)$.  If $\tau(a) \in C_{23}(f,v)$, then
  $\tau \subset \Gamma_{23}(f,v)$. Similarly, if $\tau(b) \in
  C_{01}(f,v)$, then $\tau \subset \Gamma_{01}(f,v)$.  Consequently,
  every flow line~$\tau$ of~$v|_{M \setminus \Gamma(f,v)}$ has
  \begin{align*}
    \tau(a) &\in R_-(\g) \cup C_{01}(f,v) \subset M_-, \text{ and}\\
    \tau(b) &\in R_+(\g) \cup C_{23}(f,v) \subset M_+,
  \end{align*}
  so $\tau \cap \S \neq \emptyset$.
  Since~$\S$ is positively transverse to~$v$, once an integral
  curve of~$v$ enters~$M_+$ it can never leave it, so $|\tau \cap \S|
  = 1$.
\end{proof}

Using Lemma~\ref{lem:flow}, we can endow $\S(f,v)$ with a topology as
follows.  Choose a smooth function $h \colon M \to I$ such that
$h^{-1}(0) = R(\g)$, and let $w = hv$.  Unlike~$v$, the vector field~$w$
is complete, and $v$ and~$w$ have the same phase portrait inside $M
\setminus R(\g)$.  Let $\varphi \colon \RR \times M \to M$ be the flow
of~$w$.  For surfaces~$\S$ and~$\S'$ in $\S(f,v)$, we define the
function $d_{\S',\S} \in C^{\infty}(\S)$ by requiring that
$\varphi(x,d_{\S',\S}(x)) \in \S'$ for every $x \in \S$.  This
uniquely determines $d_{\S',\S}(x)$ by Lemma~\ref{lem:flow}.  If we
fix $\S_0 \in \S(f,v)$, then the map $b_{\S_0} \colon \S(f,v) \to
C^{\infty}(\S_0)$ given by $b_{\S_0}(\S) = d_{\S,\S_0}$ is
bijective. The topology on $\S(f,v)$ is the pullback of the Whitney
$C^{\infty}$-topology on $C^{\infty}(\S_0)$ along $b_{\S_0}$. This is
independent of the choice of $\S_0$, since $d_{\S,\S_1} = d_{\S,\S_0}
\circ i_{\S_0,\S_1} + d_{\S_0,\S_1}$, where $i_{\S_0,\S_1} \colon \S_1
\to \S_0$ is the diffeomorphism given by $i_{\S_0,\S_1}(x) =
\varphi(x,d_{\S_0,\S_1}(x))$. In addition, the map $f \mapsto f \circ
i_{\S_0,\S_1}$ from $C^{\infty}(\S_0)$ to $C^{\infty}(\S_1)$ and the
map $g \mapsto g + d_{\S_0,\S_1}$ from $C^{\infty}(\S_1)$ to
$C^{\infty}(\S_1)$ are both homeomorphims.  The function $d_{\S,\S'}$
depends continuously on~$h$, hence the topology that we defined is
independent of the choice of~$h$.

\begin{proposition} \label{prop:grad-like} Suppose that $(f,v) \in
  \FV_{\le 2}(M,\g)$ is separable. Then the space $\S(f,v)$ is
  non-empty and contractible. Furthermore, every $\S \in \S(f,v)$
  divides $(M,\g)$ into two sutured compression bodies; i.e., it is a
  Heegaard surface of $(M,\g)$.
\end{proposition}

More precisely, in the indeterminate case that $(f,v) \in \FV_1(M,\g)$
and $f$ has an index~1-2 birth-death critical point, we mean that both
$\S_-(f,v)$ and $\S_+(f,v)$ are non-empty and contractible.

\begin{proof}
  By the above discussion, it is clear that if $\S(f,v) \neq
  \emptyset$, then it is homeomorphic to $C^{\infty}(\S)$, hence it is
  contractible.

  Next, we show that $\S(f,v) \neq \emptyset$. Let~$N_{01}$ be a thin
  regular neighborhood of $\Gamma_{01}(f,v) \cup R_-(\g)$, and
  consider the surface $\S_{01} = \ol{\partial N_{01}
    \setminus \partial M}$. Similarly, pick a regular neighborhood~$N_{23}$
  of $\Gamma_{23}(f,v) \cup R_+(\g)$, and define $\S_{23} =
  \ol{\partial N_{23} \setminus \partial M}$.  Choosing sufficiently
  small and nice regular neighborhoods, we can suppose that $\S_{01}
  \cap \S_{23} = \emptyset$ and that $\S_{01}$ and $\S_{23}$ are
  transverse to~$v$.
  Their union $\S_{01} \cup \S_{23}$ separates~$M$ into three
  pieces. Two of them are~$N_{01}$ and~$N_{23}$, and we call the third
  piece~$P$. Now~$v|_P$ is a nowhere vanishing vector field that
  points into~$P$ along $\S_{01}$, points out of~$P$ along~$\S_{23}$,
  and is tangent to $\g \cap P$.  In addition, $v(f) > 0$ on~$P$, so
  an isotopy from $\S_{01}$ to $\S_{23}$ relative to~$\gamma$ is given by
  flowing along $v/v(f)$.  In particular, $(P, \g \cap P)$ is a
  product sutured manifold, and the flow-lines of $v|_P$ give an
  $I$-fibration.  By isotoping $\S_{01}$ near $\g$ flowing along~$v$,
  we can obtain a surface $\S_{01}'$ such that $\partial \S_{01}' =
  s(\g)$. Hence $\S_{01}' \in \S(f,v)$ (with $M_-$ isotopic to
  $N_{01}$ and $M_+$ isotopic to $N_{23} \cup P$).

  Observe that $\S_{01}$ divides $(M,\g)$ into the sutured manifolds
  $(N_{01},\g \cap N_{01})$ and $(N_{23} \cup P,\g \cap (N_{23} \cup
  P))$. Since $\Gamma_{01}$ is either a graph (i.e., a 1-complex), or
  obtained from a graph by an elementary expansion if~$C_{01}(f,v)$
  contains an index~1-2 birth-death, $(N_{01}, \g \cap N_{01})$ is a
  sutured compression body (where $R_+(\g \cap N_{01}) = \S_{01}$ can be compressed
  to be isotopic to $R_-(\g \cap N_{01})$).
  Similarly, $(N_{23},\g \cap N_{23})$ is also a sutured compression body
  (where $R_-(\g \cap N_{23}) = \S_{23}$ can be compressed
  to be isotopic to $R_+(\g \cap N_{23})$). As $(P, \g \cap P)$ is a product,
  $(N_{23} \cup P, \g \cap (N_{23} \cup P))$ is a sutured compression
  body.  Every element of $\S(f,v)$ is isotopic to $\S_{01}$ relative to~$\g$,
  hence also divides $(M,\g)$ into two sutured compression bodies.
\end{proof}

\begin{definition}
  Let $B(M,\g)$ be the space of pairs $(f,v) \in \FV_{\le 2}(M,\g)$
  that are separable, and let $E(M,\g)$ be the space of triples
  $(f,v,\S)$, where $(f,v) \in B(M,\g)$ and $\S \in \S(f,v)$.  There
  is a
  projection $\pi \colon E(M,\g) \to B(M,\g)$ defined by
  forgetting~$\Sigma$.  For $(f,v) \in B(M,\g)$, let~$\chi(f,v)$ be
  the Euler characteristic of~$\S$ for any $\S \in \S(f,v)$ (which is
  independent of the choice of $\S$ by
  Proposition~\ref{prop:grad-like}). For $k \in \ZZ$, we define
  \[
  B_k(M,\g) = \{\,(f,v) \in B(M,\g) \colon \chi(f,v) = k\,\}.
  \]
  Finally, we set $E_k(M,\g) = \pi^{-1}(B_k(M,\g))$ and $\pi_k =
  \pi|_{E_k(M,\g)}$.
\end{definition}

Note that the total space~$E(M,\g)$ depends on whether~$\S(M,\g)$
stands for~$\S_+(M,\g)$ or~$\S_-(M,\g)$, but the base~$B(M,\g)$
is independent of this choice according to the following result.

\begin{lemma} \label{lem:euler-char} For $(f,v) \in B(M,\g)$, we have
  \[
  \chi(f,v) = \chi(R_-(\g)) + \sum_{p \in C_{01}(f,v)} i(p),
  \]
  where
  \begin{itemize}
  \item $i(p) = 2$ for $p \in C_0(f)$ or~$p$ an index~0-1-0
    birth-death-birth,
  \item $i(p) = 0$ for $p$ a birth-death, and
  \item $i(p) = -2$ for $p \in C_1(f)$, or $p$ an index 1-0-1 or 1-2-1
    birth-death-birth.
  \end{itemize}
\end{lemma}

\begin{proof}
  Recall that $(R_-(\g) \cup \Gamma_{01}(f,v),R_-(\g))$ is a relative
  CW complex of dimension at most two.  We saw in the proof of
  Proposition~\ref{prop:grad-like} that every $\S \in \S(f,v)$ is
  isotopic to the surface $\S_{01} = \ol{\partial N_{01}
    \setminus \partial M}$ relative to $\g$, where~$N_{01}$ is a
  regular neighborhood of $R_-(\g) \cup \Gamma_{01}(f,v)$. Since
  $\partial N_{01} = R_-(\g) \cup \S_{01} \cup (\g \cap N_{01})$,
  where $\g \cap N_{01}$ is a disjoint union of annuli,
  \[
  \chi(R_-(\g)) + \chi(\S_{01}) = \chi(\partial N_{01}) =
  2\chi(N_{01}).
  \]
  As $N_{01}$ deformation retracts onto $R_-(\g) \cup
  \Gamma_{01}(f,v)$, we have
  \[
  \chi(N_{01}) = \chi(R_-(\g)) + c_0 - c_1 + c_2,
  \]
  where~$c_i$ is the number of $i$-cells in $\Gamma_{01}(f,v)$.  By
  construction, each point $p \in C_{01}(f,v)$ contributes $i(p)/2$ to
  $c_0 - c_1 + c_2$.  Indeed, for $p \in C_0(f)$ or $p$ an index~0-1-0
  birth-death-birth, $W^s(p) = \{p\}$, so $p$ contributes a single
  0-cell. If $p$ is an index~0-1 birth-death, then it contributes a
  0-cell and a 1-cell, while an index~1-2 birth-death contributes a
  0-cell, two 1-cells, and a 2-cell. For $p \in C_1(f)$ or $p$ an
  index~1-0-1 or~1-2-1 birth-death-birth, $W^s(p)$
  is an arc, and $p$ contributes a 0-cell and two 1-cells.
\end{proof}

\begin{corollary} \label{cor:chi} If $(f_0,v_0)$ and $(f_1,v_1)$ lie
  in the same path-component of $\FV_0(M,\g)$ or $\FV_1(M,\g)$, then
  $\chi(f_0,v_0) = \chi(f_1,v_1)$.
\end{corollary}

Note that this corollary is false for $\FV_{\le 1}(M,\g)$.

\begin{proof}
  Take a path $\{\,(f_t,v_t) \in \FV_i(M,\g) \,\colon\, t \in I\,\}$
  connecting $(f_0,v_0)$ and $(f_1,v_1)$.  In this family, the types
  of critical points in $C_{01}(f_t,v_t)$ remain unchanged; in
  particular, the local contributions $i(p_t)$ for $p_t \in
  C_{01}(f_t,v_t)$ are also constant. By Lemma~\ref{lem:euler-char},
  we obtain that $\chi(f_0,v_0) = \chi(f_1,v_1)$.
\end{proof}

We denote by $B_k^m(M,\g)$ and $E_k^m(M,\g)$ the space of those $(f,v)
\in B_k(M,\g)$ and $(f,v,\S) \in E_k(M,\g)$ for which~$f$ is Morse.
By slight abuse of notation, we also denote the projection $(f,v,\S)
\mapsto (f,v)$ by $\pi_k$.

\begin{proposition} \label{prop:fibration} Let $(M,\g)$ be a connected
  sutured manifold and $k \in \ZZ$. Then the map $\pi_k \colon
  E_k^m(M,\g) \to B_k^m(M,\g)$ is a fiber bundle with fibre
  $C^{\infty}(\S,\R)$ for a compact, connected, orientable surface~$\S$
  with $|\partial \S| = |s(\g)|$ and $\chi(\S) = k$.
\end{proposition}

\begin{proof}
  Given $(f,v,\S) \in E_k^m(M,\g)$, there is a neighborhood~$U$ of
  $\pi_k(f,v,\S) = (f,v)$ in $B_k^m(M,\g)$ such that for every
  $(f',v') \in U$, we have $\S \in \S(f',v')$. To see this, note that
  the surface~$\S$ separates $(M,\g)$ into two sutured compression bodies
  $(M_+,\g_+)$ and $(M_-,\g_-)$, one of which contains
  $\Gamma_{01}(f,v) \cup R_-(\g)$, and the other one contains
  $\Gamma_{23}(f,v) \cup R_+(\g)$.  If $(f',v')$ is sufficiently close
  to $(f,v)$, then~$\S$ is transverse to~$v'$. Furthermore,
  $\Gamma_{01}(f',v') \cup R_-(\g) \subset M_-$ and
  $\Gamma_{23}(f',v') \cup R_+(\g) \subset M_+$.  Indeed, no critical
  point can pass through $\S$ as long as $\S$ is transverse to the
  vector field, and every critical point of~$f$ is stable.

  We now construct a local trivialization $\phi \colon U \times
  C^\infty(\S,\R) \to \pi^{-1}_k(U)$.  Choose a smooth function $h
  \colon M \to I$ such that $h^{-1}(0) = R(\g)$.  Then~$\phi$ is
  defined by the formula $\phi((f',v'),s) = \S + s$, where $(f',v')
  \in U$ and $s \in C^\infty(\S,\R)$, and we view $\S(f',v')$ as an
  affine space over $C^\infty(\S,\R)$ via flowing along~$hv'$. The
  trivializations~$\phi$ define the topology on $E(M,\g)$ that makes
  $\pi_k \colon E_k^m(M,\g) \to B_k^m(M,\g)$ into a fiber bundle,
  and is compatible with the topology on each fiber $\S(f,v)$.
\end{proof}

As a corollary, $\pi_k \colon E_k^m(M,\g) \to B_k^m(M,\g)$ is a Serre
fibration. In particular, it satisfies the path-lifting property. For
example, together with Corollary~\ref{cor:chi}, for any family of
Morse-Smale vector fields $\{\,(f_\lambda,v_\lambda) \,\colon\, \lambda
\in \Lambda\,\}$, there is a corresponding family of surfaces
$\{\,\S_\lambda \,\colon\, \lambda \in \Lambda\,\}$ such that $\S_\lambda
\in \S(f_\lambda,v_\lambda)$ for every $\lambda \in \Lambda$. As the
fiber $C^{\infty}(\S,\R)$ is contractible, we can even extend a family
of splitting surfaces defined over a closed subset of $\Lambda$.

\subsection{Codimension-0} \label{sec:codimension-0}

In the previous section, we described how to obtain a contractible
space of Heegaard splittings of the sutured manifold $(M,\g)$ from a
separable pair $(f,v) \in \mathcal{FV}_{\le 2}(M,\g)$. If we also
assume that $(f,v)$ is Morse-Smale and we make an additional discrete
choice, then we can enhance these splittings to sutured diagrams. In
the opposite direction, we also show that every diagram of $(M,\g)$
with $\alphas \pitchfork \betas$ arises from a particularly simple
Morse-Smale pair $(f,v)$, and the space of such pairs is connected.

By Proposition~\ref{prop:grad-like}, every $\S \in \S(f,v)$ is a
Heegaard surface of $(M,\g)$. If $(f,v)$ is Morse-Smale, then, for
every $p \in C_1(f)$ and $q \in C_2(f)$, the intersections $W^u(p)
\cap \S$ and $W^s(q) \cap \S$ are embedded circles, and $W^u(p) \cap
\S$ is transverse to $W^s(q) \cap \S$.

\begin{definition}
  Suppose that $(f,v) \in \FV_0(M,\g)$, and let $\S \in \S(f,v)$. Then
  the triple $H(f,v,\S) = (\S,\alphas,\betas)$ is defined by taking
  the $\a$-curves to be $W^u(p) \cap \S$ for $p \in C_1(f)$ and the
  $\b$-curves to be $W^s(q) \cap \S$ for $q \in C_2(f)$.
\end{definition}

In general, $H(f,v,\S)$ is not a diagram of $(M,\g)$ as $\alphas$ and
$\betas$ might have too many components. We will refer to
such diagrams as \emph{overcomplete}, as we can remove some components
of $\alphas$ and $\betas$ to get a sutured diagram of $(M,\g)$.

\begin{definition} \label{def:overcomplete}
  Let $(M,\g)$ be a sutured manifold. We say that
  $(\S,\alphas,\betas)$ is an \emph{overcomplete diagram} of $(M,\g)$
  if
  \begin{enumerate}
  \item $\S \subset M$ is an oriented surface with $\partial \S =
    s(\g)$ as oriented 1-manifolds,
  \item the components of the 1-manifold $\alphas \subset \S$ bound
    disjoint disks to the negative side of~$\S$, and the components of
    the 1-manifold $\betas \subset \S$ bound disjoint disks to the
    positive side of~$\S$,
  \item if we compress $\S$ along $\alphas$, we get a surface isotopic
    to $R_-(\g)$ relative to $\g$, plus some 2-spheres that bound
    disjoint balls in~$M$, and
  \item if we compress $\S$ along $\betas$, we get a surface isotopic
    to $R_+(\g)$ relative to~$\g$, plus some 2-spheres that bound
    disjoint balls in~$M$.
  \end{enumerate}
\end{definition}

Overcomplete diagrams specify handle decompositions of $(M,\g)$ that
also include 0- and 3-handles.  Note that $\alphas$ and $\betas$ might
fail to be attaching sets because $\S \setminus \alphas$ and $\S
\setminus \betas$ can have some components disjoint from $\partial
\S$; however, all such components are planar.

To actually make the overcomplete diagram $H(f,v,\S)$ into a usual
Heegaard diagram
of $(M,\g)$, in addition to assuming that $(f,v)$ is Morse-Smale, we
also need to make a discrete choice. The Morse-Smale condition rules
out flows between two index~$i$ critical points for $i \in \{1,2\}$,
hence every point of $C_1(f) \cup C_2(f)$ has valence~2 in the graph
$\Gamma(f,v)$.

\begin{definition}\label{def:Gamma-pm}
Let $\Gamma_-(f,v)$ be the graph obtained from
$\Gamma_{01}(f,v)$ by identifying all vertices lying in $R_-(\g)$ and
deleting the vertices at $C_1(f)$ (and merging the two adjacent
edges into one edge).
So the vertices of $\Gamma_-(f,v)$
are the points of $C_0(f)$, plus at most one vertex for $R_-(\g)$, and
its edges correspond to $W^s(p)$ for $p \in C_1(f)$.  In other words,
$\Gamma_-(f,v)$ is obtained from the relative CW complex
$(\Gamma_{01}(f,v) \cup R_-(\g), R_-(\g))$ by taking the factor CW
complex $(\Gamma_{01}(f,v) \cup R_-(\g))/R_-(\g)$ and removing the
vertices at $C_1(f)$.  Similarly, the graph $\Gamma_+(f,v)$ is
obtained by collapsing $\Gamma_{23}(f,v) \cap R_+(\g)$ to a single
point, and deleting the vertices at $C_2(f)$.  The edges of
$\Gamma_+(f,v)$ correspond to $W^u(q)$ for $q \in C_2(f)$.
\end{definition}

\begin{definition} \label{def:diagram} Suppose that $(f,v) \in
  \FV_0(M,\g)$.  Let $T_{\pm}$ be a spanning tree of
  $\Gamma_{\pm}(f,v)$, and choose a splitting surface $\S \in
  \S(f,v)$. Then the sutured diagram $H(f,v,\S,T_-,T_+)$ is
  defined by taking the $\alpha$-curves to be $W^u(p) \cap \S$, where
  $p \in C_1(f)$ and $W^s(p)$ is not an edge of $T_-$. Similarly, the
  $\beta$-curves are the intersections $W^s(q) \cap \S$, where $q \in
  C_2(f)$ and $W^u(q)$ is not an edge of $T_+$.
\end{definition}

For brevity, we will often write $H(f,v,\S,T_\pm)$ for
$H(f,v,\S,T_+,T_-)$.  Of course, a different choice of $T_-$ gives a
diagram that is $\a$-equivalent to the original, while changing $T_+$
gives a diagram that is $\b$-equivalent.

In the opposite direction, given a Heegaard surface~$\S$ for $(M,\g)$,
we will show that one can find a particularly nice pair $(f,v) \in
\FV_0(M,\g)$ such that $\S \in \S(f,v)$.

\begin{definition} \label{def:simple} We say that a Morse function~$f$
  on $(M,\g)$ is \emph{simple} if
  \begin{enumerate}
  \item $C_i(f) = \emptyset$ for $i \in \{ 0,3 \}$,
  \item $f(p) < 0$ for every $p \in C_1(f)$,
  \item $f(q) > 0$ for every $q \in C_2(f)$.
  \end{enumerate}
  We call a pair $(f,v) \in V(M,\g)$ \emph{simple} if $f$ is simple
  and, in addition,
  \begin{enumerate}[resume]
  \item for every $p \in C_1(f)$, there is a local coordinate system
    $(x_1,x_2,x_3)$ around $p$ in which $f = -x_1^2 + x_2^2 + x_3^2 +
    f(p)$ and $v$ has coordinates $(-2x_1,2x_2,2x_3)$,
  \item for every $q \in C_2(f)$, there is a local coordinate system
    $(x_1,x_2,x_3)$ around $q$ in which $f = -x_1^2 - x_2^2 + x_3^2 +
    f(q)$ and $v$ has coordinates $(-2x_1,-2x_2,2x_3)$.
  \end{enumerate}
\end{definition}

Let $f$ be a simple Morse function. Then observe that for any
gradient-like vector field~$v$ for~$f$, the pair $(f,v)$ is
separable. Indeed, $f(p) < f(q)$ for every $p \in C_1(f)$ and $q \in
C_2(f)$, so there is no flow-line of~$v$ from~$q$ to~$p$. Furthermore,
the surface $\S = f^{-1}(0)$ is a Heegaard surface that separates
$(f,v)$. If, in addition, $(f,v)$ is Morse-Smale, then we can uniquely
complete this to a diagram $H(f,v) = (\S,\alphas,\betas)$ of $(M,\g)$
using Definition~\ref{def:diagram}. Indeed, since $\Gamma_\pm(f,v)$ is
a wedge of circles, it has a unique spanning tree~$T_\pm$ consisting
of a single vertex. The diagram $H(f,v)$ is then
$H(f,v,\S,T_-,T_+)$. More explicitly, the $\a$-curves are $W^u(p) \cap
\S$ for $p \in C_1(f)$, and the $\b$-curves are $W^s(q) \cap \S$ for
$q \in C_2(f)$.

The following result is basically standard Morse theory, but since it
provides the crucial link between sutured Heegaard diagrams and Morse
functions, we give a complete proof.

\begin{proposition} \label{prop:existence} Let $(\S,\alphas,\betas)$
  be a diagram of the sutured manifold $(M,\g)$, and suppose that
  $\alphas \pitchfork \betas$. Then there exists a simple pair $(f,v)
  \in \FV_0(M,\g)$ such that $H(f,v) = (\S,\alphas,\betas)$.
\end{proposition}

\begin{proof}
  Given an arbitrary attaching set $\deltas \subset \S$, recall that
  $C(\deltas)$ is the sutured compression body obtained by attaching
  2-handles to $\S \times I$ along $\deltas \times \{1\}$.  We are
  going to construct a Morse function $f_{\delta} \colon C(\deltas)
  \to I$ with only index~2 critical points, and a gradient-like
  vector field $v_{\delta}$ for $f_{\delta}$.

  Consider the Morse function $h(x) = -x_1^2 - x_2^2 + x_3^2 + 1/2$
  and its gradient $v(x) = (-2x_1, -2x_2, 2x_3)$ on the unit disk
  $D^3$.  Our model 2-handle will be $Z = h^{-1}(I)$; this is a
  3-manifold with boundary and corners. Let $Z^- = h^{-1}(0)$ and $Z^+
  = h^{-1}(1)$.  The boundary of~$Z$ is $Z^- \cup Z^+ \cup A$, where
  $Z^-$ is a connected hyperboloid, hence topologically an
  annulus. The surface~$Z^+$ is the disjoint union of two disks, while
  $A = Z \cap S^2$ is the disjoint union of two annuli.  The attaching
  circle of~$Z$ is the curve $a = Z^- \cap \{x_3 = 0\}$. We isotope
  $v$ near $A$ such that it stays gradient-like for $h$, and becomes
  tangent to $A$.  The function~$h$ has a single non-degenerate
  critical point of index~2 at the origin, with stable manifold
  $W^s(0) = Z \cap \{x_3 = 0\}$.

  Pick an open regular neighborhood $N$ of $\deltas$, and let $R = (\S
  \setminus N) \times I$.  We define $f_{\delta}$ on $R$ to be the
  projection~$t$ onto the $I$-factor, and $v_{\delta}$ on $R$ is
  simply $\partial/\partial t$.  If the components of $\deltas$ are
  $\delta_1, \dots, \delta_d$, then we denote by~$N_i$ the component
  of~$N$ containing~$\delta_i$.  For each $i \in \{\, 1, \dots,
  d\,\}$, take a copy~$Z_i$ of~$Z$, together with the function $h_i$
  and the vector field~$v_i$ constructed above. We glue~$Z_i$ to~$R$
  using a diffeomorphism $A_i \to \partial N_i \times I$ that maps the
  circles $h^{-1}(t) \cap A$ to $\partial N_i \times \{t\}$ for every
  $t \in I$.  So we can extend~$f_{\delta}$ to~$Z_i$ with~$h_i$ and
  $v_{\delta}$ with~$v_i$.  After gluing $Z_1, \dots, Z_d$ to~$R$, we
  get a compression body diffeomorphic to $C(\deltas)$, together with
  the required pair $(f_{\delta}, v_{\delta})$. Note that here we
  identify $(\S \setminus N) \times \{0\}$ with $\S \setminus N$ and
  $Z_i^-$ with $\overline{N}_i$, so that $C_-(\deltas) = \S$. In
  addition, the attaching circle~$a_i$ of~$Z_i$ is identified with~$\a_i$.
  Let~$p_i$ be the center of the 2-handle~$Z_i$.  By
  construction, the stable manifold $W^s(p_i) \cap \S = \a_i$.

  The surface $\S$ cuts $(M,\g)$ into two sutured compression
  bodies. Call these~$C_-$ and~$C_+$, such that $R_{\pm}(\g) \subset
  C_{\pm}$. There are diffeomorphisms $d_- \colon C_- \to C(\alphas)$
  and $d_+ \colon C_+ \to C(\betas)$ that are the identity on~$\S$.
  Then we define the Morse function~$f$ on $(M,\g)$ by taking
  $-f_{\a} \circ d_-$ on $C_-$ and $f_{\b} \circ d_+$ on $C_+$, and
  smoothing along~$\S$. Then the vector field $v$ that agrees with
  $-(d_-)_*^{-1} \circ v_{\a} \circ d_-$ on $C_-$ and with
  $(d_+)_*^{-1} \circ v_{\b} \circ d_+$ on $C_+$ is gradient-like for
  $f$, and is Morse-Smale since $\alphas \pitchfork \betas$.  It
  follows from the construction that $H(f,v) = (\S,\alphas,\betas)$.
\end{proof}

Notice that the above construction is almost completely canonical, in
the sense that the various choices can be easily deformed into each
other. In fact, we have the following stronger statement.

\begin{proposition} \label{prop:connected-MS} Let
  $(\S,\alphas,\betas)$ be a diagram of the sutured manifold $(M,\g)$
  such that $\alphas \pitchfork \betas$, and recall that $\FV_0(M,\g)$
  is endowed with the $C^{\infty}$-topology.  Then the subspace of
  simple pairs $(f,v) \in \FV_0(M,\g)$ for which $H(f,v) =
  (\S,\alphas,\betas)$ is connected.
\end{proposition}

\begin{proof}
  Suppose that the simple pairs $(f,v)$ and $(g,w)$ in $\FV_0(M,\g)$
  satisfy $H(f,v) = (\S,\alphas,\betas)$ and $H(g,w) =
  (\S,\alphas,\betas)$.

  As in Proposition~\ref{prop:existence}, the surface~$\S$ cuts
  $(M,\g)$ into two compression bodies $C_-$ and $C_+$, such that
  $R_{\pm}(\g) \subset C_{\pm}$.  We will describe how to connect
  $(f,v)|_{C_+}$ and $(g,w)|_{C_+}$; the deformation on $C_-$ is
  analogous. Since $\S$ remains the zero level set and $v$, $w$ stay
  transverse to $\S$ throughout, it is easy to glue the deformations
  on $C_+$ and~$C_-$ together.  We construct this deformation in
  several steps.

  First, some terminology. Let $\{\,d_t \,\colon\, t \in I\,\}$ be an
  isotopy of $C_+$ such that $d_0 = \text{Id}_{C_+}$.  Then we call
  the family $(f_t,v_t) = (f \circ d_t^{-1}, (d_t)_* \circ v \circ
  d_t^{-1})$ the isotopy of $(f,v)$ along~$d_t$. Note that $v_t$ is a
  gradient-like vector field for $f_t$.

  \textbf{Step 1.} In this step, we move the stable manifolds of
  $(f,v)$ until they coincide with the stable manifolds of $(g,w)$.
  We denote the components of $\betas$ by $\b_1, \dots, \b_k$. In
  addition, let $C_2(f) = \{\,q_1,\dots,q_k\,\}$ and $C_2(g) =
  \{\,q_1',\dots,q_k'\,\}$, enumerated such that
  \[
  W^s(q_i,v) \cap \S = W^s(q_i',w) \cap \S = \b_i.
  \]
  By Lemma~\ref{lem:pi2}, we have $\pi_2(C_+) = 0$. Hence, using
  cut-and-paste techniques, the disks $W^s(q_i,v)$ and $W^s(q_i',w)$
  are isotopic relative to their boundary.  So there is an isotopy
  $\{\,d_t \,\colon t\, \in I\,\}$ of $C_+$ fixing $\partial C_+$ such
  that $d_0 = \text{Id}_{C_+}$ and $d_1(W^s(q_i,v)) = W^s(q_i',w)$;
  furthermore, $d_1(q_i) = q_i'$. Isotoping $(f,v)$ along $d_t$, we
  get a path of simple pairs $(f_t,v_t)$ in $\FV_0(M,\g)$, all
  compatible with the diagram $(\S, \alphas, \betas)$. Replacing
  $(f,v)$ with $(f_1,v_1)$, we can assume that $W^s(q_i,v) =
  W^s(q_i',w)$ and $q_i = q_i'$ for every $i \in \{\,1 \dots, k \,\}$.
  The further deformation of $(f,v)$ will preserve these properties,
  so from now on we will write $W^s(q_i)$ for every $q_i \in C_2(f) =
  C_2(g)$.

  \textbf{Step 2.} We now isotope $(f,v)$ until it coincides with
  $(g,w)$ in a neighborhood of the critical points, without ruining
  what we have already achieved in Step~1.  Let $i \in \{\,1 \dots,
  k\,\}$. Since both $(f,v)$ and $(g,w)$ are simple, there are balls
  $N_1$ and $N_2$ centered at~$q_i$ and coordinate systems $x \colon
  N_1 \to \R^3$ and $y \colon N_2 \to \R^3$ such that $f = -x_1^2 -
  x_2^2 + x_3^2 + f(q_i)$ and~$v$ has coordinates $(-2x_1,-2x_2,2x_3)$
  in $N_1$, while $g = -y_1^2 - y_2^2 + y_3^2 + g(q_i)$ and $w$ has
  coordinates $(-2y_1,-2y_2,2y_3)$ in $N_2$.  Choose an $\varepsilon >
  0$ so small that the disks $D_1 = \{\,|x| \le \varepsilon\,\}$ and
  $D_2 = \{\,|y| \le \varepsilon\,\}$ both lie in $N_1 \cap
  N_2$. Consider the diffeomorphism $d \colon D_1 \to D_2$ given by
  the formula $y^{-1} \circ x$. Then $d(D_1 \cap W^s(q_i)) = D_2 \cap
  W^s(q_i)$, as $D_1 \cap W^s(q_i)$ is given by the equation $x_3 =
  0$, while $D_2 \cap W^s(q_i)$ by $y_3 = 0$. We can choose an isotopy
  $e_t \colon D_1 \to N_1 \cap N_2$ such that $e_0 = \text{Id}_{D_1}$
  and $e_1 = d$; furthermore,
  \[
  e_t(D_1 \cap W^s(q_i)) \subset W^s(q_i)
  \]
  and $e_t(q_i) = q_i$ for every $t \in I$. This can be extended to an
  isotopy $d_t \colon C_+ \to C_+$ such that $d_t|_{D_1} = e_t$, the
  diffeomorphism~$d_t$ is the identity outside $N_1 \cap N_2$, and
  $d_t(W^s(q_i)) = W^s(q_i)$. If we isotope $(f,v)$ along~$d_t$, we
  get a pair $(f_1,v_1)$ that agrees with $(g,w)$ in $D_2$. Repeating
  this process for every~$q_i$, we can assume that $(f,v)$ and $(g,w)$
  agree in a neighborhood~$N$ of all the critical points $q_1, \dots,
  q_k$ (where $N$ is the union of the disks $D_2$ for each~$q_i$).

  \textbf{Step 3.} In this step, we arrange that $v|_{W^s(q_i)} =
  w|_{W^s(q_i)}$ for every $i \in \{\,1,\dots,k\,\}$.  By Step~2, we
  know that $v$ and $w$ coincide on the disk $B = N \cap W^s(q_i)$,
  and that they are transverse to $\partial B$. Let $A$ be the annulus
  $W^s(q_i) \setminus B$.  Take a regular neighborhood of $W^s(q_i)$
  of the form $W^s(q_i) \times [-1,1]$, where $W^s(q_i)$ is identified
  with $W^s(q_i) \times \{0\}$. We are going to construct an isotopy
  of $C_+$ that is supported in $A \times [-1,1]$ that takes $v$ to
  $w$.  For every $p \in \partial B$, we denote by $\nu_v(p,t)$ the
  flow-line of $-v$ starting at $p$ and ending at $\partial
  W^s(q_i)$. Here $t$ lies in some interval $[0,T(v,p)]$.  Similarly,
  $\nu_w(p,t)$ denotes the flow-line of $-w$ starting at $p$ and
  defined for $t \in [0,T(w,p)]$. After smoothly rescaling $v$ inside
  $A \times [-1,1]$ such that it is unchanged in a neighborhood of
  $\partial (A \times [-1,1])$, we can assume that $T(v,p) = T(w,p)$
  for every $p \in \partial B$.  Let $a \colon A \to A$ be the
  diffeomorphism defined by the formula
  \[
  a(\nu_v(p,t)) = \nu_w \left(p, t \right)
  \]
  for $t \in [0,T(v,p)]$.  This has the property that $a_* \circ v
  \circ a^{-1} = w$. There is an isotopy $\{\,a_t \,\colon\, t \in I\,\}$
  of $A$ that is fixed on $\partial B$, and such that $a_0 =
  \text{Id}_A$ and $a_1 = a$. We extend this to $A \times [-1,1]$ by
  the formula $a_t(x,s) = (a_{r(s)t}(x),s)$, where $r \colon \R \to I$
  is a bump function that is zero outside $[-1,1]$ and such that $r(0)
  = 1$. Finally, we extend $a_t$ to the whole of $C_+$ as the
  identity. Then, isotoping $(f,v)$ along $a_t$, we get a family
  $\{\,(f_t,v_t) \,\co\, t \in I \,\}$ such that $v_1|_{W^s(q_i)} =
  w|_{W^s(q_i)}$. Note that $W^s(q_i)$ is invariant
  under~$a_t$. Furthermore, even though $a_t$ is not the identity on
  $\S$, the field~$v_t$ stays transverse to~$\S$ throughout. In fact,
  $v_t$ can be made invariant under $a_t$ if we first make $v$ and $w$
  agree in a neighborhood of $\partial A$; so we can glue the
  deformation with the one on $C_-$.

  Note that we do not claim that $f = g$ anywhere outside a
  neighborhood of the critical points.  We will return to this in the
  last step.

  \textbf{Step 4.} We now make $v$ and $w$ agree on a product
  neighborhood $W^s(q_i) \times [-1,1]$ of every stable manifold
  $W^s(q_i)$. Fix $i \in \{\,1,\dots,k\,\}$, and let the ball $B$ and
  the annulus $A$ be as in Step~3. We already know that $(f,v)$ and
  $(g,w)$ agree on a neighborhood~$N$ of~$B$.  Since $v$ and $w$ agree
  on $A$ and have no zeroes there, there is a thin product
  neighborhood of $W^s(q_i)$ diffeomorphic to $W^s(q_i) \times [-2,2]$
  such that in this neighborhood the linear homotopy $(1-t)v + tw$
  from $v$ to $w$ stays gradient-like for~$f$ throughout. In addition,
  we choose this neighborhood so thin that $B \times [-2,2] \subset
  N$.  Let $\vartheta \colon \R \to I$ be a smooth function that is
  zero outside $[-2,2]$ and is identically one in $[-1,1]$. Then we
  define the isotopy $v_t$ of $v$ to be the identity outside $W^s(q_i)
  \times [-2,2]$, and
  \[
  v_t(x,s) = \left(1-\vartheta \left(s \right)t \right)v + \left(\vartheta \left(s \right)t \right)w
  \]
  for every $(x,s) \in W^s(q_i) \times [-2,2]$.  Then $v_t$ is
  gradient-like for $f$ for every $t \in I$, and~$v_1$ agrees with $w$
  on $W^s(q_i) \times [-1,1]$.

  \textbf{Step 5.} In this step, we homotope $v$ to $w$ on the rest of
  $C_+$.  Let $P$ be the manifold obtained from
  \[
  C_+ \setminus \bigcup_{i=1}^k (W^s(q_i) \times
  [-\varepsilon,\varepsilon])
  \]
  by rounding the corners. Here, we choose $\varepsilon$ so small that
  (after possibly a small perturbation) $v$ and $w$ point into $P$
  along $P_- = \partial P \setminus \text{Int}(\g \cup
  R_+(\g))$. Notice that $P_-$ consists of $\S \setminus \bigcup_{i=1}^k (\a_i \times
  [-\varepsilon,\varepsilon])$ and $W^s(q_i) \times
  \{-\varepsilon,\varepsilon\}$ for $i \in \{\,1, \dots, k\,\}$. By
  construction, the vector fields $v$ and $w$ point into $P$ along
  $\S$ and $B \times \{-\varepsilon,\varepsilon\}$ for every disk $B$
  of the form $N \cap W^s(q_i)$.  As $v$ and $w$ are tangent to the
  annuli $A$, such an $\varepsilon$ and perturbation clearly exist.
  Note that $v$ and $w$ coincide along $P_-$.

  The sutured manifold $(P, \g \cap P)$ is diffeomorphic to the
  product sutured manifold $(R_+(\g) \times I, \partial R_+(\g) \times
  I)$. For every $x \in P_-$, let $\phi_v(x,t)$ be the flow-line of
  $v$ starting at $x$ and defined for $t \in [0,T(v,x)]$. Similarly,
  let $\phi_w(x,t)$ be the flow-line of $w$ starting at $x$ and
  defined for $t \in [0,T(w,x)]$. After smoothly rescaling $v$ inside
  $P$ such that it is unchanged in a neighborhood of $\partial P$, we
  can assume that $T(v,x) = T(w,x)$ for every $x \in P_-$. As in
  Step~3, we define a diffeomorphism $d \colon P \to P$ by the formula
  \[
  d(\phi_v(x,t)) = \phi_w \left(x,t \right)
  \]
  for $t \in [0, T(v,x)]$.  This satisfies $d_* \circ v \circ d^{-1} =
  w$, and $d$ is the identity on $P_-$ and~$\g \cap P$.  Since any
  diffeomorphism of a product sutured manifold $(S \times I, \partial
  S \times I)$ that fixes $S \times \{0\}$ and $\partial S \times I$
  is isotopic to the identity through such diffeomorphisms, there is
  an isotopy $d_t$ of $P$ that fixes $P_-$ and $\g \cap P$, and $d_0 =
  \text{Id}_P$. Since $v$ and $w$ agree on each $W^s(q_i) \times
  [-1,1]$, we can extend $d_t$ to every $W^s(q_i) \times
  [-\varepsilon, \varepsilon]$ as the identity.  If we isotope $(f,v)$
  along $d_t$, we get a path of pairs $(f_t,v_t)$ such that $v_1 = w$.

  \textbf{Step 6.} In this final step, we achieve $f = g$. For this,
  the linear homotopy $\{\, f_t = (1-t)f + tg \,\co\, t \in I \,\}$
  works. Indeed, as $f$ and $g$ coincide in a neighborhood~$N$ of the
  critical points, $f_t|_N = f|_N$ for every $t \in I$.  In addition,
  \[
  v(f_t) = (1-t)v(f) + t v(g) = (1-t) v(f) + t w(g) > 0
  \]
  away from $\{\,q_1,\dots,q_k\,\}$, the common critical set of $f$
  and $g$. So $f_t$ has the same index~2 non-degenerate critical
  points as $f$ and is hence Morse, $v$ is a gradient-like vector
  field for $f_t$, and the pair $(f_t,v)$ is simple for every $t \in
  I$.
\end{proof}

\subsection{Codimension-1: Overcomplete diagrams}
\label{sec:transl-codim-1}

We start this section by proving a type of isotopy extension lemma for
families of Heegaard diagrams.

\begin{lemma} \label{lem:isot-ext} Suppose that
  $\{\,(\S_t,\alphas_t,\betas_t) \,\colon\, t \in I\,\}$ is a smooth
  1-parameter family of possibly overcomplete Heegaard diagrams in
  $(M,\g)$ such that $\alphas_t \pitchfork \betas_t$ for every $t \in
  I$.  Then there is an isotopy $D \colon M \times I \to M$ such that
  \[
  d_t(\S_0,\alphas_0,\betas_0) = (\S_t, \alphas_t,\betas_t)
  \]
  for every $t \in I$, and~$d_t$ fixes $\partial M$ pointwise (where
  $d_t = D(\cdot,t)$).  In particular, $d_1|_{\S_0} \colon \S_0 \to
  \S_1$ is isotopic to the identity in~$M$.  The space of such
  isotopies is contractible, so the space of diffeomorphisms that
  arise as $d_1$ for such an isotopy~$D$ is path-connected.

  An analogous statement holds if $\{\,\S_t \colon t \in I\,\}$ is a
  1-parameter family of Heegaard surfaces of $(M,\g)$. In particular,
  there is an induced diffeomorphism $d_1 \colon \S_0 \to \S_1$,
  well-defined up to isotopy.
\end{lemma}

\begin{proof}
  In $M \times I$, consider the submanifold
  \[
  \S_* = \bigcup_{t \in I}
  \S_t \times \{t\},
  \]
  which, in turn, contains the submanifolds $\alphas_* = \bigcup_{t
    \in I} \alphas_t \times \{t\}$ and $\betas_* = \bigcup_{t \in I}
  \betas_t \times \{t\}$.  (The fact that these are smooth
  submanifolds is in fact our definition that the family of Heegaard
  diagrams is smooth.)  Let $\mathcal{F}$ be the horizontal foliation
  of $M \times I$ by leaves $M \times \{t\}$, and we coorient
  $\mathcal{F}$ by $\partial/\partial t$.  The condition $\alphas_t
  \pitchfork \betas_t$ implies that $\alphas_* \cap \betas_*$ is a
  collection of arcs transverse to $\mathcal{F}$.  Furthermore,
  $\alphas_*$, $\betas_*$, and $\S_*$ are also transverse to
  $\mathcal{F}$.

  Pick a smooth vector field $\nu$ in the tangent bundle $T(\alphas_*
  \cap \betas_*)$ positively transverse to $\mathcal{F}$.  This can be
  extended to first a vector field in $T\alphas_*$ and $T\betas_*$
  positively transverse to $\mathcal{F}$, then to a field in $T\S_*$
  positively transverse to $\mathcal{F}$ such that $\nu_{s(\gamma)
    \times I} = \partial/\partial t$. Finally, extend the vector field
  to $M \times I$ positively transverse to $\mathcal{F}$ such that
  $\nu|_{\partial M} = \partial/\partial t$.  We also denote this
  extension by~$\nu$.  After normalizing~$\nu$ such that $\nu(t) = 1$,
  we can assume that its flow preserves the foliation~$\mathcal{F}$.
  Then the diffeomorphism $d_t \colon M \to M$ is defined by flowing
  along~$\nu$ from $M \times \{0\}$ to $M \times \{t\}$.  By
  construction, $d_t(\S_0,\alphas_0,\betas_0) =
  (\S_t,\alphas_t,\betas_t)$.  Note that the embedding $\iota_0 \colon
  \S_0 \hookrightarrow M$ is ambient isotopic to $\iota_1 \circ d_1
  \colon \S_0 \hookrightarrow M$ relative to $s(\g)$, where $\iota_1$
  is the embedding of~$\S_1$ in~$M$. So~$d_1$ is indeed isotopic to
  the identity in~$M$.

  On the other hand, every isotopy~$D$ arises from the above
  construction. Indeed, given~$D$, take~$\nu$ to be the velocity
  vector fields of the curves $t \mapsto (d_t(x),t)$ for $x \in M$.
  The space of such~$\nu$ is convex, hence contractible, so the space
  of such isotopies~$D$ is also contractible.

  The proof of the last statement about families of Heegaard surfaces
  is completely analogous, but simpler as our isotopies now do not
  have to preserve sets of attaching curves $\alphas_t$ and
  $\betas_t$.
\end{proof}

\begin{lemma} \label{lem:isot-homot} Let $\{\,\HD_t \,\colon\, t \in
  I\,\}$ and $\{\,\HD_t' \,\colon\, t \in I\,\}$ be 1-parameter families
  of possibly overcomplete Heegaard diagrams of $(M,\g)$, both
  connecting $\HD_0$ and $\HD_1$. If the two families are homotopic
  relative to their endpoints, then the induced diffeomorphisms $d_1$,
  $d_1' \colon M \to M$ (in the sense of Lemma~\ref{lem:isot-ext}) are
  isotopic through diffeomorphisms mapping $\HD_0$ to $\HD_1$.  An
  analogous statement holds for homotopic families of Heegaard
  surfaces $\{\,\S_t \,\colon\, t \in I\,\}$ and $\{\,\S_t' \,\colon\, t \in
  I\,\}$; i.e., the induced diffeomorphisms $d_1$, $d_1' \colon \S_0 \to
  \S_1$ are isotopic.
\end{lemma}

\begin{proof}
  Let $\HD_{t,u} = (\S_{t,u},\alphas_{t,u},\betas_{t,u})$ for $(t,u)
  \in I \times I$ be the homotopy between $\{\HD_t\}$ and
  $\{\HD_t'\}$; i.e., $\HD_{t,0} = \HD_t$ and $\HD_{t,1} = \HD_t'$ for
  $t \in I$, while $\HD_{i,u} = \HD_i$ for $i \in \{0,1\}$ and $u \in
  I$.  As in the proof of Lemma~\ref{lem:isot-ext}, we can construct a
  vector field~$\nu$ on $M \times I \times I$ such that $\nu(u) = 0$,
  $\nu(t) = 1$ (in particular, it is transverse to the foliation of $M
  \times I \times I$ with leaves $M \times \{t\} \times I$), and which
  is tangent to the submanifolds $\bigcup_{t,u \in I} \S_{t,u}$,
  $\bigcup_{t,u \in I} \alphas_{t,u}$, and $\bigcup_{t,u \in I}
  \betas_{t,u}$.  Then the flow of~$\nu$ defines a diffeomorphism
  \[
  g_u \colon M \times \{0\}\times \{u\} \to M \times \{1\} \times
  \{u\}
  \]
  that maps $\HD_{0,u} = \HD_0$ to $\HD_{1,u} = \HD_1$ for every $u
  \in I$.  Notice that $g_0 = d_1$ and $g_1 = d_1'$ (up to isotopy).
  Hence $\{\, g_u \colon u \in I \,\}$ provides the required isotopy
  between the diffeomorphisms $d_1$ and $d_1'$.
\end{proof}

\begin{lemma} \label{lem:isotopy} Suppose that $\{\,(f_t,v_t) \in \FV_0(M,\g)
  \,\colon\, t \in I\,\}$ is a 1-parameter family of gradient-like vector
  fields, and let $\S_i \in \S(f_i,v_i)$ be a Heegaard surface of
  $(M,\g)$ for $i \in \{0,1\}$. Choose a spanning tree $T^0_\pm$ of
  $\Gamma_\pm(f_0,v_0)$. The isotopy $\Gamma(f_t,v_t)$ takes $T^0_\pm$
  to a spanning tree $T^1_\pm$ of $\Gamma_\pm(f_1,v_1)$, and consider
  the diagrams $(\S_0,\alphas_0,\betas_0) = H(f_0,v_0,\S_0,T^0_\pm)$
  and $(\S_1,\alphas_1,\betas_1) = H(f_1, v_1, \S_1, T^1_\pm)$. Then
  there is a (non-unique) induced diffeomorphism
  \[
  d \colon (\S_0,\alphas_0,\betas_0) \to (\S_1,\alphas_1,\betas_1)
  \]
  isotopic to the identity in $M$, and the space of such
  diffeomorphisms is path-connected. If we do not pick spanning trees,
  we obtain a similar statement for overcomplete diagrams.
\end{lemma}

\begin{remark}
  If $T^0_\pm$ and $T^1_\pm$ are not related as above, then one can
  get from $(\S_0,\alphas_0,\betas_0)$ to $(\S_1,\alphas_1,\betas_1)$
  via a diffeomorphism isotopic to the identity in $M$, an
  $\a$-equivalence, and a $\b$-equivalence.
\end{remark}

\begin{proof}
  Note that if $\S \in \S(f_t,v_t)$ for some $t \in I$, then $\S$ also
  separates $(f_s,v_s)$ for every~$s$ sufficiently close to~$t$.
  Indeed, each $(f_s,v_s)$ is Morse-Smale, hence separable, so $\S$
  stays separating as long as $v_s$ is transverse to~$\S$, which is an
  open condition. By the compactness of $I$, there is a sequence $0 =
  t_0 < t_1 < \dots < t_n = 1$ and surfaces $\S_{t_i} \in
  \S(f_{t_i},v_{t_i})$ for every $i \in \{\, 0, \dots, n-1 \,\}$ such
  that $\S_{t_i} \in \S(f_s,v_s)$ for every $s \in [t_i,t_{i+1}]$.

  By the discussion preceding Proposition~\ref{prop:grad-like}, we can
  view $\S(f_t,v_t)$ as an affine space over~$C^\infty(\S)$ for any
  $\S \in \S(f_t,v_t)$. For this, choose a smooth function $h \colon M
  \to I$ such that $h^{-1}(0) = R(\g)$, and let $i \in \{\, 2, \dots,
  n \,\}$.  As both $\S_{t_{i-1}}$, $\S_{t_i} \in \S(f_{t_i},v_{t_i})$,
  we can talk about the difference $d_{\S_{t_i},\S_{t_{i-1}}} \in
  C^{\infty}(\S_{t_{i-1}})$, obtained by flowing along~$hv_{t_i}$.
  Let $\varphi_i \colon \R \to I$ be a smooth function such that
  $\varphi_i(t) = 0$ for $t \le t_{i-1}$ and $\varphi_i(t) = 1$ for $t
  \ge t_i$.  For $t \in [t_{i-1},t_i]$, let
  \[
  \S_t = \S_{t_{i-1}} + \varphi_i \left( t \right) d_{\S_{t_i},\S_{t_{i-1}}},
  \]
  where the sum is taken using the flow of~$hv_t$.  Then $\S_t$ is a
  smooth 1-parameter family of surfaces connecting $\S_0$ to~$\S_1$
  such that $\S_t \in \S(f_t,v_t)$ for every $t \in I$. (Note that
  this path-lifting also follows from
  Proposition~\ref{prop:fibration}, which claims that $E_k^m(M,\g) \to
  B_k^m(M,\g)$ is a fibre bundle with connected fiber
  $C^\infty(\S,\RR)$.)

  The isotopy $\{\,\Gamma(f_s,v_s) \,\colon\, 0 \le s \le t \,\}$ takes
  $T^0_\pm$ to a spanning tree $T^t_\pm$ of $\Gamma_\pm(f_t,v_t)$.
  Then $(\S_t,\alphas_t,\betas_t) = H(f_t,v_t,\S_t,T^t_\pm)$ provides
  a smooth 1-parameter family of diagrams connecting
  $(\S_0,\alphas_0,\betas_0)$ and $(\S_1,\alphas_1,\betas_1)$.  Since
  $(f_t,v_t)$ is Morse-Smale for every $t \in I$, we have $\alphas_t
  \pitchfork \betas_t$, so we can apply Lemma~\ref{lem:isot-ext} to
  obtain an isotopy $D \colon M \times I \to M$ such that $d_t(\S_0,
  \alphas_0,\betas_0) = (\S_t,\alphas_t,\betas_t)$ for every $t \in
  I$.  If we take $d$ to be $d_1$, then $d$ is isotopic to the
  identity in $M$.

  Also by Lemma~\ref{lem:isot-ext}, the diffeomorphism $d_1$ is unique
  up to isotopy in the space of diffeomorphisms mapping
  $(\S_0,\alphas_0,\betas_0)$ to $(\S_1,\alphas_1,\betas_1)$ once we
  fix the family of surfaces $\S_t$. For a different family of
  surfaces $\S_t' \in \S(f_t,v_t)$ connecting $\S_0$ and $\S_1$,
  consider the homotopy $\S_{t,u} = \S_t + u d_{\S_t',\S_t}$ for $t,u
  \in I$ (where the sum means flowing along $hv_t$). Then $\S_{t,0} =
  \S_t$ and $\S_{t,1} = \S_t'$ for every $t \in I$. Applying
  Lemma~\ref{lem:isot-homot} to the homotopy $\HD_{t,u} =
  H(f_t,v_t,\S_{t,u},T^t_\pm)$, we obtain that $d_1$ is also unique up
  to isotopy if we are allowed to vary the path $t \mapsto \S_t$.
\end{proof}

Summarizing the proof of Lemma~\ref{lem:isotopy},
the diffeomorphism induced by the family $\{\, (f_t,v_t) \,\colon\, t
\in I \,\}$ is obtained by first picking an arbitrary family of
surfaces $\S_t \in \S(f_t,v_t)$, and then applying
Lemma~\ref{lem:isot-ext} to the diagrams $H(f_t,v_t,\S_t,T_\pm^t)$. We
have the following analogue of Lemma~\ref{lem:isotopy} for 1-parameter
families in $\FV_1(M,\g)$, which is somewhat weaker as an element of
$\FV_1(M,\g)$ does not induce a Heegaard diagram.

\begin{lemma} \label{lem:isotopy-cod-1} Let $\{\, (f_t,v_t) \in
  \FV_1(M,\g) \,\colon\, t \in I \,\}$ be a 1-parameter family, and let
  $\S_i \in \S(f_i,v_i)$ be a Heegaard surface of $(M,\g)$ for $i \in
  \{0,1\}$. This family induces a diffeomorphism $d \colon \S_0 \to
  \S_1$ that is well-defined up to isotopy. Furthermore, there is an
  isotopy $d_t \colon \S_0 \to \S_t$ for $t \in I$ connecting
  $\text{Id}_{\S_0}$ and $d$ such that $\S_t \in \S(f_t,v_t)$ for
  every $t \in I$.
\end{lemma}

More precisely, if the $(f_t,v_t)$ have an index~1-2 birth-death
critical point, then we can choose either $\S_0 \in \S_-(f_0,v_0)$ and $\S_1 \in
\S_-(f_1,v_1)$, or $\S_0 \in \S_+(f_0, v_0)$ and $\S_1 \in
\S_+(f_1,v_1)$.

\begin{proof}
  Just like in the proof of Lemma~\ref{lem:isotopy}, there exists a
  smooth family of Heegaard surfaces $\{\, \S_t \,\colon\, t \in I \,\}$
  such that $\S_t \in \S(f_t,v_t)$ for every $t \in I$. If we apply
  the second part of Lemma~\ref{lem:isot-ext} to $\{\, \S_t \,\colon\, t
  \in I \,\}$, we obtain a family of diffeomorphisms $d_t \colon \S_0
  \to \S_t$ for $t \in I$, and $d_1$ is unique up to isotopy.
  Independence of $d_1$ from the choice of family $\{\, \S_t \,\colon\, t
  \in I \,\}$ (up to isotopy) is obtained just like in the proof of
  Lemma~\ref{lem:isotopy}, except now we apply the second part of
  Lemma~\ref{lem:isot-homot}.
\end{proof}

\begin{lemma} \label{lem:loop} Let $\Lambda \colon I \to \FV_0(M,\g)$
  be a loop of gradient-like vector fields.  Furthermore, let $\S \in
  \S(f_0,v_0)$ be a Heegaard surface, pick a spanning tree $T_\pm$ of
  $\Gamma_\pm(f_0,v_0)$, and set $\HD = H(f_0,v_0,\S,T_\pm)$.  By
  Lemma~\ref{lem:isotopy}, the loop $\Lambda$ induces a diffeomorphism
  $d \colon \HD \to \HD$.  If $\Lambda$ is null-homotopic in
  $\FV_0(M,\g)$, then $d$ is isotopic to $\text{Id}_{\HD}$ in the
  space of diffeomorphisms from $\HD$ to itself. If we do not pick a
  tree $T_\pm$, we obtain an analogous statement for overcomplete
  diagrams.
\end{lemma}

\begin{proof}
  Let $L \colon I \times I \to \FV_0(M,\g)$ be the null-homotopy;
  i.e., $L(t,0) = (f_t,v_t)$ for every $t \in I$, and $L(t,u) =
  (f_0,v_0)$ for $t \in \{0,1\}$ or $u = 1$. By
  Proposition~\ref{prop:fibration}, there is a smooth 2-parameter
  family of Heegaard surfaces $\{\, \S_{t,u} \,\colon\, t \text{, } u \in I \,\}$
  such that $\S_{t,u} \in \S(L(t,u))$ for every $t$, $u \in I$, and
  $\S_{t,u} = \S$ whenever $t \in \{0,1\}$ or $u=1$.  Furthermore,
  $T_\pm$ naturally induces a spanning tree $T_\pm^{t,u}$ of
  $\Gamma_\pm(L(t,u))$ such that $T_\pm^{t,u} = T_\pm$ for $t \in
  \{0,1\}$ or $u = 1$.  So we have a smooth 2-parameter family of
  diagrams
  \[
  \HD_{t,u} = H \left(L(t,u),\S_{t,u},T_\pm^{t,u} \right)
  \]
  such that $\HD_{t,u} = \HD$ for $t \in \{0,1\}$ or $u =1$.  Now
  Lemma~\ref{lem:isot-homot} provides the required isotopy between $d$
  and $\text{Id}_{\HD}$.
\end{proof}

\begin{corollary} \label{cor:homotopy} Let $(f_i,v_i) \in \FV_0(M,\g)$
  for $i \in \{0,1\}$, and let
  \[
  \Gamma_0 \text{, } \Gamma_1 \colon I \to \FV_0(M,\g)
  \]
  be paths such that $\Gamma_j(i) = (f_i,v_i)$ for $i$, $j \in \{0,1\}$.
  Given surfaces $\S_i \in \S(f_i,v_i)$ for $i \in \{0,1\}$, consider
  the (overcomplete) diagrams $\HD_i = H(f_i,v_i,\S_i)$. By
  Lemma~\ref{lem:isotopy}, the path $\Gamma_i$ induces a
  diffeomorphism $d_i \colon \HD_0 \to \HD_1$.  Suppose the paths
  $\Gamma_0$ and $\Gamma_1$ are homotopic in $\FV_0(M,\g)$ fixing
  their endpoints.  Then $d_0$ and $d_1$ are isotopic through
  diffeomorphisms from $\HD_0$ to $\HD_1$.
\end{corollary}

\begin{definition} \label{def:gen-stab} The sutured diagram $(\S',
  \alphas',\betas')$ is obtained from $(\S,\alphas,\betas)$ by a
  \emph{$(k,l)$-stabilization} if there is a disk $D \subset \S$ and a
  punctured torus $T \subset \S'$, and there are curves $\a \in
  \alphas'$ and $\b \in \betas'$ such that
  \begin{itemize}
  \item $\S \setminus D = \S' \setminus T$,
  \item $\alphas \setminus D = \alphas' \setminus T$ and $\betas
    \setminus D = \betas' \setminus T$,
  \item $\alphas \cap D$ and $\betas \cap D$ consist of~$l$ and~$k$
    arcs, respectively, and each component of $\alphas \cap D$
    intersects each component of $\betas \cap D$ transversely in a
    single point,
  \item $\a$, $\b \subset T$, and they intersect each other transversely
    in a single point,
  \item $(\alphas' \setminus \a) \cap T$ consists of~$l$ parallel
    arcs, each of which intersects~$\b$ transversely in a single
    point,
  \item $(\betas' \setminus \b) \cap T$ consists of~$k$ parallel arcs,
    each of which intersects $\a$ transversely in a single point,
  \item for each component of $\alphas \cap D$, there is a
    corresponding component of $\alphas' \cap T$ with the same
    endpoints, and similarly for the $\b$-curves.
  \end{itemize}
  In the above case, we also say that $(\S, \alphas, \betas)$ is
  obtained from $(\S', \betas', \alphas')$ by a
  $(k,l)$-destabilization.  Notice that a $(0,0)$-(de)stabilization
  agrees with the ``simple'' (de)stabilization of
  Definition~\ref{def:stab}.  The two diagrams in the bottom of
  Figure~\ref{fig:1-2-birth-death} are related by a
  $(3,3)$-stabilization.
\end{definition}

\begin{definition} \label{def:gen-handleslide} The sutured diagram
  $(\S,\alphas',\betas)$ is obtained from $(\S,\alphas,\betas)$ by a
  \emph{generalized $\a$-handleslide of type~$(m,n)$} if there are
  curves $\a_1, \a_2 \in \alphas$, a curve $\a_1' \in \alphas'$, and an
  embedded arc $a \subset \S$ such that
  \begin{itemize}
  \item $\alphas' \setminus \a_1' = \alphas \setminus \a_1$,
  \item $\partial a \subset \a_2$ and the interior of $a$ is disjoint
    from $\alphas$,
  \item there is a regular neighborhood $N$ of $\a_2 \cup a$ such
    that $\partial N = \a_1 \cup \a_1' \cup c$, where $c$ is a curve
    parallel to $\a_2$, and the interior of $N$ is disjoint from
    $\alphas \cup \alphas' \setminus \{\a_2\}$, and
  \item if $\a_2 \setminus \partial a = \a_2^0 \cup \a_2^1$, where
    $\a_2^0 \cup a$ is parallel to $\a_1$ and $\a_2^1 \cup a$ is
    parallel to $\a_1'$, then $|\a_2^0 \cap \betas| = m$ and $|\a_2^1
    \cap \betas| = n$.
  \end{itemize}
See Figure~\ref{fig:simplify-gen-slide}. Generalized $\b$-handleslides are defined similarly.
\end{definition}

In particular, an ``ordinary'' handleslide is a generalized
handleslide of type~$(0,n)$, where the endpoints of the arc~$a$ lie
very close to each other.

The bifurcations that appear in generic 1-parameter families of
gradient vector fields were given in Section~\ref{sec:1-param}. We now
translate these to moves on sutured diagrams.  For clarity, we state
what happens to overcomplete diagrams.

\begin{proposition} \label{prop:1-param} Suppose that
  \[
  \{\, (f_t,v_t) \,\colon\, t \in [-1,1] \,\}
  \]
  is a generic 1-parameter family of sutured functions and gradient-like
  vector fields on~$(M,\g)$ that has a bifurcation at~$t = 0$. Since
  $(f_0,v_0) \in \FV_1(M,\g)$, it is separable; pick a separating
  surface $\S \in \S(f_0,v_0)$. Then there exists an $\eps = \eps(\S) >
  0$ such that $\S \pitchfork v_t$ for every $t \in (-\eps,\eps)$.
  Furthermore, for every $x \in (-\eps,0)$ and $y \in (0,\eps)$, the
  following hold.

  If the bifurcation is not an index~1-2 birth-death, then $\S \in
  \S(f_x,v_x) \cap \S(f_y,v_y)$.
  Furthermore, the (overcomplete) diagrams
  \[
  (\S,\alphas,\betas) =
  H(f_x,v_x,\S) \quad \text{and} \quad (\S,\alphas',\betas') = H(f_y,v_y,\S),
  \]
  possibly after a small isotopy of the immersed submanifold $\alphas
  \cup \betas$, are related in one of the following ways.
  \begin{enumerate}
  \item \label{item:birthdeath} If the bifurcation is an index~0-1
    or~2-3 birth-death, adding or removing a redundant
    $\a$- or $\b$-curve, not necessarily disjoint from curves of the
    opposite type. ``Redundant'' means this $\alpha$- or $\beta$-curve is
    null-homotopic in $\Sigma$ compressed along the remaining
    $\alpha$- or $\beta$-curves, or equivalently that it bounds a
    planar region together with the other $\alpha$- or $\beta$-curves.
  \item \label{item:isot} If the
    bifurcation is a tangency of $W^u(p)$ and $W^s(q)$ for $p \in
    C_1(f_0)$ and $q \in C_2(f_0)$, an isotopy of the $\a$- and $\b$-curves
    cancelling or creating a pair of intersection points.
  \item \label{item:tangency} If the bifurcation is a tangency between
    $W^u(p)$ and $W^s(q)$ for $p$, $q \in C_1(f_0)$ or $p$, $q \in
    C_2(f_0)$, a generalized $\a$- or $\b$-handleslide. Specifically,
    the $\a$-curve corresponding to $p$ slides over the $\a$-curve
    corresponding to $q$ if $p$, $q \in C_1(f_0)$, while the $\b$-curve
    corresponding to $q$ slides over the $\b$-curve corresponding to
    $p$ if $p$, $q \in C_2(f_0)$.
  \end{enumerate}

  If the bifurcation is an index~1-2 birth, then $\S \in \S(f_x,v_x)$.
  Furthermore, there exists a surface $\S' \in \S(f_y,v_y)$ such that
  the (overcomplete) diagrams $(\S,\alphas,\betas) = H(f_x,v_x,\S)$
  and $(\S',\alphas',\betas') = H(f_y,v_y,\S')$, possibly after a
  small isotopy of the immersed submanifold $\alphas \cup \betas$, are
  related by a $(k,l)$-stabilization if there are $l$ flows from
  index~$1$ critical points into the degenerate singularity and $k$
  flows from the degenerate singularity to index~$2$ critical points.
  For an index~1-2 death, the same statements hold, but with $x$ and
  $y$ reversed.
\end{proposition}

\begin{proof}
  Since the family is generic, $(f_t,v_t) \in \FV_0(M,\g)$ for every
  $t \in [-1,1] \setminus \{0\}$, and $(f_0,v_0) \in \FV_1(M,\g)$. By
  Proposition~\ref{prop:grad-like}, the surface~$\S$ divides $(M,\g)$
  into two sutured compression bodies $(M_-,\g_-)$ and $(M_+,\g_+)$.
  Let $\eps > 0$ be so small that for every $ t \in (-\eps,\eps)$, the
  surface~$\S$ is transverse to~$v_t$.

  First, suppose we are in case~\eqref{item:birthdeath}. Without loss
  of generality, we can assume that the bifurcation is an index~0-1
  birth. The function~$f_0$ has a degenerate critical point at~$p_0
  \in M$, which splits into an index~0 critical point $p_t^0 \in
  C_0(f_t)$ and an index~1 critical point $p_t^1 \in C_1(f_t)$ for $t
  > 0$. Recall that the stable manifold $W^s(p_0)$ is a 1-manifold
  with boundary~at~$p_0$, while the unstable manifold $W^u(p_0)$ is
  locally diffeomorphic to~$\R^3_+$, with boundary the strong unstable
  manifold $W^{uu}(p_0)$; see Figure~\ref{fig:codim1}.
  The critical points $p_0$ at $t=0$ and
  $p^0_t$ for $t > 0$ both have valence $k+1$ in $\Gamma(f_t, v_t)$,
  where $k$ is the number of flow-lines from~$p_0$ to index~$1$
  critical points within $W^u(p_0)$.

  Recall that $p_0 \in C_{01}(f_0,v_0) \subset M_-$. Since $v_t$ is
  transverse to $\S$ for every $t \in (-\eps,\eps)$, both $p_t^0$ and
  $p_t^1$ lie in $M_-$, hence $C_{01}(f_t,v_t) \subset M_-$ for every
  $t \in (-\eps,\eps)$.  This implies that $\S \in \S(f_t,v_t)$ for
  every $t \in (-\eps,\eps)$.

  The attaching sets $\betas$ and $\betas'$ are just small isotopic
  translates of each other. The isotopy is provided by
  \[
  \bigcup_{q_t \in C_2(f_t)} W^s(q_t) \cap \S
  \]
  for $t \in [x,y]$.  The same holds for $\alphas$ and $\alphas'$,
  except that $\alphas'$ has one new component due to the appearance
  of the new index~1 critical point~$p_y^1$. The new $\a$-circle
  $W^u(p_y^1) \cap \S$ is a small translate of $W^{uu}(p_0) \cap
  \S$. For every index~2 critical point $q \in C_2(f_0)$ for which
  $W^s(q)$ intersects $W^{uu}(p_0)$, the corresponding $\b$-circle
  $W^s(q) \cap \S$ intersects the new, redundant, $\a$-circle. (This
  does happen generically in 1-parameter families.)

  Now we look at case~\eqref{item:isot}.  Consider the family of
  diagrams $(\S,\alphas_t,\betas_t) = H(f_t,v_t,\S)$ for $t \in
  [x,y]$. Then $(\S,\alphas,\betas) = (\S,\alphas_x,\betas_x)$ and
  $(\S,\alphas',\betas') = (\S,\alphas_y,\betas_y)$.  The 1-manifolds
  $\alphas_t$ and $\betas_t$ remain transverse, except for $t = 0$,
  when there is a generic tangency between $W^u(p) \cap \S \in
  \alphas_0$ and $W^s(q) \cap \S \in \betas_0$.

  Next, assume we are in case~\eqref{item:tangency}. Without loss of
  generality, we can suppose that $p$, $q \in C_1(f_0)$.  Then we show
  that the $\a$-curve corresponding to $p$ slides over the $\a$-curve
  corresponding to $q$.  Since $\S$ is transverse to~$v_t$ for every
  $t \in (-\eps,\eps)$, we have $C_0(f_t) \cup C_1(f_t) \cup R_-(\g)
  \subset M_-$ and $C_2(f_t) \cup C_3(f_t) \cup R_+(\g) \subset M_+$
  for every $t \in (-\eps,\eps)$.  Hence $\S \in \S(f_t,v_t)$ for
  every $t \in (-\eps,\eps)$.

  \begin{figure}
    \centering
    \includegraphics[width=5in]{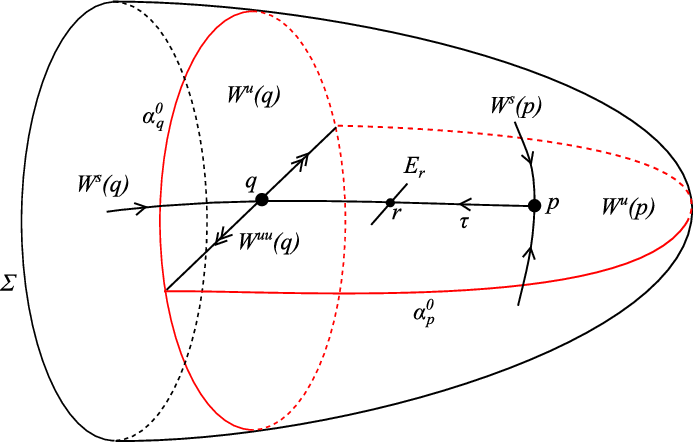}
    \caption{The situation in Case~\ref{item:tangency}, a tangency
      between $W^u(p)$ and $W^s(q)$, leading to a generalized
      handleslide.}
    \label{fig:generalized-handleslide}
  \end{figure}
  Let $\tau = W^u(p) \cap W^s(q)$ be the flow-line of $v_0$ from $p$
  to $q$.  Recall from Section~\ref{sec:1-param} that, generically,
  inside the 2-dimensional unstable manifold $W^u(q)$, there is a
  1-dimensional strong unstable manifold $W^{uu}(q)$. Furthermore, for
  every $r \in \tau$, there is a 1-dimensional subspace $E_r <
  T_rW^u(p)$ complementary to $\langle v_0(r) \rangle = T_r\tau$ that
  limits to $T_qW^{uu}(q)$ under the flow of $v_0$; see
  Figure~\ref{fig:generalized-handleslide}.  It follows that
  the curve $\a_p^0 = W^u(p) \cap \S$ is diffeomorphic to $\R$, with
  ends limiting to the two points of $W^{uu}(q) \cap \S$. Consider the
  circle $\a_q^0 = W^u(q) \cap \S$, and take a thin regular
  neighborhood $P$ of $\a_p^0 \cup \a_q^0$.  Notice that $P$ is a
  pair-of-pants, and one component $\a_q'$ of $\partial P$ is a small
  isotopic translate of $\a_q^0$.  For $t \in (-\eps,\eps)$, let $p_t$
  and $q_t$ be the points of $C_1(f_t)$ corresponding to $p = p_0$ and
  $q = q_0$, and let $\a_p^t = W^u(p_t) \cap \S$ and $\a_q^t = W^u(q_t)
  \cap \S$. Then $\a_q^x$ and $\a_q^y$ are small isotopic translates
  of $\a_q^0$, while $\a_p^x$ and $\a_p^y$ are small isotopic
  translates of the other two components of $\partial P$. Hence
  $\a_p^y \in \alphas'$ is obtained (up to a small
  isotopy) by a generalized handleslide of $\a_p^x \in \alphas$ over
  $\a_q^x \in \alphas$ using the arc $a = \text{cl}\left( \a_p^0
  \right)$, and every other component of $\alphas'$ is a small
  translate of a component of $\alphas$. The type~$(m,n)$ of the
  generalized handleslide is given by the number of flow lines from
  $q$ to index~2 critical points that intersect the two components of
  $W^u(q) \setminus W^{uu}(q)$.

\begin{figure}
  \centering
  \includegraphics[width = 5in]{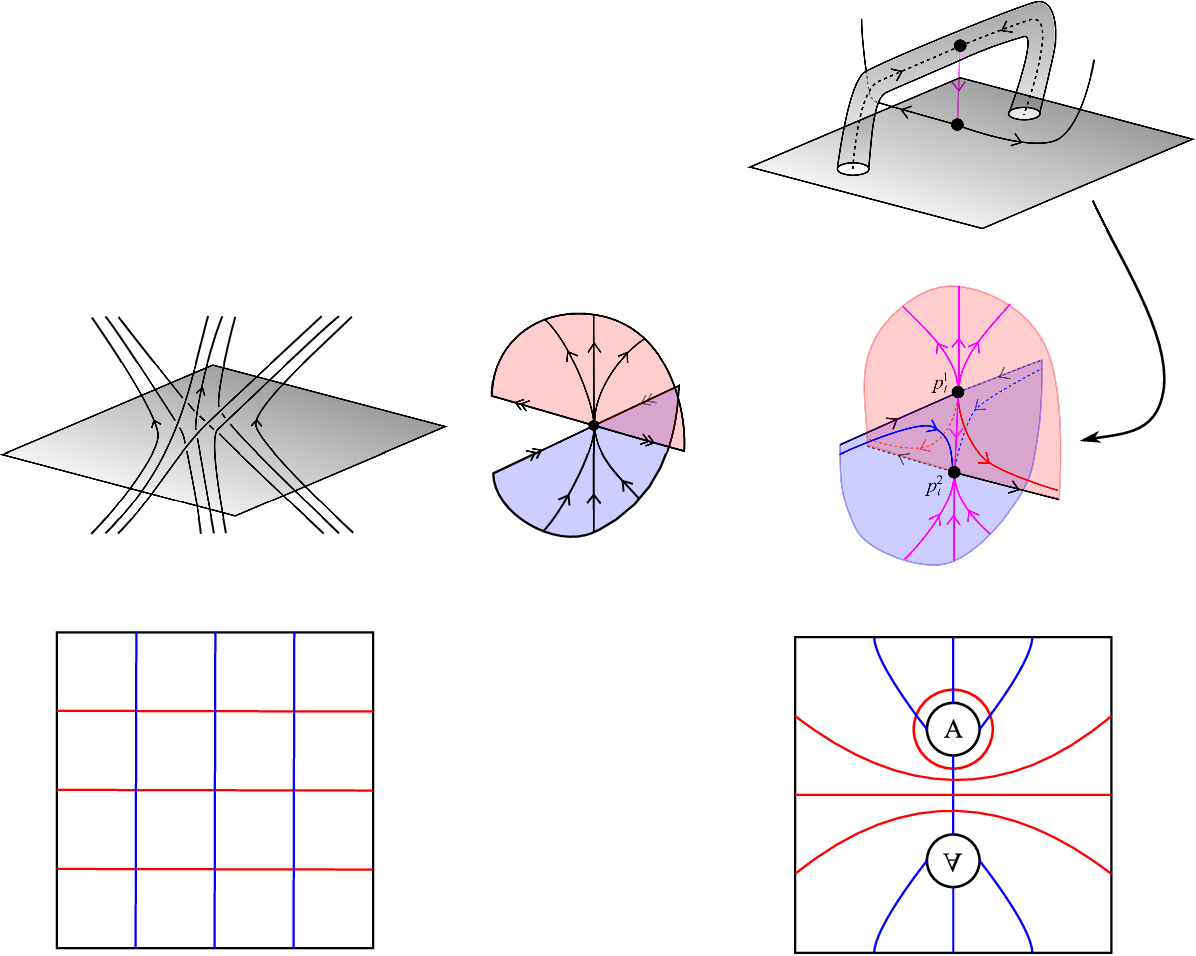}
  \caption{\textbf{An index 1-2 birth death.}  Here, we see the two
    sides of a codimension-1 index~1-2 birth-death singularity, which is included in $C_{23}(f_0,v_0)$.
    There may be flows from index~1 critical points to this
    singularity, or from this singularity to index~2 critical points;
    in the example, there are three of each.
    The top row shows locally the gradient flow, together with the Heegaard surface (drawn in grey).
    In the bottom row, we illustrate the corresponding Heegaard diagrams.
    As usual, we identify the two circles labeled ``$A$''.
    On the side of the singularity where the two critical points die
    (on the left), we see a grid of flows between these critical
    points.}
  \label{fig:1-2-birth-death}
\end{figure}

Finally, consider the case of an index~1-2 birth-death, as
illustrated in Figure~\ref{fig:1-2-birth-death}.  Without loss of
generality, we can assume that a pair of index~1 and~2 critical points
are born.  So $f_0$ has a degenerate singularity at $p_0$ that splits
into $p_t^1 \in C_1(f_t)$ and $p_t^2 \in C_2(f_t)$ for $t > 0$. Recall
that we can either include $p_0$ in $C_{01}(f_0,v_0)$ or
in $C_{23}(f_0,v_0)$. For now, we assume that $p_0 \in C_{23}(f_0,v_0)$,
but the other choice works as well.

Observe that $\S \in \S(f_t,v_t)$ for every $t \in (-\eps,0)$, as $\S
\pitchfork v_t$, $C_{01}(f_t,v_t) \subset M_-$, and $C_{23}(f_t,v_t)
\subset M_+$.  We have $p_0 \in C_{23}(f_0,v_0) \subset M_+$, thus
$p_t^1$, $p_t^2 \in M_+$ for every $t \in (0,\eps)$.  Indeed, neither of
the points~$p_t^1$ and~$p_t^2$ can pass through $\S$ as $v_t$ remains
transverse to $\S$ throughout.

Since $v_y$ is generic, $W^s(p^1_y)$ has both ends in $C_0(f_y) \cup
R_-(\g) \subset M_-$. This implies that $W^s(p^1_y)$ intersects $\S$
transversely in two points. On the other hand, as $p^2_y \in M_+$ and
both ends of $W^u(p^2_y)$ lie in $C_3(f_y) \cup R_+(\g) \subset M_+$,
we have $W^u(p^2_y) \cap \S = \emptyset$.  We obtain $\S'$ by
smoothing the corners of
\[
\partial\left( M_- \cup N \left(W^s (p^1_y) \right) \right)
\setminus \partial M,
\]
where $N \left(W^s(p^1_y) \right)$ is a thin tubular neighborhood of
$W^s(p^1_y)$ whose boundary in $M_+$ is transverse to $v_y$. It is
apparent that $\S'$ is transverse to $v_y$, and $\S'$ cuts $(M,\g)$
into two sutured compression bodies, one of which contains
$C_{01}(f_y,v_y) \cup R_-(\g)$, while the other one contains
$C_{23}(f_y,v_y) \cup R_+(\g)$.  Hence $\S' \in \S(f_y,v_y)$.  Notice
that $\S' \setminus \S$ is an annulus $A$, and $\S \setminus \S'$ is a
disjoint union of two disks~$D_1$ and~$D_2$.

We now describe the attaching sets $\alphas'$ and $\betas'$. Observe
that $\a = W^u(p^1_y) \cap \S' \in \alphas'$ is a homologically
non-trivial curve in $A$.  The curve $\b = W^s(p^2_y) \cap \S' \in
\betas'$ intersects $\a$ transversely in a single point, and $\b \cap
A$ is an arc connecting $\partial D_1$ and $\partial D_2$. Let $T$ be
a thin regular neighborhood of $A \cup \b$. Then $T$ is homeomorphic
to a punctured torus. In addition, let
\[
D = (T \setminus A) \cup D_1
\cup D_2;
\]
this is diffeomorphic to a disk.  Observe that $\alphas'
\cap (\S \setminus T)$ is a small isotopic translate of $\alphas \cap
(\S \setminus D)$, and similarly, $\betas' \cap (\S \setminus T)$ is a
small isotopic translate of $\betas \cap (\S \setminus D)$.

If $q_0 \in C_1(f_0)$ is a non-degenerate critical point, and $q_y \in
C_1(f_y)$ is the corresponding critical point, then $W^u(q_y) \cap \S'
\in \alphas'$ intersects~$\b$ in precisely $|W^u(q_0) \cap W^s(p_0)|$
points. In addition, $W^u(q_y) \cap T$ consists of parallel arcs that
do not enter~$A$. Similarly, for every $r_0 \in C_2(f_0)$ with
corresponding $r_y \in C_2(f_y)$, the $\b$-curve $W^s(r_y) \cap \S'
\in \betas'$ intersects $\a$ in $|W^s(r_0) \cap W^u(p_0)|$ points, and
$A$ in the same number of parallel arcs. Hence
$(\S',\alphas',\betas')$ is indeed obtained from $(\S,\alphas,\betas)$
by a $(k,l)$-stabilization, as stated.

Note that if we include $p_0$ in $C_{01}(f_0,v_0)$, then an analogous
argument applies, with the difference that inside the stabilization
tube $A$ we have a $\b$-curve and an $\a$-arc.
\end{proof}

\begin{remark}
  In case of an index 1-2 stabilization,
  the stabilized surface $\S' \in \S(f_y,v_y)$ does not separate $(f_t,v_t)$
  if $t \in (0,y)$ is sufficiently small. Indeed, generically, the
  saddle-node $p_0 \not\in \S'$, and consequently the index~1 and~2
  critical points~$p_1^t$ and~$p_2^t$ will both lie on the same side
  of $\S'$ for $t > 0$ sufficiently small.
\end{remark}

Essentially the same argument gives the following analogue of
Proposition~\ref{prop:1-param} for 2-parameter families.

\begin{proposition} \label{prop:2-param-cod-1} Suppose that $\{\,
  (f_\mu,v_\mu) \in \FV(M,\g) \,\colon\, \mu \in \R^2 \,\}$ is a generic
  2-parameter family that has a codimension-1 bifurcation at $\mu =
  0$. Let~$S$ be the stratum of the bifurcation set passing through
  the origin ($S$ is a non-singular curve near $0$). Since $(f_0,v_0)
  \in \FV_1(M,\g)$, it is separable; pick a separating surface $\S \in
  \S(f_0,v_0)$. Then there exists an $\eps = \eps(\S) > 0$ such that
  $D^2_\eps \setminus S$ consists of two components~$C_1$ and~$C_2$,
  and for every $x \in C_1$ and $y \in C_2$ the same conclusion holds
  as in Proposition~\ref{prop:1-param}.
\end{proposition}

Recall that, in Definition~\ref{def:distinguished-rect}, we introduced
the notion of distinguished rectangles of Heegaard moves.

\begin{definition}
  A \emph{generalized distinguished rectangle} is defined just like in
  Definition~\ref{def:distinguished-rect}, except we replace the word
  ``stabilization'' with ``$(k,l)$-stabilization,'' and allow
  overcomplete diagrams.
\end{definition}

The following result relates isotopies with codimension-1 moves.

\begin{proposition} \label{prop:rectangle} Suppose that $\{\,
  (f_\mu,v_\mu) \in \FV_{\le 1}(M,\g) \,\colon\, \mu \in \R^2 \,\}$ is a generic
  2-parameter family, and let
  \[
  V_1 = \{\, \mu \in \R^2 \,\colon\, (f_\mu, v_\mu) \in \FV_1(M,\g) \,\}
  \]
  be the codimension-1 bifurcation set.  Let $a \subset V_1$ be an arc
  with endpoints~$\mu_0$ and~$\mu_1$, and suppose we are given
  surfaces $\S_i \in \S(f_{\mu_i},v_{\mu_i})$ for $i \in \{0,1\}$.
  Let~$b_0$ and~$b_1$ be arcs transverse to $V_1$ such that the only
  bifurcation value in~$b_i$ is~$\mu_i$ and $\S_i \pitchfork v_\mu$
  for every~$\mu \in b_i$. Orient~$b_0$ and~$b_1$ such that they have
  the same intersection sign with~$V_1$, and let $\partial b_0 = y_0 -
  x_0$ and $\partial b_1 = y_1-x_1$; see Figure~\ref{fig:rectangle}.

  After possibly flipping the orientation of $b_0$ and $b_1$, we can
  assume that $\S_i \in \S(f_{x_i},v_{x_i})$ for $i \in \{0,1\}$.
  Furthermore, suppose we are given surfaces $\S_i' \in
  \S(f_{y_i},v_{y_i})$ for $i \in \{0,1\}$ such that $\S_i'$ is
  obtained from $\S_i$ by a stabilization if $(f_{\mu_i},v_{\mu_i})$
  is an index~1-2 birth, and $\S_i' = \S_i$ otherwise. (Such surfaces
  always exist by applying Proposition~\ref{prop:1-param} to the
  1-parameter family parametrized by $b_i$.)  Then the isotopy
  diagrams $H_1 = [H(f_{x_0},v_{x_0},\S_0)]$, $H_2 =
  [H(f_{y_0},v_{y_0},\S_0')]$, $H_3 = [H(f_{x_1},v_{x_1},\S_1)]$, and
  $H_4 = [H(f_{y_1},v_{y_1},\S_1')]$ fit into a generalized
  distinguished rectangle
  \[
  \xymatrix{H_1 \ar[r]^{e} \ar[d]^{f} & H_2 \ar[d]^{g} \\ H_3
    \ar[r]^{h} & H_4}
  \]
  of type~\eqref{item:rect-alpha-diff} if $(f_{\mu_i},v_{\mu_i})$ is a
  handleslide, or type~\eqref{item:rect-stab-diff} if
  $(f_{\mu_i},v_{\mu_i})$ is an index~1-2 birth-death;
  cf.~Definition~\ref{def:distinguished-rect}. For other types of
  bifurcations, we have a rectangle with $e$ and $h$ the identity or
  adding/removing a redundant $\a$- or $\b$-curve, and $f = g$ a
  diffeomorphism.  If we pick any curves~$a_1$ and~$a_2$ outside the
  bifurcation set parallel to $a$ with $\partial a_1 = x_1 - x_0$ and
  $\partial a_2 = y_1 - y_0$ and apply Lemma~\ref{lem:isotopy}, then
  $a_1$ will induce a diffeomorphism isotopic to $f$, and $a_2$ will
  induce a diffeomorphism isotopic to $g$.  In particular, $f \in
  \G_\d^0(H_1,H_3)$ and $g \in \G_\d^0(H_2,H_4)$.  The arrows~$e$ and~$h$
  are given by Proposition~\ref{prop:1-param}.
\end{proposition}

\begin{figure}
  \centering
  \includegraphics{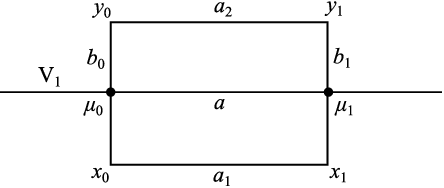}
  \caption{}
  \label{fig:rectangle}
\end{figure}

Note that, in case of an index~1-2 birth-death singularity, we mean
$\S_0 \in \S_\pm(f_{\mu_0},v_{\mu_0})$ and $\S_1 \in
\S_\pm(f_{\mu_1},v_{\mu_1})$, and we allow all four combinations of
signs.

\begin{proof}
  For now, assume that in case of an index~1-2 birth death, we have
  either $\S_0 \in \S_+(f_{\mu_0},v_{\mu_0})$ and $\S_1 \in
  \S_+(f_{\mu_1},v_{\mu_1})$, or $\S_0 \in \S_-(f_{\mu_0},v_{\mu_0})$
  and $\S_1 \in \S_-(f_{\mu_1},v_{\mu_1})$.

  Choose an arbitrary parametrization $a(t)$ of the arc $a$, then
  apply Lemma~\ref{lem:isotopy-cod-1} to the 1-parameter family $\{\,
  (f_{a(t)},v_{a(t)}) \,\colon\, t \in I \,\}$ inside $\FV_1(M,\g)$. We
  obtain a family of diffeomorphisms $d_t \colon \S_0 \to \S_t$ such
  that $\S_t = d_t(\S_0) \in \S(f_{a(t)},v_{a(t)})$ for every $t \in
  I$. There exists an $\eps > 0$ such that for every $t \in I$ and
  $\mu \in \R^2$ with $|a(t) - \mu| < \eps$, we have $\S_t \pitchfork
  v_\mu$.  Indeed, as transversality is an open relation, the set
  \[
  U = \{\, (t,\mu) \in I \times \R^2 \,\colon\, \S_t \pitchfork v_\mu \,\}
  \]
  is an open neighborhood of the graph $\ol{a} = \{\, (t,a(t)) \,\colon\,
  t \in I \,\}$ in $I \times \R^2$.  In particular, we can take $\eps$
  to be the distance of $\ol{a}$ and $(I \times R^2) \setminus U$.
  Furthermore, we take~$\eps$ so small that for every $\mu \in
  N_\eps(a)$, the pair $(f_\mu,v_\mu)$ is Morse-Smale unless $\mu$
  lies in the component of $V_1$ containing~$a$.
  We denote by $C_1$ and $C_2$ the components of $N_\eps(a) \setminus
  V_1$, labeled such that $b_0$ and $b_1$ are both oriented from $C_1$
  to $C_2$.

  First, suppose that $(f_{\mu_i},v_{\mu_i})$ is not an index~1-2
  birth-death. Then, as explained in the proof of
  Proposition~\ref{prop:1-param}, for every $t \in I$ and $\mu \in
  \R^2$ with $|a(t) -\mu| < \eps$, we even have $\S_t \in
  \S(f_\mu,v_\mu)$.
  Furthermore, $\S_0 = \S_0'$ and $\S_1 = \S_1'$, and the
  corresponding isotopy diagrams~$H_1$ and~$H_2$, and similarly, $H_3$
  and $H_4$, are related by a handleslide, adding or removing a
  redundant $\a$- or $\b$-curve, or they are the same (for an
  orbit of tangency between an index~1 and an index~2 critical
  point). In this case, we take $f = g = d_1 \colon \S_0 \to
  \S_1$. What we need to show is that $d_1(H_1) = H_3$ and $d_1(H_2) =
  H_4$. Pick points $x_0' \in C_1 \cap b_0$, $x_1' \in C_1 \cap b_1$,
  $y_0' \in C_2 \cap b_0$, and $y_1' \in C_2 \cap b_1$, then choose
  arcs $c_1 \subset C_1$ and $c_2 \subset C_2$ with $\partial c_1 =
  x_1' - x_0'$ and $\partial c_2 = y_1' - y_0'$.  These can be
  parametrized such that $|a(t) - c_j(t)| < \eps$ for every $t \in I$
  and $j \in \{1,2\}$.  Since $\S_t \in \S(f_{c_j(t)},v_{c_j(t)})$ and
  $(f_{c_j(t)},v_{c_j(t)})$ is Morse-Smale, there is an induced
  overcomplete diagram
  \[
  \HD^j_t = H(f_{c_j(t)},v_{c_j(t)},\S_t).
  \]
  If we apply the first part of Lemma~\ref{lem:isot-ext} to the family
  of diagrams $\{\, \HD^j_t \,\colon\, t \in I \,\}$, we obtain an induced
  diffeomorphism $d^j_1 \colon \S_0 \to \S_1$ such that
  $d^1_1(\HD^1_0) = \HD^1_1$ and $d_1^2(\HD^2_0) = \HD^2_1$.  Since
  $\S_i \in \S(f_\mu,v_\mu)$ for every $\mu \in b_i$ and $i \in
  \{0,1\}$, the isotopy diagrams $[\HD^1_0] = H_1$, $[\HD^1_1] = H_3$,
  $[\HD^2_0] = H_2$, and $[\HD^2_1] = H_4$.  Hence $d_1^1(H_1) = H_3$
  and $d_1^2(H_2) = H_4$. The second part of Lemma~\ref{lem:isot-ext}
  implies that $d_1^j$ is isotopic to $d_1$ for $j \in \{1,2\}$, so
  indeed $d_1(H_1) = H_3$ and $d_1(H_2) = H_4$.

  Let $a_1' \subset \R^2$ be the path obtained by going from~$x_0$ to~$x_0'$
  along $b_0$, then from~$x_0'$ to~$x_1'$ along~$c_1$, finally,
  from~$x_1'$ to~$x_1$ along $b_1$.  We define the path $a_2' \subset
  \R^2$ from~$y_0$ to~$y_1$ in an analogous manner.  Since $\S_i \in
  \S(f_\mu,v_\mu)$ for every $\mu \in b_i$ and $i \in \{0,1\}$, the
  path $a_j'$ induces a diffeomorphism $\delta^j_1 \colon H_j \to
  H_{j+2}$ isotopic to $d_1^j$ for $j \in \{1,2\}$. If $a_1$ is an
  arbitrary path from $x_0$ to $x_1$ in the complement of the
  bifurcation set and parallel to $a$, then $a_1$ is homotopic to
  $a_1'$ relative to their boundary. So, by
  Corollary~\ref{cor:homotopy}, the path $a_1$ induces a
  diffeomorphism $f \colon H_1 \to H_3$ isotopic to $\delta^1_1$,
  hence also to $d_1$. Similarly, a path $a_2$ from $y_0$ to $y_1$
  avoiding the bifurcation set and parallel to $a$ induces a
  diffeomorphism $g \colon H_2 \to H_4$ isotopic to $\delta^2_1$,
  hence also to $d_1$.

  Suppose that $(f_{\mu_i},v_{\mu_i})$ is an index~1-2 birth;
  furthermore, $\S_0 \in \S_+(f_{\mu_0},v_{\mu_0})$ and $\S_1 \in
  \S_+(f_{\mu_1},v_{\mu_1})$.  The case when $\S_0 \in
  \S_-(f_{\mu_0},v_{\mu_0})$ and $\S_1 \in \S_-(f_{\mu_1},v_{\mu_1})$
  is completely analogous.  Then, by
  Proposition~\ref{prop:2-param-cod-1}, the diagram~$H_2$ is obtained
  from $H_1$ by a stabilization, and similarly, $H_4$ is obtained from
  $H_3$ by a stabilization. Pick arcs $c_1 \colon I \to C_1$ and $c_2
  \colon I \to C_2$ as above, and extend them to arcs $a_1' \colon
  [-1,2] \to \R^2$ connecting $x_0$ and $x_1$ and $a_2' \colon [-1,2]
  \to \R^2$ connecting $y_0$ and $y_1$ in a similar manner.  I.e.,
  $a_j'([-1,0]) \subset b_0$, $\left. a_j' \right|_I = c_j$, and
  $a_j'([1,2]) \subset b_1$ for $j \in \{1,2\}$.
  For every $\mu$ on the ``birth'' side~$C_2$ of~$V_1$, let the index~1
  and~2 critical points born be $p^1(\mu)$ and $p^2(\mu)$,
  respectively.  The surface~$\S_t$ divides $M$ into two sutured
  compression bodies $M_-(t)$ and $M_+(t)$.  Let $\S_t = \S_0$ for $t
  \in [-1,0]$ and $\S_t = \S_1$ for $t \in [1,2]$.  Furthermore, we
  write $p^j(t)$ for~$p^j(a_2'(t))$, where $j \in \{1,2\}$ and $t \in
  [-1,2]$.  For each $t \in [-1,2]$, we construct a surface
  \[
  \S_t^* \in \S(f_{a_2'(t)},v_{a_2'(t)})
  \]
  from $\S_t$ by adding a tube around $W^s(p^1(t))$ as in the proof of
  Proposition~\ref{prop:1-param}, but now in a way that the
  construction depends smoothly on~$t$. For this, simply pick a thin
  regular neighborhood $N$ of
  \[
  \bigcup_{t \in [-1,2]} W^s(p^1(t)) \times \{t\} \subset M \times
  [-1,2],
  \]
  let $N_t = N \cap (M \times \{t\})$, then take
  \[
  \S_t^* = \partial (M_-(t) \cup N_t) \setminus \partial M.
  \]
  Finally, let $A_t = \ol{\S_t^* \setminus \S_t}$ be the added tube
  (i.e., annulus). We can do this in such a way that $\S_{-1}^* =
  \S_0'$ and $\S_2^* = \S_1'$.

  We take $f$ to be $d_1$, defined at the beginning of this proof.
  Furthermore, we define~$g$ to be $d_1$
  outside the extra tube $A_{-1}$, and extend it to $A_{-1}$ using the
  family $\S_t^*$ (this follows from a straightforward relative
  version of Lemma~\ref{lem:isot-ext}).  Similarly to the other cases,
  take $\HD^1_t = H(f_{a_1'(t)},v_{a_1'(t)},\S_t)$ and $\HD^2_t =
  H(f_{a_2'(t)},v_{a_2'(t)},\S_t^*) =
  (\S_t^*,\alphas^2_t,\betas^2_t)$, then apply
  Lemma~\ref{lem:isot-ext} to these families of diagrams to obtain
  diffeomorphisms $d_1^1$ and $d_1^2$, respectively, such that
  $d_1^1(H_1) = H_3$ and $d_1^2(H_2) = H_4$. As above, $d_1^1$ is
  isotopic to $d_1$, hence $d_1(H_1) = H_3$. Similarly, $d_1^2$ agrees
  with $d_1$ up to isotopy outside $A_{-1}$, and inside $A_{-1}$ it
  has to map the curve $\alphas^2_{-1} \cap A_{-1}$ to $\alphas^2_2
  \cap A_2$ up to isotopy, hence $d_1^2(H_2) = H_4$.  So we indeed
  have a generalized distinguished rectangle of
  type~\eqref{item:rect-stab-diff}.  The fact that any curve $a_i$
  homotopic to $a_i'$ relative to their boundary induces an isotopic
  diffeomorphism follows from Corollary~\ref{cor:homotopy}.

  Finally, we consider the case of an index~1-2 birth-death
  singularity with $\S_0 \in \S_+(f_{\mu_0},v_{\mu_0})$ and $\S_1 \in
  \S_-(f_{\mu_1},v_{\mu_1})$, or $\S_0 \in \S_-(f_{\mu_0},v_{\mu_0})$
  and $\S_1 \in \S_+(f_{\mu_1},v_{\mu_1})$. We will only discuss the
  former possibility, as the latter is completely analogous. We first
  assume that $a$ is a constant path mapping to the point $\om \in
  V_1$, and $b_0 = b_1$.  We denote the arc $b_0 = b_1$ by $b$, and
  write $\partial b = y - x$. Let $b_x$ and $b_y$ be the components of
  $b \setminus \{\om\}$ containing $x$ and $y$, respectively.  Then
  $\S_0$, $\S_1 \in \S(f_x,v_x)$, while $\S_0'$, $\S_1' \in \S(f_y,v_y)$.
  If $a_1$ is the constant path at $x$, then it induces the
  diffeomorphism $f \colon \S_0 \to \S_1$ obtained by flowing along
  $v_x$. Similarly, if $a_2$ is the constant path at $y$, then it
  induces the diffeomorphism $g \colon \S_0' \to \S_1'$ obtained by
  flowing along $v_y$.  Then $f(H_1) = H_3$ and $g(H_2) = H_4$.

  All we need to show is that $g$ is isotopic to the stabilization of
  $f$.  Let $p \in M$ be the degenerate critical point of $f_{\om}$,
  and let $p^1(\mu) \in C_1(f_\mu)$ and $p^2(\mu) \in C_2(f_\mu)$ be
  the corresponding critical points of $f_\mu$ for $\mu \in b_y$.  Let
  $N \subset M$ be a $(v_{\om})$-saturated regular neighborhood of
  $W^s(p) \cup W^u(p)$.  Then the disk $D_0 = \S_0 \cap N$ is a
  regular neighborhood of the arc $W^s(p) \cap \S_0$, and the disk
  $D_1 = \S_1 \cap N$ is a regular neighborhood of the arc $W^u(p)
  \cap \S_1$.  Furthermore, for $i \in \{0,1\}$, let
  \[
  A_i = \ol{\S_i' \setminus \S_i}
  \]
  be the stabilization tubes, and $\a_i = \S_i' \cap W^u(p^1(y))$ and
  $\b_i = \S_i' \cap W^s(p^2(y))$ the new $\a$- and $\b$-curves. Note
  that $\a_0 \subset A_0$ and $\b_0 \cap A_0$ is an arc, whereas $\b_1
  \subset A_1$ and $\a_1 \cap A_1$ is an arc.  Recall that $B_i = \S_i
  \setminus \S_i'$ is a pair of open disks; we choose $D_i$ such that
  $B_i \subset D_i$. Then
  \[
  T_i = (D_i \setminus B_i) \cup A_i
  \]
  is a punctured torus that is a regular neighborhood of $\a_i \cup
  \b_i$ for $i \in \{0,1\}$.  By construction, $\S_i' = (\S_i
  \setminus D_i) \cup T_i$.  The flow of $v_{\om}$ induces a
  diffeomorphism
  \[
  d \colon \S_0 \setminus D_0 \to \S_1 \setminus D_1.
  \]
  If $i \colon \S_1 \setminus D_1 \hookrightarrow \S_1$ is the
  embedding, then both
  \[
  f|_{\S_0 \setminus D_0} \colon \S_0 \setminus D_0 \to \S_1 \text{ and}
  \]
  \[
  g|_{\S_0' \setminus T_0} \colon \S_0' \setminus T_0 = \S_0 \setminus
  D_0 \to \S_1'
  \]
  are isotopic to $i \circ d$. Hence, $f|_{\S_0 \setminus D_0}$ is
  isotopic to $g|_{\S_0' \setminus T_0}$.  Since $g(\a_0) = \a_1$ and
  $g(\b_0) = \b_1$, we can isotope $g$ such that it maps the regular
  neighborhood $T_0$ of $\a_0 \cup \b_0$ to the regular neighborhood
  $T_1$ of $\a_1 \cup \b_1$.  So, up to isotopy of $f$ and $g$, the
  diagram
  \[
  \xymatrix{[H(f_x,v_x,\S_0)] \ar[r]^{e} \ar[d]^{f} & [H(f_y,v_y,\S_0')] \ar[d]^{g} \\
    [H(f_x,v_x,\S_1)] \ar[r]^{h} & [H(f_y,v_y,\S_1')]}
  \]
  is a distinguished rectangle of type~\eqref{item:rect-stab-diff}.

  We are now ready to prove the general case, when $a \subset V_1$ is
  an arbitrary arc, and we have an index 1-2 birth-death singularity
  with $\S_0 \in \S_+(f_{\mu_0},v_{\mu_0})$ and $\S_1 \in
  \S_-(f_{\mu_1},v_{\mu_1})$. Choose a surface $\S \in
  \S_+(f_{\mu_1},v_{\mu_1})$.  There exists an $\eps = \eps(\S) > 0$
  such that $\S \pitchfork v_\mu$ for every $\mu \in D^2_\eps(\mu_1)$,
  and let $b \subset D^2_\eps(\mu_1)$ be a sub-arc of $b_1$ such that
  $\mu_1 \in \text{Int}(b)$.  Suppose that $\partial b = y - x$, and
  denote by $b_{x,x_1}$ the sub-arc of $b_1$ between $x$ and $x_1$,
  and by $b_{y,y_1}$ the sub-arc of $b_1$ between $y$ and $y_1$.  If
  we apply Proposition~\ref{prop:1-param} to the 1-parameter family
  $b$, then we see that $\S \in \S(f_x,v_x)$, and we obtain a surface
  $\S' \in \S(f_y,v_y)$ stabilizing $\S \in \S(f_x,v_x)$.  We write $H
  = [H(f_x,v_x,\S)]$ and $H' = [H(f_y,v_y,\S')]$.

  Let $a_1 \subset \R^2$ be the path obtained by going from $x_0$ to
  $x$ along an arc~$a_1'$ parallel to $a$, then from $x$ to $x_1$
  along $b_{x,x_1}$.  The path $a_2 \subset \R^2$ is obtained by going
  from $y_0$ to $y$ along an arc $a_2'$ parallel to $a$, then from $y$
  to $y_1$ along $b_{y,y_1}$.  We also assume that $a_1'$ and $a_2'$
  are disjoint from the bifurcation set.  Then $a_1'$ induces a
  diffeomorphism
  \[
  f' \colon H_1 = [H(f_{x_0},v_{x_0},\S_0)] \to H = [H(f_x,v_x,\S)],
  \] and the arc $a_2'$ induces a diffeomorphism
  \[
  g' \colon H_2 = [H(f_{y_0},v_{y_0},\S_0')] \to H' =
  [H(f_y,v_y,\S')].
  \]
  Furthermore, the constant $x$ path induces a diffeomorphism
  \[
  f'' \colon H = [H(f_x,v_x,\S)] \to [H(f_x,v_x,\S_1)],
  \]
  and the constant $y$ path induces a diffeomorphism
  \[
  g'' \colon H' = [H(f_y,v_y,\S')] \to [H(f_y,v_y,\S_1')].
  \]

  Since $\S \in \S(f_\mu, v_\mu)$ for every $\mu \in b_{x,x_1}$, both
  $H(f_x,v_x,\S_1)$ and $H(f_{x_1},v_{x_1},\S_1)$ define the same
  isotopy diagram~$H_3$.  Similarly, as $\S' \in \S(f_\mu, v_\mu)$ for
  every $\mu \in b_{y,y_1}$, both $H(f_y,v_y,\S_1')$ and
  $H(f_{y_1},v_{y_1},\S_1')$ define the same isotopy diagram~$H_4$.
  Furthermore, the arc $b_{x,x_1}$ induces a diffeomorphism isotopic
  to~$f''$, and the path $b_{y,y_1}$ induces a diffeomorphism isotopic
  to~$g''$.  Hence, the path $a_1$ induces a diffeomorphism $f \colon
  H_1 \to H_3$ isotopic to $f'' \circ f'$, and the path $a_2$ induces
  a diffeomorphism $g \colon H_2 \to H_4$ isotopic to $g'' \circ g'$.
  Let $s$ denote the stabilization from $H$ to $H'$, and consider the
  following subgraph of $\G$:
  \[
  \xymatrix{H_1 \ar[r]^{e} \ar[d]^{f'} & H_2 \ar[d]^{g'} \\
    H \ar[r]^{s} \ar[d]^{f''} & H' \ar[d]^{g''} \\
    H_3 \ar[r]^{h} & H_4.}
  \]
  The top rectangle is distinguished of
  type~\eqref{item:rect-stab-diff}, as we already know the result for
  $\S_0 \in \S_+(f_{\mu_0},v_{\mu_0})$ and $\S \in
  \S_+(f_{\mu_1},v_{\mu_1})$, together with the arc $a \subset V_1$
  and transverse arcs~$b_0$ and~$b$.  The bottom rectangle is also
  distinguished of type~\eqref{item:rect-stab-diff}, since we have
  proved the proposition for the special case $\S \in
  \S_+(f_{\mu_1},v_{\mu_1})$ and $\S_1 \in \S_-(f_{\mu_1},v_{\mu_1})$,
  the constant $\mu_1$ path, and the transverse arc~$b$.  It follows
  that the large rectangle is also distinguished of
  type~\eqref{item:rect-stab-diff}, and by the above discussion, it
  agrees with the rectangle in the statement.
\end{proof}

\subsection{Codimension-1: Ordinary diagrams}
\label{sec:codimension-1-actual}

In this section, we will show how to choose spanning trees
appropriately in Propositions~\ref{prop:1-param}
and~\ref{prop:2-param-cod-1} to pass from overcomplete to actual
Heegaard diagrams, without altering the relationship of the diagrams
before and after the bifurcation in an essential way. We are going to
write $\Gamma$ for $\Gamma(f_x,v_x)$ and $\Gamma'$ for
$\Gamma(f_y,v_y)$. Similarly, we use the shorthand $\Gamma_\pm$ for
$\Gamma_\pm(f_x,v_x)$ and $\Gamma_\pm'$ for $\Gamma_\pm(f_y,v_y)$,
where $\Gamma_\pm(f,v)$ is defined in Definition~\ref{def:Gamma-pm}.
By abuse of notation, if $p$ is a non-degenerate critical point of $f_0$,
then we also write $p$ for the corresponding critical points of $f_x$
and $f_y$.  Furthermore, if $p$ is index~1 or~2, we also view $p$ as
the midpoint of the
appropriate edge of $\Gamma_+ \cup \Gamma_-$.
Even though this graph is not strictly speaking a subset of $M$, it is obtained from
$\Gamma$ -- which does contain $p$ -- by identifying its vertices lying in $R_+(\g)$
and its vertices lying in $R_-(\g)$. Similarly, we can view $p$ as the midpoint of the appropriate
edge of $\Gamma_+' \cup \Gamma_-'$.

Suppose we are in case~\eqref{item:birthdeath} of
Proposition~\ref{prop:1-param} (0-1 or 2-3 birth-death), and without
loss of generality,
consider the case of the birth of the critical points $p \in C_0(f_y)$
and $q \in C_1(f_y)$. Then $\Gamma$ is obtained from~$\Gamma'$ by a
small isotopy, deleting the vertex~$q$ of valence two along with
its two adjacent edges, and merging the two vertices in~$\Gamma'$ it
was connected to (one of which is~$p$). There is a map $b$ from
spanning trees of $\Gamma_\pm$ to spanning trees of $\Gamma'_\pm$,
given by small isotopy and adding the edge whose midpoint is~$q$;
then $H(f_x,v_x,\Sigma,T_\pm)$ and $H(f_y,v_y,\Sigma,b(T_\pm)$
are the same isotopy diagram.

Similarly, in case~\eqref{item:isot} (1-2 tangency), the graphs $\Gamma$ and
$\Gamma'$ are the same, except for a small isotopy. This
induces a bijection $b$ of spanning trees of $\Gamma_\pm$ and
$\Gamma_\pm'$ such that $H(f_x,v_x,\S,T_\pm)$ and
$H(f_y,v_y,\S,b(T_\pm))$ represent the same isotopy diagram.  So
bifurcations~\eqref{item:birthdeath}
and~\eqref{item:isot} have no effect on isotopy diagrams if we choose
the spanning trees consistently.

Now consider the case of an index 1-2 birth. Then $\Gamma_-'$ is
obtained from
$\Gamma_-$ by adding an edge corresponding to
the new index~1 critical point, and similarly, $\Gamma_+'$ is
obtained from $\Gamma_+$ by adding an edge corresponding to the new
index 2 critical point. Furthermore, $\Gamma_-$ and $\Gamma_+$ are
both connected. So spanning trees $T_\pm$ of $\Gamma_\pm$ remain
spanning trees $T_\pm'$ of $\Gamma_\pm'$.  The diagram
$H(f_y,v_y,\S',T_\pm')$ is obtained from $H(f_x,v_x,\S,T_\pm)$ by a
$(k',l')$-stabilization, where $l'$ is the number of flows from
index~1 critical points of $f_0$ not in $T_-$ to the saddle-node
singular point, and $k'$ is the number of flows from the saddle-node
to index~2 critical points not in $T_+$.  Note that, in this case,
not all spanning trees of $\Gamma_\pm'$ come from
spanning trees of $\Gamma_\pm$.

Finally, consider case~\eqref{item:tangency} (same-index
tangency). Without loss of
generality, assume that the curve $\a_p$ slides over $\a_q$, yielding
$\a_p'$, where $p$, $q \in C_1(f_x)$, the curve $\a_p = W^u(p) \cap \S$,
and $\a_q = W^u(q) \cap \S$.  Then the graph $\Gamma_-'$ is obtained
from $\Gamma_-$ by sliding the edge $e_q \in E(\Gamma_-)$ containing
$q$ over the edge $e_p \in E(\Gamma_-)$ containing $p$, yielding the
edge $e_q'$ (note the change of roles as we pass to the spanning
trees).  Issues arise when $e_q \in T_-$ and $e_p \not \in T_-$,
since then the curve $\a_p$ is sliding over the ``invisible'' curve
$\a_q$.  In fact, there are situations where, for any spanning tree
$T_-$ of $\Gamma_-$ and any spanning tree $T_-'$ of
$\Gamma_-'$, the corresponding Heegaard diagrams do not differ by a
single handleslide.  For such a situation, see
Figure~\ref{fig:graph-slide}. This motivates the following definition.

\begin{figure}
  \centering
  \includegraphics{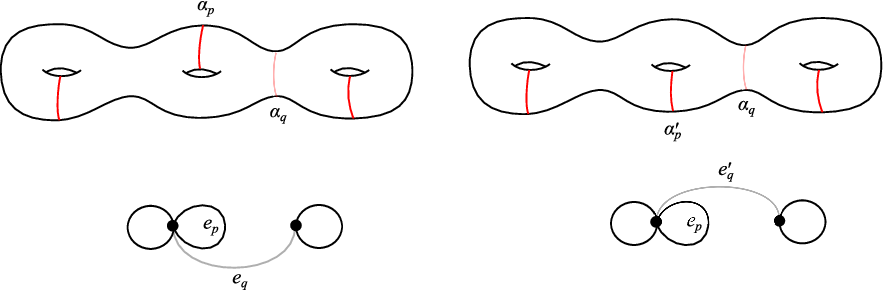}
  \caption{A handleslide in an overcomplete diagram. In this example,
    $T_- = \{e_q\}$ is the only spanning tree of $\Gamma_-$ and $T_-'
    = \{e_q'\}$ is the unique spanning tree of $\Gamma_-'$.  However,
    the diagram $\HD'$ cannot be obtained from $\HD$ by a single
    handleslide.  The $\a$-curves and graph edges corresponding to
    $T_-$ and $T_-'$ are drawn in a lighter color.  The tree $T_-$ is
    not adapted to the handleslide.}
  \label{fig:graph-slide}
\end{figure}

\begin{definition}
  Suppose we have a handleslide of $\a_p$ over $\a_q$ as above. Then
  we say that the spanning tree $T_\pm$ of $\Gamma_\pm$ is
  \emph{adapted to the handleslide} if either
  \begin{itemize}
  \item $e_q \not \in T_\pm$, or
  \item both $e_p, e_q \in T_\pm$.
  \end{itemize}
  We denote by $A_{\a_p/\a_q}(\Gamma_\pm)$ the set of spanning trees
  of $\Gamma_\pm$ adapted to sliding the curve $\a_p$ over $\a_q$.
\end{definition}

\begin{lemma} \label{lem:graph-slide} Given a handleslide as above,
  $A_{\a_p/\a_q}(\Gamma_\pm) \neq \emptyset$ if either $e_p$ is not a
  loop or $e_q$ is not a cut-edge.  Furthermore, there is a bijection
  \[
  b \colon A_{\a_p/\a_q}(\Gamma_\pm) \to A_{\a_p'/\a_q}(\Gamma_\pm')
  \]
  such that, for every spanning tree $T_\pm \in
  A_{\a_p/\a_q}(\Gamma_\pm)$, the sutured diagrams $\HD =
  H(f_{-\eps},v_{-\eps},\S,T_\pm)$ and $\HD' =
  H(f_\eps,v_\eps,\S,b(T_\pm))$ are related by sliding $\a_p$ over
  $\a_q$ if~$e_p, e_q \not\in T_\pm$, and represent the same isotopy
  diagram otherwise.
\end{lemma}

\begin{proof}
  If $e_p$ is not a loop, then there is a spanning tree $T_\pm$ of
  $\Gamma_\pm$ that contains~$e_p$.  Alternatively, if $e_q$ is not a
  cut-edge, then there is a spanning tree
  $T_\pm$ of $\Gamma_\pm$ such that $e_q \not \in T_\pm$. In either case
  $T_\pm \in A_{\a_p/\a_q}(\Gamma_\pm)$ and
  $A_{\a_p/\a_q}(\Gamma_\pm) \neq \emptyset$.

  We now define the map~$b$. If $e_q \not \in T_\pm$, then
  $b(T_\pm)=T_\pm$.  In this case, the diagram~$\HD'$ is obtained
  from~$\HD$ by sliding~$\a_p$ over~$\a_q$ if $e_p \not \in T_\pm$,
  and~$\HD'$ represents the same isotopy diagram as~$\HD$ otherwise.
  If $e_p$, $e_q \in T_\pm$, then $b(T_\pm)$ is
  $T_\pm \setminus \{e_q\} \cup \{e_q'\}$, where~$e_q'$ is obtained by
  sliding~$e_q$ across~$e_p$.
  Now~$\HD$ and~$\HD'$ represent the same isotopy diagram.
\end{proof}

Even if $A_{\a_p/\a_q}(\Gamma_\pm) = \emptyset$ (with $p$, $q$ of
index~$1$), since
$\Gamma_{23}(f_x,v_x)$ and $\Gamma_{23}(f_y,v_y)$ are small isotopic
translates of each other, there is a natural bijection $b$ between
spanning trees of $\Gamma_+$ and $\Gamma_+'$.  If $T_\pm$ and $T_\pm'$
are spanning trees of $\Gamma_\pm$ and $\Gamma_\pm'$, respectively,
such that $T_+' = b(T_+)$, then the corresponding diagrams are
$\a$-equivalent. As the example in Figure~\ref{fig:graph-slide} shows,
this is the best we can hope for, unless we are in one of the lucky
situations of Lemma~\ref{lem:graph-slide}.

\subsection{Codimension-1: Converting Heegaard moves to function
  moves}
\label{sec:codim-1:hd-to-func}

We now turn to the other direction: Given a move on Heegaard diagrams,
can it be converted to a path of functions?

\begin{proposition} \label{prop:lift-moves} Suppose that $\HD_i =
  (\S_i,\alphas_i,\betas_i)$ for $i \in \{0,1\}$ are diagrams of the
  sutured manifold $(M,\g)$ such that $\alphas_i \pitchfork
  \betas_i$. In addition, let $(f_i,v_i) \in \FV_0(M,\g)$ for $i \in
  \{0,1\}$ be simple Morse-Smale pairs with $H(f_i,v_i) = \HD_i$.
  \begin{enumerate}
  \item \label{item:lift-diffeo} Given a diffeomorphism $d \colon
    \HD_0 \to \HD_1$ isotopic to the identity in $M$, there is a
    family $\{\, (f_t,v_t) \,\colon\, t \in [0,1] \,\}$ of simple
    Morse-Smale pairs connecting $(f_0,v_0)$ and $(f_1,v_1)$ that
    induces $d$ in the sense of Lemma~\ref{lem:isotopy}.
  \item \label{item:lift-equiv} If\/ $\HD_0$ and $\HD_1$ are $\a$- or
    $\b$-equivalent, then $(f_0,v_0)$ and $(f_1,v_1)$ can be connected
    by a family of simple (but not necessarily Morse-Smale) pairs
    $(f_t,v_t)$ such that $\S_0 = \S_1 \in \S(f_t,v_t)$ for every $t
    \in [0,1]$.  In particular, every isotopy and handleslide can be
    realized by such a family.
  \item \label{item:lift-stab} If\/ $\HD_1$ is obtained from $\HD_0$
    by a (de)stabilization, then there is a generic family $(f_t,v_t)$
    of sutured functions connecting $(f_0,v_0)$ and $(f_1,v_1)$ such
    that for every $t \neq 1/2$, the pair $(f_t,v_t)$ is simple and
    Morse-Smale, and at $t = 1/2$, there is an index 1-2 birth-death
    bifurcation of $(f_t,v_t)$ realizing the stabilization.
  \end{enumerate}
\end{proposition}

\begin{proof}
  We first prove claim~\eqref{item:lift-diffeo}. Let $\iota_i \colon
  \S_i \hookrightarrow M$ be the embedding.  The statement that $d$ is isotopic to
  the identity in $M$ means that there exists an isotopy $e_t \colon
  \S_0 \to M$ such that $e_0 = \iota_0$ and $e_1 = \iota_1 \circ d$,
  while $e_t(\partial \S_0) = s(\g)$ for every $t \in [0,1]$. This can
  be extended to a diffeotopy $E_t \colon M \to M$ such that
  $E_t|_{\S_0} = e_t$ and $E_0 = \text{Id}_M$. Consider the function
  $g_t = f_0 \circ E_t^{-1}$ and the vector field $w_t = dE_t \circ
  v_0 \circ E_t^{-1}$. Then $(g_t,w_t)$ is a simple Morse-Smale
  pair. If $\S_t = e_t(\S_0)$, $\alphas_t = e_t(\alphas_0)$, and
  $\betas_t = e_t(\betas_0)$, then we have $(\S_t,\alphas_t,\betas_t) \in
  \S(g_t,w_t)$.  Clearly, $(g_0,w_0) = (f_0,v_0)$, but $(g_1,w_1)$ and
  $(f_1,v_1)$ might differ. We define $(f_t,v_t)$ to be
  $(g_{2t},w_{2t})$ for $t \in [0,1/2]$.  By
  Proposition~\ref{prop:connected-MS}, the pairs $(g_1,w_1)$ and
  $(f_1,v_1)$ can be connected by a family $\{\, (f_t,v_t) \,\colon\, t
  \in [1/2,1] \,\}$ of simple Morse-Smale pairs, all adapted to
  $\HD_1$.  In the proof of Lemma~\ref{lem:isotopy}, if we take $d_t$
  to be $e_{2t}$ for $t \in [0,1/2]$ and to be $e_1$ for $t \in [1/2,1]$,
  then $d_t$ satisfies $d_t(\HD_0) = \HD_t \in \S(f_t,v_t)$
  for every $t \in [0,1]$.  Hence the family $\{\, (f_t,v_t) \,\colon\, t
  \in [0,1] \,\}$ indeed induces the diffeomorphism $d_1 = d$, which
  concludes the proof of~\eqref{item:lift-diffeo}.

  Now consider claim~\eqref{item:lift-equiv}, and suppose that $\HD_0
  = (\S,\alphas_0,\betas)$ and $\HD_1 = (\S,\alphas_1,\betas)$ are
  $\a$-equivalent. Then Lemma~\ref{lem:handleslide} implies that,
  after applying a sequence of handleslides to $\alphas_0$, we get an
  attaching set $\alphas_1'$ that is isotopic to $\alphas_1$. Hence,
  it suffices to prove the claim when $\HD_2$ can be obtained from
  $\HD_1$ by an isotopy of the $\a$-curves, or by an $\a$-handleslide.

  First, assume that~$\alphas_0$ and~$\alphas_1$ are related by an
  isotopy. As described by Milnor~\cite[Section~4]{Milnor}, there is
  an isotopy $\{\, w_t \,\colon\, t \in [0,1] \,\}$ of $v_0$, supported in
  a collar neighborhood of $\S$ in the $\a$-handlebody, such that $w_0
  = v_0$, every $w_t$ is gradient-like for $f_0$, and $H(f_0,w_1) =
  \HD_1$.  Once we have arranged that the Heegaard diagrams are equal,
  by Proposition~\ref{prop:connected-MS}, we can
  connect~$(f_0,w_1)$ and~$(f_1,v_1)$ through a family of simple
  Morse-Smale pairs, all adapted to~$\HD_1$.

  Now suppose that $\alphas_0$ and $\alphas_1$ are related by a
  handleslide. In particular, the circle $\a_p = W^u(p) \cap \S$
  corresponding to $p \in C_1(f_0)$ slides over the curve $\a_q =
  W^u(q) \cap \S$ corresponding to $q \in C_1(f_0)$ along some arc $a
  \subset \S$ connecting $\a_p$ and $\a_q$. Again, by
  Milnor~\cite[Section~4]{Milnor}, there is a deformation
  $\{\,(g_t,w_t) \,\colon\, t \in [0,1] \,\}$ of $(f_0,v_0)$ such that
  $(g_0,w_0) = (f_0,v_0)$, every $(g_t,w_t)$ is a simple Morse-Smale
  pair, and~$p$, $q$ are neighboring index~1 critical points of~$g_1$.
  \emph{Neighboring} means that, if~$\xi = g_1(p)$ and~$\eta = g_1(q)$,
  then~$\xi < \eta$, and the only critical points of~$g_1$ in
  $M_{[\xi,\eta]} = g_1^{-1}([\xi,\eta])$ are~$p$ and~$q$.
  We can also assume that $H(g_t,w_t) = \HD_0$ for
  every $t \in [0,1]$, and that $(g_t,w_t)$ coincides with $(f_0,v_0)$
  outside a small regular neighborhood of
  \[
  W^s(p) \cup W^u(p) \cup W^s(q) \cup W^u(q).
  \]
  Let $c = (\xi+\eta)/2$ and $M_c = g_1^{-1}(c)$.  By flowing backwards
  along $w_1$, the arc $a$ gives rise to an arc $a' \subset M_c$.
  Then there is an isotopy $\{\, w_t \,\colon\, t \in [1,2] \,\}$ of $w_1$
  such that
  \begin{itemize}
  \item the isotopy is supported in $M_{[\xi,\eta]}$,
  \item $w_t$ is a gradient-like vector field for $g_1$ for every $t \in [1,2]$,
  \item it isotopes the circle $W^u(p) \cap M_c$ by a finger move
    along $a'$ across one of the points of the 0-sphere $W^s(q) \cap
    M_c$,
  \item $W^u(r) \cap M_c$ is fixed for every $r \in C_1(g_1) \setminus
    \{p\}$.
  \end{itemize}
  The last condition can be satisfied because $a'$ is disjoint from
  the circles $W^u(r) \cap M_c$.  This realizes the handleslide of
  $\a_p$ over $\a_q$; i.e., $H(g_2,w_2) = \HD_1$.  Again, using
  Proposition~\ref{prop:connected-MS}, the pairs~$(g_2,w_2)$
  and~$(f_1,v_1)$ can be connected by a family of simple Morse-Smale
  pairs, all adapted to~$\HD_1$, concluding the proof of
  claim~\eqref{item:lift-equiv}.  Notice that~$(f_t,v_t)$ ceases to be
  Morse-Smale at values of $t$ for which there is a tangency between
  an $\a$- and a $\b$-curve, or when there is an $\a$-handleslide.

  Finally, consider statement~\eqref{item:lift-stab}. Without loss of
  generality, we can suppose that~$\HD_1$ can be obtained from~$\HD_0$
  by a stabilization.  The case of a destabilization follows by
  time-reversal.

  By definition, there is a disk~$D \subset \S_0$ and a punctured
  torus~$T \subset \S_1$ such that $\S_0 \setminus D = \S_1 \setminus T$.
  Furthermore, $\alphas_0 = \alphas_1 \cap (\S_1 \setminus T)$,
  $\betas_0 = \betas_1 \cap (\S_1 \setminus T)$, and there are circles
  $\a = \alphas_1 \cap T$ and $\b = \betas_1 \cap T$ that intersect each other
  transversely in a single point. Let $p \in C_1(f_1)$ and
  $q \in C_2(f_1)$ be the critical points of~$f_1$ for which $W^u(p) \cap
  \S_2 = \a$ and $W^s(q) \cap \S_2 = \b$.  Let~$Z_0$ be the union of
  the flow-lines of~$v_0$ passing through~$D$.
  As $D \cap (\alphas_0 \cup \betas_0) = \emptyset$, the
  manifold $Z_0$ is diffeomorphic to $D \times I$.  Define $Z_1$ to be
  the union of the flow-lines of $v_1$ passing through $T$, together
  with
  \[
  W^s(p) \cup W^u(p) \cup W^s(q) \cup W^u(q).
  \]
  Then $Z_1$ is also diffeomorphic to $D \times I$, since it can be
  obtained from $T \times I$ by attaching 3-dimensional 2-handles
  along $\alpha \times \{0\}$ and $\beta \times \{1\}$.

  The vertical boundary of $Z_i$ is the annulus $A_i$ obtained by
  taking the union of the flow-lines of $v_i$ passing through
  $\partial D = \partial T$. There is an isotopy $\{\, d_t \,\colon\, t
  \in [0,1] \,\}$ of~$M$ such that $d_0 = \text{Id}_M$, $d_1(A_0) =
  A_1$, and $d_t$ fixes $\S_0$ pointwise. Then consider the
  1-parameter family $(f_0 \circ d_t^{-1}, (d_t)_* \circ v_0 \circ
  d_t^{-1})$ of simple Morse-Smale pairs. This isotopes $A_0$ to
  $A_1$. Hence, we can assume that $A_0 = A_1$, which implies that
  $Z_0 = Z_1$. So we will write $A$ for $A_i$ and $Z$ for $Z_i$.  The
  attaching sets $\alphas_0$ and $\betas_0$ might move during this
  process via an isotopy avoiding $D$, but we can undo this using
  claim~\eqref{item:lift-equiv} without changing $A$ anymore. So we
  still have $\alphas_0 = \alphas_1 \cap (\S_1 \setminus T)$ and
  $\betas_0 = \betas_1 \cap (\S_1 \setminus T)$.  It is
  straightforward to arrange that $(f_0,v_0)$ and $(f_1,v_1)$ agree on
  a regular neighborhood of $A$.

  Take the sutured manifold
  \[
  (N,\nu) = \left(\overline{M \setminus Z}, \g \cup A \right).
  \]
  Then $\HD_0' = (\S_0 \setminus D, \alphas_0,\betas_0)$ and
  \[
  \HD_1' = (\S_1 \setminus T, \alphas_1 \setminus \{\a\} ,\betas_1
  \setminus \{\b\})
  \]
  are both diagrams of $(N,\nu)$. If we write $(f_i',v_i') =
  (f_i,v_i)|_N$ for $i \in \{0,1\}$, then $\HD_i' = H(f_i',v_i')$.
  However, as $\HD_0' = \HD_1'$, we can apply
  Proposition~\ref{prop:connected-MS} to get a family $(f_t',v_t')$ of
  simple Morse-Smale pairs on $(N,\nu)$ connecting $(f_0',v_0')$ and
  $(f_1',v_1')$.  On the other hand, observe that
  $(D,\emptyset,\emptyset)$ and $(T,\a,\b)$ are both diagrams of the
  product sutured manifold $(Z,A)$ that are related by a
  stabilization, hence it now suffices to prove
  claim~\eqref{item:lift-stab} for this special case. Indeed, we can
  simply glue the family connecting $(f_0,v_0)|_Z$ and $(f_1,v_1)|_Z$
  to the family $(f_t',v_t')$.

  Consider $\R^3$ with the standard coordinates $(x,y,z)$. Let
  \[
  G_t(x,y,z) = x^3 - y^2 + z^2 + (1/2 - t)x,
  \]
  with gradient vector field
  \[
  W_t(x,y,z) = (3x^2 + 1/2 - t, -2y, 2z).
  \]
  Then $G_t$ has a bifurcation at $t = 1/2$, where a pair of index~1
  and~2 critical points are born. Let
  \[
  B_t = G_t^{-1}([-1,1]) \cap D^3_2 \quad \text{and} \quad \eta_t =
  B_t \cap \partial D^3_2,
  \]
  where $D^3_2$ is the unit disk in $\R^3$ of radius~2.  Furthermore,
  let $g_t = G_t|_{B_t}$ and $w_t = W_t|_{B_t}$. It is straightforward
  to check that $(B_t,\eta_t)$ is diffeomorphic to the product sutured
  manifold $(D^2,\partial D^2 \times I)$ for every $t \in [0,1]$.  In
  addition, $H(g_0,w_0) = (D',\emptyset,\emptyset)$, where $D' =
  g_0^{-1}(0)$ is a disk, while $H(g_1,w_1) = (T',\a',\b')$, where $T'
  = g_1^{-1}(0)$ is a punctured torus, and $\a'$ and $\b'$ are simple
  closed curves that intersect each other in a single point. There
  exists a smooth family of diffeomorphisms $h_t \colon (B_t,\eta_t)
  \to (Z,A)$ such that $h_0(D') = D$, $h_1(T') = T$, $h_1(\a') = \a$,
  and $h_1(\b') = \b$. Pushing $(g_t,w_t)$ forward along $h_t$, we get
  a family on $(Z,A)$ that we also denote by $(g_t,w_t)$.  According
  to Proposition~\ref{prop:connected-MS}, for $i \in \{0,1\}$, the
  pair $(f_i,v_i)|_Z$ can be connected with $(g_i,w_i)$ via a family
  of simple Morse-Smale pairs.  This concludes the proof of
  claim~\eqref{item:lift-stab}.
\end{proof}

\subsection{Codimension-2}
\label{sec:transl-codim-2}
The singularities of gradient vector fields that appear in ge\-ner\-ic
2-parameter families were given in Section~\ref{sec:2-param}. This
also applies to gradient-like vector fields on sutured manifolds by
Proposition~\ref{prop:grad-like-metric}. In this section, we associate a loop
of sutured Heegaard diagrams and Heegaard moves to each codimension-2 singularity.

Let $(f_\mu,v_\mu)$ for $\mu \in \R^2$ be a generic 2-parameter family
of sutured functions and gradient-like vector fields on the sutured
manifold $(M,\g)$ that has a codimension-2 singularity for $\mu = 0$;
i.e., $(f_0,v_0) \in \FV_2(M,\g)$.  Recall the notion of the
bifurcation set in the parameter space from
Definition~\ref{def:bifurcation}; this is the set $S$ of parameter
values $\mu \in \R^2$ for which $v_\mu$ fails to be Morse-Smale.
Then, for $\eps > 0$ sufficiently small, the set $(S \cap D^2_\eps)
\setminus \{0\}$ is the disjoint union of smooth arcs (strata)
$S_1,\dots,S_r$ with $0 \in \partial S_i$ and $\partial S_i \setminus
\{0\} \in S^1_\eps$. We label the arcs $S_i$ in a clockwise
manner. The components of $D^2_\eps \setminus S$ are chambers
$C_1,\dots, C_r$, labeled such that $C_i$ lies between $S_{i-1}$ and
$S_i$ for $i \in \{\,1, \dots, r \,\}$ (where $S_0 = S_r$ by
definition).

In this section, the bifurcation diagrams that we draw illustrate the
bifurcation set $S \subset \R^2$ in a neighborhood $D^2_\eps$ of $0$,
and for each chamber $C_i$, we indicate the relevant part of the
corresponding (overcomplete) Heegaard diagram $H(f_\mu,v_\mu,\S)$ for
$\mu \in C_i$ near $0$ and some Heegaard surface $\S \in
\S(f_\mu,v_\mu)$.  (Note that if $\mu$, $\mu' \in C_i$, then the
vector fields $v_\mu$ and $v_{\mu'}$ are topologically equivalent,
hence the corresponding diagrams are homeomorphic and close to each
other.)  We only show certain subsurfaces of~$\S$ in our illustrations
and draw the boundary of these in green.  Outside these subsurfaces,
the diagrams are related by a small isotopy of $\alphas \cup \betas$.
Following our previous conventions, $\a$-circles are drawn in red,
while $\b$-circles are drawn in blue.

Consider an arc $S_i$, and pick a short curve $c \colon [-\nu,\nu] \to
\R^2$ transverse to $S_i$ at~$c(0)$. This gives rise to a 1-parameter
family
\[
\{\, (f_{c(t)},v_{c(t)}) \,\co\, t \in [-\nu,\nu] \,\}
\]
to which
we can apply Proposition~\ref{prop:1-param}.  If the diagrams for
$(f_{c(-\nu)},v_{c(-\nu)})$ and $(f_{c(\nu)},v_{c(\nu)})$ are related
by an $\a$-equivalence, then we draw $S_i$ in red; if they are related
by a $\b$-equivalence, then we draw $S_i$ in blue; and $S_i$ is black
if they are related by a (de)stabilization.

\begin{definition}
  Suppose that $\{\, (f_\mu,v_\mu) \,\colon\, \mu \in \R^2 \,\}$ is a
  generic 2-parameter family such that $(f_0,v_0) \in \FV_2(M,\g)$.
  For $\eps > 0$ as above, a \emph{link of the bifurcation} at $0$ is
  an embedded polygonal curve $P \subset D^2_\eps$ such that
  \begin{itemize}
  \item the bifurcation value $0$ lies in the interior of $P$,
  \item $P \pitchfork S$ and $|S_i \cap P| = 1$ for every $i \in \{\,
    1,\dots,r \,\}$,
  \item each chamber $C_i$ contains exactly one or two vertices of
    $P$.
  \end{itemize}
  We say that $P$ is \emph{minimal} if each $C_i$ contains precisely
  one vertex of $P$.  We orient the curve~$P$ in a clockwise manner.

  A \emph{surface enhanced link} of the bifurcation at $0$ is a link
  $P$, together with a choice of Heegaard surface $\S_\mu \in
  \S(f_\mu,v_\mu)$ for each vertex $\mu$ of $P$.
\end{definition}

We will use the following notational convention. If $C_i$ contains one
vertex of $P$, then we denote that by $\mu_i$. The edge of $P$ that
intersects $S_i$ is called $a_i$.  If $C_i$ contains two vertices,
then they are denoted by $\mu_i$ and $\mu_i'$, ordered coherently with
the orientation of the edge $a_i'$ of $P$ between them. So $\partial
a_i$ is either $\mu_{i+1} - \mu_i$ or $\mu_{i+1} - \mu_i'$, and
$\partial a_i' = \mu_i' - \mu_i$.
In particular, if $P$ is minimal, then the vertices of $P$ are
$\mu_1,\dots,\mu_r$ and its edges are $a_1,\dots,a_r$.  For
simplicity, we write $(f_i,v_i)$ for $(f_{\mu_i},v_{\mu_i})$,
$(f_i',v_i')$ for $(f_{\mu_i'},v_{\mu_i'})$, $\S_i$ for $\S_{\mu_i}$,
and $\S_i'$ for $\S_{\mu_i'}$. Furthermore, we write $H_i =
(\S_i,[\alphas_i],[\betas_i])$ for the (overcomplete) isotopy diagram
$[H(f_i,v_i,\S_i)]$ and $H_i' = (\S_i',[\alphas_i'],[\betas_i'])$ for
$[H(f_i',v_i',\S_i')]$.

The cases distinguished in the following result are labeled
consistently with the ones appearing in the bifurcation analysis of
Section~\ref{sec:2-param}.

\begin{theorem} \label{thm:2-param} Suppose that
  \[
   \mathcal{F} = \{\,(f_\mu,v_\mu) \,\colon\, \mu \in \R^2 \,\}
  \]
  is a generic 2-parameter
  family such that $(f_0,v_0) \in \FV_2(M,\g)$.  Using the above notation,
  for every $\eps > 0$ sufficiently small, there exists a
  surface enhanced link $P \subset D^2_\eps$ of the bifurcation at $0$
  such that the following hold.  The polygon $P$ is minimal unless the
  bifurcation at $0$ is of type~\ref{item:C} or~\ref{item:E1}.  For $i
  \in \{\, 1, \dots, r \,\}$, there is a point $x_i \in \partial a_i$
  such that $\S_{x_i} \pitchfork v_\mu$ for every $\mu \in
  a_i$. Consecutive isotopy diagrams $H_i$ and $H_{i+1}$, or $H_i'$
  and $H_{i+1}$, are related by a move corresponding to the type of
  the stratum $S_i$.  As in Lemma~\ref{lem:isotopy}, each edge $a_i'$
  induces a diffeomorphism $d_i \colon H_i \to H_i'$ isotopic to the
  identity in $M$.
  We explicitly describe the surface enhanced link~$P$ for each
  bifurcation of type~\ref{item:A}--\ref{item:E} in the proof below.
\end{theorem}

\begin{proof}
We may ignore the strata $S_i$ that correspond to an index~0-1 or an
index 2-3 birth-death, or a tangency between the unstable manifold of
an index~1 critical point and the stable manifold of an index~2
critical point, as the isotopy diagrams defined by
$H(f_i,v_i,\S_i,T_\pm^i)$ and
$H(f_{i+1},v_{i+1},\S_{i+1},T_\pm^{i+1})$ coincide if we take $\S_i =
\S_{i+1}$ and choose $T_\pm^i$ and $T_\pm^{i+1}$ consistently (see the
discussion of trees following Proposition~\ref{prop:1-param}). If this
reduces the bifurcation set to a single curve of codimension-1
bifurcations or eliminates it completely, then we do not list the
bifurcation below. This simplification reduces the number of cases
considerably, though no extra technical difficulty arises in
the proof of Theorem~\ref{thm:2-param} in the
omitted cases.  We use the notation of Section~\ref{sec:2-param}, with
the codimension-2 bifurcation appearing at the parameter value $\om =
0$. Whenever we talk about handleslides, we mean generalized
handleslides, as in Definition~\ref{def:gen-handleslide}.

In all the cases where $(f_0,v_0)$ is separable (i.e., everywhere
except in case~\ref{item:E1}) we construct the surfaces
$\S_1,\dots,\S_r$ (and $\S_3'$ in case~\ref{item:C}) from a common
surface $\S \in \S(f_0,v_0)$ with the aid of
Proposition~\ref{prop:grad-like}.  In these cases, we take $\eps$ so
small that $\S \pitchfork v_\mu$ for every $|\mu| < \eps$.  Often, $\S
\in \S(f_i,v_i)$ for every $i \in \{\, 1,\dots,r \,\}$; for example,
when $f_0$ is Morse (this includes all bifurcations of
type~\ref{item:A}), or has an index~0-1 or~2-3 birth-death
singularity, or an index 0-1-0, 1-0-1, 2-3-2, or 3-2-3
birth-death-birth. When $f_0$ has an index 1-2 birth-death, then we
can construct surfaces on the two sides of the corresponding stratum
as in the proof of Proposition~\ref{prop:rectangle}.  We only explain
how to construct the surfaces $\S_1,\dots,\S_r$ whenever a new idea is
needed.

As stated above, in cases~\ref{item:C} and~\ref{item:E1}, the link $P$
is not minimal. In the corresponding figures, if $C_i$ contains two
vertices of $P$, we will draw a yellow ray in $C_i$ emanating from $0$
that separates $\mu_i$ and $\mu_i'$. The reader should think of this
ray as a ``diffeomorphism stratum'' of the bifurcation set. The
purpose of this will be explained in the following section.

We start by looking at bifurcations of type~\ref{item:A}, which were
illustrated schematically in Figure~\ref{fig:flow-pairs}.  For
cases~\ref{item:B} and~\ref{item:C}, the reader should consult
Figure~\ref{fig:codim-two-cases}, while for cases~\ref{item:D}
and~\ref{item:E1}, see Figure~\ref{fig:codim-two-local}.

In all subcases of case~\ref{item:A}, we can take an arbitrary minimal
link $P \subset D^2_\eps$ and $\S_i = \S$ for $i \in \{\, 1, \dots, r
\,\}$. First, we describe the possibilities in case~\ref{item:A1};
see Figure~\ref{fig:link-A1}. In each case, $r = 4$ and the
bifurcation set $S$ is the union of two smooth curves that intersect
transversely at~$0$.

\begin{figure}
\centering
\includegraphics{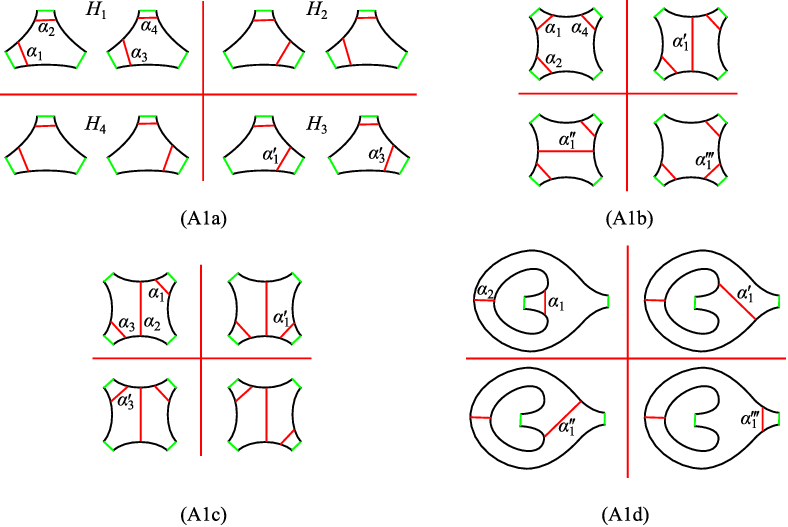}
\caption{The links of bifurcations of
  type~\ref{item:A1a}--\ref{item:A1d}.  The surfaces shown
  should be doubled along their black boundary arcs to obtain the
  relevant subsurface of the Heegaard diagram. This way the red arcs
  become the $\a$-circles and the green arcs become the boundary
  components of the subsurface. We do not draw the $\b$-circles as
  they remain unchanged.}
\label{fig:link-A1}
\end{figure}

\begin{enumerate}[label = (A1\alph*)]
\item \label{item:A1a} Suppose $\mathcal{F}$ is of type~\ref{item:A1}, and using
  the notation of Section~\ref{sec:2-param}, all $p_i^0$ are
  distinct, $\I(p_1^0) = \I(p_2^0) \in \{1,2\}$ and $\I(p_3^0) =
  \I(p_4^0) \in \{1,2\}$. We describe what happens to the diagrams
  $H_i$ when all the $p_i^0$ have index~1; the other cases are
  analogous. Then $\betas_i = \betas_1$ for $i \in
  \{\,2,3,4\,\}$. Furthermore, the attaching set $\alphas_1$ contains
  four distinct curves $\a_1,\dots,\a_4$ (corresponding to $p_1^0,
  \dots, p_4^0$, respectively), and $\alphas_3$ contains two distinct
  curves $\a_1'$ and $\a_3'$ such that $\a_1'$ is obtained by sliding
  $\a_1$ over $\a_2$ and $\a_3'$ is obtained by sliding $\a_3$ over
  $\a_4$. In addition, $\alphas_2 = (\alphas_1 \setminus \a_1) \cup
  \a_1'$, $\alphas_4 = (\alphas_1 \setminus \a_3) \cup \a_3' $, and
  $
  \alphas_3  =\left( \alphas_1 \setminus (\a_1 \cup \a_3) \right) \cup (\a_1' \cup \a_3').
  $
\item \label{item:A1b} Suppose $\mathcal{F}$ is of type~\ref{item:A1}
  and $p_1^0 = p_3^0$. The points $p_1^0$, $p_2^0$, and $p_4^0$ all have index~1,
  or they all have index~2. We discuss the case when they are all
  index~1.  Then there are curves $\a_1$, $\a_2$, $\a_4 \in
  \alphas_1$, and curves $\a_1' \in H_2$, $\a_1'' \in H_4$, and
  $\a_1''' \in H_3$ such that $\a_1'$ is obtained by sliding $\a_1$
  over $\a_2$, the curve $\a_1''$ is obtained by sliding $\a_1$ over
  $\a_4$, while $\a_1'''$ can be obtained by either sliding $\a_1'$
  over $\a_4$, or $\a_1''$ over $\a_2$. Furthermore, $\alphas_2 =
  (\alphas_1 \setminus \a_1) \cup \a_1'$, $\alphas_3 = (\alphas_1
  \setminus \a_1) \cup \a_1'''$, and $\alphas_4 = (\alphas_1 \setminus
  \a_1) \cup \a_1''$.  In other words, $H_2$ is obtained from $H_1$ by
  sliding $\a_1$ over $\a_2$, the diagram $H_4$ is obtained from $H_1$
  by sliding $\a_1$ over $\a_4$, and $H_3$ is obtained from $H_1$ by
  sliding $\a_1$ over $\a_2$, and then sliding the resulting curve
  over $\a_4$.
\item \label{item:A1c} Suppose $\mathcal{F}$ is of type~\ref{item:A1}
  and $p_2^0 = p_4^0$. The set $\alphas_1$ contains three distinct curves $\a_1$,
  $\a_2$, and $\a_3$; furthermore, there is an arc $a_1$ with
  $\partial a_1 \subset \a_2$ and an arc $a_3$ with $\partial a_3
  \subset \a_2$ such that $a_1$ and $a_3$ reach $\a_2$ from opposite
  sides, and $H_2$ is obtained from $H_1$ by sliding $\a_1$ over
  $\a_2$ using $a_1$ (resulting in a curve $\a_1'$), $H_4$ is obtained
  from $H_1$ by sliding $\a_3$ over $\a_2$ using $a_3$ (resulting in
  the curve $\a_3'$), while $H_3$ differs from $H_1$ by removing
  $\a_1$ and $\a_3$ and adding $\a_1'$ and $\a_3'$.
\item \label{item:A1d} Suppose $\mathcal{F}$ is of type~\ref{item:A1},
  $p_1^0 = p_3^0$, and $p_2^0 = p_4^0$. Both $p_1^0$ and $p_2^0$ have index~1, or they
  both have index~2.  This is similar to case~\ref{item:A1b}, except
  that $\alpha_1$ is sliding over the same curve $\alpha_2$
  in two different ways from opposite sides.
\end{enumerate}

\begin{enumerate}[label = \ref{item:A2}]
\item \label{item:link-A2} In this case, we have $\I(p_1^0) =
  \I(p_2^0) = \I(p_3^0) \in \{1,2\}$ and $r = 5$.  We consider the
  case when all the $p_i^0$ are index~1. The $\betas_i$ coincide up to
  a small isotopy.  The attaching set $\alphas_1$ contains three
  distinct curves $\a_1$, $\a_2$, and $\a_3$ corresponding to
  $p_1^0$, $p_2^0$, and $p_3^0$, respectively.  Then the pentagon is
  formed by $\a_1$ sliding over $\a_2$, which is itself sliding over
  $\a_3$.  More precisely, let $\a_1'$ be the curve obtained from
  $\a_1$ by sliding it over $\a_2$, let $\a_2'$ be the curve obtained
  from $\a_2$ by sliding it over $\a_3$, and finally let $\a_1''$ be
  the curve
  obtained from $\a_1$ by sliding it over $\a_2'$. Then $\alphas_2 =
  (\alphas_1 \setminus \a_2) \cup \a_2'$, $\alphas_3 = (\alphas_2
  \setminus \a_1) \cup \a_1''$, $\alphas_4 = (\alphas_3 \setminus
  \a_2') \cup \a_2$, and $\alphas_5 = (\alphas_4 \setminus \a_1'') \cup
  \a_1'$.  In particular, this implies that $\alphas_1 = (\alphas_5
  \setminus \a_1') \cup \a_1$. For a schematic illustration, see
  Figure~\ref{fig:link-chain-flows}.
  \begin{figure}
    \centering
    \includegraphics{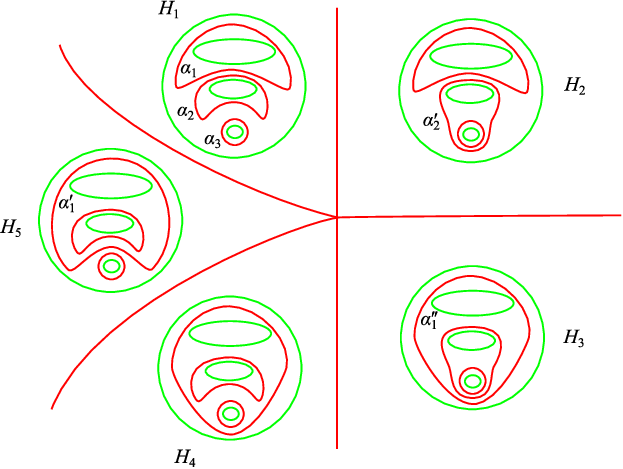}
    \caption{The link of a bifurcation of type~\ref{item:A2}.}
    \label{fig:link-chain-flows}
  \end{figure}
\end{enumerate}

We now look at bifurcations of type~\ref{item:B}; i.e., codimension-2
singularities that include a single stabilization.  See
Figure~\ref{fig:codim-two-cases} for schematic drawings.  The link $P$
and the surfaces $\S_i$ are obtained as follows. We label the strata
such that $S_1$ and $S_3$ are the stabilizations, and there is a single
stratum~$S_2$ on the stabilized side.
We choose $\S \in \S(f_0,v_0)$ and $\eps$ as above.
For $i \not\in \{2,3\}$, the vertex $\mu_i$ of $P$ is an arbitrary point of
$C_i$ and $\S_i = \S$.  Pick a parameter value $\nu \in S_2$ with
$|\nu| < \eps$. Let $p^0 \in C(f_0)$ be the index 1-2 birth-death
singularity; it breaks into the critical points $p^1_\nu \in
C_1(f_\nu)$ and $p^2_\nu \in C_2(f_\nu)$. The surface
\[
\S_\nu \in \S(f_\nu,v_\nu)
\]
is obtained from $\S$ by attaching a tube around
$W^u(p^2_\nu)$ if $p^0 \in C_{01}(f_0,v_0)$, or a tube around
$W^s(p^1_\nu)$ if $p^0 \in C_{23}(f_0,v_0)$.
The side $a_2$ of $P$ is chosen short enough so that $\S_\nu \pitchfork
v_\mu$ for every $\mu \in a_2$. The endpoints of $a_2$ are $\mu_2$ and
$\mu_3$. Both $\S_2$ and $\S_3$ are defined to be $\S_\nu$.  Every
side $a_i$ of $P$ for $i \neq 2$ can be an arbitrary arc connecting
$\mu_i$ and $\mu_{i+1}$ that intersects $S_i$ in a single point.

\begin{enumerate}[label = \ref{item:B\arabic*}]
\item \label{item:link-B1}
  For definiteness, suppose that $p^0$ is an index 1-2 birth, while
  $p_1^0$ and $p_2^0$ are index~1 critical points. Then $r=4$, and the
  strata $S_1$ and $S_3$ are stabilizations, while $S_2$ and $S_4$ are
  handleslides. Recall from Definition~\ref{def:partition} that, in
  this case, $p^0 \in C_{01}(f_0,v_0)$.
  The type of the stabilizations $H_1 \to H_2$ and $H_4 \to H_3$
  depend on the number of flows from $p_1^\mu$ to $p^\mu$ for $\mu \in
  S_1$ and $\mu \in S_3$, respectively.  Recall from \ref{item:eigenvalues}
  that $W^{uu}(p_2^0)$ is 1-dimensional, and it is transverse to~$W^s(p^0)$
  by \ref{item:failure-H};
  i.e., disjoint from it. Hence, the flows from $p_2^0$ to $p^0$
  are split into two parts by $W^{uu}(p_2^0)$.
  For~$\mu \in S_1$, flows in one part will be glued to the orbit of
  tangency from~$p_1^0$ to~$p_2^0$, while for $\mu \in S_3$ flows in
  the other part will be glued to the orbit of tangency.
  If there are~$k_1$ and~$k_2$ flows from~$p_2^0$ to~$p^0$ in the two
  parts, and $l$ flows from~$p^0$ to index~2 critical points and $m$
  flows to~$p^0$ from index 1 critical points, then the two
  stabilizations $H_1 \to H_2$ and $H_4 \to H_3$ are of types $(l,
  m+k_1)$ and $(l, m+k_2)$, respectively.

  Figure~\ref{fig:link-6a}
  shows an example with $k_1 = k_2 = 1$. In general, the $\a$-curve
  corresponding to~$p_2^\lambda$ intersects the green disk in~$k_1 + k_2$
  horizontal segments, $k_1$ of which lie on one side of~$W^{uu}(p_2^\lambda)$
  and~$k_2$ on the other side. So the $\a$-curve corresponding to~$p_1^\lambda$
  intersects the green disk in~$k_1$ arcs on one side of the handleslide stratum,
  and in~$k_2$ arcs on the other side.
  When $p_1^0$ and $p_2^0$ are index~2, then we
  obtain a similar picture, but with red and blue reversed. (This
  is ensured by the convention of Definition~\ref{def:separable}
  that now $p^0 \in C_{23}(f_0,v_0)$.)%
  \begin{figure}
    \centering
    \includegraphics[width=4.5in]{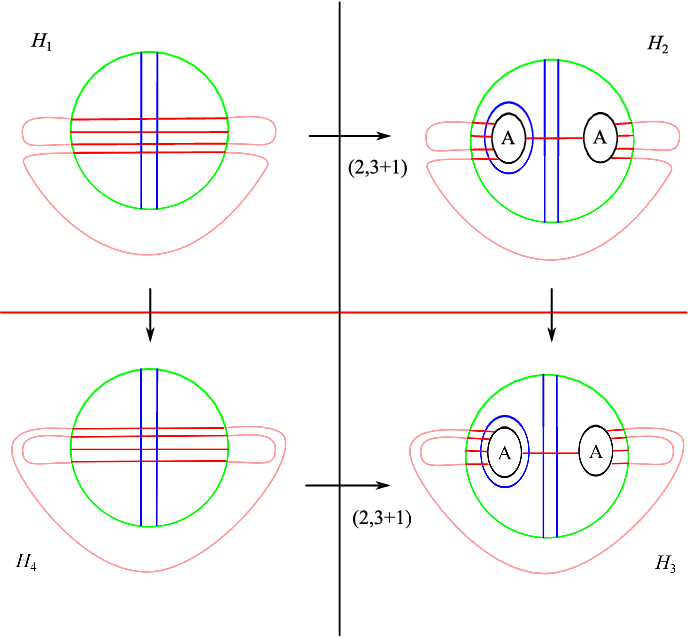}
    \caption{The link of a singularity of type~\ref{item:B1}.  This
      example has $l=2$, $m=3$, and $k_1 = k_2 = 1$.}
    \label{fig:link-6a}
  \end{figure}
\item \label{item:link-B2} An orbit of tangency from an index 1-2
  birth-death point~$p^0$ to an index 1 critical point $\ol{p}^0$ (in
  which case $p^0 \in C_{01}(f_0,v_0)$), or an orbit of tangency from
  an index~$2$ critical point to an index 1-2 birth-death point (in
  which case $p^0 \in C_{23}(f_0,v_0)$). For definiteness, we consider
  the first case.  The bifurcation diagram has at least $r \ge 3$
  strata, where $S_1$ and $S_3$ are stabilizations and the other $S_i$
  for $i \not\in \{1,3\}$ are $\a$-handleslides.  Indeed, for any flow
  from an index~1 critical point $p^0_*$ to $p^0$, we can perturb the
  neighborhood of $p^0$ on the ``death'' side of $S_1 \cup S_3$ so
  that there is a flow from $p^\mu_*$ to $\ol{p}^\mu$. The number of
  flows from index 1 critical points to $p^0$ is equal to $r-3$.

  For the types of the stabilizations, suppose that there are $k$
  flows from index 1 critical points to~$p^0$ and $l$ flows from $p^0$
  to index 2 critical points. Furthermore, the flows from $\ol{p}^0$
  to index 2 critical points are divided into two parts by
  $W^{uu}(\ol{p}^0)$; let these two parts have $m_1$ and $m_2$ flows,
  respectively. Then the two stabilizations $H_1 \to H_2$ and $H_4 \to
  H_3$ have types $(l+m_1,k)$ and $(l+m_2,k)$, respectively (where
  $H_4 = H_1$ if $r=3$).  The pair $(m_1,m_2)$ can be seen as the type
  of the generalized handleslide $H_2 \to H_3$.
  Figure~\ref{fig:link-7a} shows an example. When $\ol{p}$ is index~2,
  we obtain a similar picture, but with red and blue reversed.
  \begin{figure}
    \centering
    \includegraphics{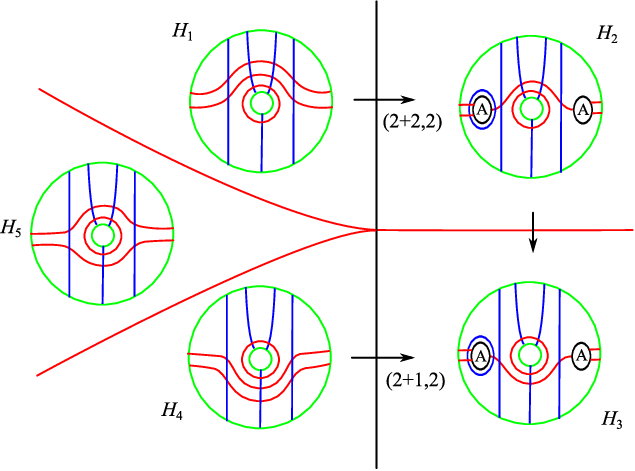}
    \caption{The link of a singularity of type~\ref{item:B2}.  This example has
      $k = 2$, $l = 2$, $m_1 = 2$, and $m_2 = 1$.}
    \label{fig:link-7a}
  \end{figure}

\item \label{item:link-B3} An orbit of tangency between the strong
  stable manifold of an index 1-2 birth-death point~$p$ and the
  unstable manifold of an index 1 critical point~$\ol{p}$, or between
  the strong unstable manifold of an index 1-2 birth-death point and
  the stable manifold of an index 2 critical point. Without loss of
  generality, suppose we are in the former case. Then $r = 3$, the
  strata $S_1$ and $S_3$ are stabilizations, while the stratum $S_2$
  is a handleslide. Recall that we chose $p \in C_{01}(f_0,v_0)$.  If
  there are $k$ flow lines from index~1 critical point to $p$ (not
  counting the flow from $\ol{p}$ in $W^{ss}(p)$) and $l$ flows from
  $p$ to index 2 critical points, then the stabilizations $H_1 \to
  H_2$ and $H_1 \to H_3$ are of types $(k,l+1)$ and $(k,l)$,
  respectively. For $\mu \in S_2$, the 2-dimensional unstable manifold
  $W^u(\ol{p})$ has a tangency with the 1-dimensional stable manifold
  of the index~1 critical point born from $p$; see
  Figure~\ref{fig:1-2-birth-death}.  Hence, the $\a$-curve
  $W^u(\ol{p}) \cap \S$ slides over the $\a$-curve appearing in the
  stabilization as we move from $H_3$ to $H_2$.
  Figure~\ref{fig:link-9a} shows an example. As before,
  Definition~\ref{def:separable} ensures that we obtain a picture with
  colors reversed when $\ol{p}$ is index~2.
  \begin{figure}
    \centering
    \includegraphics{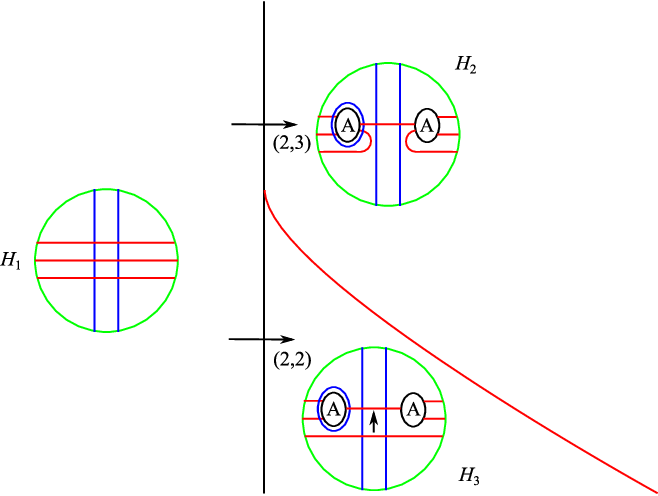}
    \caption{The link of a singularity of type~\ref{item:B3}.  This example has
      $k=2$ and $l=2$.}
    \label{fig:link-9a}
  \end{figure}
\end{enumerate}

\begin{enumerate}[label=\ref{item:C}]
\item \label{item:link-C} Two simultaneous index 1-2 birth-death
  critical points at $p_1$ and~$p_2$, labeled such that $f(p_1) <
  f(p_2)$. Recall that, in this case, $(f_0,v_0)$ is separable with
  $p_1 \in C_{01}(f_0,v_0)$ and $p_2 \in C_{23}(f_0,v_0)$. Let the
  number of flows from $p_1$ to $p_2$ be $t$, and let the number of flows from
  $p_1$ to index~2 critical points and from index~1 critical points to
  $p_1$ be $m$ and $n$, respectively. Similarly, let the number of flows
  to index~2 critical points from $p_2$ and from index~1 critical
  points to $p_2$ be $k$ and $l$, respectively. Then $r = 4$, and
  each stratum $S_i$ is a (de)stabilization.  The strata $S_1$ and
  $S_3$ correspond to the birth-death at $p_1$, while $S_2$ and $S_4$
  are the birth-death strata for $p_2$.  The type of the stabilization
  $H_1 \to H_2$ is $(kt + m, n)$, for $H_2 \to H_3$ it is $(k, l+t)$,
  for $H_1 \to H_4$ it is $(k, nt + l)$, and finally, for $H_3 \to
  H_4$ it is $(m+t, n)$.  Figure~\ref{fig:link-8} shows an example
  with $t=1$, and Figure~\ref{fig:link-8-alt} shows an example with
  $t=2$.
  \begin{figure}
    \centering
    \includegraphics{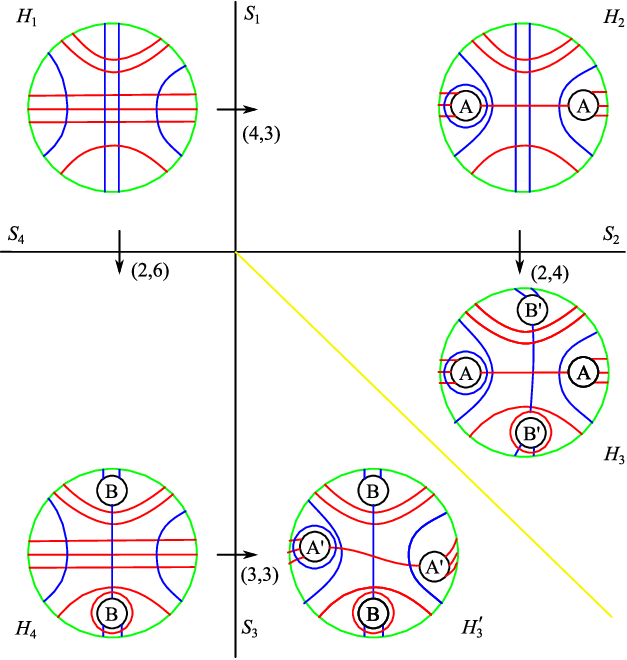}
    \caption{The link of a singularity of type~\ref{item:C}. This
      example has $t=1$, $(k,l)=(2,3)$, and $(m,n)=(2,3)$.}
    \label{fig:link-8}
  \end{figure}
  \begin{figure}
    \centering
    \includegraphics[width=4in]{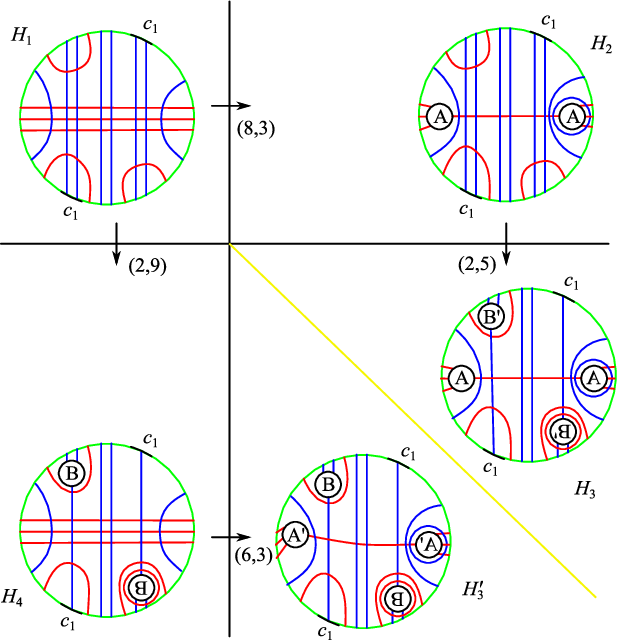}
    \caption{The link of a more complicated singularity of
      type~\ref{item:C}.  This example has $t=2$, $(k,l) = (2,3)$, and
      $(m,n) = (4,3)$. The black arcs labeled with $c_1$ in the
      boundary of the green circle are identified.}
    \label{fig:link-8-alt}
  \end{figure}

  The vertices $\mu_1$, $\mu_2$, $\mu_3$, $\mu_3'$, and $\mu_4$ of $P$
  and the Heegaard surfaces $\S_i \in \S(f_i,v_i)$ for $i \in
  \{\,1,\dots,4\,\}$ and $\S_3' \in \S(f_3',v_3')$ are obtained as
  follows. As before, pick a surface $\S \in \S(f_0,v_0)$, and let
  $\eps > 0$ be so small that $\S \pitchfork v_\mu$ for every $|\mu| <
  \eps$. Then $\mu_1 \in C_1$ is arbitrary and $\S_1 = \S$.  For $\nu
  \in C_2 \cup S_2 \cup C_3$, let $q_1(\nu) \in C_2(f_\nu)$ be the
  index~2 critical point born from $p_1$. Similarly, for $\eta \in C_3
  \cup S_3 \cup C_4$, let $q_2(\eta) \in C_1(f_\eta)$ be the index~1
  critical point born from $p_2$.  By taking $\eps$ to be sufficiently
  small, we can assume that $W^u(q_1(\nu)) \cap W^s(q_2(\eta)) =
  \emptyset$ for every $\nu$, $\eta \in S_2 \cup C_3 \cup S_3$.
  Choose points $\nu \in S_2$ and $\eta \in S_3$. The surface $\S_\nu$
  is obtained from $\S$ by attaching a tube around $W^u(q_1(\nu))$
  such that $\S_\nu \in \S(f_\nu,v_\nu)$. Similarly, $\S_\eta$ is
  obtained from $\S$ by attaching a tube around $W^s(q_2(\eta))$ such
  that $\S_\eta \in \S(f_\eta,v_\eta)$.  Pick short arcs $a_2$ and
  $a_3$ transverse to $S_2$ and $S_3$ at $\nu$ and $\eta$,
  respectively, such that $\S_\nu \pitchfork v_\mu$ for every $\mu \in
  a_2$ and $\S_\eta \pitchfork v_\mu$ for every $\mu \in a_3$.  We
  take $\mu_2 = \partial a_2 \cap C_2$, $\mu_3 = \partial a_2 \cap
  C_3$, $\mu_3' = \partial a_3 \cap C_3$, and $\mu_4 = \partial a_3
  \cap C_4$. Furthermore, $\S_2 = \S_\nu$ and $\S_4 = \S_\eta$. To
  obtain $\S_3 \in \S(f_3,v_3)$, add a tube to $\S_2$ around
  $W^s(q_2(\mu_3))$. Similarly, $\S_3' \in \S(f_3',v_3')$ is obtained
  from $\S_4$ by adding a tube around $W^u(q_1(\mu_3'))$. The edges
  $a_1$, $a_3'$, and $a_4$ of $P$ are chosen arbitrarily (subject
  to $|a_i \cap S_i| = 1$ for $i \in \{1,4\}$ and $a_3' \subset C_3$).

  The regions of $\S$ shown in Figures~\ref{fig:link-8}
  and~\ref{fig:link-8-alt} are obtained by taking a regular
  neighborhood $N$ of $(W^u(p_1) \cup W^s(p_2)) \cap \S$; so the green
  curves are the components of $\partial N$.  Recall that both
  $W^u(p_1) \cap \S$ and $W^s(p_2) \cap \S$ are arcs, which intersect
  each other in $t$ points $x_1, \dots, x_t$. For $i \in \{\,1, \dots,
  t-1\,\}$, let $c_i$ be a properly embedded arc in $N$ that
  intersects $W^s(p_2) \cap \S$ transversely in a single point between
  $x_i$ and $x_{i+1}$.  Cutting $N$ along $c_1,\dots, c_{t-1}$, we
  obtain a disk with distinguished pairs of arcs in its boundary. This
  is the disk that we draw in our figures, with the $c_i$ shown in
  black.
\end{enumerate}

\begin{remark} \label{rem:diffeo-C}
  To avoid the ``diffeomorphism stratum'' in case (C), one would need
  a quantitative result that the arcs $a_2$ and $a_3$ can be chosen to
  be so long that they intersect, in which case we could take
  $\mu_3 = a_2 \cap a_3$ and drop $\mu_3'$. This does not seem
  possible for an arbitrary 2-parameter family.

  Note that the diffeomorphism $d_3 \colon H_3 \to H_3'$ induced by
  $a_3'$ can be destabilized to a diffeomorphism $d_3' \colon \S \to
  \S$. This follows from the fact that $\S \pitchfork v_\mu$ for every
  $\mu \in a_3'$ and both $\S_3$ and $\S_3'$ are obtained by attaching
  tubes to $\S$. Indeed, consider the family of surfaces $\S_\mu \in
  \S(f_\mu,v_\mu)$ for $\mu \in a_3'$ obtained by adding tubes around
  $W^u(q_1(\mu))$ and $W^s(q_2(\mu))$.  Then one can apply
  Lemma~\ref{lem:isot-ext} to lift this family of surfaces to an
  isotopy that preserves the ``tubes.''
\end{remark}

\begin{enumerate}[label = \ref{item:D}]
\item \label{item:link-D} An index 1-2-1 ($A_3^+$) or 2-1-2 ($A_3^-$)
  degenerate critical point, a birth-death-birth singularity.  See
  Figure~\ref{fig:link-10a} for an example in the index 2-1-2 case,
  which we will discuss. In this case, $r = 2$, and on the stabilized
  side $C_2$, we have three critical points, $p_1$, $p_2$, and $p_3$,
  with $p_1$ and $p_3$ of index~2 and $p_2$ of index~1. In the
  birth-death strata $S_1$ and $S_2$, the critical points cancel each
  other in two different ways: $p_2$ cancels against either $p_1$ or
  $p_3$.  For both cancellations to be possible, there is necessarily
  a unique flow from $p_2$ to both $p_1$ and $p_3$, and no other flows
  from~$p_2$ to index~2 critical points.
  To see there are no other flows from~$p_2$, recall that
  the local form of an~$A_3^-$ singularity~$p$ is $-x_1^2 + x_2^2 - x_3^4$,
  hence it has a 1-dimensional unstable manifold, which is generically
  disjoint from the stable manifolds of all index~2 critical points.
  So, after a sufficiently small deformation of~$f_0$, these stable manifolds
  will still avoid a neighborhood of~$p$.
  The parameters are the numbers $k$ and $l$ of flows
  from index~1 critical points to $p_1$ and $p_3$, respectively, not
  counting the flows from~$p_2$. The two stabilizations corresponding
  to passing $S_1$ and $S_2$ are of types $(1,k)$ and $(1,l)$,
  respectively.  Note that, on the common destabilized diagram $H_1$,
  there is a single $\beta$-circle meeting $k+l$ of the $\alpha$-strands.

  The link $P$ is an arbitrary bigon around 0 inside $D^2_\eps$.  The
  Heegaard surface $\S_1 = \S$. The part shown in
  Figure~\ref{fig:link-10a} is a neighborhood of $W^s(p) \cap \S$,
  where $p$ is the degenerate critical point of $f_0$.  The surface
  $\S$ divides $M$ into two pieces $M_-$ and $M_+$ such that $p \in
  M_+$.  To obtain $\S_2$, we add a tube around $W^s(p_2)$ to $\S$ so
  thin that it separates $p_2$ from $p_1$ and $p_3$.  Recall that
  $W^s(p)$ is a 2-disk, while $W^{ss}(p)$ is a curve inside it.  The
  numbers $k$ and $l$ are in fact the number of flow-lines from
  index~1 critical points to $p$ on the two sides of $W^{ss}(p)$
  in~$W^s(p)$.
  \begin{figure}
    \centering
    \includegraphics{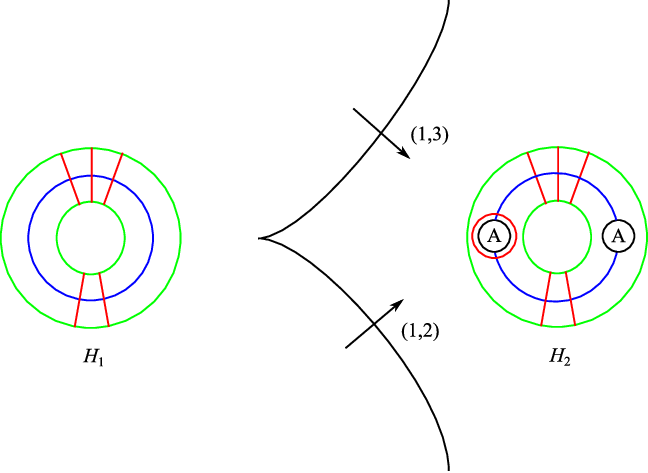}
    \caption{The link of a birth-death-birth singularity,
      type~\ref{item:D}.  In this example, $k=3$ and $l=2$.}
    \label{fig:link-10a}
  \end{figure}
\end{enumerate}

\begin{enumerate}[label = \ref{item:E1}]
\item \label{item:link-E1} A flow from an index~2 critical
  point $p_1$ to an index~1 critical point~$p_2$. Suppose there are
  $k$ flows from $p_2$ to index~2 critical points and $l$ flows from
  index~1 critical points to $p_1$. Then $r = k+l$, and $k$ of the
  strata $S_i$ are $\b$-handleslides while $l$ of them are
  $\a$-handleslides. Indeed, as we move the parameter value $\mu$ in a
  circle around 0, for each flow from another index~1 critical
  point~$q$ to $p_1$, we pass a stratum $S_i$ where we see an orbit of
  tangency in $W^u(q) \cap W^s(p_2)$, which translates to an
  $\a$-handleslide. Similarly, for each flow from $p_2$ to an index~2
  critical point~$r$, for some value $\mu \in S_i$, we see an orbit of
  tangency in $W^u(p_1) \cap W^s(r)$, which translates to a
  $\b$-handleslide. In each case, another curve slides over the
  $\a$-circle $W^u(p_2) \cap \S$ or the $\b$-circle $W^s(p_1) \cap \S$.

  We now explain how to choose $\eps$, the link $P$ -- which is a
  $2r$-gon -- and the corresponding diagrams $H_1$, $H_1', \dots, H_r$,
  $H_r'$. Since $(f_0,v_0)$ is not separable, we describe the
  construction in detail. For an illustration, see
  Figure~\ref{fig:E1-surface}.  Take~$N_0$ to be a thin regular
  neighborhood of
  \[
  \bigcup \left\{\, W^s(p) \,\colon\, p \in C_0(f_0) \cup C_1(f_0)
    \setminus \{p_2^0\} \,\right\} \cup R_-(\g),
  \]
  and let $\S = \ol{\partial N_0 \setminus \partial M}$.  Generically,
  $\S$ is transverse to $v_0$, and if we choose~$\eps$ small enough,
  $\S$ is transverse to $v_\mu$ for every $\mu \in D^2_\eps$.
  Consider the circle $\b = W^s(p_1) \cap \S$, and let $B$ be an
  annular neighborhood of $\b$ in $\S$ so small that it is disjoint
  from the stable flow line into~$p_2$ not starting at~$p_1$. Pick values $\nu_i \in S_i$,
  and let $A_i$ be a thin tube $\partial N(W^s(p_2^{\nu_i})) \setminus
  N_0$ disjoint from $W^u(p_1^{\nu_i}) \cup W^s(p_1^{\nu_i})$, and let
  $D_i \cup D_i'$ be the pair of disks $N(W^s(p_1^{\nu_i})) \cap \S$
  (the feet of the tube $A_i$).  If $\nu_i$ lies sufficiently close to
  $0$, then we can assume that $D_i \subset B$.
  Define $\S_i$ to be $(\S \setminus (D_i \cup D_i')) \cup A_i$; this
  is a separating surface for the Morse-Smale pair $(f_i,v_i)$.  For
  every $i \in \{\, 1,\dots,r \,\}$, pick an arc $a_i$ transverse to
  $S_i$ at $\nu_i$ so short that $\S_i \in \S(f_\mu,v_\mu)$ and $D_i
  \cap W^s(p_1^{\mu}) = \emptyset$ for every $\mu \in a_i$.
  According to our conventions, $\partial a_i \cap C_i = \mu_i'$ and
  $\partial a_i \cap C_{i+1} = \mu_{i+1}$.  Note that $A_i$ is an
  annular neighborhood of the circle $\a^i = W^u(p_2^{\nu_i}) \cap
  \S_i$. We also pick a small disk $D \subset \S$ around the point
  $W^s(p_2) \cap (\S \setminus B)$. Again, taking~$A_i$ sufficiently
  thin, all the~$D_i'$ will lie in~$D$.  The side $a_i'$ of the link
  $P$ is an arbitrary curve in $C_i$ connecting $\mu_i$ and
  $\mu_i'$. To obtain the surface enhanced link, we take~$\S_i' =
  \S_{i+1}$.

  The diagrams $H_i = H(f_i,v_i,\S_i)$ and $H_i' = H(f_i',v_i',\S_i')$
  all agree outside the subsurfaces $T_i = (D \setminus D_i') \cup A_i
  \cup (B \setminus D_i)$, up to a small isotopy of $\alphas \cup
  \betas$. What happens inside the twice punctured disks $T_i$ is
  depicted in Figure~\ref{fig:link-11}. There, $\partial T_i$ are the
  green circles and the two components of $\partial A_i$ are labeled
  by~$A$. (We have omitted the tubes $A_i$ for clarity.) The only
  difference between $\S_i$ and $\S_{i+1}$ is that the foot of the
  tube $A_i$ lying in $B$ is moved around by an isotopy. The diagrams
  $H_i'$ and $H_{i+1}$ are related by a handleslide, while~$a_i'$
  induces a diffeomorphism $\varphi_i \colon H_i \to H_i'$ isotopic to
  the identity in $M$ (and well-defined up to isotopy). The
  composition $\varphi_n \circ \dots \circ \varphi_1 \colon \S_1 \to
  \S_1$ is a diffeomorphism that is the product of Dehn twists about
  the components of $\partial T_1$; see
  Definition~\ref{def:handleswap}.
  \begin{figure}
    \centering
    \includegraphics{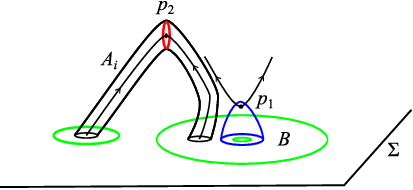}
    \caption{Construction of the Heegaard surface in
      case~\ref{item:E1}.}
    \label{fig:E1-surface}
  \end{figure}

  \begin{figure}
    \centering
    \includegraphics{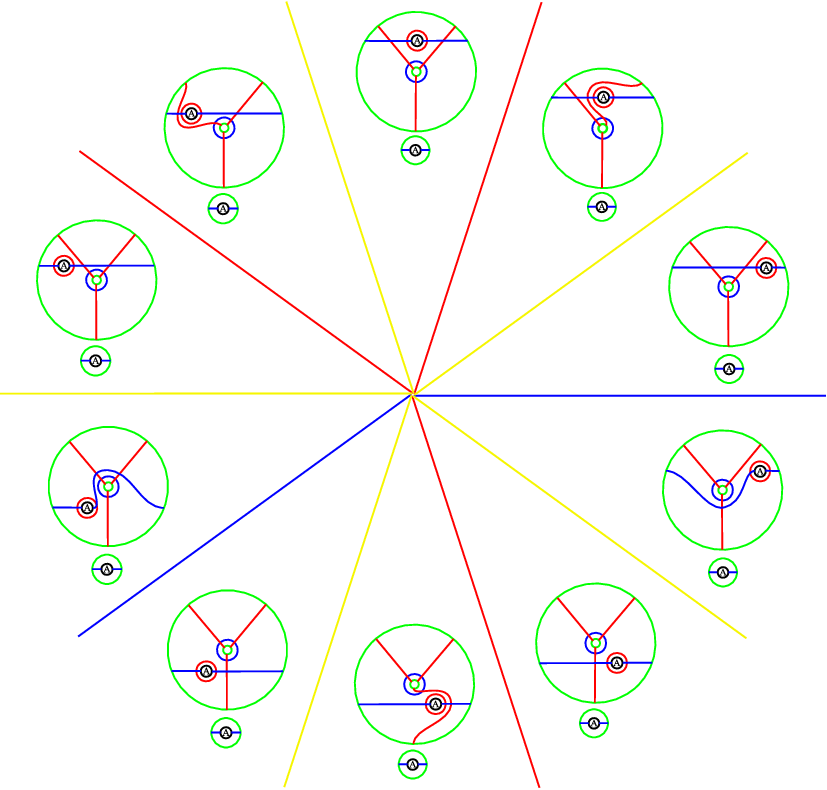}
    \caption{The bifurcation diagram for singularity~\ref{item:E1}, a
      flow from an index~1 critical point to an index~2 critical
      point.  In this example, $(k,l) = (2,3)$.  Outside the green
      circles, all five diagrams are small isotopic translates of each
      other. The Heegaard surface inside the green circles is not
      constant, there is a tube that moves around joining the two
      black boundary circles labeled by A.}
    \label{fig:link-11}
  \end{figure}
\end{enumerate}

In cases \ref{item:E2}--\ref{item:E4}, the same splitting surface $\S$
can be chosen for every $(f_i,v_i)$, and the curves $\alphas_i$
(resp.\ $\betas_i$) are all isotopic to each other for an appropriate
choice of spanning trees. Since we are going to pass to isotopy
diagrams, this description suffices for our purposes.
This concludes the proof of Theorem~\ref{thm:2-param}.
\end{proof}
\section{Simplifying moves on Heegaard diagrams}
\label{sec:simplify}

In this section, we break down $(k,l)$-stabilizations, generalized
handleslides, and the loops of diagrams of type
\ref{item:A}--\ref{item:E} appearing in Theorem~\ref{thm:2-param} into
the simpler moves and loops that come up in the definition of strong
Heegaard invariants, Definition~\ref{def:strong-Heegaard}.  During the
simplification procedure, we work with overcomplete diagrams, and will
only later choose spanning trees to pass to actual sutured diagrams.

We now overview the relevant combinatorial structures we are going to use.
A generic 2-parameter family of gradient-like vector fields
$\mathcal{F} \colon D^2 \to \FV(M,\g)$ gives rise to a bordered stratification
of~$D^2$; see Definition~\ref{def:stratification}. We subdivide this stratification
into a certain CW decomposition of~$D^2$, and consider a dual CW decomposition.
We label the vertices of this dual decomposition with overcomplete diagrams
such that neighboring diagrams are related by (generalized) Heegaard moves.
The simplification procedure consists of modifying the CW decompositions in a combinatorial
way, as opposed to modifying the 2-parameter family of gradients. Our goal is to obtain
a CW decomposition of~$D^2$ where the boundary of every 2-cell corresponds to one of the
elementary loops in the definition of strong Heegaard invariants (Definition~\ref{def:strong-Heegaard}).
We now give the relevant definitions.

\begin{definition} \label{def:polyhedral}
  A \emph{polyhedral decomposition} of $D^2$ is a regular CW decomposition of~$D^2$
  (i.e., the attaching map of every cell is an embedding) such that every closed
  1-cell is smoothly embedded in~$D^2$.

  A \emph{bordered polyhedral decomposition} of $D^2$ is a partition
  of $D^2$ that arises as follows:
  Choose a polyhedral decomposition of~$D^2$ such that every 0-cell in the boundary~$S^1$
  has valence~3, every closed 1-cell not contained in~$S^1$
  intersects~$S^1$ in at most one 0-cell, and every
  2-cell intersects~$S^1$ in at most one 1-cell.
  Then take the union of each open $i$-cell in $S^1$ with the
  neighboring open $(i+1)$-cell in $\text{Int}(D^2)$ for $i \in
  \{0,1\}$ (where an open 0-cell is just a 0-cell).
  We call these \emph{bordered $(i+1)$-cells}.

  A polyhedral decomposition $\mathcal{P}$ and a bordered polyhedral
  decomposition $\mathcal{R}$ are \emph{dual} if
  $\text{sk}_0(\mathcal{P}) \cap \text{sk}_1(\mathcal{R}) = \emptyset$
  and $\text{sk}_0(\mathcal{R}) \cap \text{sk}_1(\mathcal{P}) = \emptyset$,
  in each 2-cell of $\mathcal{P}$ there is a unique vertex of $\mathcal{R}$,
  and in each (bordered) 2-cell of $\mathcal{R}$ there is a unique vertex of $\mathcal{P}$.
  Furthermore, for each 1-cell~$e$ of~$\mathcal{P}$, we have $|e \cap \text{sk}_1(\mathcal{R})| = 1$,
  and for each (bordered) 1-cell~$e^*$ of~$\mathcal{R}$, we have $|e^* \cap \text{sk}_1(\mathcal{P})| = 1$.
  For an example, see Figure~\ref{fig:polyhedral-dual}.
\end{definition}

\begin{figure}
    \centering
    \input{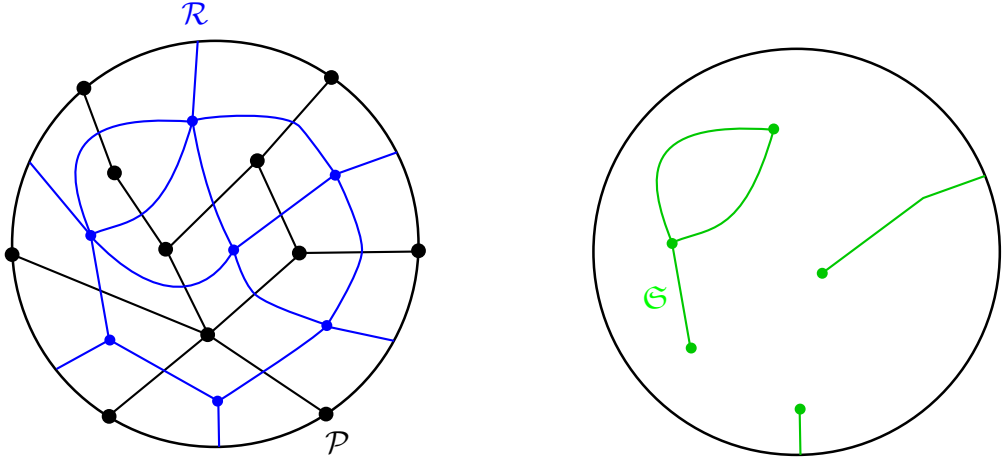}
    \caption{On the left, a polyhedral decomposition~$\mathcal{P}$ of $D^2$ is
      shown in black. It is dual to the blue bordered polyhedral decomposition~$\mathcal{R}$
      that refines the green bordered stratification~$\mathfrak{S}$
      of~$D^2$ on the right.}
    \label{fig:polyhedral-dual}
\end{figure}

\begin{definition} \label{def:stratification}
  We say that the partition $\mathfrak{S} = V_0 \sqcup V_1 \sqcup V_2$
  is a \emph{bordered stratification} of the disk $D^2$ if the
  following hold:
  \begin{enumerate}
  \item $V_0$ is a finite set of points in the interior of $D^2$,
  \item $V_1$ is a properly embedded 1-dimensional submanifold-with-boundary of $D^2
    \setminus V_0$, and
  \item each point $x \in V_0$ has a neighborhood $N_x$ such
    that the pair $(N_x, V \cap N_x)$ is homeomorphic to a cone
    $(D^2,I \cdot H)$ for some finite set $H \subset S^1$, where $V = V_0 \cup V_1$,
    and $I \cdot H$ is the union of the line segments connecting the origin with each point of~$H$.
  \end{enumerate}

  Note that we can view every bordered polyhedral decomposition
  as a bordered stratification if we set $V_0$ to be the union of the
  0-cells and $V_1$ to be the union of the (bordered) 1-cells.
  A bordered polyhedral decomposition
  $\mathcal{R}$ of $D^2$ is a \emph{refinement of} $\mathfrak{S}$ if
  $\text{sk}_i(\mathcal{R}) \supset V_i$ for $i \in \{0,1\}$
  (in particular, every open 1-cell of
  $\mathcal{R}$ is either contained in $V_1$ or is disjoint from it).
  We say that a polyhedral decomposition of $D^2$ is \emph{dual to
  $\mathfrak{S}$} if it is dual to some bordered polyhedral
  decomposition $\mathcal{R}$ refining $\mathfrak{S}$.
  For an example, see Figure~\ref{fig:polyhedral-dual}.
\end{definition}

A generic 2-parameter family of gradient-like vector fields
$\mathcal{F} \colon D^2 \to \FV(M,\g)$
gives rise to a bordered stratification $\mathfrak{S}(\mathcal{F})$ of
$D^2$ by taking
\[
V_i = \{\, \mu \in D^2 \,\colon\, \mathcal{F}(\mu) \in
\FV_{2-i}(M,\g)\,\}
\]
for $i \in \{\,0,1,2\,\}$.

\begin{definition}\label{def:adapted}
  A polyhedral decomposition $\mathcal{P}$ of $D^2$ is \emph{adapted
    to the family $\mathcal{F}$} if
  \begin{enumerate}
  \item $\mathcal{P}$ is dual to $\mathfrak{S}(\mathcal{F})$,
  \item \label{item:short} each edge intersecting $V_1$ is so short
    that Proposition~\ref{prop:1-param} applies to it,
  \item if $\om \in V_0$ and $\sigma$ is the 2-cell of $\mathcal{P}$
    containing $\om$, then $\partial \sigma$ is a link of $\om$ as in
    Theorem~\ref{thm:2-param},
  \item every 2-cell $\sigma$ of $\mathcal{P}$ that intersects $V_1$
    but is disjoint from $V_0$ is a quadrilateral, and $\sigma \cap
    V_1$ is an arc connecting opposite sides of $\sigma$,
  \item \label{item:disjoint-edges} any two closed 2-cells of
    $\mathcal{P}$ containing two different points of $V_0$ are
    disjoint, and any two closed 1-cells of $\mathcal{P}$ that
    intersect $V_1 \setminus S^1$ are either disjoint, or they both
    belong to a 2-cell containing a point of $V_0.$
  \end{enumerate}
\end{definition}

\begin{lemma} \label{lem:adapted} Let $\mathcal{F} \colon D^2 \to
  \FV(M,\g)$ be a generic 2-parameter family. Then there exists a
  polyhedral decomposition $\mathcal{P}$ of $D^2$ adapted to
  $\mathcal{F}$. Furthermore, given a triangulation of $S^1$ such that
  each 1-cell contains at most one bifurcation point of $\mathcal{F}$
  and satisfies Condition~(\ref{item:short}) of
  Definition~\ref{def:adapted}, then we can choose $\mathcal{P}$
  such that it extends this triangulation.
\end{lemma}

\begin{proof}
  First, choose the 2-cells of $\mathcal{P}$ containing the points of
  $V_0$ using Theorem~\ref{thm:2-param}, all taken to be sufficiently
  small, and denote by $N(V_0)$ their union.  Pick short arcs
  transverse to $V_1 \setminus N(V_0)$ such that
  Proposition~\ref{prop:1-param} applies to each, and such that there
  is an arc through each boundary point of $V_1$ lying inside $S^1$.
  These will all be 1-cells of $\mathcal{P}$.  Next, as in
  Proposition~\ref{prop:rectangle}, connect the endpoints of
  neighboring 1-cells intersecting $V_1$ so that we obtain a
  collection of rectangles that, together with~$N(V_0)$, completely
  cover $V_1$.  The rectangles are 2-cells of $\mathcal{P}$.  Finally,
  we subdivide the remaining regions until the attaching map of each
  2-cell becomes an embedding.  This is possible if we choose
  sufficiently many 1-cells intersecting $V_1$.  It is apparent from
  the construction that $\mathcal{P}$ is dual to a bordered polyhedral
  decomposition of $D^2$ refining the bordered stratification
  $\mathfrak{S}(\mathcal{F})$. If we are already given $\mathcal{P}
  \cap S^1$, then the extension to $D^2$ proceeds in an analogous
  manner.
\end{proof}

If $\mathcal{P}$ is adapted to the generic 2-parameter family
$\mathcal{F} \colon D^2 \to \FV(M,\g)$, then we label each edge of
$\mathcal{P}$ with the type of move that occurs as we move along it,
which is either a diffeomorphism isotopic to the identity if the edge
does not cross $V_1$, or a generalized handleslide, a
$(k,l)$-stabilization, or birth/death of a redundant $\a/\b$
curve if the edge does cross $V_1$.

\begin{definition} \label{def:coherent-surfaces} Let $\mathcal{F}
  \colon D^2 \to \FV(M,\g)$ be a generic 2-parameter family and
  $\mathcal{P}$ an adapted polyhedral decomposition. A choice of
  Heegaard surfaces
  \[
  \{\, \S_\mu \in \S(\mathcal{F}(\mu)) \,\colon\, \mu \in
  \text{sk}_0(\mathcal{P}) \,\}
  \]
  is \emph{coherent} with $\mathcal{P}$ if, for every edge $e$ of
  $\mathcal{P}$ with $\partial e = \mu - \mu'$, the isotopy diagrams
  $[H(\mathcal{F}(\mu),\S_\mu)]$ and $[H(\mathcal{F}(\mu'),\S_{\mu'})]$
  are related as indicated by the label of $e$. A \emph{surface
    enhanced polyhedral decomposition of $D^2$ adapted to
    $\mathcal{F}$} is a polyhedral decomposition of $D^2$ adapted to
  $\mathcal{F}$, together with a coherent choice of Heegaard surfaces.
\end{definition}

\begin{lemma} \label{lem:coherent} Let $\mathcal{F} \colon D^2 \to
  \FV(M,\g)$ be a generic 2-parameter family, and suppose that
  $\mathcal{P}$ is a polyhedral decomposition of $D^2$ adapted to
  $\mathcal{F}$.  If we are given Heegaard surfaces $\S_\mu \in
  \S(\mathcal{F}(\mu))$ for $\mu \in \text{sk}_0(\mathcal{P}) \cap
  S^1$ such that, for every edge $e$ of
  $\mathcal{P}$ in $S^1$ with $\partial e = \mu - \mu'$, the isotopy diagrams
  $[H(\mathcal{F}(\mu),\S_\mu)]$ and $[H(\mathcal{F}(\mu'),\S_{\mu'})]$
  are related as indicated by the label of $e$, then this can be
  extended to a choice of Heegaard surfaces coherent with
  $\mathcal{P}$.
\end{lemma}

\begin{proof}
  For the vertices of each 2-cell containing a point of $V_0$, we
  choose the surfaces~$\S_\mu$ using Theorem~\ref{thm:2-param}.  Then,
  for the remaining vertices of edges $e$ that intersect $V_1
  \setminus S^1$, we pick the $\S_\mu$ using
  Proposition~\ref{prop:1-param}.  This is possible because these
  edges have no vertices in common by (\ref{item:disjoint-edges}).
  For the rest of the vertices in $\text{sk}_0(\mathcal{P}) \setminus
  S^1$, we choose $\S_\mu$ arbitrarily.
\end{proof}

From now on, let
\[
\mathcal{F} \colon D^2 \to \FV(M,\g)
\]
be a generic 2-parameter family, $\mathfrak{S} = \mathfrak{S}(\mathcal{F})$ the
induced bordered stratification of $D^2$, and $\mathcal{P}$ an adapted surface
enhanced polyhedral decomposition of $D^2$ with dual bordered
polyhedral decomposition $\mathcal{R}$ refining~$\mathfrak{S}$.  In
the rest of Section~\ref{sec:simplify}, we give a method for resolving $\mathcal{R}$, giving
rise to a new bordered polyhedral decomposition $\mathcal{R}'$ of
$D^2$. This consists of first replacing the strata in $V_1 \setminus
N(V_0)$ corresponding to $(k,l)$-stabilizations by a collection of
parallel strata labeled by simple stabilizations and handleslides.
Then, at each point $\om$ of $V_0$, we connect these strata in a
particular manner depending on the type of $\om$.  We do not claim the
existence of a family $\mathcal{F}'$ giving rise to the new
decomposition $\mathcal{R}'$, though constructing such is probably
straightforward but tedious. (This would be the 2-parameter analogue of
Proposition~\ref{prop:lift-moves}.)

The role of the resolved stratification $\mathcal{R}'$ is that we can
refine the polyhedral decomposition $\mathcal{P}$ adapted to
$\mathcal{F}$ to obtain a decomposition $\mathcal{P}'$ dual to
$\mathcal{R}'$, and we can choose (overcomplete) isotopy diagrams for
the new vertices $\text{sk}_0(\mathcal{P}') \setminus
\text{sk}_0(\mathcal{P})$ in a natural manner such that neighboring
diagrams are now related by simple stabilizations, simple
handleslides, or diffeomorphisms. Furthermore, along the boundary of
each 2-cell of $\mathcal{P}'$, after an appropriate choice of spanning
trees, each strong Heegaard invariant will commute by definition.

As in the previous sections, we suppress the strata corresponding to
index 0-1 and 2-3 saddle-nodes, since these disappear for any choice
of spanning trees. For simplicity, we will often only draw the
bordered polyhedral decomposition $\mathcal{R}'$, possibly the dual
decomposition $\mathcal{P}'$, and the Heegaard diagrams for a few
vertices $\mu$ of $\mathcal{P}'$ if the other intermediate diagrams
are easy to recover.  Consistently with our previous color
conventions, edges of $\mathcal{R}$ and $\mathcal{R}'$ are red if the
diagrams on the two sides are related by an $\a$-equivalence, blue for
$\b$-equivalences, black for (de)stabilizations, and
yellow for diffeomorphisms.

\subsection{Codimension-1}
\label{sec:simplify-codim-1}

Suppose that the possibly overcomplete isotopy diagram $H' =
(\S',[\alphas'],[\betas'])$ is obtained from $H =
(\S,[\alphas],[\betas])$ by a $(k,l)$-stabilization, as in Definition~\ref{def:gen-stab}.
In particular, we
remove the disk $D \subset \S$ and replace it with the punctured torus
$T$ to obtain $\S'$. Inside $T$, we have two new attaching curves;
namely, $\a \in \alphas'$ and $\b \in \betas'$.

Such a $(k,l)$-stabilization can be replaced by a simple
stabilization, $k$ consecutive $\b$-handleslides, and $l$ consecutive
$\a$-handleslides.  For convenience, we describe this procedure in the
direction of the destabilization going from $H'$ to $H$.
Specifically, pick an orientation on both $\a$ and $\b$.  Let $\a_1,
\dots, \a_l$ be the $\a$-curves that intersect $\b$, labeled in order
given by the orientation of $\b$ and possibly listing the same
$\a$-curve several times. This gives rise to a sequence of diagrams
\[
H' = H_0 \text{, } H_1,\dots, H_l,
\]
where $H_i$ is obtained from $H_{i-1}$ by sliding $\a_i$ over $\a$ in
the direction opposite to the orientation of $\b$.  Similarly, let
$\b_1, \dots,\b_k$ be the $\b$-curves that intersect $\a$, labeled in order
given by the orientation of $\a$. Sliding these over $\b$ one by one
in the direction opposite to the orientation of $\a$, we obtain the
diagrams $H_{l+1},\dots,H_{l+k}$.  The result is a rectangular grid
between the arcs coming from $\a_1,\dots,\a_l$ and $\b_1, \dots,
\b_k$, plus the handle $T$ over which $\a$ and $\b$ run.  We can now
perform a simple destabilization on $(T,\a,\b)$ to obtain $H$.  See
Figure~\ref{fig:simpl-stab} for the resolution of a $(k,l)$-stabilization
stratum in a 2-parameter family.

\begin{figure}
  \centering
  \includegraphics{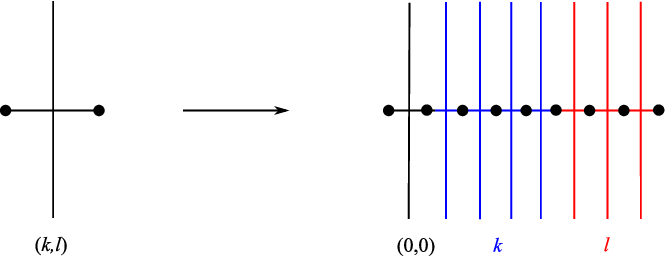}
  \caption{Resolving a stabilization. A stabilization of type
    $(k,l)$ can be replaced by a simple stabilization, followed by~$k$
    consecutive $\b$-handleslides and~$l$ consecutive
    $\a$-handleslides. When forming the resolution~$\mathcal{R}'$
    of the bordered polyhedral decomposition~$\mathcal{R}$, we replace the $(k,l)$-stabilization
    stratum of $\mathcal{R}$ on the left with a simple stabilization stratum,
    $k$ consecutive $\b$-handleslide strata (shown in blue),
    and $l$ consecutive $\a$-handleslide strata (shown in red) of $\mathcal{R}'$ on the right.}
  \label{fig:simpl-stab}
\end{figure}

Note that there were several choices involved in this construction,
namely the orientations on $\alpha$ and $\beta$, and also whether to
do the $\alpha$-handleslides or the $\beta$-handleslides first.  (In
the opposite direction, going from $H$ to $H'$ via a
$(k,l)$-stabilization, the choice of orientations corresponds to a
choice of which quadrant around the grid of intersections to stabilize
in.)  It will be helpful here to introduce the notion of a
stabilization slide.

\begin{definition} \label{def:stab-slide} A \emph{stabilization slide}
  is a subgraph of $\G$ of the form
  \[
  \xymatrix{H_1 \ar[r]^e \ar[rd]^f & H_2 \ar[d]^g \\ & H_3}
  \]
  such that
  \begin{enumerate}
  \item $H_i = (\S_i,[\alphas_i],[\betas_i])$ are (possibly
    overcomplete) isotopy diagrams for $i \in \{1,2,3\}$ such that
    $\S_2 = \S_3$,
  \item the edges $e$ and $f$ are stabilizations, while $g$ is an $\a$- or
    $\b$-equivalence,
  \item there are a disk $D \subset \S_1$ and a punctured torus $T
    \subset \S_2 = \S_3$ such that the restrictions $H_1|_D$, $H_2|_T$
    and $H_3|_T$ are conjugate to the pictures in
    Figure~\ref{fig:stab-slide} if the edge $g$ is an
    $\a$-equivalence, and to the same pictures with red and blue
    reversed if $g$ is a $\b$-equivalence, and
  \item we have $H_1|_{\S_1 \setminus D} = H_2|_{\S_2 \setminus T} =
    H_3|_{\S_3 \setminus T}$.
  \end{enumerate}
\end{definition}

\begin{figure}
  \centering
  \includegraphics{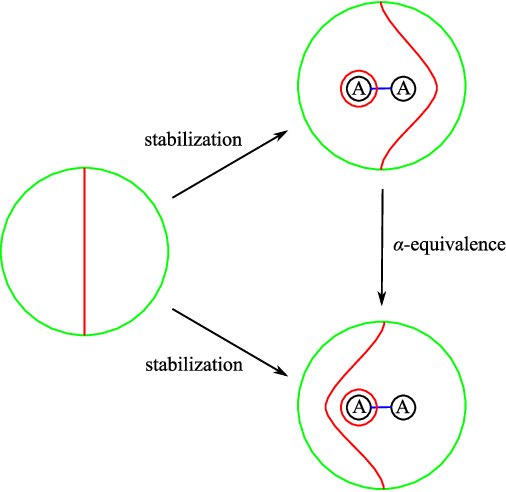}
  \caption{A stabilization slide. Such a loop of diagrams is a degenerate
    case of a distinguished rectangle of
    type~\eqref{item:rect-alpha-stab}.}
  \label{fig:stab-slide}
\end{figure}

Note that, if we apply a strong Heegaard invariant to a stabilization
slide, we obtain a commutative triangle.  Indeed, consider the
rectangle obtained from the stabilization slide triangle by taking two
copies of $H_1$ and connecting them by an edge labeled by the identity
of $H_1$. Then this is a distinguished rectangle of
type~\eqref{item:rect-alpha-stab}, with two opposite edges being
stabilizations and the other two being $\a$- or $\b$-equivalences.

\begin{figure}
  \centering
  \includegraphics{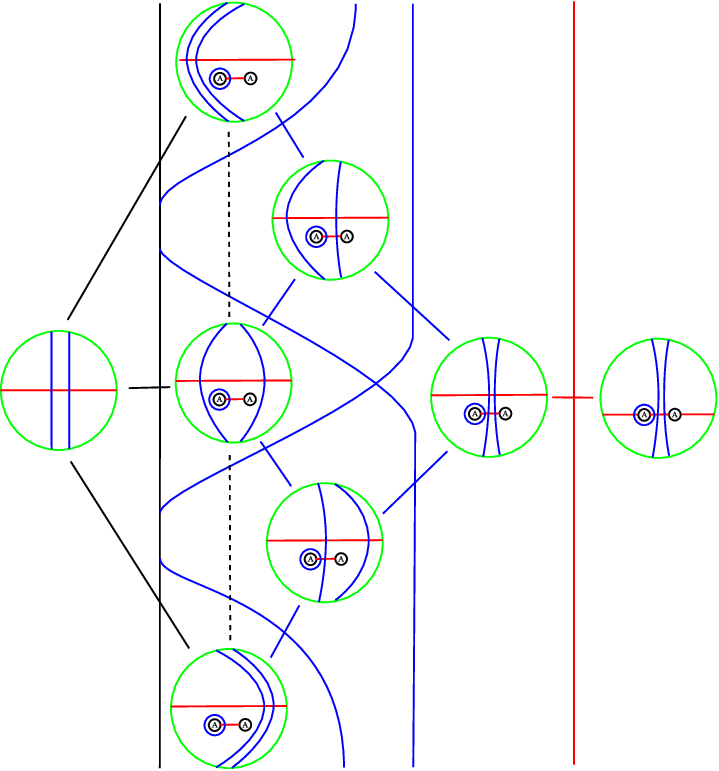}
  \caption{Switching the orientation involved in resolving a
    $(2,1)$-stabilization.  On the top and bottom are two different
    ways of resolving a $(2,1)$-stabilization, with different
    choices for the orientation of $\alpha$.  The two different
    choices can be related by loops of $\b$-equivalences and
    stabilization slides, as shown. Dashed edges are diagonals of
    rectangles whose other diagonal intersects the stabilization stratum.}
  \label{fig:stab-orient}
\end{figure}

It might happen that at the two ends of a $(k,l)$-stabilization
stratum, we need to use resolutions with different orientations
of the curves~$\a$ and~$\b$ (recall that the resolution process
depends on these orientations). Then we can interpolate between
the different resolutions as follows.

\begin{lemma} \label{lem:ori-change}
Suppose that the isotopy diagrams $H$ and $H'$ are related by a $(k,l)$-stabilization.
Using the notation introduced at the beginning of Section~\ref{sec:simplify-codim-1},
fix orientations of the curves~$\a$ and~$\b$, and let $H' = H_0, \dots, H_{k+l+1} = H$
be the resolution where~$H_i$ and~$H_{i+1}$ are related by an $\a$-handleslide for
$i \in \{0,\dots,l-1\}$, a $\b$-handleslide for $i \in \{l,\dots,l+k-1\}$, and by
a simple stabilization for $i = l+k$. If we reverse the orientation of~$\a$,
we obtain a different resolution
\[
H' = H_0, \dots, H_l, H_{l+1}', \dots, H_{k+l+1}' = H
\]
(the first $l+1$ diagrams are the same as we have kept the orientation of~$\b$).
Then there is a regular CW decomposition~$\mathcal{C}_k$ of~$D^2$ and a labeling of its vertices with diagrams such that
\begin{enumerate}
\item \label{it:boundary-decoration} the consecutive vertices along $S^1$ are labeled
\[
H, H_{k+l}, \dots, H_l, H_{l+1}', \dots, H_{k+l}',
\]
\item the boundary of every 2-cell is either labeled by a slide triangle,
or all of its vertices are labeled by $\b$-equivalent diagrams.
\end{enumerate}
\end{lemma}

Before proving Lemma~\ref{lem:ori-change},
we consider the example of switching the orientation of~$\a$ in a $(2,1)$-stabilization.
Figure~\ref{fig:stab-orient} illustrates how to modify the bordered polyhedral decomposition~$\mathcal{R}'$
and the dual polyhedral decomposition~$\mathcal{P}'$
along a $(2,1)$-stabilization stratum of~$\mathfrak{S}$ to obtain new
decompositions $\mathcal{R}''$ and $\mathcal{P}''_0$ that interpolate
between resolutions obtained using the two different orientations of~$\a$.
Essentially, we cut~$\mathcal{P}'$ and~$\mathcal{R}'$ along an arc transverse to the simple stabilization stratum and all
the $\b$-handleslide strata, and glue in the CW decomposition~$\mathcal{C}_2$ of the disk given by Lemma~\ref{lem:ori-change}.
We can also extend the stratification and its bordered polyhedral decomposition
to the disk we glued in.
Note that, in this figure, we have
introduced four codimension-2 bifurcations of type~\ref{item:B3}.
The modified decomposition~$\mathcal{R}''$ has some 1-cells corresponding
to $(1,0)$-stabilizations. We obtain $\mathcal{P}''_0$ by taking the
dual polyhedral decomposition $\mathcal{P}''$ of $\mathcal{R}''$, then
for each $(1,0)$-stabilization edge~$e$ of~$\mathcal{R}''$, we delete
the edge of~$\mathcal{P}''$ passing through~$e$, and replace it with
the other diagonal of the quadrilateral in~$\mathcal{P}''$ containing~$e$.
We indicated these new diagonals by dashed lines in the figure.
Each such diagonal divides the corresponding quadrilateral in~$\mathcal{P}''_0$
into a stabilization slide and a triangle all of whose
vertices are $\b$-equivalences. (A strong Heegaard invariant commutes along these
by the Functoriality Axiom.)

\begin{proof}[Proof of Lemma~\ref{lem:ori-change}]
Along~$S^1$, the CW decomposition~$\mathcal{C}_k$ and the labeling of the vertices
by diagrams are given by condition~\eqref{it:boundary-decoration}.
We denote each vertex on~$S^1$ with the corresponding diagram.
Assume that $H = (-1,0)$, $H_l = (1,0)$, while $H_i$ has negative $y$-coordinate
and $H_i'$ has positive $y$-coordinate for every $i \in \{l+1,\dots,l+k\}$.
We also suppose that the vertices are evenly spaced along~$S^1$.

Before extending the CW decomposition to the interior of~$D^2$,
it is helpful to define a bordered stratification~$\mathfrak{S}_k$ first:
Connect the midpoint of the 1-cell between~$H$ and~$H_{k+l}$
with the midpoint of the 1-cell between~$H$ and~$H_{k+l}'$ by a straight arc~$a$.
Let $v_1, \dots, v_k$ be distinct consecutive points in the interior of~$a$ such that~$v_1$
is closest to~$H_{k+l}$. The straight arc~$b_i$ connects $v_i$ with the midpoint of the
1-cell between $H_{k+l+1-i}$ and $H_{k+l-i}$ for $i \in \{1,\dots,k\}$.
Similarly, the straight arc~$b_i'$ connects $v_i$ with the midpoint of the
1-cell between $H_{l+i-1}'$ and $H_{l+i}'$. Then the set of vertices~$V$ of~$\mathfrak{S}_k$
are the points $b_i \cap b_j'$ for $1 \le j \le i \le k$, and the 1-cells are the components
of
\[
\left(a \cup (b_1 \cup b_1') \cup \dots \cup (b_k \cup b_k')\right) \setminus V;
\]
see Figure~\ref{fig:CW} for an illustration.
\begin{figure}
  \centering
  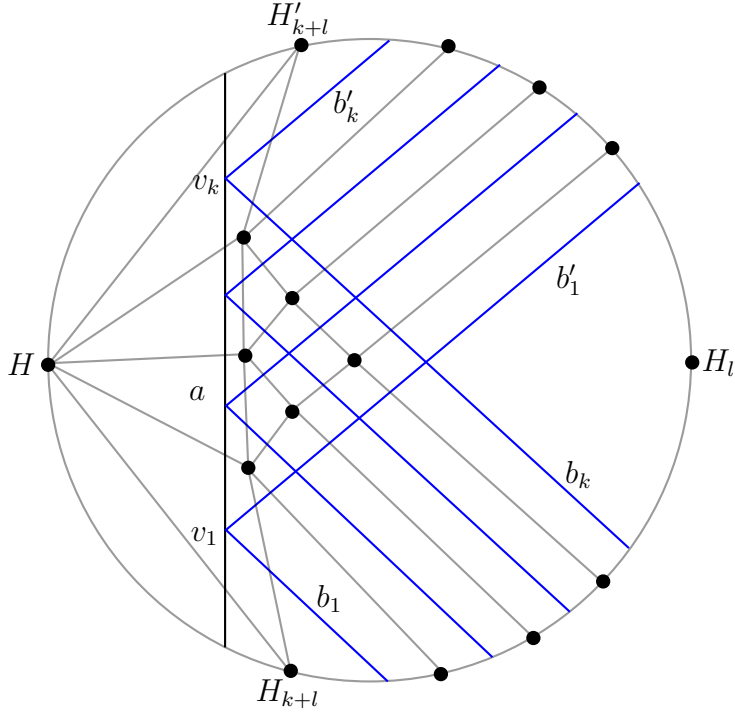
  \caption{The CW decomposition~$\mathcal{C}_k$ of $D^2$ from Lemma~\ref{lem:ori-change} is shown in grey,
  and the arcs $b_1,\dots,b_k$ and $b_1',\dots,b_k'$ in blue.}
  \label{fig:CW}
\end{figure}
The bordered stratification~$\mathfrak{S}_k$ can also be viewed as a bordered polyhedral decomposition.
The vertices of~$\mathcal{C}_k$ lying in the interior of~$D^2$ are the centers
of the 2-cells of~$\mathfrak{S}_k$ disjoint from~$S^1$.
The open 1-cells of $\mathcal{C}_k$ lying in the interior of~$D^2$
are obtained by connecting any two vertices that lie in 2-cells of~$\mathfrak{S}_k$
that share an edge by a straight line. Finally, for every 2-cell of~$\mathcal{C}_k$ that has~$H$ as a vertex
(these are all quadrilaterals), we add the diagonal disjoint from~$H$ to the 1-skeleton of~$\mathcal{C}_k$.

For $1 \le j \le i \le k$, we denote by~$v_{i,j}$ the vertex of~$\mathcal{C}_k$ that
lies between~$b_i$ and~$b_{i+1}$ and between~$b_j'$ and~$b_{j+1}'$.
Every interior vertex of~$\mathcal{C}_k$ is of this form.
We label~$v_{i,j}$ with the diagram~$H_{i,j}$ that is obtained from $H_l$ by
handlesliding $\b_k,\dots,\b_{i+1}$, in this order, over $\b$ according to the positive orientation of~$\a$,
and $\b_1, \dots, \b_j$ over~$\b$ according to the negative orientation of~$\a$.
The diagrams labeling the vertices of
an edge of~$\mathcal{C}_k$ that intersects~$(b_1 \cup b_1') \cup \dots \cup (b_k \cup b_k')$ in~$m \in \{1,2\}$ points
are related by~$m$ consecutive $\b$-handleslides.
Finally, the diagrams~$H$, $H_{i,i}$, and~$H_{i+1,i+1}$ form a slide triangle
for $i \in \{0,\dots,k-1\}$, where $H_{0,0} = H_{k+l}$
and $H_{k,k} = H_{k+l}'$.
\end{proof}

In order to prove Theorem~\ref{thm:iso},
it suffices to construct the modified resolution $\mathcal{P}''_0$ of $\mathcal{P}$ in a purely
combinatorial manner, without actually showing the existence of a
corresponding modification of the 2-parameter family of gradient-like
vector fields. Thus, by cutting the parameter space along an arc
transverse to the stabilization stratum and the $k$ $\b$-handleslide strata,
and gluing in the CW decomposition~$\mathcal{C}_k$ of~$D^2$
together with the labeling of its vertices given in Lemma~\ref{lem:ori-change}, we may assume that the
orientations are picked conveniently along both ends of each $(k,l)$-stabilization stratum.

\begin{figure}
  \centering
  \includegraphics{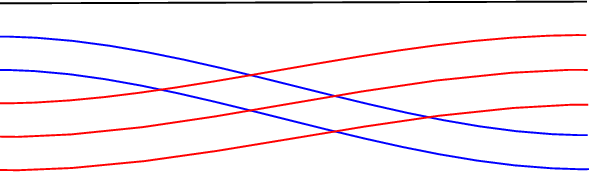}
  \caption{Interpolating in $\mathcal{R}'$ between the resolution of a
    $(2,3)$\hyp stabilization stratum starting with the
    $\b$-handleslides and the one starting with the
    $\a$-handleslides. For this, we introduce a grid of distinguished
    rectangles of type~\eqref{item:rect-alpha-beta}. As before, the handleslide
    stratum is black, the $\a$-handleslide strata are red, and the $\b$-handleslide
    strata are blue.}
  \label{fig:stab-ab-switch}
\end{figure}

Now suppose that at one end of a $(k,l)$-stabilization stratum,
we resolve by doing the $\b$-handleslides first, while at the other end,
we do the $\a$-handleslides first. We can interpolate between these
two choices by introducing a grid of distinguished rectangles of
type~\eqref{item:rect-alpha-beta} of Definition~\ref{def:distinguished-rect};
see Figure~\ref{fig:stab-ab-switch}.

\begin{figure}
  \centering
  \includegraphics{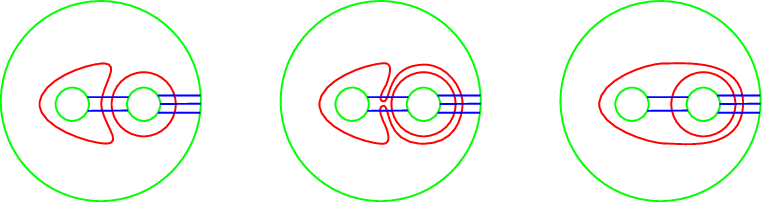}
  \caption{Writing a generalized handleslide of type $(2,3)$ as the
    composition of a simple handleslide and an isotopy of the resulting
    $\a$-curve.}
  \label{fig:simplify-gen-slide}
\end{figure}

If the diagram $\HD'$ is obtained from $\HD$ by a
generalized $\a$-handleslide of type $(m,n)$,
as in Definition~\ref{def:gen-handleslide},
then $\HD'$ can also be
obtained from $\HD$ by a simple (i.e., type $(0,m+n)$) handleslide,
followed by an isotopy of the resulting $\a$-curve.  For an
illustration, see Figure~\ref{fig:simplify-gen-slide}.  Since we are
passing to isotopy diagrams, we do not have to distinguish between
simple and generalized handleslides.

We now prove Proposition~\ref{prop:diag-connected-strong}, which
claims that for any balanced sutured manifold~$(M,\g)$,
in the graph~$\G_{(M,\g)}$, any two vertices can be
connected by an oriented path. Our argument refines the
proofs of \cite[Proposition~2.2]{OS06:HolDiskFour} and \cite[Proposition~2.15]{Juhasz06:Sutured},
in order to include diffeomorphisms isotopic to the identity among the moves
connecting two diagrams.

\begin{proof}[Proof of Proposition~\ref{prop:diag-connected-strong}]
Let~$H$ and~$H'$ be isotopy diagrams of~$(M,\g)$. Then pick
representatives~$\HD = (\S,\alphas,\betas)$ and~$\HD' = (\S',\alphas',\betas')$
such that~$\alphas \pitchfork \betas$ and~$\alphas' \pitchfork \betas'$.

By Proposition~\ref{prop:existence}, there are
simple Morse-Smale pairs~$(f,v)$, $(f',v') \in \FV_0(M,\g)$ such
that~$H(f,v) = \HD$ and~$H(f',v') = \HD'$. By Corollary~\ref{cor:FV-contractible},
there exists a generic 1-parameter family
\[
\{\, (f_t,v_t) \in \FV_{\le 1}(M,\g) \,\colon\, t \in I \,\}
\]
of sutured functions and gradient-like vector fields
such that $(f_0,v_0) = (f,v)$ and $(f_1,v_1) = (f',v')$.
Let~$0 < b_1 < \dots < b_n < 1$ be the set of parameter values such that~$(f_t,v_t) \in \FV_1(M,\g)$
if and only if~$t \in \{b_1,\dots,b_n\}$.

Using Proposition~\ref{prop:1-param} and the fact that for a given splitting surface
any two attaching sets are $\a$/$\b$-equivalent, for every~$i \in \{1,\dots,n\}$,
we can choose points~$b_i^- < b_i < b_i^+$ close to~$b_i$, separating
surfaces~$\S_i^\pm \in \S\left(f_{b_i^\pm},v_{b_i^\pm}\right)$, and spanning trees~$T_i^\pm$
such that the diagrams~$\HD_i^-$ and $\HD_i^+$ for
\[
\HD_i^\pm = H\left(f_{b_i^\pm},v_{b_i^\pm},\S_i^\pm,T_i^\pm\right)
\]
are related by an~$\a$- or $\b$-equivalence, or a~$(k,l)$-(de)stabilization.
As explained at the beginning of Section~\ref{sec:simplify-codim-1}, every $(k,l)$-stabilization
can be written as a simple stabilization, followed by an $\a$-equivalence and a $\b$-equivalence.

Finally, by Lemma~\ref{lem:isotopy},
$\HD$ and~$\HD_1^-$, $\HD_i^+$ and~$\HD_{i+1}^-$ for~$i \in \{\,1,\dots,n-1\,\}$,
and~$\HD_n^+$ and $\HD'$ are related by a diffeomorphism isotopic to the identity in~$M$,
followed by an~$\a$-equivalence and a $\b$-equivalence.
\end{proof}

\subsection{Codimension-2}
\label{sec:simplify-codim-2}

We consider the various types of singularities from
Theorem~\ref{thm:2-param} in Section~\ref{sec:transl-codim-2}, in an
order that is more convenient for this section.
For each type of singularity, we will construct a resolved bordered
decomposition $\mathcal{R}'$, as described at the beginning of
Section~\ref{sec:simplify}.

The links of singularities of type~\ref{item:A1a}--\ref{item:A1d}
and~\ref{item:link-A2} from Theorem~\ref{thm:2-param} (involving pairs
of handleslides) are easy, and we do not modify these during the
resolution process.  After choosing arbitrary spanning trees, we get a
loop in $\G_{(M,\g)}$ where each edge is an $\a$- or
$\b$-equivalence. Any strong Heegaard invariant $F$ applied to this
loop commutes. Indeed, such a loop can be subdivided into triangles
where each edge is of the same color, and some rectangles with two
opposite edges blue and two opposite edges red.  The commutativity of
$F$ along a triangle is guaranteed by the Functoriality Axiom of
Definition~\ref{def:strong-Heegaard}, whereas for the rectangles --
each of which is a distinguished rectangle of
type~\eqref{item:rect-alpha-beta} in the sense of
Definition~\ref{def:distinguished-rect} -- we can use the
Commutativity Axiom.

\begin{figure}
  \centering
  \includegraphics{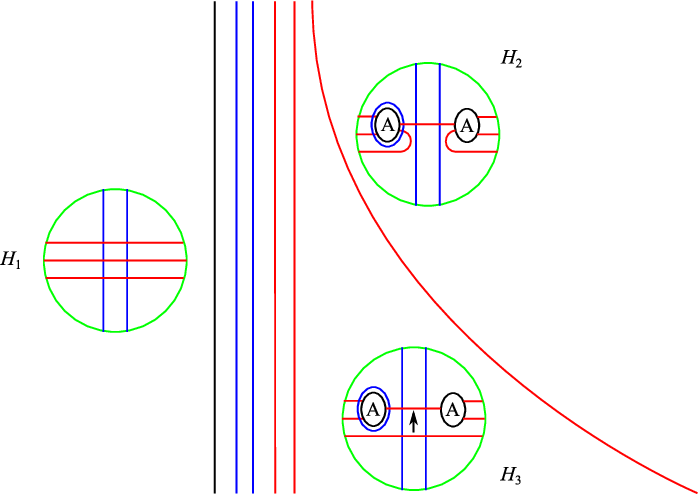}
  \caption{Resolving the singularity of type~\ref{item:link-B3} from
    Figure~\ref{fig:link-9a}.}
  \label{fig:simplify-9a}
\end{figure}

Next, we consider a singularity of type~\ref{item:link-B3}. This
essentially is just changing the type of destabilization, and can be
done with no singularities in the resolved bifurcation diagram
$\mathcal{R}'$, as long as the choice of orientation for resolving
the stabilization is consistent along the stabilization stratum.  An example is shown in
Figure~\ref{fig:simplify-9a}.

\begin{figure}
  \centering
  \includegraphics{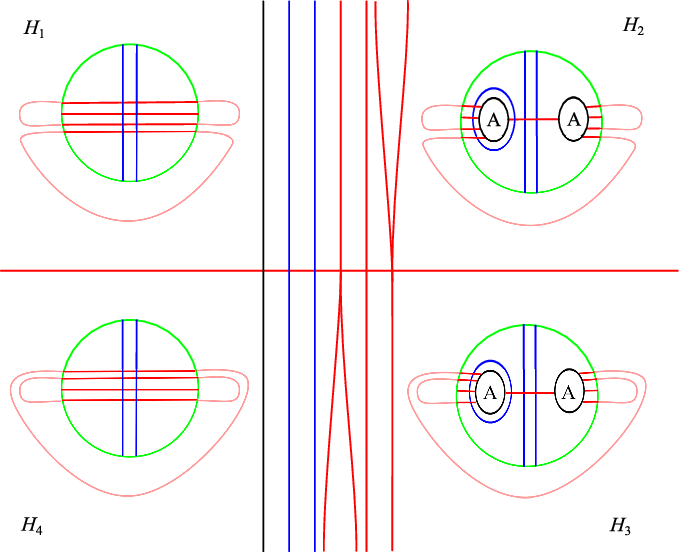}
  \caption{Resolving the singularity of type~\ref{item:link-B1} from
    Figure~\ref{fig:link-6a}.}
  \label{fig:simplify-6a}
\end{figure}

For singularities of type~\ref{item:link-B1} (a birth-death singularity
at~$p$ simultaneous with a handleslide of $p_1$ over $p_2$), recall
that a crucial feature was the number~$k = k_1 + k_2$ of flows from
$p_2$ to $p$. If $k=0$, the resolution can be done easily with several
handleslide commutations (links of type~\ref{item:A1a}) and one
stabilization-handleslide commutation.  For $k > 0$, we need to
introduce $k$ slide pentagons (links of type~\ref{item:link-A2}), as
$\alpha_1 = W^u(p_1) \cap \S$ is sliding over $\alpha_2 = W^u(p_2)
\cap \S$, which in turn slides over the circle $\alpha$ introduced at
the stabilization corresponding to $p$.  Note that $k_1$ of these
pentagons have exactly two vertices on the side
of the $\a$-handleslide stratum that existed before the resolution
containing $C_3 \cup C_4$
(below the horizontal red line in Figure~\ref{fig:simplify-6a}),
while $k_2$ of them have exactly two vertices on the other side.
See Figure~\ref{fig:simplify-6a} for an example with $k=2$.

Before proceeding, we introduce handleswaps; loops of overcomplete
diagrams that generalize the notion of simple handleswaps.  As we
shall see, these arise during the simplification procedure of links of
type~\ref{item:link-E1}, and are in fact quite close to them.

\begin{figure}
  \centering
  \includegraphics{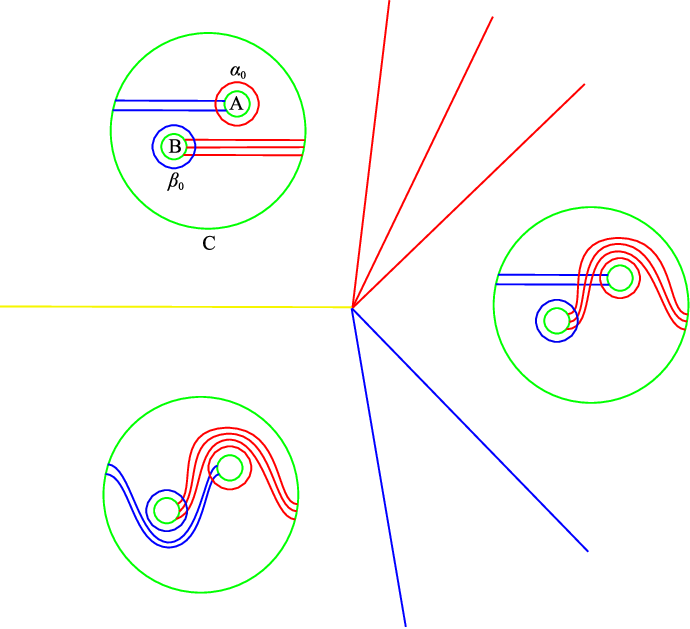}
  \caption{A $(2;3)$-handleswap.}
  \label{fig:handleswap-separated}
\end{figure}

\begin{definition} \label{def:handleswap} A \emph{$(k;l)$-handleswap}
  is a loop of overcomplete diagrams $\HD_0, \dots, \HD_{k+l}$ as follows.
  There is a surface $\S$ such that $\HD_i = (\S,\alphas_i,\betas_i)$
  for every $i \in \{\, 0,\dots,k+l \,\}$.  Furthermore, there is a
  pair of pants $P \subset \S$ such that
  \[
  \alphas_i \cap (\S \setminus P) = \alphas_j \cap (\S \setminus P)\text{ and }
  \betas_i \cap (\S \setminus P) = \betas_j \cap (\S \setminus P)
  \]
  for every
  $i$, $j$. Inside $P$, we have one full $\a$-curve that we denote by
  $\a_0$ and one full $\b$-curve that we call $\b_0$. The boundary
  $\partial P$ consists of three curves, $A$ being parallel to $\a_0$,
  a curve $B$ parallel to $\b_0$, and the third we denote by $C$. The
  set $(\alphas_0 \cap P) \setminus \a_0$ consists of $l$ parallel
  arcs connecting $B$ and $C$, while $(\betas_0 \cap P) \setminus
  \b_0$ consists of $k$ parallel arcs connecting $A$ and $C$. We also
  require that none of the $\a$-arcs intersect the $\b$-arcs in
  $P$. For $0 \le i < l$, the diagram $\HD_{i+1}$ is obtained from $\HD_i$
  by sliding one of the $\a$-arcs over $\a_0$, and for $l \le i <
  k+l$, the diagram $\HD_{i+1}$ is obtained from $\HD_i$ by sliding one of
  the $\b$-arcs over $\b_0$. The diagram $\HD_0$ is obtained from
  $\HD_{k+l}$ by a diffeomorphism that is the composition of a
  left-handed Dehn twist about $C$ and right-handed Dehn-twists about
  $A$ and $B$.  The case of a $(2;3)$-handleswap is depicted in
  Figure~\ref{fig:handleswap-separated}.

  Similarly, we say that
  a loop $H_0, \dots, H_{k+l}$ of isotopy diagrams is a $(k;l)$-handleswap
  if every $H_i$ has a representative $\HD_i$ such that
  $\HD_0, \dots, \HD_{k+l}$ is a $(k;l)$-handleswap.
\end{definition}

We will show in Section~\ref{sec:simpl-handleswap} that any
$(k;l)$-handleswap can be resolved into a number of simple
handleswaps, and thus that any strong Heegaard invariant applied to
the $(k;l)$-handleswap commutes.

Suppose we have a loop of diagrams as in
Definition~\ref{def:handleswap}, but where the $\a$- and
$\b$-handleslides are not necessarily separated from each other. Using
commutations, these can easily be rearranged in the standard form, so
we will also refer to these as handleswaps.
Figure~\ref{fig:commutations} shows how to write a handleswap loop
with mixed $\a$- and $\b$-handleslides as a product of a standard
handleswap and some distinguished rectangles corresponding to
commuting handleslides. The procedure is easier to understand on the
level of the bordered polyhedral decomposition $\mathcal{R}$, where
one spirals the blue strata and red strata in opposite directions to
separate them.

\begin{figure}
  \centering
  \includegraphics{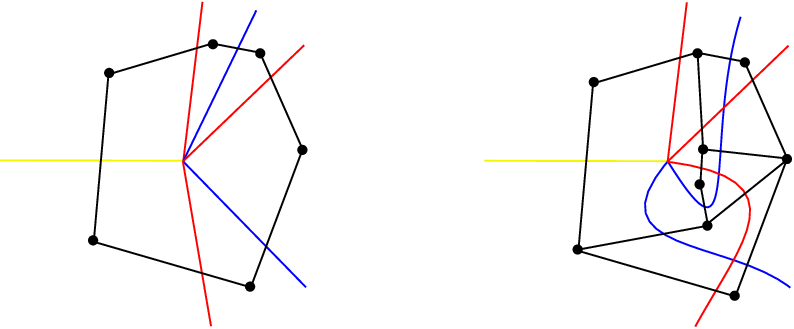}
  \caption{Rearranging the order of $\a$- and $\b$-handleslides in a
    handleswap using commutations.  On the left, we see $\mathcal{P}$
    and $\mathcal{R}$, on the right the modified decompositions
    $\mathcal{P}'$ and $\mathcal{R}'$.}
  \label{fig:commutations}
\end{figure}

As mentioned above, the link of a singularity of
type~\ref{item:link-E1} (a flow from an index~2 critical point to an
index~1 critical point) is quite close to a handleswap.  We can make
it exactly a handleswap by introducing some commutation moves between
diffeomorphisms and handleslides.  On the level of the bordered
polyhedral decomposition $\mathcal{R}$, in a small neighborhood of the
\ref{item:link-E1} singularity, we spiral the yellow diffeomorphism strata
corresponding to the diffeomorphisms $\varphi_1,\dots,\varphi_n$ to
all lie next to each other, then compose the diffeomorphisms.  For an
illustration, see Figure~\ref{fig:simplify-11}.
\begin{figure}
  \centering
  \includegraphics{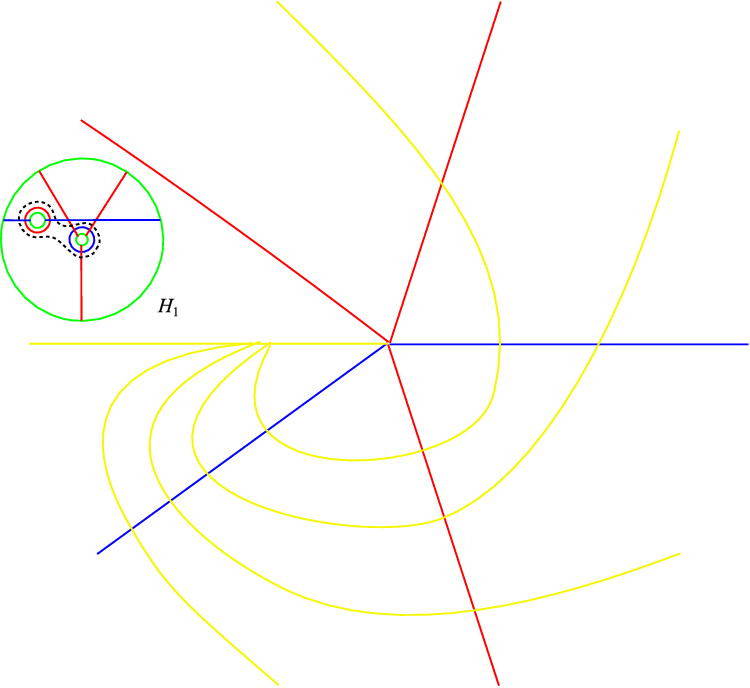}
  \caption{Simplifying a link of type~\ref{item:link-E1}.}
  \label{fig:simplify-11}
\end{figure}
Recall that the composition $d = \varphi_n \circ \dots \circ \varphi_1
\colon \S_1 \to \S_1$ is the product of Dehn twists about the boundary
components of the pair-of-pants $T_1$ (the three green circles in the
figure).  Finally, we rearrange the $\a$- and $\b$-handleswaps as
above to first have the $\a$-handleslides, followed by the
$\b$-handleslides. We denote this new surface enhanced polyhedral
decomposition by $\mathcal{P}'$, and the dual bordered polyhedral
decomposition by~$\mathcal{R}'$.

The link of the~\ref{item:link-E1} singularity in $\mathcal{P}'$
appears slightly different from a standard handleswap, since in the
diagram $H_1$ right above the diffeomorphism stratum, the $\a$- and
$\b$-arcs intersect each other.  However, this is not an issue as we
are dealing with isotopy diagrams. Indeed, consider the smaller pair
of pants $T_1'$ bounded by the dashed curve and the two small green
circles in Figure~\ref{fig:simplify-11}.  We choose $T_1'$ so small
that inside it all the $\a$- and $\b$-arcs are disjoint. If we now
perform all handleslides and the diffeomorphism within $T_1'$, we get
a standard handleswap loop, and each diagram is isotopic to the
corresponding diagram in $\mathcal{P}'$. If we even replace the
diffeomorphism $d$ by the diffeomorphism $d'$ that is a product of
Dehn twists about the boundary components of $T_1'$, then $d$ and $d'$
are isotopic, and by the Continuity Axiom of strong Heegaard
invariants,
\[
F(d) = F(d') \colon F(H_1) \to F(H_1).
\]

\begin{figure}
  \centering
  \includegraphics{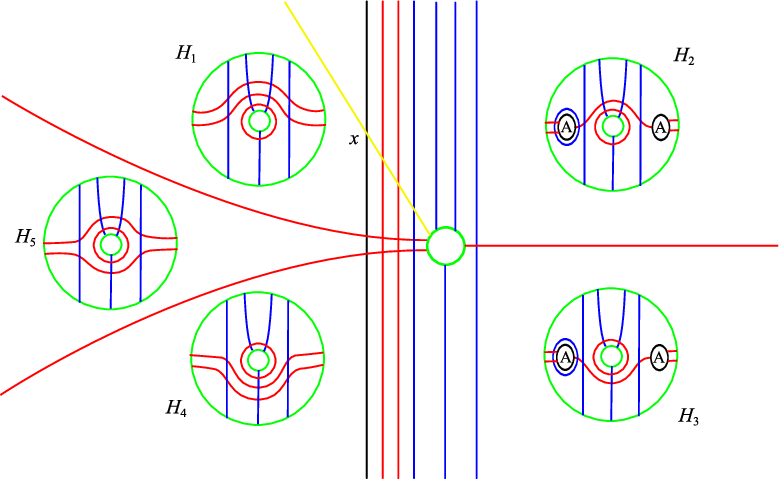}
  \caption{Resolving the singularity of type~\ref{item:link-B2} from
    Figure~\ref{fig:link-7a}, which has $k=2$ and $m=3$. Here, for
    clarity, the handleswap has not yet been put in the standard
    form.}
  \label{fig:simplify-7a}
\end{figure}

\begin{figure}
  \centering
  \includegraphics{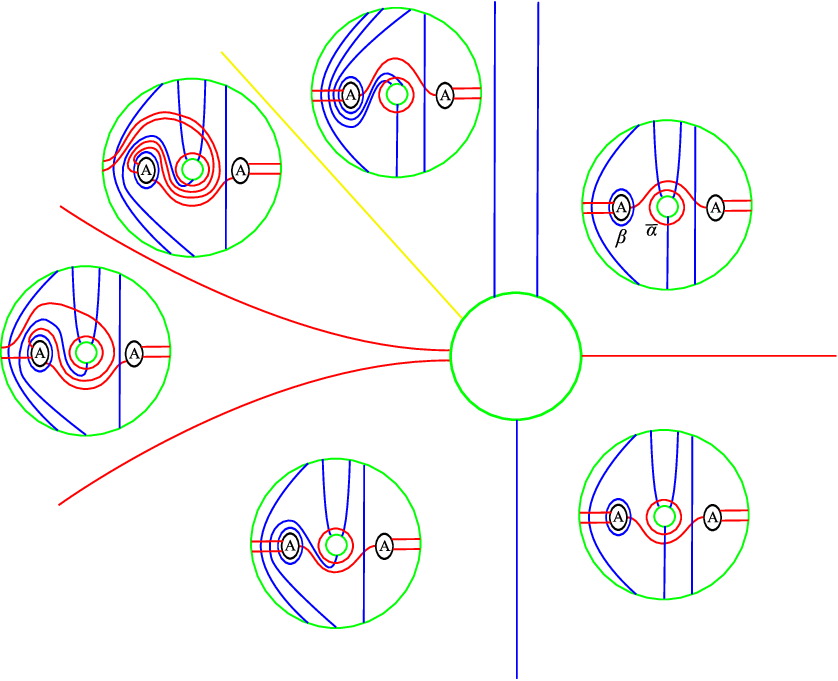}
  \caption{A closeup of the handleswap loop of diagrams in
    $\mathcal{P}'$ around the green circle from
    Figure~\ref{fig:simplify-7a}.}
  \label{fig:simplify-7a-closeup}
\end{figure}

Next, we consider singularities of type~\ref{item:link-B2}, a flow from a
birth-death singularity~$p$ to an index~1 critical point~$\ol{p}$.
As explained in the proof of Theorem~\ref{thm:2-param} (also see Figure~\ref{fig:link-7a}),
the crucial features are the number~$k$ of flows from index~1
critical points to $p$ and the number $m = m_1 + m_2$ of flows from
$\ol{p}$ to index~2 critical points. There are two codimension-1 stabilization
strata of types $(l+m_1,k)$ and a $(l+m_2,k)$, respectively. We resolve both of
these according to Section~\ref{sec:simplify-codim-1}. This results in a simple
stabilization stratum, followed by $k$ consecutive $\a$-handleslide strata
and $l+m_1$ consecutive $\b$-handleslide strata in one of the resolutions,
and $l+m_2$ consecutive $\b$-handleslide strata in the other one.
We connect the $k$ consecutive $\a$-handleslide strata, but the $\b$-handleslide
strata do not match up in number. In order to obtain the resolution~$\mathcal{R}'$,
we remove a small disk~$D$ around the central bifurcation value in the parameter
space~$D^2$. See Figure~\ref{fig:simplify-7a} for an
example, where we draw~$\partial D$ in green.
In this step, we only construct $\mathcal{R}'$ and
the dual polyhedral decomposition~$\mathcal{P}'$ outside~$D$.
The simple stabilization stratum and the $k$ consecutive $\a$-handleslide
strata are to one side of~$D$. Let $\ol{\a}$ be the $\a$-curve corresponding to~$\ol{p}$.
We connect the $\b$-handleslide strata
that correspond to $\b$-curves disjoint from $\ol{\a}$ (there are $l$ of these), and the
$m_1+m_2$ remaining $\b$-handleslide strata end on $\partial D$. In a neighborhood
of~$\partial D$, the $\b$-handleslide strata look the same as the corresponding $\b$-curves
in a neighborhood of $\ol{\a}$.

In the polyhedral decomposition~$\mathcal{P}'$ dual to~$\mathcal{R}'$,
we have a number of commutations as usual, as well as
an $(m;k+1)$-handleswap between the circles $\beta = W^s(p) \cap \S$
and $\ol{\alpha} = W^u(\ol{p}) \cap \S$.  To obtain this handleswap,
we add a single diffeomorphism stratum to~$\mathcal{R}'$, drawn in
yellow.  The handleswap
is the loop of diagrams in~$\mathcal{P}'$ around the green circle;
this loop is illustrated in Figure~\ref{fig:simplify-7a-closeup}.  We
will explain in Section~\ref{sec:simpl-handleswap} how to extend
$\mathcal{R}'$ to the interior of the green circle so that the
$(m;k+1)$-handleswap is reduced to a simple handleswap. The edge~$e$
of $\mathcal{P}'$ dual to the yellow stratum on the destabilized side
corresponds to a diffeomorphism that is isotopic to the identity of
the Heegaard surface, hence the two vertices of~$e$ correspond to the
same isotopy diagram. So we can terminate the yellow diffeomorphism
stratum at a point~$x$ of the black stabilization stratum, giving rise
to a triangle in the dual polyhedral decomposition~$\mathcal{P}'$
containing~$x$.

\begin{figure}
  \centering
  \includegraphics{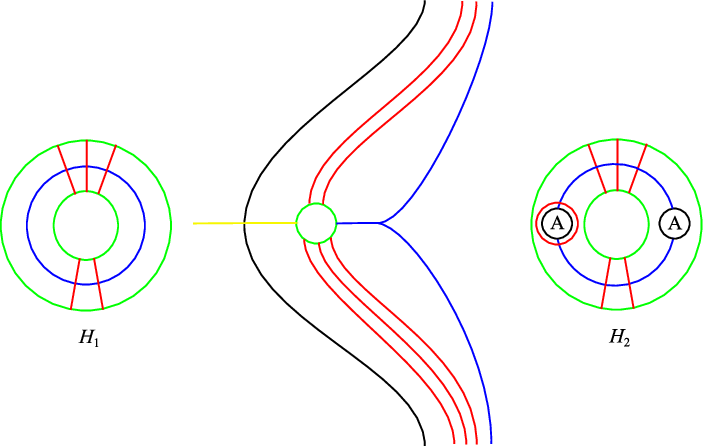}
  \caption{Resolving the singularity of type~\ref{item:link-D} from
    Figure~\ref{fig:link-10a}.}
  \label{fig:simplify-10a}
\end{figure}

\begin{figure}
  \centering
  \includegraphics{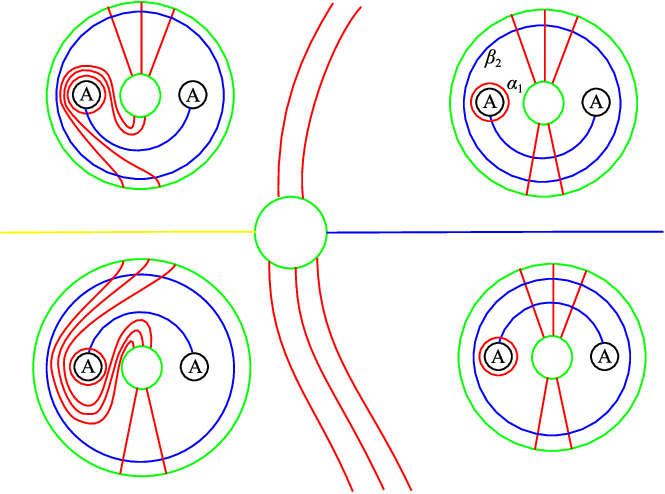}
  \caption{A closeup of the handleswap loop around the green circle
    from Figure~\ref{fig:simplify-10a}.  The handleswap is between
    $\alpha_1$ and $\beta_2$.}
  \label{fig:simplify-10a-closeup}
\end{figure}

For singularities of type~\ref{item:link-D}, a 2-1-2 birth-death-birth
singularity, we can, as usual, replace the stabilization by a simple
stabilization and a number of handleslides.  This time, we can replace
the cusp singularity by a slide triangle and a $(1;k+l)$-handleswap,
as shown in Figure~\ref{fig:simplify-10a}. As in
case~\ref{item:link-B2}, we add a diffeomorphism stratum passing
through the stabilization stratum. The corresponding diffeomorphism is
isotopic to the identity on the destabilized side. For a closeup of
the handleslide loop, see Figure~\ref{fig:simplify-10a-closeup}. The
handleswap is between $\a_1 = W^u(p_1) \cap \S$ and $\b_2 = W^s(p_2)
\cap \S$.

\begin{figure}
  \centering
  \includegraphics{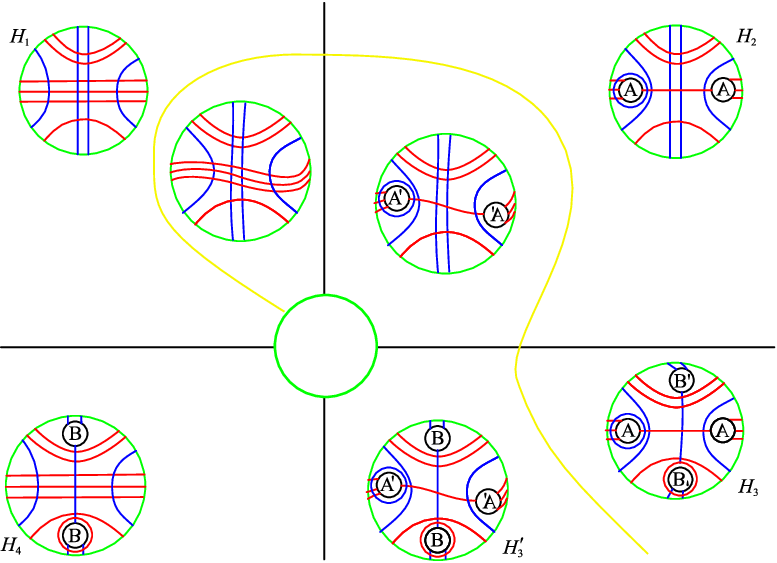}
  \caption{First step in resolving the singularity of
    type~\ref{item:link-C} from Figure~\ref{fig:link-8}. The two
    diagrams in the upper left quadrant are isotopic, so we manage to
    eliminate the diffeomorphism this way.}
  \label{fig:simplify-8-first}
\end{figure}

\begin{figure}
  \centering
  \includegraphics{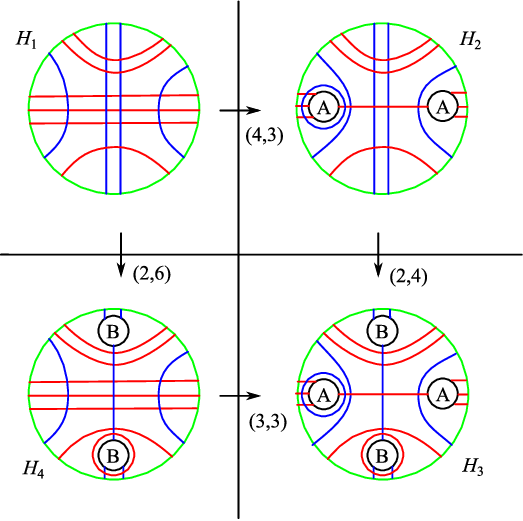}
  \caption{After the first reduction step in Figure~\ref{fig:simplify-8-first},
    the link in Figure~\ref{fig:link-8} can be replaced by this simpler link.}
  \label{fig:link-8-simpler}
\end{figure}

Finally, we consider the case of a double stabilization,
type~\ref{item:link-C}.  As shown in
Figure~\ref{fig:simplify-8-first}, we can eliminate the diffeomorphism
and assume that we are dealing with the link in
Figure~\ref{fig:link-8-simpler}.
More precisely, we first remove from the parameter space~$D^2$
a small disk~$D'$ around the codimension-2 bifurcation point.
Let~$D \subset D'$ be a smaller concentric disk whose boundary is shown in green
in Figure~\ref{fig:simplify-8-first}. We extend the stratification~$\mathcal{R}$ from
$D^2 \setminus D'$ to $D^2 \setminus D$ by extending the stabilization strata straight
to~$\partial D$, and spiralling the yellow diffeomorphism stratum across the region
(codimension-0 stratum) $C_2$ labelled by~$H_2$ and into the region~$C_1$ labeled by~$H_1$,
ending on~$\partial D$. This creates two new regions. The first
is split off~$C_2$, which we label by the destabilization~$H_2'$
of~$H_3'$ along the tube~$B$. The second one is split off~$C_1$,
and we label it by the diagram~$H_1'$
obtained from~$H_2'$ by destabilizing it along the tube~$A'$.
As explained in Remark~\ref{rem:diffeo-C}, the diffeomorphism
$d_3 \colon H_3 \to H_3'$ destabilizes to a diffeomorphism $d_3' \colon H_1 \to H_1'$ that
is isotopic to $\id_\S$, and so $H_1 = H_1'$ as isotopy diagrams.
Consequently, we can remove the diffeomorphism stratum dividing the regions labeled by~$H_1$
and~$H_1'$. We can then extend the stabilization strata straight across the disk~$D$,
and their intersection point will have the link $H_1$, $H_2'$, $H_3'$, $H_4$
shown in Figure~\ref{fig:link-8-simpler}.
The first step in the simplification of Figure~\ref{fig:link-8-alt} is shown in
Figure~\ref{fig:link-8-alt-simpler}.

\begin{figure}
  \centering
  \includegraphics{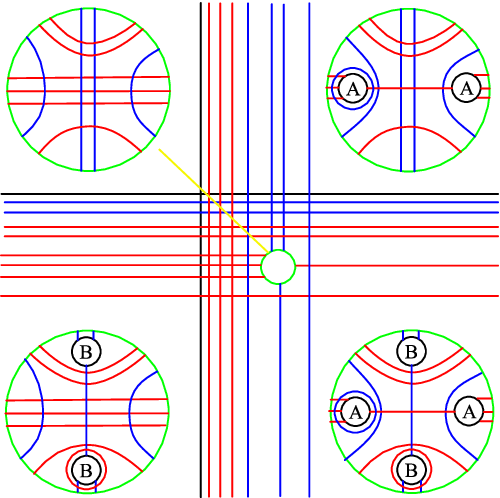}
  \caption{Resolving the singularity of type~\ref{item:link-C} with
    $t=1$ from Figure~\ref{fig:link-8-simpler}.}
  \label{fig:simplify-8}
\end{figure}

\begin{figure}
  \centering
  \includegraphics{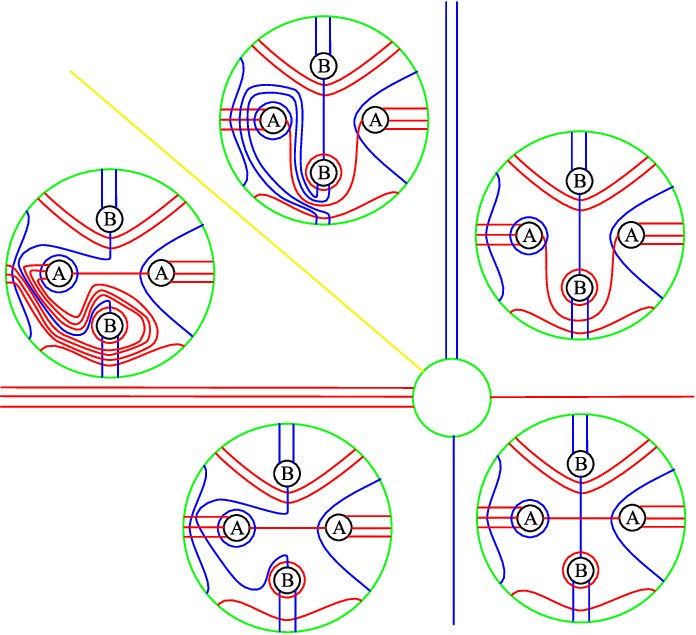}
  \caption{A closeup of the handleswap loop around the small green
    circle in Figure~\ref{fig:simplify-8}.}
  \label{fig:simplify-8-local}
\end{figure}

Recall that a key feature in case~\ref{item:link-C} was the number~$t$
of flows between the two stabilization points (from $p_1$ to $p_2$).
If $t=0$ and the resolutions of the generalized stabilization strata (the sequence of
a stabilization and a number of handleslides) are compatible with each
other, then we can fill in the link with a number of commuting squares.
Otherwise, we will get a total of $t$ different handleswaps, as shown
by example in Figure~\ref{fig:simplify-8} for $t=1$. For a closeup of
the handleswap loop, see Figure~\ref{fig:simplify-8-local}.  An
example for the $t=2$ case is shown in
Figure~\ref{fig:simplify-8-alt}.

\begin{figure}
  \centering
  \includegraphics{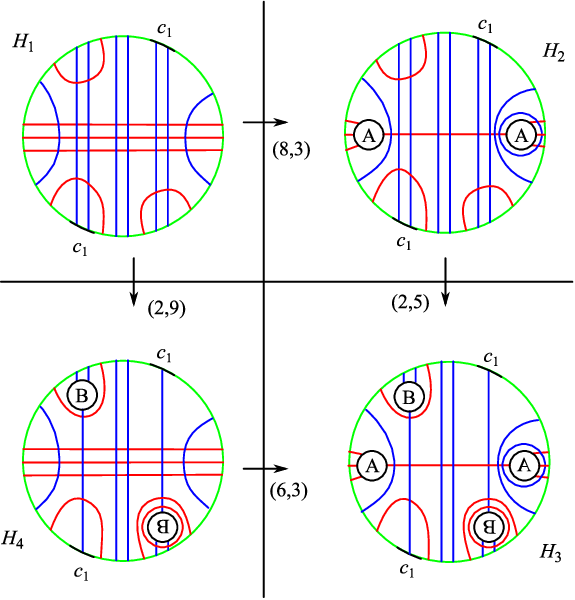}
  \caption{The first step in the simplification of the more
    complicated loop of Figure~\ref{fig:link-8-alt}.}
  \label{fig:link-8-alt-simpler}
\end{figure}

\begin{figure}
   \centering
  \includegraphics[width=4in]{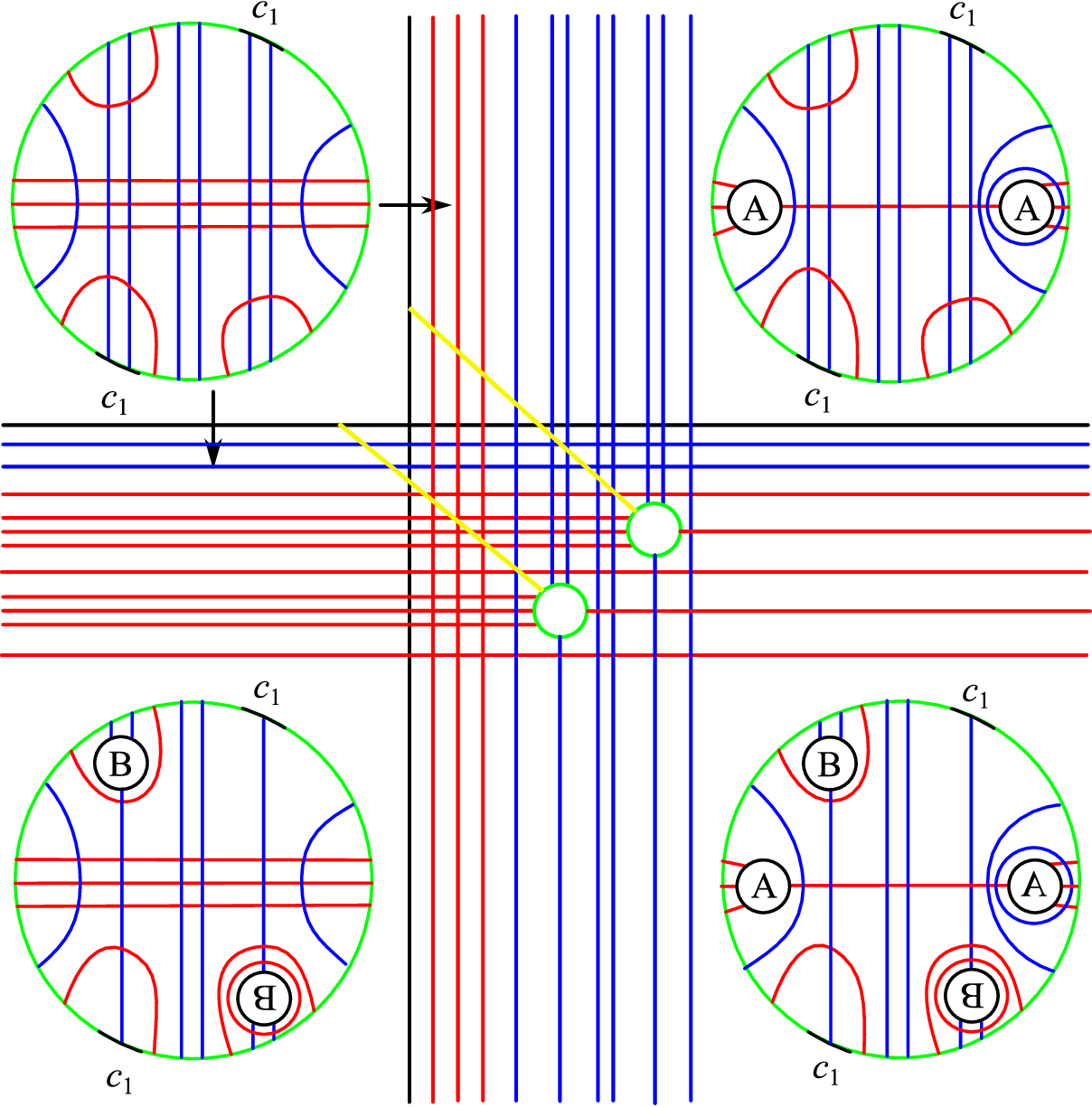}
  \caption{Resolving the singularity of type~\ref{item:link-C} with
    $t=2$ from Figure~\ref{fig:link-8-alt-simpler}.}
  \label{fig:simplify-8-alt}
\end{figure}
\subsection{Simplifying handleswaps}
\label{sec:simpl-handleswap}

Note that, in Definition~\ref{def:handleswap}, a $\b$-curve might
intersect $\a_0$ multiple times, hence several $\b$-arcs in the pair
of pants $P$ might belong to the same $\b$-curve.

\begin{definition}
  A $(k,1;l)$-handleswap is a $(k+1;l)$-handleswap between $\a_0$
  and $\b_0$ such that there is a $\b$-curve that intersects $\a_0$ in
  a single point.  Similarly, a $(k,1;l,1)$ handleswap is a
  $(k,1;l+1)$-handleswap such that there is an $\a$-curve that
  intersects $\b_0$ in a single point. A $(k;l,1)$-handleswap is
  defined in an analogous manner.
\end{definition}

The final ingredient in the proof of Theorem~\ref{thm:iso} is to
replace an arbitrary $(k;l)$-handleswap by simple handleswaps.  This
proceeds in several stages:
\begin{itemize}
\item We first stabilize the diagram, to guarantee that in each
  handleswap between $\a_0$ and $\b_0$, at least one of the
  $\b$-circles meeting $\a_0$ meets it exactly once, giving a
  $(k,1;l)$-handleswap.  Similarly, we do the same thing for the
  $\a$-circles meeting~$\b_0$, giving a $(k,1; l,1)$-handleswap.
\item Given a $(k,1;l,1)$-handleswap between $\a_0$ and $\b_0$ in which
  $\a_1$ intersects $\b_0$ once, we can perform handleslides of each
  of the $\a$-circles intersecting $\b_0$ over $\a_1$ to get rid of
  these intersections and reduce to the case of
  a $(k,1;1)$-handleswap. Similarly, we can
  perform handleslides on the $(k,1;1)$-handleswap
  to reduce it to the case of a $(1;1)$-handleswap.
\item Finally, in a $(1;1)$-handleswap between $\a_0$ and $\b_0$ with
  $\b_1$ intersecting $\a_0$ and $\a_1$ intersecting $\b_0$, we can
  perform handleslides of each $\a$-circle intersecting $\b_1$ over
  $\a_0$ to guarantee that $\b_1$ has no intersections besides the one
  with~$\a_0$.  We can similarly guarantee that $\a_1$ has no
  intersections besides the one with~$\b_0$. This is now, by
  definition, a simple handleswap.
\end{itemize}
In this overview, we have talked rather loosely about ``stabilizing''
and ``performing handleslides'' on a codimension two singularity (the
handleswap). In fact, we have to perform these operations consistently
in 2-parameter families, and see that we reduce our original loop of
Heegaard diagrams to the elementary loops of
Section~\ref{sec:strong-invar}.  We will carry this out in the
following sections. Recall that, in the previous section, each time we
encountered a handleswap in a figure we removed a disk -- indicated by
a green circle -- from the parameter space and only drew
$\mathcal{P}'$ and $\mathcal{R}'$ outside this disk.  In
$\mathcal{P}'$, the handleswap loop is parallel to this green circle.
In each step, we extend $\mathcal{P}'$ and $\mathcal{R}'$ to an
annulus in the interior of the disk removed, until we reduce to simple
handleswaps.

\subsubsection{Reducing to \texorpdfstring{$(k,1; l,1)$}{(k,1; l,1)}-handleswaps}
\label{sec:red-k+1-l+1}

It is easiest to understand this reduction by using non-simple
stabilizations.  Specifically, we will reduce a $(k;l)$-handleswap to
a $(k,1;l)$-handleswap and a $(1;l)$-handleswap (which is, of course, a
$(0,1;l)$-handleswap).  Start with a $(k;l)$-handleswap involving
$\a_0$ and $\b_0$.  Let the $\b$-strands crossing $\a_0$ be $\b_1,
\dots, \b_k$, and let the $\a$-strands crossing $\b_0$ be $\a_1,
\dots, \a_l$ (both lists with multiplicities).
In the diagrams involved in a handleswap without extra
crossings (on the top and bottom in
Figure~\ref{fig:handleswap-separated}), we can do a
$(k,1)$-stabilization on $\a_0$ and $\b_1,\dots,\b_k$.  Similarly, on
the diagram with $\b_1,\dots,\b_k$ crossing $\a_1,\dots,\a_l$ (on the
right in Figure~\ref{fig:handleswap-separated}), we can do a
$(k,l+1)$-stabilization on $\b_1,\dots,\b_k$ and $\a_0$,
$\a_1,\dots,\a_l$.  Let $\a'$ and $\b'$ be the new circles introduced
in the stabilization.  These two stabilizations in fact fit into a
2-parameter family (with the same boundary as the original $(k;l)$
handleswap): each $\a_i$ sliding over $\a_0$ for $i\in \{\,1,\dots,
l\,\}$ introduces a singularity of type~\ref{item:link-B1}, while
$\b'$ sliding over $\b_0$ introduces a singularity of
Type~\ref{item:link-B2}.  See Figure~\ref{fig:handleswap-red1} for an
example.  Note that the original handleswap is now a $(1;l)$-handleswap.

\begin{figure}
  \centering
  \includegraphics{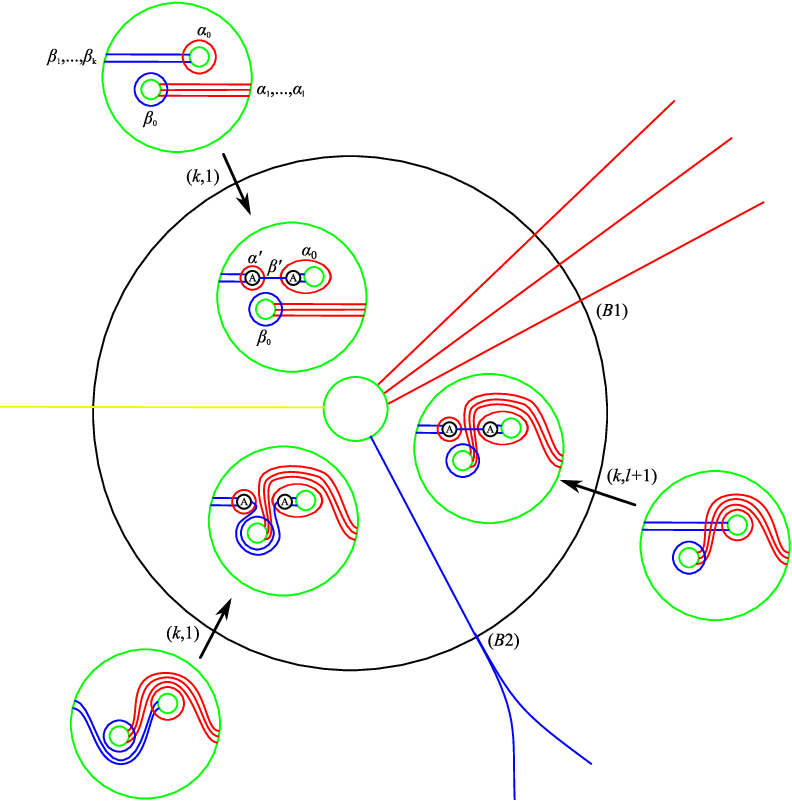}
  \caption{Reducing from a general $(k;l)$-handleswap to a $(k,1;
    l)$-handleswap, step~1.  Here we have
    introduced a circle of non-simple stabilizations to the
    $(2;3)$-handleswap from Figure~\ref{fig:handleswap-separated}.}
  \label{fig:handleswap-red1}
\end{figure}

We can resolve the non-simple stabilization introduced in this
procedure, following the algorithm of
Sections~\ref{sec:simplify-codim-1} and~\ref{sec:simplify-codim-2}, to
obtain a diagram involving only simple stabilizations and handleslides.
An example of the result is shown in
Figure~\ref{fig:handleswap-red1a}.  Resolving the singularities of
type~\ref{item:link-B1} introduces only slide pentagons, but
resolving the singularity of type~\ref{item:link-B2} introduces
another handleswap, of $\b_0$ and~$\a'$.  Since $\b'$ intersects $\a'$
in only one point, this is a $(k,1;l)$-handleswap, as desired.
\begin{figure}
  \centering
  \includegraphics{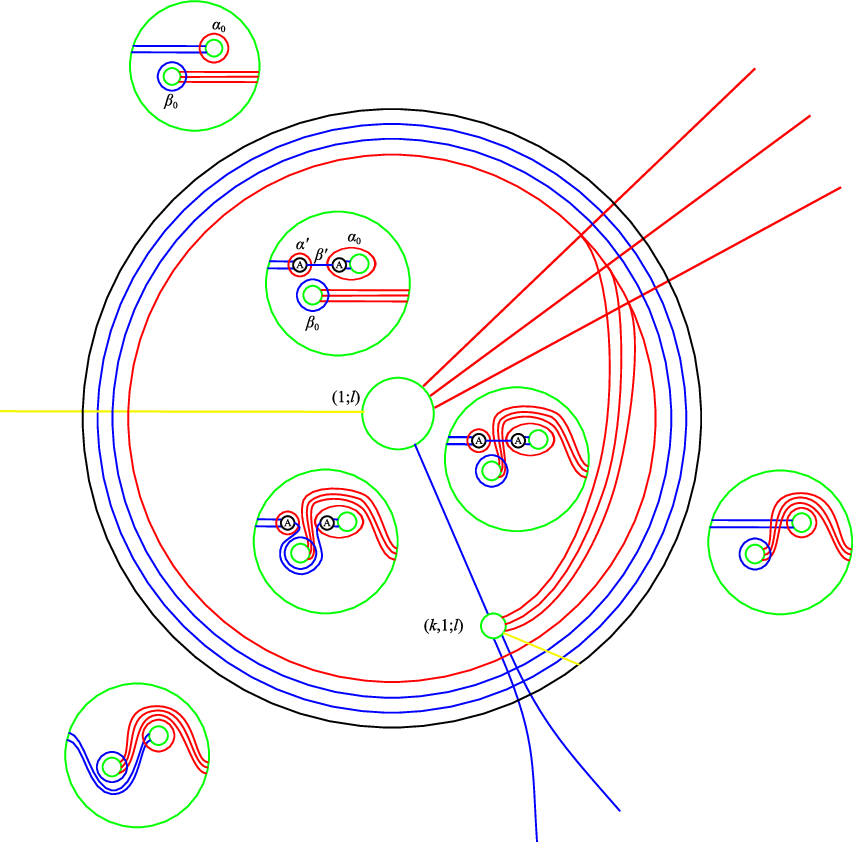}
  \caption{Reducing from a general $(k;l)$-handleswap to a $(1;l)$-
    and a $(k,1;l)$-handleswap, step~2.  This is
    the resolution (following Section~\ref{sec:simplify}) of
    Figure~\ref{fig:handleswap-red1}.}
  \label{fig:handleswap-red1a}
\end{figure}

Observe that the set of $\a$-strands involved in these two handleswaps
did not change.  Thus we can perform the same procedure again, but
with the roles of $\a$ and $\b$ switched, to reduce to handleswaps of
type $(k,1;l,1)$.

\subsubsection{Reducing to \texorpdfstring{$(1;1)$}{(1; 1)}-handleswaps}
\label{sec:red-1-1}

Next, we reduce a $(k;l,1)$-handleswap between $\alpha_0$ and
$\beta_0$ to a $(k;1)$-handleswap.  Again, let the $\beta$-strands
intersecting $\alpha_0$ be $\beta_1,\dots, \beta_k$ and let the
$\alpha$-strands intersecting $\beta_0$ be $\alpha_1, \dots,
\alpha_{l+1}$.  Assume that the circle containing $\alpha_1$
intersects $\beta_0$ only once.  Then, by sliding $\alpha_2, \dots,
\alpha_{l+1}$ over $\alpha_1$, we can reduce all three stages of the
handleswap to diagrams where only $\alpha_1$ intersects $\beta_0$,
which can in turn be related by a $(k;1)$-handleswap.  These
handleslides can be done consistently in a family with the
introduction of commuting squares and slide pentagons, that arise
when $\alpha_i$ for $i > 1$ slides over $\alpha_1$, which in turn
slides over $\alpha_0$.  See Figure~\ref{fig:handleswap-red2} for an
example.

\begin{figure}
  \centering
  \includegraphics{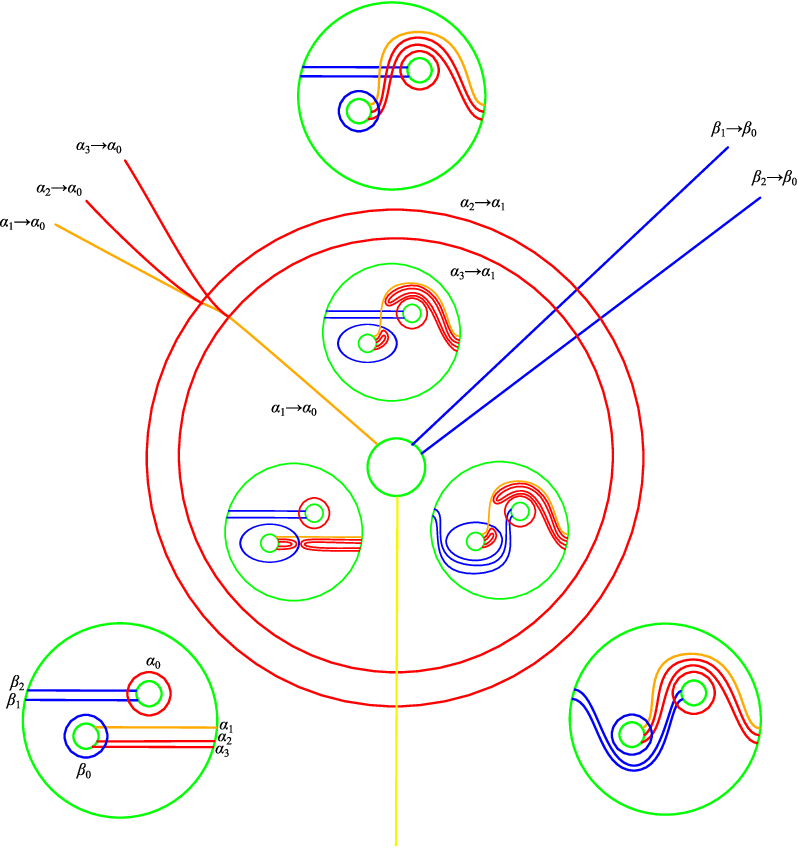}
  \caption{Here, we illustrate reduction from a $(k;l,1)$-handleswap to
    a $(k;1)$-handleswap.  In this example, $k=2$ and $l=2$.  The
    curve~$\alpha_1$, which, by hypothesis, intersects $\beta_0$ only
    once, is shown in orange.}
  \label{fig:handleswap-red2}
\end{figure}

This reduction did not affect the $\beta$-strands intersecting
$\alpha_0$.  Thus, if we start with a $(k,1;l,1)$-handleswap, we can
first reduce it to a $(k,1;1)$-handleswap as above, and then perform
the same operation on the $\beta$-strands to reduce to a
$(1;1)$-handleswap.

\subsubsection{Reducing to simple handleswaps}
\label{sec:red-simple}

Finally, we reduce a $(1;1)$-handleswap to a simple handleswap; this
is illustrated in Figure~\ref{fig:red-simple}.
\begin{figure}
  \centering
  \includegraphics{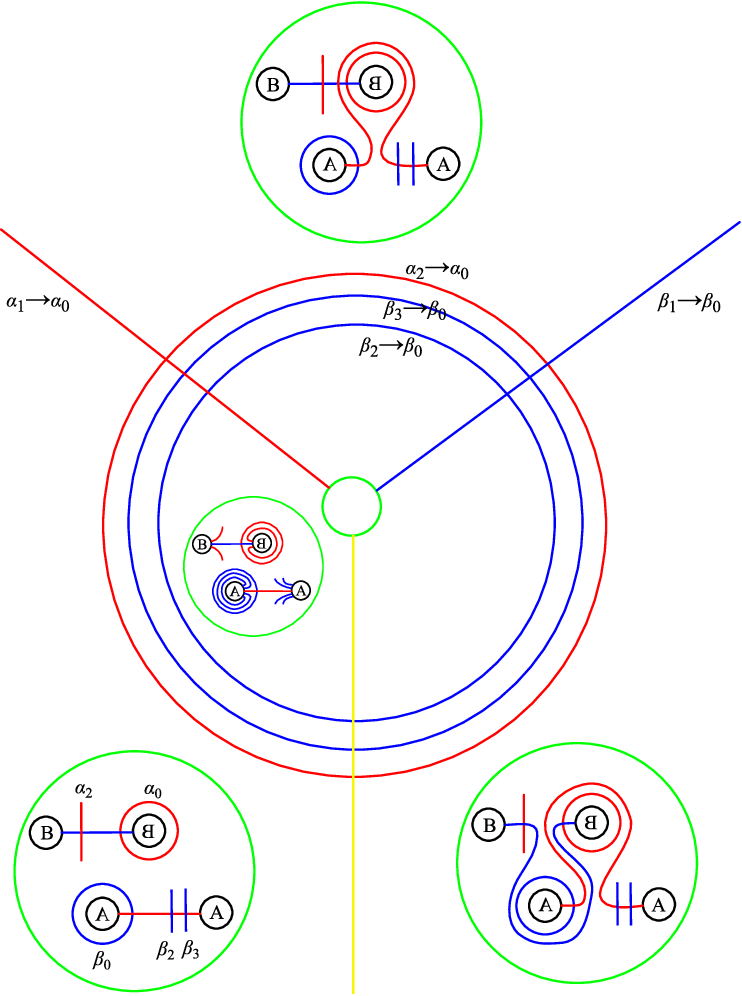}
  \caption{Reducing from a $(1;1)$-handleswap to a simple handleswap.}
  \label{fig:red-simple}
\end{figure}
Suppose the $(1;1)$-handleswap involves $\a_0$ and $\b_0$, a single
curve $\b_1$ intersecting $\a_0$, and a single curve $\a_1$
intersecting~$\b_0$.  Let the other strands intersecting $\b_1$ be
$\a_2, \dots, \a_{k+1}$, and let the other strands intersecting $\a_1$
be $\b_2, \dots, \b_{l+1}$, numbered such that, in the stage of the
handleswap where $\a_1$ and $\b_1$ cross, the intersections along
$\a_1$ are $\b_0$, $\b_1$, $\b_2, \dots$ in that order, and similarly,
the intersections along $\b_1$ are $\a_0$, $\a_1$, $\a_2, \dots$.
We can now slide (in order) $\b_{l+1}, \dots, \b_2$ over $\b_0$, from the
opposite side of the slide of $\b_1$ over $\b_0$ that appears in the
handleswap.  This commutes with
all three moves in the handleswap.  (It commutes with the slide of
$\b_1$ over $\b_0$ because we are sliding $\b_{l+1}, \dots, \b_2$ from
the opposite side of $\b_0$.)  Similarly, slide $\a_{k+1}, \dots,
\a_2$ over $\a_0$.  Again, if we slide from the opposite side from the
$\a_1$ slide, this commutes with all three moves in the handleswap.
But after these slides, $\a_1$ and $\b_1$ do not intersect any other
strands, and we have a simple handleswap, as in
Figure~\ref{fig:handleswap}.
\section{Strong Heegaard invariants have no monodromy}
\label{sec:proof}

We now have all the ingredients ready to prove
Theorem~\ref{thm:iso}. For the reader's convenience, we restate it
here.

\begin{thm}
  Let $\cS$ be a set of diffeomorphism types of sutured manifolds
  containing $[(M,\g)]$. Furthermore, let $F \colon \G(\cS) \to \C$ be
  a strong Heegaard invariant. Given isotopy diagrams $H$, $H' \in
  |\G_{(M,\g)}|$ and any two oriented paths $\eta$ and $\nu$ in
  $\G_{(M,\g)}$ connecting $H$ to $H'$, we have
  \[
  F(\eta) = F(\nu).
  \]
\end{thm}

\begin{proof}
  Since $F$ satisfies the Functoriality Axiom of
  Definition~\ref{def:strong-Heegaard}, it suffices to show that for
  any loop $\eta$ in $\G_{(M,\g)}$ of the form
  \[
  H_0 \stackrel{e_1}{\longrightarrow} H_1
  \stackrel{e_2}{\longrightarrow} \dots
  \stackrel{e_{n-1}}{\longrightarrow} H_{n-1}
  \stackrel{e_n}{\longrightarrow} H_0,
  \]
  we have $F(\eta) = \text{Id}_{F(H_0)}$. By
  Lemma~\ref{lem:handleslide}, every $\a$- and $\b$-equivalence
  between isotopy diagrams can be written as a product of
  handleslides. So, by the functoriality of $F$, we can assume that, for
  every $k \in \{\, 1,\dots,n \,\}$, if $e_k$ is an $\a$- or
  $\b$-equivalence, then it is actually a handleslide.

  We are going to construct a generic 2-parameter family $\mathcal{F}
  \colon D^2 \to \FV(M,\g)$ of sutured functions and gradient-like
  vector fields, together with a surface enhanced polyhedral
  decomposition $\mathcal{P}$ of $D^2$ such that along $S^1$ we have
  the loop $\eta$. First, for every $k \in \{\, 0,\dots,n-1 \,\}$,
  pick a representative $\HD_k = (\S_k,\alphas_k,\betas_k)$ of the
  isotopy diagram~$H_k$ such that $\alphas_k \pitchfork \betas_k$.
  We can then apply Proposition~\ref{prop:existence} to obtain a
  simple Morse-Smale pair $(f_k,v_k) \in \FV_0(M,\g)$ such that
  $H(f_k,v_k) = \HD_k$.  Let $p_k = e^{2\pi i k/n}$ be a vertex of
  $\mathcal{P}$ for every $k \in \{\, 0, \dots, n-1 \,\}$.
  We define $\mathcal{F}(p_k) = (f_k,v_k)$, and the surface enhancement
  assigns $\S_k \in \S(f_k,v_k)$ to $p_k$. In fact, the vertices of
  $\mathcal{P}$ along $S^1$ are precisely $p_0, \dots, p_{n-1}$ and
  the edges are the arcs in between them.  We extend $\mathcal{F}$ to
  the edge
  \[
  \ol{p_k p_{k+1}} = \{\, e^{2\pi i t/n} \,\colon\, t \in [k, k+1] \,\}
  \]
  between $p_k$ and $p_{k+1}$ using
  Proposition~\ref{prop:lift-moves}. By construction, each edge
  $\ol{p_k p_{k+1}}$ contains at most one bifurcation point of
  $\mathcal{F}$. Furthermore, if $\ol{p_k p_{k+1}}$ does contain a
  bifurcation point $p$, then at least one of $\S_k$ and $\S_{k+1}$ is
  in $\S(\mathcal{F}(p))$ and is transverse to $v_\mu$ for every $\mu
  \in \ol{p_kp_{k+1}}$, so Proposition~\ref{prop:1-param} applies to
  the whole edge for this separating surface. Hence~$\mathcal{P}$ and~$\mathcal{F}$
  satisfy the boundary conditions of Lemma~\ref{lem:adapted}.

  For $\mu \in S^1$, let $\mathcal{F}(\mu) = (f_\mu,v_\mu)$.  By
  Proposition~\ref{prop:grad-like-metric}, the space $G(f_\mu,v_\mu)$
  of Riemannian metrics $g$ on $M$ for which $v_\mu = \grad_g(f_\mu)$
  is non-empty and contractible.  So we can choose a generic family of
  metrics $\{\, g_\mu \in G(f_\mu,v_\mu) \,\colon\, \mu \in S^1 \,\}$.
  Choose a generic extension of $\{\, f_\mu \,\colon\, \mu \in S^1 \,\}$
  to a family of sutured functions $\{\, f_\mu \,\colon\, \mu \in D^2
  \,\}$, and similarly, extend $\{\, g_\mu \,\colon\, \mu \in S^1 \,\}$ to
  a generic family of metrics $\{\, g_\mu \,\colon\, \mu \in D^2
  \,\}$. For $\mu \in D^2$, let $v_\mu = \grad_{g_\mu}(f_\mu)$,
  modified near $\g$ such that it becomes a gradient-like vector
  field; see condition~\eqref{item:suturedfn-on-gamma} of
  Definition~\ref{def:grad-like}. Then, away from a neighborhood of
  $\g$, the family $\{\, v_\mu \,\colon\, \mu \in D^2 \,\}$ is a generic
  2-parameter family of gradients, as in
  Definition~\ref{def:family-of-gradients}.  The possible bifurcations
  of generic 2-parameter families of gradients were all listed in
  Section~\ref{sec:2-param}.  Even though the boundary behavior of
  $v_\mu$ on $\gamma$
  is not generic, this will not cause any problems since $\g$
  is an invariant subset of $v_\mu$ containing no singular points.
  Finally, let $\mathcal{F}(\mu) = (f_\mu,v_\mu) \in \FV(M,\g)$ for
  every $\mu \in D^2$. By Lemma~\ref{lem:adapted}, we can extend
  $\mathcal{P}$ to a polyhedral decomposition of $S^1$ adapted to
  $\mathcal{F}$. The surface enhancement assigning $\S_k$ to the
  boundary vertices $p_k \in \text{sk}_0(\mathcal{P}) \cap S^1$ can be
  extended to a choice of Heegaard surfaces
  \[
  \{\, \S_\mu \in \S(\mathcal{F}(\mu)) \,\colon\, \mu \in
  \text{sk}_0(\mathcal{P}) \,\}
  \]
  coherent with $\mathcal{P}$ according to Lemma~\ref{lem:coherent}.

  As in Section~\ref{sec:simplify}, let
  \[
  \mathfrak{S} = \mathfrak{S}(\mathcal{F}) = V_0 \sqcup V_1 \sqcup V_2
  \]
  be the bordered stratification given by the bifurcation strata of
  the family $\mathcal{F}$. Furthermore, pick a bordered polyhedral
  decomposition $\mathcal{R}$ of $D^2$ refining $\mathfrak{S}$ that is
  dual to $\mathcal{P}$. After applying the resolution process of
  Section~\ref{sec:simplify}, we obtain a new surface enhanced
  polyhedral decomposition $\mathcal{P}'$ of $D^2$, with dual bordered
  polyhedral decomposition $\mathcal{R}'$.  Since along $S^1$ we only
  have simple stabilizations and because we can assume that none of
  the 2-cells of $\mathcal{P}$ that intersect $S^1$ contain
  codimension-2 bifurcations of $\mathcal{F}$, after the resolution
  $\mathcal{P} \cap S^1 = \mathcal{P}' \cap S^1$, with the same
  surface enhancement. Note that we no longer claim that
  $\mathcal{P}'$ is adapted to some family of gradient-like vector
  fields, but along the boundary of each 2-cell of $\mathcal{P}'$, we
  have a loop of overcomplete diagrams that appears in
  Definition~\ref{def:strong-Heegaard} (or a stabilization slide,
  which is a degenerate distinguished rectangle; see Definition~\ref{def:stab-slide}).
  So it is either a loop of
  $\a$-equivalences, a loop of $\b$-equivalences, a loop of
  diffeomorphisms, a distinguished rectangle, a simple handleswap, or
  a stabilization slide. In addition, if we have a loop of
  diffeomorphisms, their composition is isotopic to the
  identity. Indeed, the composition $d$ of the diffeomorphisms around
  a 2-cell $\sigma$ is the same as the one induced by
  $\mathcal{F}|_{\partial \sigma} \colon \partial\sigma \to
  \FV_0(M,\g)$.  Since $\mathcal{F}$ has no bifurcations inside
  $\sigma$, the loop $\mathcal{F}|_{\partial \sigma}$ is
  null-homotopic in $\FV_0(M,\g)$, so $d$ is isotopic to the identity
  by Lemma~\ref{lem:loop}.

  The diagrams assigned to the vertices of $\mathcal{P}'$ might be
  overcomplete (except along the boundary). We now explain how to pass
  to a polyhedral decomposition $\mathcal{P}''$ that is decorated by
  actual (non-overcomplete) isotopy diagrams without altering anything
  along $S^1$. We obtain $\mathcal{P}''$ as follows. Let $v$ be a
  vertex of $\mathcal{P}'$ lying in the interior of $D^2$ that is the
  endpoint of $k_v$ one-cells. Then pick a $k_v$-gon $\sigma_v$
  centered at $v$ such that it has one vertex in each component of
  $D_\eps(v) \setminus \text{sk}_1(\mathcal{P}')$ for some $\eps$ very
  small. For every such $v$, the polygon $\sigma_v$ is a 2-cell of
  $\mathcal{P}''$.  Then, for each edge $e$ of $\mathcal{P}'$ in the
  interior of $D^2$ with $\partial e = v - w$, connect the sides of
  $\sigma_v$ and $\sigma_w$ that intersect $e$ by two arcs parallel to
  $e$; these will be edges of $\mathcal{P}''$.  If $e$ is an edge with
  one endpoint $w$ in $S^1$ and the other endpoint $v$ in the interior
  of $D^2$, then we connect the side of $\sigma_v$ intersecting~$e$
  with $w$, forming a 2-cell of $\mathcal{P}''$ that is a triangle.
  So each 2-cell of $\mathcal{P}'$ is replaced by a smaller 2-cell in
  $\mathcal{P}''$, each interior vertex of $\mathcal{P}'$ is ``blown
  up'' to a 2-cell, and each edge to a rectangle or triangle. For an
  illustration, see Figure~\ref{fig:polyhedral}.
  \begin{figure}
    \centering
    \includegraphics[width=3in]{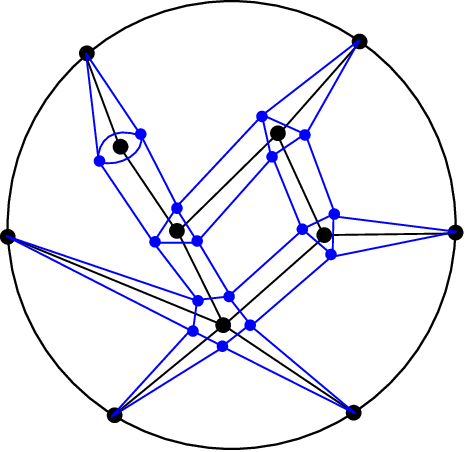}
    \caption{The polyhedral decomposition $\mathcal{P}'$ of $D^2$ is
      shown in black, and the ``blown-up'' decomposition
      $\mathcal{P}''$ in blue (along the boundary $S^1$ the two
      coincide).}
    \label{fig:polyhedral}
  \end{figure}

  We are going to decorate the vertices of $\mathcal{P}''$ with
  (non-overcomplete) isotopy diagrams by choosing spanning trees for
  the overcomplete diagram at the ``nearest'' vertex of
  $\mathcal{P}'$. If $\sigma$ is a 2-cell of
  $\mathcal{P}''$ with $r$ vertices, then we will write
  $K_1,\dots,K_r$ for the loop of overcomplete diagrams along
  $\partial \sigma$.

  Recall that, to a Morse-Smale gradient $(f,v) \in \FV_0(M,\g)$, we
  assigned the graphs $\Gamma_\pm(f,v)$, and any separating surface
  $\S \in \S(f,v)$ gives rise to an overcomplete diagram $H(f,v,\S) =
  (\S,\alphas,\betas)$.  However, we can obtain graphs
  $\Gamma_\pm(\S,\alphas,\betas)$ directly from the overcomplete
  diagram $(\S,\alphas,\betas)$ as follows.  First, consider the graph
  whose vertices correspond to the components of $\S \setminus
  \alphas$, and for each component $\a$ of $\alphas$, connect the
  vertices corresponding to the components on the two sides of $\a$ by
  an edge (possibly introducing a loop). Then
  $\Gamma_-(\S,\alphas,\betas)$ is obtained by identifying all the
  vertices that correspond to a component of $\S \setminus \alphas$
  that intersects $\partial \S$ non-trivially. We define
  $\Gamma_+(\S,\alphas,\betas)$ in an analogous manner. In case
  $(\S,\alphas,\betas) = H(f,v)$, we have
  \[
  \Gamma_\pm(\S,\alphas,\betas) = \Gamma_\pm(f,v).
  \]
  If $D$ is an overcomplete diagram and $T_\pm$ is a spanning tree of
  $\Gamma_\pm(D)$, then we denote by $H(D,T_\pm)$ the diagram obtained
  from $D$ by removing the $\a$- and $\b$-curves corresponding to
  edges in $T_\pm$.  A diffeomorphism of isotopy diagrams $d \colon
  D_1 \to D_2$ induces a map $d_* \colon \Gamma_\pm(D_1) \to
  \Gamma_\pm(D_2)$.

  Note that each vertex of $\mathcal{P}''$ in the interior of $D^2$
  lies in a unique 2-cell $\sigma$ that corresponds to a 2-cell of
  $\mathcal{P}'$. Hence, we can pick spanning trees for each such
  2-cell separately to make $F$ commute along their boundaries. Then
  we need to check that $F$ also commutes along 2-cells of
  $\mathcal{P}''$ corresponding to 0-cells and 1-cells of
  $\mathcal{P}'$.

\begin{definition}
  The isotopy diagrams $(\S_1,A_1,B_1)$ and $(\S_2,A_2,B_2)$ are
  $\a/\b$-e\-quiv\-a\-lent if $\S_1 = \S_2$, $A_1 \sim A_2$, and $B_1
  \sim B_2$.
\end{definition}

Clearly, an $\a$-equivalence or a $\b$-equivalence is a special case
of an $\a/\b$-equivalence.  From $\G(\cS)$, we obtain a graph
$\G'(\cS)$ by adding an edge for every $\a/\b$-equivalence that is not
an $\a$-equivalence or a $\b$-equivalence, and similarly, from
$\G_{(M,\g)}$ we obtain the graph $\G'_{(M,\g)}$.  The strong Heegaard
invariant $F \colon \G(\cS) \to \mathcal{C}$ extends to $\G'(\cS)$ as
follows.  Given an edge $e$ from $(\S,A_1,B_1)$ to $(\S,A_2,B_2)$,
there is an $\a$-equivalence $h$ from $(\S,A_1,B_1)$ to $(\S,A_2,B_1)$
and a $\b$-equivalence $g$ from $(\S,A_2,B_1)$ to $(\S,A_2,B_2)$.  We
let $F(e) = F(g) \circ F(h)$. Note that we could have taken the
intermediate diagram to be $(\S,A_1,B_2)$, but that gives the same map
by the Commutativity Axiom of strong Heegaard invariants applied to a
distinguished rectangle of type~\eqref{item:rect-alpha-beta}.

\begin{lemma} \label{lem:alpha-beta} Suppose that
  \[
  D_1 \stackrel{a_1}{\longrightarrow} D_2
  \stackrel{a_2}{\longrightarrow} \dots
  \stackrel{a_{r-1}}{\longrightarrow} D_r
  \stackrel{a_r}{\longrightarrow} D_1
  \]
  is a loop of isotopy diagrams in $\G'(\cS)$ such that each edge
  $a_i$ is an $\a/\b$-equivalence. Furthermore, let $F \colon \G(\cS)
  \to \mathcal{C}$ be a strong Heegaard invariant. Then
  \[
  F(a_r) \circ \dots \circ F(a_1) = \text{Id}_{F(D_1)}.
  \]
\end{lemma}

\begin{proof}
  As above, we can write every $\a/\b$-equivalence as a product of an
  $\a$-equivalence and a $\b$-equivalence.  By the Commutativity
  Axiom, it suffices to prove the lemma when $a_1, \dots, a_{i-1}$ are
  $\a$-equivalences and $a_i,\dots,a_r$ are $\b$-equivalences for
  some $i$. However, in this case $D_1 = D_i$, so we only have to
  prove the lemma when $a_1,\dots,a_r$ are all $\a$-equivalences, or
  when they are all $\b$-equivalences.  This is a simple consequence
  of the Functoriality Axiom of strong Heegaard invariants.
\end{proof}

If $\sigma$ is a 2-cell of $\mathcal{P}''$ corresponding to a vertex
$v$ of $\mathcal{P}'$ and $v$ is decorated by the overcomplete diagram
$K$, then choosing arbitrary spanning trees $T^1_\pm, \dots, T^r_\pm$
for $\Gamma_\pm(K)$ gives diagrams $D_i = H(K,T^i_\pm)$ for $i \in
\{\, 1,\dots,r \,\}$ such that any two of them are $\a/\b$-equivalent.
Hence $F$ applied to the loop of diagrams $D_1,\dots,D_r$ along
$\partial\sigma$ commutes by Lemma~\ref{lem:alpha-beta}.

Next, suppose that $\sigma$ is a 2-cell of $\mathcal{P}''$ that
corresponds to a 2-cell $\sigma_0$ of $\mathcal{P}'$.  We distinguish
several cases. In all the cases, we make sure that if the edge between
$K_i$ and $K_{i+1}$ is a diffeomorphism, then we choose spanning trees
$T^i_\pm$ and $T^{i+1}_\pm$ such that $T^{i+1}_\pm =
d_*(T^i_\pm)$. Furthermore, if this edge is an index~1-2
stabilization, then $T^{i+1}_\pm$ is the same as $T^i_\pm$ (in
particular, it does not contain the edges corresponding to the new
$\a$- and $\b$-curve).

If all the edges of $\partial\sigma_0$ are diffeomorphisms
$d_1,\dots,d_r$, then we showed above that their composition is
isotopic to the identity. Choose a spanning tree $T_\pm^1$ for
$\Gamma_\pm(K_1)$. Given $T^i_\pm$, we define $T^{i+1}_\pm =
d_{i*}(T^i_\pm)$ for $i \in \{\,1,\dots,r-1\,\}$. Note that $T^1_\pm =
d_{r*}(T^r_\pm)$, since $d_r \circ \dots \circ d_1$ is isotopic to the
identity and hence it cannot permute the $\a$-curves or the
$\b$-curves, which are both linearly independent in $H_1(\S_1)$.  By
taking $D_i = H(K_i,T^i_\pm)$ at the vertices of $\partial \sigma$, we
obtain the loop of diffeomorphisms
\[
D_1 \stackrel{d_1}{\longrightarrow} D_2
\stackrel{d_2}{\longrightarrow} \dots
\stackrel{d_{r-1}}{\longrightarrow} D_r
\stackrel{d_r}{\longrightarrow} D_1
\]
in $\G_{(M,\g)}$. The invariant $F$ commutes along this loop, since
\[
F(d_r) \circ \dots \circ F(d_1) = F(d_r \circ \dots \circ d_1) =
\text{Id}_{F(D_1)}
\]
by the Functoriality and Continuity Axioms.

If $\partial \sigma_0$ is a loop of $\a$- or $\b$-equivalences (e.g.,
a link of a singularity of type~(A)), or a commutative rectangle of
type~\eqref{item:rect-alpha-beta}, then, after choosing arbitrary
spanning trees, we get a loop of $\a/\b$-equivalences along $\partial
\sigma$. Then the strong Heegaard invariant $F$ commutes by
Lemma~\ref{lem:alpha-beta}.

Suppose that along $\sigma_0$, we have the distinguished rectangle
\[
\xymatrix{K_1 \ar[r]^e \ar[d]^f & K_2 \ar[d]^g \\ K_3 \ar[r]^h & K_4.}
\]
If this is of type~\eqref{item:rect-alpha-stab}, with $e$ and $h$
being $\a$-equivalences and $f$, $g$ being stabilizations, then we
choose a spanning tree $T^1_\pm$ of $\Gamma_\pm(K_1)$ and then a
spanning tree $T^2_\pm$ of $\Gamma_\pm(K_2)$ such that $T^2_+ =
T^1_+$. We can view $T^1_\pm$ as a spanning tree $T^3_\pm$ of
$\Gamma_\pm(K_3)$, and we can view $T^2_\pm$ as a spanning tree
$T^4_\pm$ of $\Gamma_\pm(K_4)$.  Then the vertices of $\sigma$ are
decorated by the diagrams $D_i = H(K_i,T^i_\pm)$ for $i \in \{\,
1,\dots, 4 \,\}$, which also form a distinguished rectangle of
type~\eqref{item:rect-alpha-stab}. A distinguished rectangle of
overcomplete diagrams of type~\eqref{item:rect-alpha-diff}, where
$f$ and $g$ are diffeomorphisms, can be reduced to a distinguished rectangle
of non-overcomplete diagrams of the same type in an analogous manner.
In case of a rectangle of
type~\eqref{item:rect-stab-stab} including only stabilizations, we
start with a spanning tree $T^1_\pm$ for $\Gamma_\pm(K_1)$, which then
gives rise to $T^2_\pm$ and $T^3_\pm$ in a natural manner. Both
$T^2_\pm$ and $T^3_\pm$ give the same spanning tree $T^4_\pm$ of
$\Gamma_\pm(K_4)$, as this is also the image of $T^1_\pm$ under the
embedding of $\Gamma_\pm(K_1)$ into $\Gamma_\pm(K_4)$. Finally, for a
rectangle of type~\eqref{item:rect-stab-diff}, where $e$ and $h$ are
stabilizations and $f$, $g$ are diffeomorphisms, we first choose
$T^1_\pm$, then let $T^3_\pm = f_*(T^1_\pm)$. We let $T^2_\pm$ be the
image of $T^1_\pm$ under the embedding of $\Gamma_\pm(K_1)$ into
$\Gamma_\pm(K_2)$, and $T^4_\pm$ is the image of $T^3_\pm$ under the
embedding of $\Gamma_\pm(K_3)$ into $\Gamma_\pm(K_4)$. By
construction, $T^4_\pm = g_*(T^2_\pm)$, hence reducing to a loop of
non-overcomplete diagrams of type~\eqref{item:rect-stab-diff}
along $\partial\sigma$.

The last possible type of loop along $\partial \sigma_0$ is a simple
handleswap, a triangle in $\G_{(M,\g)}$ with vertices decorated by
isotopy diagrams $K_1,K_2$, and $K_3$ on the common Heegaard surface
$\S$.  Let the $\a$- and $\b$-curves involved in the handleswap be
$\a_1$, $\a_2$, and $\b_1$, $\b_2$. Recall that the other $\a$- and
$\b$-curves coincide in $K_1$, $K_2$, and $K_3$, so the
graphs~$\Gamma_\pm(K_i)$ only differ in the 4 edges corresponding to $\a_1$,
$\a_2$, $\b_1$, $\b_2$.  Since $\S \setminus (\a_1 \cup \a_2)$ has the
same number of components as $\S$, there exists a common spanning tree
$T_-$ of $\Gamma_-(K_i)$ for $i \in \{1,2,3\}$ not containing the
edges corresponding to $\a_1$ and $\a_2$.  Similarly, $\b_1 \cup \b_2$
is non-separating, so there is a common spanning tree $T_+$ of
$\Gamma_+(K_i)$ for $i \in \{1,2,3\}$.  If we take the
non-overcomplete sutured diagrams $D_i = H(K_i, T_\pm)$ for $i \in
\{1,2,3\}$, then $D_1$, $D_2$, and $D_3$ also form a simple
handleswap. Indeed, they all contain $\a_1$, $\a_2,$ $\b_1$, and
$\b_2$, and all other curves coincide.

Finally, let~$\sigma$ be a 2-cell of~$\mathcal{P}''$ that corresponds
to an edge~$e$ of~$\mathcal{P}'$ not lying entirely in~$S^1$.
Then~$\sigma$ is a rectangle if~$e$ lies in the interior of~$D^2$, and is a
triangle if~$e \cap S^1 \neq \emptyset$. In the latter case,
we view~$\sigma$ as a rectangle in~$\G_{(M,\g)}$ with one edge being the
identity. Let~$\sigma_0$ and~$\sigma_1$ be the 2-cells of~$\mathcal{P}'$
lying on the two sides of~$e$, and the edges
corresponding to~$e$ in~$\mathcal{P}''$ are~$g_0 \subset \sigma_0$
and~$g_1 \subset \sigma_1$.  We denote the other two edges of~$\sigma$
by~$h_0$ and~$h_1$. The vertices of $e$ are decorated by the overcomplete
diagrams~$K_0$ and~$K_1$.  We distinguish three cases depending on the
type of~$e$. If~$e$ is an~$\a$- or $\b$-equivalence, then no matter
how we choose trees for~$K_0$ and~$K_1$ in~$\sigma_0$ and~$\sigma_1$,
along~$\sigma$ we get a loop of $\a/\b$-equivalences for which~$F$
commutes by Lemma~\ref{lem:alpha-beta}.

If $e$ is a stabilization, then in both $\sigma_0$ and $\sigma_1$, we
chose trees such that $g_0$ and $g_1$ are decorated by stabilizations.
Furthermore, the edges $h_0$ and $h_1$ are decorated by
$\a/\b$-equivalences, coming from the fact that we chose spanning
trees for the same overcomplete diagram to decorate the endpoints of
$h_i$. If $h_1$ is on the stabilized side, then this
$\a/\b$-equivalence leaves the $\a$- and $\b$-curve involved in the
stabilizations unchanged.  To see that applying $F$ to $\sigma$ we get
a commutative square, bisect both $h_0$ and $h_1$ and write them as a
product of an $\a$-equivalence and a $\b$-equivalence. Connect the
midpoints of $h_0$ and $h_1$ by a stabilization edge, hence
decomposing $\sigma$ into two distinguished rectangles of
type~\eqref{item:rect-alpha-stab}.  Then $F$ commutes when applied to
each of these distinguished rectangles.  If $e$ is a diffeomorphism,
then we proceed in a way analogous to the previous case; we can
decompose $\sigma$ into two distinguished rectangles.

So we now have a polyhedral decomposition $\mathcal{P}''$ of $D^2$,
together with a morphism of graphs
$H \colon \text{sk}_0(\mathcal{P}'') \to \G_{(M,\g)}$,
such that $F \circ H$ commutes along the boundary of each 2-cell of
$\mathcal{P}''$. What remains to show is that this implies that $F$
commutes along the boundary of $D^2$; i.e.,
\[
F(\eta) = F(e_n) \circ \dots \circ F(e_1) = \text{Id}_{F(H_0)}.
\]
The proof of this requires some care as the composition of morphisms
is not commutative.
For this, we show that there is a ``combinatorial 0-homotopy'' from
$S^1$ to the boundary of a 2-cell of $\mathcal{P}''$. By this, we mean
that there is a sequence of curves $\eta_0, \dots, \eta_k$ in $D^2$
such that
\begin{enumerate}
\item $\eta_0 = \eta$ and $\eta_k = \partial \sigma_0$ for some
  two-cell $\sigma_0$ of $\mathcal{P}''$,
\item every $\eta_i$ is a properly embedded curve in
  $\text{sk}_1(\mathcal{P}'')$, and
\item the 1-chain $\eta_i - \eta_{i+1}$ is the boundary of a single
  2-cell $\sigma_i$ of $\mathcal{P}''$.
\end{enumerate}
This clearly implies that $F(\eta)= \text{Id}_{F(H_0)}$, since
\[
F(\eta_i) \circ F(\eta_{i+1})^{-1} = F(\partial \sigma_i) = \text{Id}
\]
for every $i \in \{\,1,\dots,k-1\,\}$,  and
$F(\eta_k) = F(\partial \sigma_0) = \text{Id}$.

To construct the combinatorial 0-homotopy, we proceed
recursively. Suppose we have already obtained $\eta_i$. Then $\eta_i$
bounds a disk $D^2_i$ in $D^2$, and $\mathcal{P}''$ restricts to a
polyhedral decomposition of $D^2_i$. It suffices to show that if
$D^2_i$ has more than one 2-cells, then there exists a 2-cell
$\sigma_i$ in $D^2_i$ that intersects $\eta_i$ in a single
arc. Indeed, we then take $\eta_{i+1} = \eta_i - \partial \sigma_i$,
which is a simple closed curve.  The existence of such a $\sigma_i$
follows from the following lemma.

\begin{lemma}
  For any polyhedral decomposition of $D^2$ with more than one
  2-cells, there exists a 2-cell that intersects $S^1$ in a single
  arc.
\end{lemma}

\begin{proof}
  We proceed by induction on the number $t$ of 2-cells. If $t = 2$,
  then let the 2-cells be $\sigma_1$ and $\sigma_2$.  Since the
  attaching map of each 2-cell is an embedding, $\sigma_1 \cap
  \sigma_2$ consists of some disjoint arcs, and to obtain $D^2$, it
  has to be a single arc $a$. Hence $\sigma_i \cap S^1
  = \partial\sigma_i \setminus \text{Int}(a)$ is a single arc for
  $i \in \{1,2\}$.

  Now suppose that the statement holds for polyhedral decompositions
  for which the number of 2-cells is less than $t$ for some $t > 2$,
  and consider a decomposition where the number of 2-cells is $t$.
  There is a 2-cell $\sigma_1$ such that $\text{Int}(\sigma_1 \cap
  S^1) \neq \emptyset$.  If $\sigma_1 \cap S^1$ has a single
  component, it has to be an arc, and we are done. Otherwise, $D^2
  \setminus \sigma_1$ consists of at least two components, each of
  whose closure is homeomorphic to a disk; let $D_1$ be one of these.
  Observe that $s_1 = D_1 \cap \sigma_1$ is an arc.  If there are at
  least two 2-cells in $D_1$, then, by induction, there is a 2-cell
  $\sigma_2$ in $D_1$ for which $\sigma_2 \cap \partial D_1$ is a
  single arc $a_1$.  Since
  \[
  \sigma_2 \cap S^1 = a_1 \setminus \text{Int}(s_1),
  \]
  we are done if this is a single arc. Otherwise,
  either $a_1 \subset s_1$ or $a_1 \supset s_1$. In both cases, we
  merge $\sigma_1$ and $\sigma_2$ by removing all the vertices and
  edges in $\text{Int}(a_1 \cap s_1)$.  We obtain a polyhedral
  decomposition of~$D^2$ where the number of 2-cells is $t-1 \ge
  2$, so, by the induction hypothesis, there is a 2-cell $\sigma_3$
  that intersects $S^1$ in a single arc. Since $\sigma_1 \cup
  \sigma_2$ intersects $S^1$ in the same number of components as
  $\sigma_1$, which is more than one, $\sigma_3 \neq \sigma_1 \cup
  \sigma_2$, and so $\sigma_3$ is also a 2-cell of the original
  decomposition.  Finally, if $D_1$ consists of a single 2-cell
  $\sigma_2$, then $\sigma_2 \cap S^1 = \partial\sigma_2 \setminus
  \text{Int}(s_1)$, which is a single arc.
\end{proof}

Since the polyhedral decomposition $\mathcal{P}''|_{D^2_{i+1}}$
contains one less 2-cell than $\mathcal{P}''|_{D^2_i}$, the process
ends when $\mathcal{P}''|_{D^2_k}$ consists of a single 2-cell, and
we obtain the combinatorial 0-homotopy.
This concludes the proof of Theorem~\ref{thm:iso}.
\end{proof}
\section{Heegaard Floer homology}
\label{sec:HeegaardFloer}

In this section, we prove Theorem~\ref{thm:strong}. We first explain
how the different versions of Heegaard Floer homology fit into the
framework of weak Heegaard invariants in the sense of
Definition~\ref{def:weak-Heegaard}. We then show
they are strong Heegaard invariants in the sense of Definition~\ref{def:strong-Heegaard}.
Most of our constructions build on the work of Ozsv\'ath and Szab\'o~\cite[Section
2.5]{OS06:HolDiskFour}.  To show Theorem~\ref{thm:strong}, we perform the following additional checks:
\begin{enumerate}
\item We keep track of the embedding of the
Heegaard surface in the 3-manifold, and include diffeomorphisms isotopic
to the identity among the Heegaard moves, which allows us to define functorial
diffeomorphism maps on Heegaard Floer homology.
\item We verify the Continuity Axiom~\eqref{item:strong-cont} of Definition~\ref{def:strong-Heegaard}.
\item Most importantly, we verify simple handleswap invariance.
\end{enumerate}

For concreteness, we explain the case of sutured Floer
homology in detail as it includes $\HFa$ and $\HFLa$
as special cases, and only remark on the differences for the other
versions. In particular, we show that $\SFH$ is a strong Heegaard
invariant of the class~$\cS_\bal$. However, to emphasize that
all the arguments are essentially the same for the other
versions of Heegaard Floer homology, we will write
$\HF^\circ$ instead of $\SFH$. All Heegaard diagrams appearing in
this section are assumed to be balanced.

\subsection{Heegaard Floer homology as a weak Heegaard invariant}
\label{sec:HeegaarFloerInvariant}

We start by explaining how Heegaard Floer homology fits into the
framework of weak Heegaard invariants. This was proved, in a different form,
by Ozsváth and Szabó (with the exception of~$\SFH$);
we remind the reader of the proof in order to fill
in details and because we will later extend the arguments to prove that
Heegaard Floer homology is a strong Heegaard invariant.
One complication arises from the
fact that Heegaard Floer homology is, in fact, not an invariant
associated to arbitrary Heegaard diagrams; rather, these Heegaard
diagrams must satisfy the additional property of {\em admissibility}.
There are several forms of admissibility. We will focus presently on
the case of {\em weak admissibility} in the sense of Ozsv\'ath and
Szab\'o~\cite[Definition~4.10]{OS04:HolomorphicDisks} and
Juh\'asz~\cite[Definition~3.11]{Juhasz06:Sutured}, which is sufficient
for defining $\HFa$, $\SFH$, and $\HFp$.  The stronger variant, used
in the construction of $\HFm$ and $\HFinf$, is defined in reference to
an auxiliary $\SpinC$ structure.

We briefly discuss $\SpinC$ structures. Let $\HD =
(\S,\alphas,\betas)$ be an abstract (i.e., non-embedded) Heegaard
diagram. Our aim is to explain what we mean by a $\SpinC$-structure
for $\HD$. If $\HD$ is a diagram of the sutured manifolds $(M,\g)$ and
$(M',\g')$, then there is a diffeomorphism $\phi \colon (M,\g) \to
(M',\g')$ that is well-defined up to isotopy fixing $\S$.
So $\phi$ induces a bijection
\[
b_{(M,\g),(M',\g')} \colon \SpinC(M,\g) \to \SpinC(M',\g'),
\]
which intertwines the $H_1(M)$-action on the set $\SpinC(M,\g)$ and
the $H_1(M')$-action on the set $\SpinC(M',\g')$.  For $\spinc \in
\SpinC(M,\g)$ and $\spinc' \in \SpinC(M',\g')$, we write $\spinc \sim
\spinc'$ if and only if $b_{(M,\g),(M',\g')}(\spinc) = \spinc'$. Then
``$\mathord{\sim}$'' defines an equivalence relation on the class of
elements of $\SpinC(M,\g)$ for all $(M,\g)$ such that $\HD$ is a
diagram of $(M,\g)$.  We define $\SpinC(\HD)$ to be the collection of
these equivalence classes.
Given $\spinc \in \SpinC(M,\g)$, we denote its
equivalence class by $[\spinc]$; this is an element of
$\SpinC(\S,\alphas,\betas)$.

Let $\x \in \Torus_\a \cap \Torus_\b$ be a Heegaard Floer generator.
By Proposition~\ref{prop:existence}, there exists a simple pair $(f,v)
\in \FV_0(M,\g)$ such that $H(f,v) = \HD$.  We can associate to $\x$ a
nowhere vanishing vector field~$v_\x$ on $M$ as follows.
Let~$\g_\x$
be the union of the flow lines of $v$ passing through the points
of~$\x$. Then we delete~$v$ on a thin regular neighborhood~$N_\x$
of~$\g_\x$, and extend it to~$N_\x$ as a nowhere vanishing vector
field. This is possible since each component of $N_\x$ contains two
critical points of $v$ of opposite sign, and hence the degree of~$v$
is zero along each component of~$\partial N_\x$. If we take a
different simple pair $(\ol{f},\ol{v}) \in \FV_0(M,\g)$ with
$H(\ol{f},\ol{v}) = \HD$, then Proposition~\ref{prop:connected-MS}
implies that $(f,v)$ and $(\ol{f},\ol{v})$ can be connected by a path
within $\FV_0(M,\g)$.  This show that the vector fields $v_\x$ and
$\ol{v}_\x$ are homologous relative to $\partial M$. (Recall that two
vector fields are \emph{homologous} relative to $\partial M$ if they
are homotopic in the
complement of a ball, where the homotopy is the identity on $\partial
M$.) In particular, the
$\SpinC$ structures defined by $v_\x$ and $\ol{v}_\x$ coincide; we
denote it by $\spinc_{(M,\g)}(\x)$. As above, we can also assign to
$\x$ an element $\spinc_{(M',\g')}(\x) \in \SpinC(M',\g')$. By
construction, $\spinc_{(M,\g)}(\x) \sim \spinc_{(M',\g')}(\x)$, so we
can define $\spinc(\x) \in \SpinC(\HD)$ to be $[\spinc_{(M,\g)}(\x)]$
for any $(M,\g)$ such that $\HD$ is a diagram of $(M,\g)$.

As explained by Ozsv\'ath and Szab\'o~ \cite[Section~4 and
Theorem~6.1]{OS04:HolomorphicDisks}, the Floer homology groups
depend on a choice of complex structure $\mathfrak{j}$ on $\S$ and a
generic path $J_s \subset \mathcal{U}$ of perturbations of the induced
complex structure over $\text{Sym}^g(\S)$, where $\mathcal{U}$ is a
certain contractible set of almost complex structures. The following result
is due to Ozsv\'ath and Szab\'o~\cite[Lemma~2.11]{OS06:HolDiskFour}.

\begin{lemma} \label{lem:complex}
  Let $(\S,\alphas,\betas)$ be admissible.  Fix two different choices
  $(\mathfrak{j}, J_s)$ and $(\mathfrak{j}', J_s')$ of complex
  structures and perturbations.  Then there is an isomorphism
  \[
  \Phi_{J_s \to J_s'} \colon \HF^\circ_{J_s}(\S,\alphas,\betas,\spinc)
  \to \HF^\circ_{J_s'}(\S,\alphas,\betas,\spinc).
  \]
  These isomorphisms are natural in the sense that
  \[
  \Phi_{J_s' \to J_s''} \circ \Phi_{J_s \to J_s'} = \Phi_{J_s \to
    J_s''},
  \]
  and $\Phi_{J_s \to J_s}$ is the identity.
\end{lemma}

Hence, we can define
\[
\HF^\circ(\S,\alphas,\betas,\spinc) = \coprod_{(\mathfrak{j},J_s)}
\HF^\circ_{J_s}(\S,\alphas,\betas,\spinc) \Bigr/\mathord{\sim},
\]
where $x \sim y$ if and only if $y = \Phi_{J_s \to J_s'}(x)$ for some
$(\mathfrak{j},J_s)$ and $(\mathfrak{j}',J_s')$.

Let $(\S,\alphas,\betas,\gammas)$ be an admissible triple diagram. Then,
by counting pseudo-holomorphic triangles,
Ozsv\'ath and Szab\'o~\cite[Theorem~8.12]{OS04:HolomorphicDisks} in the case of
ordinary Heegaard triple-diagrams, and Grigsby and
Wehrli~\cite[Section~3.3]{GW} for sutured triple-diagrams
defined maps
\[
F_{\a,\b,\g} \colon \HF^\circ(\S,\alphas,\betas) \otimes
\HF^\circ(\S,\betas,\gammas) \to \HF^\circ(\S,\alphas,\gammas).
\]

\begin{lemma} \label{lem:triangle}
  Let $(\S,\betas,\gammas)$ be an admissible sutured diagram such that
  $\betas \sim \gammas$. Then there is a unique $\SpinC$ structure
  $\spinc_0 \in \SpinC(\S,\betas,\gammas)$ such that $c_1(\spinc_0) = 0$.
  Furthermore, in the highest non-zero relative homological grading
  $\HF^\circ(\S,\betas,\gammas,\spinc_0)$ is isomorphic to~$\FF_2$;
  we denote its generator by $\Theta_{\b,\g}$.
\end{lemma}

\begin{proof}
Suppose we have a diagram $(\S,\betas,\gammas)$ such that $\betas
\sim \gammas$, and let $k = |\betas| = |\gammas|$. Then
$(\S,\betas,\gammas)$ defines a sutured manifold diffeomorphic to
\[
M(R_+,k) =
\left(R_+ \times I, \partial R_+ \times I\right) \conn \left( \#^{k} (S^1
  \times S^2)\right)
\]
for some compact oriented surface~$R_+$.  There is a unique
$\SpinC$ structure $\spinc_0$ on $M(R_+,k)$ such that $c_1(\spinc_0) = 0
\in H^2(M(R_+,k);\ZZ)$, and which can be represented by a vector field
that is vertical on the summand $(R_+ \times I, \partial R_+ \times
I)$.  By the connected sum formula for sutured manifolds of
Juh\'asz~\cite[Proposition~9.15]{Juhasz06:Sutured},
\[
\HF^\circ(M(R_+,k), \spinc_0) \cong \Lambda^* H_1(S^1 \times S^2; \FF_2)
\]
as relatively $\ZZ$-graded groups. Here, we do not use naturality,
only that Heegaard Floer homology is well-defined up to isomorphism
in each $\SpinC$ structure and relative homological grading, as shown
by Ozsv\'ath and Szab\'o~\cite{OS04:HolomorphicDisks}.
Hence, in the ``top'' non-zero homological grading, the group
\[
\HF^\circ(\S,\betas,\gammas,[\spinc_0]) \cong \HF^\circ(M(R_+,k),\spinc_0)
\]
is isomorphic to $\FF_2$; we denote its generator by $\Theta_{\b,\g}$.
Since $[\spinc_0]$ is independent of the concrete manifold representing
$M(R_+,k)$, we see that $\Theta_{\b,\g}$ is a well-defined element of
$\HF^\circ(\S,\betas,\gammas)$.
\end{proof}

\begin{definition} \label{def:psi}
  Let $(\S,\alphas,\betas,\gammas)$ be an admissible triple diagram.
  If $\betas \sim \gammas$, then we write
  $\Psi^\alphas_{\betas \to \gammas}$ for the map
  \[
  F_{\a,\b,\g}(- \otimes \Theta_{\b,\g}) \colon
  \HF^\circ(\S,\alphas,\betas) \to \HF^\circ(\S,\alphas,\gammas).
  \]
  Similarly, if $\alphas \sim \betas$, then let
  \[
  \Psi^{\alphas \to \betas}_\gammas(-) = F_{\b,\a,\g}(\Theta_{\b,\a}
  \otimes -) \colon \HF^\circ(\S,\alphas,\gammas) \to
  \HF^\circ(\S,\betas,\gammas).
  \]
\end{definition}

\begin{lemma} \label{lem:grading}
Let $(\S,\alphas,\betas,\gammas)$ be an admissible triple diagram such that $\betas \sim \gammas$.
Then there is an identification $\SpinC(\S,\alphas,\betas) \cong \SpinC(\S,\alphas,\gammas)$
such that, for every $\spinc \in \SpinC(\S,\alphas,\betas)$, the map
$\Psi^\alphas_{\betas \to \gammas}$ maps $\HF^\circ(\S,\alphas,\betas,\spinc)$ to
$\HF^\circ(\S,\alphas,\gammas,\spinc)$ and preserves the relative $\ZZ_\mathfrak{d}(\spinc)$-grading,
where $\mathfrak{d(\spinc)}$ is the divisibility of $c_1(\spinc) \in H^2(M)$.
\end{lemma}

\begin{proof}
If $(\S,\alphas,\betas)$ is a diagram of~$(M,\g)$, then $\betas \sim \gammas$
implies that $(\S,\alphas,\gammas)$ is also a diagram of $(M,\g)$. Hence,
for $\spinc \in \SpinC(M,\g)$, we identify its equivalence class $[\spinc] \in \SpinC(\S,\alphas,\betas)$
with $[\spinc] \in \SpinC(\S,\alphas,\gammas)$.

Suppose that the index zero homotopy class $\psi \in \pi_2(\x,\theta,\y)$ contributes
to~$\Psi^\alphas_{\betas \to \gammas}$, where $\x \in \Torus_\a \cap \Torus_\b$ is such that $\spinc(\x) = \spinc$,
while $\y \in \Torus_\a \cap \Torus_\g$, and $\theta \in \Torus_\b \cap \Torus_\g$
represents~$\Theta_{\b,\g}$.
By definition, $\spinc(\theta) = \spinc_0$.
There is a cobordism of sutured manifolds $\mathcal{W} = \mathcal{W}_{\a,\b,\g}$
from $M_{\a,\b} \sqcup M_{\b,\g}$ to $M_{\a,\g}$ associated to the triple diagram $(\S,\alphas,\betas,\gammas)$
(where $M_{i,j}$ is the sutured manifold defined by $(\S,i,j)$ for every $i$, $j \in \{\,\alphas, \betas, \gammas\,\}$),
together with a set of $\SpinC$ structures $\underline{\SpinC}(\mathcal{W})$; see~\cite[p.~961]{Juhasz16:Cobordism}.
By \cite[Proposition~5.5]{Juhasz16:Cobordism}, there is a map
\[
\underline{\spinc} \colon \pi_2(\x,\theta,\y) \to \underline{\SpinC}(\mathcal{W})
\]
such that $\underline{\spinc}(\psi)|_{M_{\a,\b}} = \spinc$,
while $\underline{\spinc}(\psi)|_{M_{\b,\g}} = \spinc_0$,
and $\underline{\spinc}(\psi)|_{M_{\a,\g}} = \spinc(\y)$.
Since $\betas \sim \gammas$, we can glue a cobordism from $\emptyset$ to $M_{\b,\g}$
to $\mathcal{W}$ and obtain a cobordism $\mathcal{W}'$ from $M_{\a,\b}$ to $M_{\a,\g}$
diffeomorphic to the product $M_{\a,\b} \times I$.
Since $\underline{\spinc}(\psi)|_{M_{\b,\g}} = \spinc_0$, it follows that
$\underline{\spinc}(\psi)|_{M_{\b,\g}}$ extends to a $\SpinC$ structure $\mathfrak{t}'$ on $\mathcal{W}'$.
As $\mathfrak{t}'|_{M_{\a,\b}} = \mathfrak{s}$, we have $\mathfrak{t}'|_{M_{\a,\g}} = \spinc$. Consequently,
$\spinc(\y) = \underline{\spinc}(\psi)|_{M_{\a,\g}} = \spinc$, hence
$\Psi^\alphas_{\betas \to \gammas}$ maps $\HF^\circ(\S,\alphas,\betas,\spinc)$ to
$\HF^\circ(\S,\alphas,\gammas,\spinc)$.

It also follows from the above argument that there is a unique $\SpinC$ structure
$\mathfrak{t}$ on~$\mathcal{W}$ that satisfies $\mathfrak{t}|_{M_{\a,\b}} = \spinc$
and $\mathfrak{t}|_{M_{\b,\g}} = \spinc_0$.
Let $\psi' \in \pi_2(\x',\theta',\y')$ be another index zero homotopy class that contributes
to~$\Psi^\alphas_{\betas \to \gammas}$, where $\x' \in \Torus_\a \cap \Torus_\b$ is such that $\spinc(\x') = \spinc$,
while $\y' \in \Torus_\a \cap \Torus_\g$, and $\theta' \in \Torus_\b \cap \Torus_\g$
represents~$\Theta_{\b,\g}$.
Choose a domain~$D_x$ from $\x$ to $\x'$, a domain~$D_y$ from $\y$ to $\y'$, and a domain~$D_\theta$
from $\theta$ to $\theta'$. We write $D$ for the domain of $\psi$ and $D'$ for the domain of $\psi'$.
Then
\[
D_1 := D + D_x + D_\theta - D_y \in \pi_2(\x',\theta',\y'),
\]
and $\underline{\spinc}(D_1) = \underline{\spinc}(D')$.
Hence, by \cite[Proposition~5.9]{Juhasz16:Cobordism}, we can write the domain $D_1 - D'$
as a sum $P_{\a,\b} + P_{\b,\g} + P_{\a,\g}$ for a periodic domain~$P_{\a,\b}$ in
$(\S,\alphas,\betas)$, a periodic domain~$P_{\b,\g}$ in
$(\S,\betas,\gammas)$, and a periodic domain~$P_{\a,\g}$ in
$(\S,\alphas,\gammas)$. Then, on one hand,
\[
\mu(D_1 - D') = \mu(D) + \mu(D_x) + \mu(D_\theta) - \mu(D_y) - \mu(D') = \mu(D_x) - \mu(D_y)
\]
since $\mu(D) = \mu(D') = 0$ as they contribute to the map~$\Psi^\alphas_{\betas \to \gammas}$,
and $\mu(D_\theta) = 0$ since $\D_\theta$ connects $\theta$ and $\theta'$ that have the same
homological grading as they are homologous.
On the other hand,
\[
\mu(D_1 - D') = \mu(P_{\a,\b}) + \mu(P_{\b,\g}) + \mu(P_{\a,\g}).
\]
As $\spinc(\theta') = \spinc_0$ and $c_1(\spinc_0) = 0$, we have
\[
\mu(P_{\b,\g}) = \langle\, c_1(\spinc(\theta')), H(P_{\b,\g}) \,\rangle = 0.
\]
Similarly, $\mu(P_{\a,\b}) = \langle\, c_1(\spinc), H(P_{\a,\b}) \,\rangle$
and $\mu(P_{\a,\g}) = \langle\, c_1(\spinc), H(P_{\a,\g}) \,\rangle$
are divisible by~$\mathfrak{d}(\spinc)$. It follows that
$\mathfrak{d}(\spinc)$ divides $\mu(D_x) - \mu(D_y) = \text{gr}(\x,\x') - \text{gr}(\y,\y')$,
hence $\Psi^\alphas_{\betas \to \gammas}$ preserves the relative $\ZZ_\mathfrak{d}(\spinc)$-grading.
\end{proof}

Before proceeding, we state two key lemmas that will be
used multiple times. Recall that
a \emph{Hamiltonian isotopy} of a symplectic manifold $(\S,\omega)$
is a 1-parameter family of diffeomorphisms of~$\S$ obtained by integrating a
time-dependent Hamiltonian vector field; see Ozsv\'ath and Szab\'o~\cite[Section~7.3]{OS04:HolomorphicDisks}.

\begin{lemma} \label{lem:adm-multi} For every $i \in \{\, 1,\dots,k \,\}$, let
  \[
  (\S,\etas_0^i,\dots,\etas_{n-1}^i,\etas_n)
  \]
  be a sutured multi-diagram such that $(\S,\etas_0^i,\dots,\etas_{n-1}^i)$ is admissible.
  Then there is a Hamiltonian isotopic translate $\etas_n'$ of
  $\etas_n$ such that $(\S,\etas_0^i,\dots,\etas_{n-1}^i,\etas_n')$ is
  admissible for every $i \in \{\, 1,\dots, k \,\}$.
\end{lemma}

\begin{proof}
  The case $i=1$ was shown by Grigsby and Wehrli~ \cite[proof of Lemma
  3.13]{GW}. We proceed the same way, and isotope $\etas_n$ to $\etas_n'$ using
  finger moves along oriented arcs representing a basis of
  $H_1(\S,\partial \S)$ and their parallel opposites. Since this
  isotopy and hence $\etas_n'$ is independent of~$i$,
  the diagram $(\S,\etas_0^i,\dots,\etas_{n-1}^i,\etas_n')$ is
  admissible for every $i \in \{1,\dots,k\}$.
  Note that the finger moves of $\etas_n$ can be
  achieved by a Hamiltonian isotopy.
\end{proof}

\begin{lemma} \label{lem:triangles-compose} Suppose that the quadruple
  diagram $(\S,\alphas,\betas_1,\betas_2,\betas_3)$ is admissible,
  $\betas_1 \sim \betas_2 \sim \betas_3$, and
  $\Psi^{\betas_1}_{\betas_2 \to \betas_3}$ is an isomorphism.
  Then
  \[
  \Psi^\alphas_{\betas_1 \to \betas_3} = \Psi^\alphas_{\betas_2 \to
    \betas_3} \circ \Psi^\alphas_{\betas_1 \to \betas_2}.
  \]
\end{lemma}

\begin{proof}
  Pick an element $x \in \HF^\circ(\S,\alphas,\betas_1)$.  Since
  $(\S,\alphas,\betas_1,\betas_2,\betas_3)$ is admissible, we can use
  the associativity of the triangle maps, which was proved by
  Ozsv\'ath and Szab\'o~ \cite[Theorem 8.16]{OS04:HolomorphicDisks},
  to conclude that
  \begin{align*}
    \Psi^\alphas_{\betas_2 \to \betas_3} \circ \Psi^\alphas_{\betas_1
      \to \betas_2}(x)
    &= F_{\a,\b_2,\b_3}\left(F_{\a,\b_1,\b_2}(x \otimes \Theta_{\b_1,\b_2}) \otimes \Theta_{\b_2,\b_3}\right) \\
    &= F_{\a,\b_1,\b_3}\left(x \otimes F_{\b_1,\b_2,\b_3}(\Theta_{\b_1,\b_2} \otimes \Theta_{\b_2,\b_3})\right) \\
    &= F_{\a,\b_1,\b_3}\left(x \otimes \Psi^{\betas_1}_{\betas_2 \to
        \betas_3}(\Theta_{\b_1,\b_2})\right).
  \end{align*}
  So it suffices to show that $\Psi^{\betas_1}_{\betas_2 \to
    \betas_3}(\Theta_{\b_1,\b_2}) = \Theta_{\b_1,\b_3}$.  We assumed
  that $\Psi^{\betas_1}_{\betas_2 \to \betas_3}$ is an isomorphism.
  Hence, by Lemma~\ref{lem:grading}, it induces an isomorphism between the top groups
  $\HF^\circ_{\text{top}}(\S,\betas_1,\betas_2,\spinc_0) = \FF_2\langle
  \Theta_{\b_1,\b_2} \rangle$ and
  $\HF^\circ_{\text{top}}(\S,\betas_1,\betas_3, \spinc_0) = \FF_2
  \langle \Theta_{\b_1,\b_3} \rangle$, where $\spinc_0$ is the torsion
  $\SpinC$ struc\-ture, and has to map the generator
  $\Theta_{\b_1,\b_2}$ to the generator~$\Theta_{\b_1,\b_3}$.
\end{proof}

\begin{lemma} \label{lem:cont-triangle} Let $(\Sigma,\alphas,\betas)$
  and $(\Sigma,\alphas,\betas')$ be two admissible diagrams, let
  $\omega$ be a symplectic form on $\S$, and suppose we are given a
  Hamiltonian isotopy $I$ from~$\betas$ to~$\betas'$.  Then the
  isotopy $I$ induces an isomorphism
  \[
  \Gamma^\alphas_{\betas \to \betas'}\colon \HF^\circ(\Sigma,\alphas,
  \betas) \to \HF^\circ(\Sigma,\alphas,\betas').
  \]
  These isomorphisms compose under juxtaposition of isotopies, and we have
  $\Gamma^\alphas_{\betas \to \betas} = \text{Id}_{\HF^\circ(\S,\alphas,\betas)}$.  If,
  moreover, the triple $(\Sigma,\alphas,\betas,\betas')$ is
  admissible, then
  \begin{equation} \label{eqn:gamma-psi} \Gamma^\alphas_{\betas \to
      \betas'} = \Psi^\alphas_{\betas \to \betas'},
  \end{equation}
  and in particular it is independent of the isotopy $I$ (i.e., it
  depends only on the endpoints).
  Analogous results hold for admissible diagrams $(\S,\alphas,\betas)$ and $(\S,\alphas',\betas)$,
  and isomorphisms $\Gamma^{\alphas \to \alphas'}_{\betas} \colon \HF^\circ(\Sigma,\alphas,
  \betas) \to \HF^\circ(\Sigma,\alphas',\betas)$.
\end{lemma}

\begin{proof}
  The continuation maps~$\Gamma^\alphas_{\betas \to \betas'}$ and~$\Gamma^{\alphas \to \alphas'}_{\betas}$
  were defined by Ozsv\'ath and Sza\-b\'o~\cite[Section~7.3]{OS04:HolomorphicDisks}, and they showed they are
  natural under juxtaposition in~\cite[Lemma~2.12]{OS06:HolDiskFour}.
  The result $\Gamma^\alphas_{\betas \to \betas} = \text{Id}_{\HF^\circ(\S,\alphas,\betas)}$
  follows immediately from the definition of~$\Gamma$ on \cite[p.~1087]{OS04:HolomorphicDisks},
  as we count Maslov index zero disks in $\text{Sym}^g(\S)$ with half of the boundary on~$\Torus_\a$ and
  the other half on the now constant torus $\Torus_\b$, which are exactly the constant disks whith image
  being some $\x \in \Torus_\a \cap \Torus_\b$.

  Equation~\eqref{eqn:gamma-psi} follows from commutativity of the continuation and triangle
  maps. Indeed, if $(\S,\alphas,\betas,\gammas)$ is an
  admissible triple, and $\betas'$ is a
  Hamiltonian translate of $\betas$ such that
  $(\S,\alphas,\betas',\gammas)$ is also admissible, then the last commutative diagram
  in~\cite[Theorem~2.3]{OS06:HolDiskFour} is
  \[
  \xymatrix{\HF^\circ(\S,\alphas,\betas) \otimes \HF^\circ(\S,\betas,\gammas)  \ar[r]^-{F_{\a,\b,\g}}
  \ar[d]_{\Gamma^\alphas_{\betas \to \betas'} \otimes \Gamma^{\betas \to \betas'}_{\gammas}} &
    \HF^\circ(\S,\alphas,\gammas) \ar[d]^{\text{Id}} \\
    \HF^\circ(\S,\alphas,\betas') \otimes \HF^\circ(\S,\betas',\gammas) \ar[r]^-{F_{\a,\b',\g}} &
    \HF^\circ(\S,\alphas,\gammas);}
  \]
  see also~\cite[Proposition~8.14]{OS04:HolomorphicDisks}.
  If $\betas \sim \gammas$, then
  \[
  \Gamma^{\betas \to \betas'}_{\gamma} \colon \HF^\circ(\S,\betas,\gammas,\spinc_0)
  \to \HF^\circ(\S,\betas',\gammas,\spinc_0)
  \]
  is an isomorphism preserving the relative $\ZZ$-gradings, hence
  $\Gamma^{\betas \to \betas'}_{\gamma}(\Theta_{\b,\g}) = \Theta_{\b',\g}$. Together with
  Definition~\ref{def:psi}, we obtain that the diagram
  \[
  \xymatrix{\HF^\circ(\S,\alphas,\betas) \ar[r]^{\Psi^\alphas_{\betas
        \to \gammas}} \ar[d]_{\Gamma^\alphas_{\betas \to \betas'}} &
    \HF^\circ(\S,\alphas,\gammas) \ar[d]^{\text{Id}} \\
    \HF^\circ(\S,\alphas,\betas') \ar[r]^{\Psi^\alphas_{\betas' \to
        \gammas}} & \HF^\circ(\S,\alphas,\gammas)}
  \]
  is commutative if $\betas \sim \gammas$.

  A priori, the maps $\Psi^\alphas_{\betas \to \gammas}$ and
  $\Psi^\alphas_{\betas' \to \gammas}$ might not be isomorphisms; we
  choose $\gammas$ such that they are, as follows: Let $\gammas$ be
  a sufficiently small Hamiltonian translate of~$\betas'$ such that
  each component of $\gammas$ intersects the corresponding component
  of $\betas'$ transversely in two points.
  Since the triple $(\S,\alphas,\betas,\betas')$ is admissible, we can
  choose $\gammas$ such that the quadruple
  $(\S,\alphas,\betas,\betas',\gammas)$ is also admissible.
  Indeed, if $\mathcal{P}$ is a quadruply-periodic domain in $(\S,\alphas,\betas,\betas',\gammas)$
  and $\mathcal{D}_i$ is the doubly-periodic domain traced by the isotopy from $\b_i' \in \betas'$ to
  $\g_i \in \gammas$ for $i \in \{1,\dots,d\}$ (in particular, $\partial \mathcal{D}_i = \b_i' - \g_i$),
  then
  \[
  \mathcal{P}' = \mathcal{P} + a_1\mathcal{D}_1 + \dots + a_d \mathcal{D}_d
  \]
  is a triply-periodic domain in $(\S,\alphas,\betas,\betas')$, where $a_i$
  is the multiplicity of $\g_i$ in $\partial \mathcal{P}$. Since $(\S,\alphas,\betas,\betas')$
  is admissible, $\mathcal{P}'$ has both positive and negative multiplicities,
  and hence so does $\mathcal{P}$ if $\gammas$ is so close to~$\betas'$ that none
  of the components of $\S \setminus (\alphas \cup \betas \cup \betas')$ is contained
  in a bigon between~$\betas'$ and~$\gammas$.
  In particular, both $(\Sigma,\alphas,\betas,\gammas)$ and
  $(\Sigma,\alphas,\betas',\gammas)$ are admissible, satisfying the
  conditions for the above diagram to be commutative.  Since
  $\gammas$ is close to $\betas'$, a result of Ozsv\'ath and
  Szab\'o~\cite[Proposition~9.8]{OS04:HolomorphicDisks} implies that
  the maps
  \[
  \begin{split}
  \Psi^\alphas_{\betas' \to \gammas} &\colon \HF^\circ(\S,\alphas,\betas)
  \to \HF^\circ(\S,\alphas,\gammas) \text{ and} \\
  \Psi^\betas_{\betas' \to \gammas} &\colon \HF^\circ(\S,\betas,\betas')
  \to \HF^\circ(\S,\betas,\gammas)
  \end{split}
  \]
  are isomorphisms. Then the commutativity of the above rectangle
  gives that
  \[
  \Gamma^\alphas_{\betas \to \betas'} = \left(\Psi^\alphas_{\betas'
      \to \gammas}\right)^{-1} \circ \Psi^\alphas_{\betas \to
    \gammas}.
  \]
  Since the quadruple $(\S,\alphas,\betas,\betas',\gammas)$ is
  admissible and $\Psi^\betas_{\betas' \to \gammas}$
  is an isomorphism, we can apply Lemma~\ref{lem:triangles-compose}
  to conclude that
  \[
  \left(\Psi^\alphas_{\betas' \to \gammas}\right)^{-1} \circ
  \Psi^\alphas_{\betas \to \gammas} = \Psi^\alphas_{\betas \to
    \betas'}.
  \]

  An alternate elegant argument can be given using monogons; see the
  work of Lipshitz~\cite[Proposition~11.4]{Lipshitz06:CylindricalHF}.
\end{proof}

\begin{remark}
  Continuation maps in general symplectic manifolds do
  depend on the homotopy class of the isotopy, and hence cannot be
  written in terms of triangle maps; see Seidel~\cite{Seidel97:pi1}.
  The above lemma is specific to Heegaard Floer homology for two reasons: First,
  in Lemma~\ref{lem:triangles-compose}, we used that if $\betas \sim \gammas$,
  then $\HF^\circ(\S,\betas,\gammas)$ is relatively $\ZZ$-graded as opposed to
  just being $\ZZ_2$-graded, and hence the top-graded generator $\Theta_{\b,\g}$ is well-defined up to
  grading-preserving automorphism (cf.~Lemma~\ref{lem:grading}). Secondly,
  the admissibility conditions allow us to work over a polynomial ring instead
  of Novikov ring coefficients (cf.~\cite[Section~10.0.1]{OS04:HolomorphicDisks}).
\end{remark}

Another way to view Lemma~\ref{lem:cont-triangle} is that the triangle
map $\Psi^\alphas_{\betas \to \betas'}$ is an isomorphism whenever the
triple $(\S,\alphas,\betas,\betas')$ is admissible and $\betas$ and
$\betas'$ are Hamiltonian isotopic. Our next goal is to relax
the second condition and show that $\Psi^\alphas_{\betas \to \betas'}$
is also an isomorphism whenever $\betas \sim
\betas'$.

\begin{proposition}\label{prop:Compatibility1}
\hspace{2em}
  \begin{enumerate}
  \item \label{item:psi-iso} Suppose that
    $(\S,\alphas,\betas,\betas')$ is an admissible triple and we have
    $\betas \sim \betas'$.  Then the map
    \[
    \Psi^{\alphas}_{\betas \to \betas'} \colon
    \HF^\circ(\S,\alphas,\betas)\to \HF^\circ(\S,\alphas,\betas')
    \]
    is an isomorphism.
  \item \label{item:psi-compose} These isomorphisms are compatible in
    the sense that if the triple diagrams
    $(\S,\alphas,\betas,\betas')$, $(\S,\alphas,\betas',\betas'')$,
    and $(\S,\alphas,\betas,\betas'')$ are admissible, then
    \begin{equation*}
      \Psi^{\alphas}_{\betas'\to\betas''}\circ \Psi^{\alphas}_{\betas\to\betas'} =
      \Psi^{\alphas}_{\betas\to\betas''}.
    \end{equation*}
  \item \label{item:psi-alpha-beta} Similarly, if
    $(\S,\alphas',\alphas,\betas)$ is admissible and $\alphas \sim
    \alphas'$, then the map
    \[
    \Psi^{\alphas \to \alphas'}_{\betas}\colon
    \HF^\circ(\S,\alphas,\betas)\to \HF^\circ(\S,\alphas',\betas)
    \]
    is an isomorphism, and satisfies the analogue of
    \eqref{item:psi-compose}. Finally, we have
    \begin{equation*}
      \Psi^{\alphas\to\alphas'}_{\betas'}\circ \Psi^{\alphas}_{\betas\to\betas'}=
      \Psi^{\alphas'}_{\betas\to\betas'}\circ \Psi^{\alphas\to\alphas'}_{\betas},
    \end{equation*}
    assuming all four triple diagrams involved are admissible.
  \end{enumerate}
\end{proposition}

\begin{proof}
  We first show~\eqref{item:psi-iso}. By Lemma~\ref{lem:handleslide},
  we can get from $\betas$ to $\betas'$ by a sequence of isotopies and
  handleslides; let $h(\betas,\betas')$ be the minimal number of
  handleslides required in such a sequence. We prove the claim by
  induction on $h(\betas,\betas')$.

  Suppose that $h(\betas,\betas') = 0$. Since the triple
  $(\S,\alphas,\betas,\betas')$ is admissible, the pair
  $(\S,\betas,\betas')$ is also admissible. According to Ozsv\'ath and
  Szab\'o~\cite[Lemma 4.12]{OS04:HolomorphicDisks}, there exists a
  volume form $\omega$ on $\S$ for which every periodic domain has
  signed area equal to zero. Let $\betas = \{\,\b_1,\dots,\b_k\,\}$
  and $\betas' = \{\,\b_1',\dots,\b_k'\,\}$, labeled such that $\b_i$ and~$\b_i'$
  are isotopic for every $i \in \{1,\dots,k\}$.
  Then the cycle $\b_i - \b_i'$ is the boundary of
  a 2-chain $\mathcal{P}_i$, which can be viewed as a periodic
  domain. Since $\mathcal{P}_i$ has signed area zero with respect to $\omega$,
  it follows that $\b_i$ and $\b_i'$ are Hamiltonian isotopic for every $i \in \{1,\dots,k\}$.
  Let~$d_1$ be a diffeomorphism Hamiltonian isotopic to the identity
  such that $d_1(\b_1) = \b_1'$. Then $d_1(\b_2)$ and $\b_2'$ are
  also Hamiltonian isotopic, and since they are both disjoint and homologically
  linearly independent from $d_1(\b_1) = \b_1'$,
  they can be connected by a Hamiltonian isotopy that is the identity in a neighborhood
  of $\b_1'$. Let $d_2$ be a diffeomorphism Hamiltonian isotopic to the identity
  that fixes a neighborhood of $\b_1'$, and such that $d_2 \circ d_1(\b_2) = \b_2'$.
  Continuing this process, we recursively obtain diffeomorphisms $d_1, \dots, d_k$ such that
  $d_i$ is Hamiltonian isotopic to the identity, and fixes a neighborhood of
  $\b_1', \dots, \b_{i-1}'$ for every $i \in \{1, \dots, k\}$.
  Hence $d_k \circ \dots \circ d_1(\betas) =  \betas'$, and
  $d_k \circ \dots \circ d_1$ is Hamiltonian isotopic to the identity.
  By Lemma~\ref{lem:cont-triangle}, the triangle map
  $\Psi^\alphas_{\betas \to \betas'}$ is an isomorphism for any
  complex structure compatible with $\omega$. However, the triangle
  maps commute with the maps $\Phi_{J_s \to J_s'}$, hence it is an
  isomorphism for any complex structure and perturbation.

  Suppose we know the statement for $h(\betas,\betas') < n$ for some
  $n > 0$.  If $h(\betas,\betas') = n$, then we can choose an
  attaching set $\gammas$ such that $h(\betas,\gammas) = 1$ and
  $h(\gammas,\betas') = n-1$; furthermore, $\gammas$ is obtained from
  $\betas$ by a model handleslide as described by Ozsv\'ath and
  Szab\'o~\cite[Section~9]{OS04:HolomorphicDisks}. Then, according to
  Ozsv\'ath and Szab\'o~\cite[Theorem~9.5]{OS04:HolomorphicDisks}, the
  triple $(\S,\alphas,\betas,\gammas)$ is admissible and the map
  $\Psi^\alphas_{\betas \to \gammas}$ is an isomorphism.  The triple
  diagram $(\S,\alphas,\gammas,\betas')$ might not be admissible, but,
  by Lemma~\ref{lem:adm-multi}, there is a Hamiltonian
  translate $\gammas'$ of $\gammas$ for which both
  $(\S,\alphas,\betas,\betas',\gammas')$ and
  $(\S,\alphas,\betas,\gammas,\gammas')$ are admissible.  Then
  consider the following diagram:
  \[
  \xymatrix{ \HF^\circ(\S,\alphas,\betas) \ar[r]^{\Psi^\alphas_{\betas
        \to \betas'}} \ar[d]_{\Psi^\alphas_{\betas \to \gammas}}
    \ar[rd]^{\Psi^\alphas_{\betas \to \gammas'}}
    & \HF^\circ(\S,\alphas,\betas')  \\
    \HF^\circ(\S,\alphas,\gammas) \ar[r]_{\Psi^\alphas_{\gammas \to
        \gammas'}} & \HF^\circ(\S,\alphas,\gammas')
    \ar[u]_{\Psi^\alphas_{\gammas' \to \betas'}}.  }
  \]
  We will prove that it is commutative. Since
  $(\S,\alphas,\gammas,\gammas')$ is admissible and $\gammas'$ is a
  Hamiltonian translate of $\gammas$,
  Lemma~\ref{lem:cont-triangle} implies that the map
  $\Psi^\alphas_{\gammas \to \gammas'} = \Gamma^\alphas_{\gammas \to
    \gammas'}$ is an isomorphism.
  Similarly, $\Psi^{\betas}_{\gammas \to \gammas'}$ is also an isomorphism,
  and as the tuple
  $(\S,\alphas,\betas,\gammas,\gammas')$ is admissible, we can apply
  Lemma~\ref{lem:triangles-compose} to conclude that
  \[
  \Psi^\alphas_{\gammas \to \gammas'} \circ \Psi^\alphas_{\betas \to
    \gammas} = \Psi^\alphas_{\betas \to \gammas'}.
  \]
  We have seen that both $\Psi^{\alphas}_{\gammas \to \gammas'}$ and
  $\Psi^\alphas_{\betas \to \gammas}$ are isomorphisms, so
  $\Psi^\alphas_{\betas \to \gammas'}$ is an isomorphism.  Since
  $h(\gammas',\betas') = n-1$, the map $\Psi^\alphas_{\gammas' \to
    \betas'}$ is an isomorphism by the induction hypothesis. So we are
  done if we show that
  \[
  \Psi^\alphas_{\betas \to \betas'} = \Psi^\alphas_{\gammas' \to
    \betas'} \circ \Psi^\alphas_{\betas \to \gammas'}.
  \]
  This also follows from Lemma~\ref{lem:triangles-compose}. Indeed,
  the quadruple diagram $(\S,\alphas,\betas,\betas',\gammas')$ is
  admissible; furthermore, the map $\Psi^\betas_{\gammas' \to
    \betas'}$ is an isomorphism by the induction hypothesis (the
  diagram $(\S,\betas,\gammas',\betas')$ is admissible and
  $h(\gammas',\betas') = n-1$).  It follows that $\Psi^\alphas_{\betas
    \to \betas'}$ is an isomorphism, concluding the proof of
  \eqref{item:psi-iso}.

  A useful consequence of \eqref{item:psi-iso} is that, in
  Lemma~\ref{lem:triangles-compose}, the condition that
  $\Psi^{\betas_1}_{\betas_2 \to \betas_3}$ is an isomorphism
  automatically follows from the others (we only need that the triple
  $(\S,\betas_1,\betas_2,\betas_3)$ is admissible and $\betas_2 \sim
  \betas_3$). Armed with this fact, we proceed to the proof of
  \eqref{item:psi-compose}. By Lemma~\ref{lem:adm-multi}, there is a
  Hamiltonian translate $\betas_1'$ of $\betas'$ such that the
  quadruple diagrams $(\S,\alphas,\betas,\betas',\betas_1')$,
  $(\S,\alphas,\betas',\betas'',\betas_1')$, and
  $(\S,\alphas,\betas,\betas'',\betas_1')$ are all admissible. Then
  consider the following diagram:
  \[
  \xymatrix{
    & \HF^\circ(\S,\alphas,\betas') \ar[dd]^{\Psi^\alphas_{\betas' \to \betas_1'}} \ar[rddd]^{\Psi^\alphas_{\betas' \to \betas''}} \\
    \\
    & \HF^\circ(\S,\alphas,\betas_1') \ar[dr]^{\Psi^\alphas_{\betas_1' \to \betas''}} \\
    \HF^\circ(\S,\alphas,\betas) \ar[ru]^{\Psi^\alphas_{\betas \to
        \betas_1'}} \ar[ruuu]^{\Psi^\alphas_{\betas \to \betas'}}
    \ar[rr]^{\Psi^\alphas_{\betas \to \betas''}} & &
    \HF^\circ(\S,\alphas,\betas'').  }
  \]
  Commutativity of the three small triangles follows from the above
  improved version of Lemma~\ref{lem:triangles-compose}. Hence the
  large triangle is also commutative; i.e.,
  \[
  \Psi^\alphas_{\betas \to \betas''} = \Psi^\alphas_{\betas' \to
    \betas''} \circ \Psi^\alphas_{\betas \to \betas'},
  \]
  which concludes the proof of \eqref{item:psi-compose}.

  We finally prove \eqref{item:psi-alpha-beta}. First, we verify this
  when $(\S,\alphas,\alphas',\betas,\betas')$ is admissible. Pick an
  element $x \in \HF^\circ(\S,\alphas,\betas)$. Then, using the
  associativity of the triangle maps
  \cite[second part of Theorem~8.16]{OS04:HolomorphicDisks},
  \begin{align*}
    \Psi^{\alphas \to \alphas'}_{\betas'} \circ \Psi^\alphas_{\betas
      \to \betas'}(x) &= F_{\a',\a,\b'}(\Theta_{\a',\a} \otimes
    F_{\a,\b,\b'}(x \otimes \Theta_{\b,\b'}) ) \\ &=
    F_{\a',\b,\b'}(F_{\a',\a,\b}(\Theta_{\a',\a} \otimes x) \otimes
    \Theta_{\b,\b'}) \\ &= \Psi^{\alphas'}_{\betas \to \betas'} \circ
    \Psi^{\alphas \to \alphas'}_\betas(x).
  \end{align*}
  We now consider the general case. By Lemma~\ref{lem:adm-multi},
  there is an isotopic copy $\ol{\betas}$ of $\betas$ for which both
  $(\S,\alphas,\alphas',\betas,\ol{\betas})$ and
  $(\S,\alphas,\alphas',\betas',\ol{\betas})$ are admissible.  Then
  \begin{align*}
    \Psi^{\alphas \to \alphas'}_{\betas'} \circ \Psi^\alphas_{\betas
      \to \betas'} &=
    \Psi^{\alphas \to \alphas'}_{\betas'} \circ \Psi^\alphas_{\ol{\betas} \to \betas'} \circ \Psi^\alphas_{\betas \to \ol{\betas}} \\
    &= \Psi^{\alphas'}_{\ol{\betas} \to \betas'} \circ \Psi^{\alphas \to \alphas'}_{\ol{\betas}} \circ \Psi^\alphas_{\betas \to \ol{\betas}} \\
    &= \Psi^{\alphas'}_{\ol{\betas} \to \betas'}  \circ \Psi^{\alphas'}_{\betas \to \ol{\betas}} \circ \Psi^{\alphas \to \alphas'}_{\betas} \\
    &= \Psi^{\alphas'}_{\betas \to \betas'} \circ \Psi^{\alphas \to
      \alphas'}_{\betas}.
  \end{align*}
  Here, the first and fourth equalities follow from
  \eqref{item:psi-compose}, while the second and third follow from the
  previous special case, assuming the admissibility conditions.  This
  concludes the proof of \eqref{item:psi-alpha-beta}.
\end{proof}

\begin{definition}
Suppose that the quadruple diagram $(\S,\alphas,\alphas',\betas,\betas')$ is admissible,
$\alphas \sim \alphas'$, and $\betas \sim \betas'$. Then let
\[
\Psi^{\alphas \to \alphas'}_{\betas \to \betas'} =
\Psi^{\alphas \to \alphas'}_{\betas'} \circ \Psi^{\alphas}_{\betas \to \betas'} =
\Psi^{\alphas'}_{\betas \to \betas'} \circ \Psi^{\alphas \to \alphas'}_{\betas}.
\]
Note that the second equality holds by part~\eqref{item:psi-alpha-beta} of
Proposition~\ref{prop:Compatibility1}.
\end{definition}

\begin{lemma} \label{lem:psi-compose}
Suppose that the six-tuple $(\S,\alphas,\alphas',\alphas'',\betas,\betas',\betas'')$ is admissible.
Then
\[
\Psi^{\alphas' \to \alphas''}_{\betas' \to \betas''} \circ \Psi^{\alphas \to \alphas'}_{\betas \to \betas'}=
\Psi^{\alphas \to \alphas''}_{\betas \to \betas''}.
\]
\end{lemma}

\begin{proof}
By parts~\eqref{item:psi-compose} and~\eqref{item:psi-alpha-beta} or Proposition~\ref{prop:Compatibility1},
 \begin{align*}
    \Psi^{\alphas' \to \alphas''}_{\betas' \to \betas''} \circ
    \Psi^{\alphas \to \alphas'}_{\betas \to \betas'} &=
    \Psi^{\alphas''}_{\betas' \to \betas''} \circ \Psi^{\alphas' \to
      \alphas''}_{\betas'} \circ
    \Psi^{\alphas \to \alphas'}_{\betas'} \circ \Psi^{\alphas}_{\betas \to \betas'} \\
    &= \Psi^{\alphas''}_{\betas' \to \betas''} \circ \Psi^{\alphas \to \alphas''}_{\betas'} \circ \Psi^{\alphas}_{\betas \to \betas'} \\
    &= \Psi^{\alphas''}_{\betas' \to \betas''} \circ \Psi^{\alphas''}_{\betas \to \betas'} \circ \Psi^{\alphas \to \alphas''}_{\betas} \\
    &= \Psi^{\alphas''}_{\betas \to \betas''} \circ \Psi^{\alphas \to
      \alphas''}_{\betas} = \Psi^{\alphas \to \alphas''}_{\betas \to
      \betas''}.
  \end{align*}
\end{proof}

So triangle maps give canonical isomorphisms $\Psi^{\alphas \to \alphas'}_{\betas
  \to \betas'}$ between $\HF^\circ(\S,\alphas,\betas)$ and
$\HF^\circ(\S,\alphas',\betas')$ (in the sense of Definition~\ref{def:transitive-system}) whenever the quadruple
$(\S,\alphas,\alphas',\betas,\betas')$ is admissible,
$\alphas \sim \alphas'$, and $\betas \sim \betas'$.
But what do we do when the admissibility condition fails?  If the
triple $(\S,\alphas,\betas,\betas')$ is not admissible, then the
triangle count in $\Psi^{\alphas}_{\betas \to \betas'}$ might not be finite,
and even if it is, there are
simple examples where it does not give a natural isomorphism, even
though $\betas$ and $\betas'$ are isotopic. To overcome this obstacle,
we first apply a Hamiltonian isotopy to $\alphas$ and $\betas$ so that the
quadruple $(\alphas,\alphas',\betas,\betas')$ becomes admissible.
According to Lemma~\ref{lem:adm-multi}, this is always possible.

\begin{proposition} \label{prop:isotopy-map} Suppose that the diagrams
  $(\S,\alphas,\betas)$ and $(\S,\alphas',\betas')$ are both
  admissible, $\alphas \sim \alphas'$, and $\betas \sim \betas'$.  According to
  Lemma~\ref{lem:adm-multi}, there exist attaching sets $\ol{\alphas}$ and
  $\ol{\betas}$ isotopic to $\alphas$ and $\betas$, respectively, and such that the
  quadruples $(\S, \alphas, \ol{\alphas}, \betas, \ol{\betas})$ and
  $(\S, \ol{\alphas}, \alphas',\ol{\betas}, \betas')$ are both admissible.  Then the map
  \[
  \Psi^{\ol{\alphas} \to \alphas'}_{\ol{\betas} \to \betas'} \circ \Psi^{\alphas \to \ol{\alphas}}_{\betas
    \to \ol{\betas}} \colon \HF^\circ(\S,\alphas,\betas) \to
  \HF^\circ(\S,\alphas',\betas')
  \]
  is an isomorphism. Furthermore, it is independent of the choice of~$\ol{\alphas}$ an~$\ol{\betas}$;
  we denote it by $\Phi^{\alphas \to \alphas'}_{\betas \to \betas'}$.
  Finally, if $(\S,\alphas'',\betas'')$ is also admissible, $\alphas'' \sim \alphas$, and $\betas''
  \sim \betas$, then
  \begin{equation} \label{eqn:phi-compose} \Phi^{\alphas' \to \alphas''}_{\betas' \to
      \betas''} \circ \Phi^{\alphas \to \alphas'}_{\betas \to \betas'} =
    \Phi^{\alphas \to \alphas''}_{\betas \to \betas''}.
  \end{equation}
\end{proposition}

\begin{proof}
  The map $\Psi^{\ol{\alphas} \to \alphas'}_{\ol{\betas} \to \betas'} \circ
  \Psi^{\alphas \to \ol{\alphas}}_{\betas \to \ol{\betas}}$ is an isomorphism by
  part~\eqref{item:psi-iso} of Proposition~\ref{prop:Compatibility1}.
  We now show that it is independent of the choice of~$\ol{\alphas}$ and~$\ol{\betas}$.
  Let~$\ol{\alphas}_1$, $\ol{\betas}_1$ and $\ol{\alphas}_2$, $\ol{\betas}_2$ be two
  different choices. Using Lemma~\ref{lem:adm-multi}, we isotope~$\alphas$ and~$\betas$
  until we get attaching sets $\ol{\alphas}$ and $\ol{\betas}$ such that the
  six-tuples obtained by adding them to the quadruples $(\S,\alphas, \ol{\alphas}_1,\betas, \ol{\betas}_1)$,
  $(\S,\alphas, \ol{\alphas}_2,\betas, \ol{\betas}_2)$,
  $(\S,\ol{\alphas}_1, \alphas', \ol{\betas}_1, \betas')$, and
  $(\S, \ol{\alphas}_2, \alphas', \ol{\betas}_2, \betas')$ are all admissible.  Then
  we can consider the following diagram:
  \begin{equation*} 
    \xymatrix{
      &\HF^\circ\left(\S,\ol{\alphas}_1,\ol{\betas}_1\right) \ar[d]^{\Psi^{\ol{\alphas}_1 \to \ol{\alphas}}_{\ol{\betas}_1 \to \ol{\betas}}}
      \ar[rd]^{\Psi^{\ol{\alphas}_1 \to \alphas'}_{\ol{\betas}_1 \to \betas'}} \\
      \HF^\circ(\S,\alphas,\betas) \ar[ru]^{\Psi^{\alphas \to \ol{\alphas}_1}_{\betas \to \ol{\betas}_1}}
      \ar[r]^{\Psi^{\alphas \to \ol{\alphas}}_{\betas \to \ol{\betas}}} \ar[rd]_{\Psi^{\alphas \to \ol{\alphas}_2}_{\betas \to \ol{\betas}_2}} &
      \HF^\circ\left(\S,\ol{\alphas},\ol{\betas}\right) \ar[d]^{\Psi^{\ol{\alphas} \to \ol{\alphas}_2}_{\ol{\betas} \to \ol{\betas}_2}}
      \ar[r]^{\Psi^{\ol{\alphas} \to \alphas'}_{\ol{\betas} \to \betas'}} &
      \HF^\circ(\S,\alphas',\betas'). \\
      &\HF^\circ\left(\S,\ol{\alphas}_2,\ol{\betas}_2\right) \ar[ru]_{\Psi^{\ol{\alphas}_2 \to \alphas'}_{\ol{\betas}_2 \to \betas'}}
    }
  \end{equation*}
  Each of the four small triangles is commutative by
  part~\eqref{item:psi-compose} of
  Proposition~\ref{prop:Compatibility1}.  Hence, the outer square also
  commutes; i.e.,
  \[
  \Psi^{\ol{\alphas}_1 \to \alphas'}_{\ol{\betas}_1 \to \betas'} \circ \Psi^{\alphas \to \ol{\alphas}_1}_{\betas
    \to \ol{\betas}_1} = \Psi^{\ol{\alphas_2} \to \alphas'}_{\ol{\betas}_2 \to \betas'}
  \circ \Psi^{\alphas \to \ol{\alphas}_2}_{\betas \to \ol{\betas}_2},
  \]
  so the map for $\ol{\alphas}_1$, $\ol{\betas}_1$ is the same as the map for
  $\ol{\alphas}_2$, $\ol{\betas}_2$.

  Finally we show equation~\eqref{eqn:phi-compose}.  Using
  Lemma~\ref{lem:adm-multi}, pick isotopic copies $\ol{\alphas}$, $\ol{\betas}$,
  $\ol{\alphas}'$, $\ol{\betas}'$ of $\alphas$, $\betas$, $\alphas'$ and $\betas'$, respectively, such that
  the six-tuples obtained by adding these four attaching sets to the diagrams $(\S,\alphas,\betas)$,
  $(\S,\alphas',\betas')$, and $(\S,\alphas'',\betas'')$
  are all admissible.  Applying part~\eqref{item:psi-compose} of
  Proposition~\ref{prop:Compatibility1} to the left-hand side of
  equation~\eqref{eqn:phi-compose},
  \begin{align*}
    \Psi^{\ol{\alphas}' \to \alphas''}_{\ol{\betas}' \to \betas''} \circ
    \Psi^{\alphas' \to \ol{\alphas}'}_{\betas' \to \ol{\betas}'} \circ
    \Psi^{\ol{\alphas} \to \alphas'}_{\ol{\betas} \to \betas'} \circ \Psi^{\alphas \to \ol{\alphas}}_{\betas
      \to \ol{\betas}} &=
    \Psi^{\ol{\betas}' \to \alphas''}_{\ol{\betas}' \to \betas''} \circ
    \Psi^{\ol{\alphas} \to \ol{\alphas}'}_{\ol{\betas} \to \ol{\betas}'} \circ
    \Psi^{\alphas \to \ol{\alphas}}_{\betas \to \ol{\betas}} \\
    &= \Psi^{\ol{\alphas} \to \alphas''}_{\ol{\betas} \to \betas''} \circ
    \Psi^{\alphas \to \ol{\alphas}}_{\betas \to \ol{\betas}} =
    \Phi^{\alphas \to \alphas''}_{\betas \to \betas''},
  \end{align*}
  as required.
\end{proof}

\begin{definition}
Suppose that the diagrams $(\S,\alphas,\betas)$ and $(\S,\alphas,\betas')$ are both
admissible and $\betas \sim \betas'$. Then let
\[
\Phi^\alphas_{\betas \to \betas'} = \Phi^{\alphas \to \alphas}_{\betas \to \betas'}.
\]
Similarly, when we have admissible diagrams $(\S,\alphas,\betas)$ and $(\S,\alphas',\betas)$
such that $\alphas \sim \alphas'$, we write
\[
\Phi^{\alphas \to \alphas'}_\betas = \Phi^{\alphas \to \alphas'}_{\betas \to \betas}.
\]
\end{definition}

\begin{lemma} \label{lem:identity}
Suppose that the diagrams $(\S,\alphas,\betas)$ and $(\S,\alphas,\betas')$ 
are both admissible and $\betas \sim \betas'$.
Let $\ol{\betas}$ be an isotopic copy of $\betas$ such that the triples $(\S,\alphas,\betas,\ol{\betas})$
and $(\S,\alphas,\betas',\ol{\betas})$ are admissible. Then
\[
\Phi^\alphas_{\betas \to \betas'} = \Psi^\alphas_{\ol{\betas} \to \betas'} \circ 
\Psi^\alphas_{\betas \to \ol{\betas}}.
\]
An analogous statement holds for $\Phi^{\alphas \to \alphas'}_\betas$.
Finally,
  \begin{equation} \label{eqn:id} \Phi^{\alphas \to \alphas}_{\betas
      \to \betas} = \Phi^\alphas_{\betas \to \betas} = \Phi^{\alphas
      \to \alphas}_\betas = \text{Id}_{\HF^\circ(\S,\alphas,\betas)}.
  \end{equation}
\end{lemma}

\begin{proof}
  Let $\ol{\alphas}$ be a Hamiltonian translate of $\alphas$ such that
  the quadruples $(\S,\alphas,\ol{\alphas},\betas,\ol{\betas})$ and
  $(\S,\alphas,\ol{\alphas},\betas',\ol{\betas})$ are admissible.
  By Lemma~\ref{lem:cont-triangle} and the naturality of the continuation maps under juxtaposition,
  \[
  \Psi^{\ol{\alphas} \to \alphas}_{\ol{\betas}}  \circ \Psi^{\alphas \to \ol{\alphas}}_{\ol{\betas}} =
  \Gamma^{\ol{\alphas} \to \alphas}_{\ol{\betas}}  \circ \Gamma^{\alphas \to \ol{\alphas}}_{\ol{\betas}} =
  \text{Id}_{\HF^\circ(\S,\alphas,\ol{\betas})}.
  \]
It follows that
\begin{align*}
\Phi^\alphas_{\betas \to \betas'} &= \Phi^{\alphas \to \alphas}_{\betas \to \betas'} =
\Psi^{\ol{\alphas} \to \alphas}_{\ol{\betas} \to \betas'} \circ \Psi^{\alphas \to \ol{\alphas}}_{\betas \to \ol{\betas}}
= \Psi^{\alphas}_{\ol{\betas} \to \betas'} \circ \Psi^{\ol{\alphas} \to \alphas}_{\ol{\betas}}  \circ
\Psi^{\alphas \to \ol{\alphas}}_{\ol{\betas}} \circ
\Psi^{\alphas}_{\betas \to \ol{\betas}} = \\
&= \Psi^{\alphas}_{\ol{\betas} \to \betas'} \circ \Psi^{\alphas}_{\betas \to \ol{\betas}},
\end{align*}
as claimed. The statement for $\Phi^{\alphas \to \alphas'}_\betas$ follows similarly.

  We now prove the last statement regarding $\Phi^{\alphas \to \alphas}_{\betas \to \betas}$.
  Let $\ol{\betas}$ be a Hamiltonian translate of $\betas$ such
  that $(\S,\alphas,\betas,\ol{\betas})$ is admissible.  If we apply the first part with $\betas = \betas'$,
  we get that
  \[
  \Phi^\alphas_{\betas \to \betas} = \Psi^\alphas_{\ol{\betas} \to
    \betas} \circ \Psi^\alphas_{\betas \to \ol{\betas}}.
  \]
  Using Lemma~\ref{lem:cont-triangle}, the right-hand side is
  $\Gamma^\alphas_{\ol{\betas} \to \betas} \circ
  \Gamma^\alphas_{\betas \to \ol{\betas}}$.  By Lemma~\ref{lem:cont-triangle}, the naturality of the
  continuation maps under juxtaposition, this is
  $\Gamma^\alphas_{\betas \to \betas} = \text{Id}_{\HF^\circ(\S,\alphas,\betas)}$.
\end{proof}

\begin{corollary} \label{cor:id} Let $(\S,\alphas,\betas,\betas')$ be
  an admissible triple such that $\betas \sim \betas'$. Then
  \[
  \left(\Psi^\alphas_{\betas \to \betas'}\right)^{-1} =
  \Psi^\alphas_{\betas' \to \betas}.
  \]
  An analogous result holds for the maps $\Psi^{\alphas \to \alphas'}_\betas$.
\end{corollary}

\begin{proof}
By Lemma~\ref{lem:identity},
\[
\Psi^\alphas_{\betas' \to \betas} \circ \Psi^\alphas_{\betas \to \betas'} = \Phi^\alphas_{\betas \to \betas} = \text{Id}_{\HF^\circ(\S,\alphas,\betas)}.\qedhere
\]
\end{proof}

Let $(\S,A,B)$ be an isotopy diagram. Then we denote by $M_{(\S,A,B)}$
the set of admissible diagrams $(\S,\alphas,\betas)$ such that
$[\alphas] = A$ and $[\betas] = B$. This is non-empty by
Lemma~\ref{lem:adm-multi}.  It follows from
equations~\eqref{eqn:phi-compose} and~\eqref{eqn:id} that the
groups $\HF^\circ(\S,\alphas,\betas)$ for $(\S,\alphas,\betas) \in
M_{(\S,A,B)}$, together with the isomorphisms $\Phi^{\alphas \to
  \alphas'}_{\betas \to \betas'}$ form a transitive system of groups,
as in Definition~\ref{def:transitive-system}.

\begin{definition} \label{def:HF-diagram}
  Given an isotopy diagram $H$, let $\HF^\circ(H)$ be the colimit
  of the transitive system of groups $\HF^\circ(\S,\alphas,\betas)$ for
  $(\S,\alphas,\betas) \in M_H$ and $\Phi^{\alphas \to
    \alphas'}_{\betas \to \betas'}$. In other words,
  \[
  \HF^\circ(H) = \coprod_{(\S,\alphas,\betas) \in M_H}
  \HF^\circ(\S,\alphas,\betas)\Bigr/\mathord{\sim},
  \]
  where $x \in \HF^\circ(\S,\alphas,\betas)$ and $x' \in
  \HF^\circ(\S,\alphas',\betas')$ are equivalent if and only if $x' =
  \Phi^{\alphas \to \alphas'}_{\betas \to \betas'}(x)$.
\end{definition}

We would like to show that $\HF^\circ$ is a weak Heegaard invariant. To
this end, we need to define isomorphisms induced by $\a$-equivalences,
$\b$-equivalences, diffeomorphisms, and (de)stabilizations between
isotopy diagrams. We start with $\a$- and $\b$-equivalences.

\begin{lemma} \label{lem:phi-ab-welldef} Suppose that we are given
  admissible diagrams $(\S,\alphas_1,\betas_1)$,
  $(\S,\alphas_1,\betas_1')$, $(\S,\alphas_2,\betas_2)$, and
  $(\S,\alphas_2,\betas_2')$ such that $\alphas_1 \sim \alphas_2$ and
  $\betas_1 \sim \betas_2 \sim \betas_1' \sim \betas_2'$. Then the following diagram is commutative:
  \[
  \xymatrix{ \HF^\circ(\S,\alphas_1,\betas_1)
    \ar[r]^{\Phi^{\alphas_1}_{\betas_1 \to \betas_1'}}
    \ar[d]_{\Phi^{\alphas_1 \to \alphas_2}_{\betas_1 \to \betas_2}} &
    \HF^\circ(\S,\alphas_1,\betas_1') \ar[d]^{\Phi^{\alphas_1 \to \alphas_2}_{\betas_1' \to \betas_2'}} \\
    \HF^\circ(\S,\alphas_2,\betas_2) \ar[r]^{\Phi^{\alphas_2}_{\betas_2
        \to \betas_2'}} & \HF^\circ(\S,\alphas_2,\betas_2').  }
  \]
\end{lemma}

\begin{proof}
  By equation~\eqref{eqn:phi-compose},
  \[
  \Phi^{\alphas_1 \to \alphas_2}_{\betas_1' \to \betas_2'} \circ
  \Phi^{\alphas_1}_{\betas_1 \to \betas_1'} = \Phi^{\alphas_1 \to
    \alphas_2}_{\betas_1 \to \betas_2'} = \Phi^{\alphas_2}_{\betas_2
    \to \betas_2'} \circ \Phi^{\alphas_1 \to \alphas_2}_{\betas_1 \to
    \betas_2}.\qedhere
  \]
\end{proof}

\begin{definition} \label{def:phi-ab} Suppose that the isotopy
  diagrams $H = (\S,A,B)$ and $H' = (\S,A,B')$ are $\b$-equivalent.
  Pick admissible representatives $(\S,\alphas,\betas)$ and
  $(\S,\alphas,\betas')$ of $H$ and $H'$, respectively (this is
  possible by Lemma~\ref{lem:adm-multi}). By
  Lemma~\ref{lem:phi-ab-welldef}, the isomorphisms
  $\Phi^\alphas_{\betas \to \betas'}$ descend to the direct limit,
  giving an isomorphism
  \[
  \Phi^A_{B \to B'} \colon \HF^\circ(H) \to \HF^\circ(H').
  \]
  For $\a$-equivalent diagrams $(\S,A,B)$ and $(\S,A',B)$, we define
  the isomorphism $\Phi^{A \to A'}_B$ analogously.
\end{definition}

\begin{remark} \label{rem:equivalence}
Note that the isomorphisms Ozsv\'ath and Szab\'o~\cite[p. 344]{OS06:HolDiskFour} associate
to $\a$- and $\b$-equivalences are defined by composing continuation maps~$\Gamma$ and triangle maps.
The isomorphisms we define in Definition~\ref{def:phi-ab} only involve triangle maps, hence are better suited to computations,
but agree with the isomorphisms of Ozsv\'ath and Szab\'o by Lemma~\ref{lem:cont-triangle}.
\end{remark}

Next, we go on to define isomorphisms induced by diffeomorphisms.

\begin{definition} \label{def:diffeo}
  Let $(\S,\alphas,\betas)$ be an admissible diagram and $d \colon \S
  \to \S'$ a diffeomorphism.  We write $\alphas' = d(\alphas)$ and
  $\betas' = d(\betas)$.  Then $d$ induces an isomorphism
  \[
  d_* \colon \HF^\circ(\S, \alphas, \betas) \to \HF^\circ(\S', \alphas',
  \betas'),
  \]
  as follows. Let $k = |\alphas| = |\betas|$.  Choose a complex
  structure $\mathfrak{j}$ on $\S$ and a perturbation~$J_s$ of
  $\Sym^k(\mathfrak{j})$ on $\Sym^k(\S)$. Pushing $\mathfrak{j}$ and $J_s$
  forward along $d$, we get a complex structure $\mathfrak{j}'$ on
  $\S'$ and
  a perturbation $J_s'$ of $\Sym^k(\mathfrak{j}')$ on
  $\text{Sym}^k(\S')$. Clearly, $d$ induces an isomorphism
  \[
  d_{J_s,J_s'} \colon \HF^\circ_{J_s}(\S,\alphas,\betas) \to
  \HF^\circ_{J_s'}(\S',\alphas', \betas').
  \]
  Since the maps $d_{J_s,J_s'}$ commute with the isomorphisms
  $\Phi_{J_s \to \ol{J_s}}$ of Lemma~\ref{lem:complex}, these diffeomorphism maps descend to a map $d_*$
  on the direct limit $\HF^\circ(\S,\alphas,\betas)$.
\end{definition}

\begin{lemma} \label{lem:diffeo-triangle} The maps
  $\Psi^\alphas_{\betas \to \betas'}$ commute with the diffeomorphism
  maps $d_*$ defined above.  More precisely, suppose that
  $(\S,\alphas,\betas,\betas')$ is an admissible triple. Let $d \colon
  \S \to \ol{\S}$ be a diffeomorphism, and write $\ol{\alphas} =
  d(\alphas)$, $\ol{\betas} = d(\betas)$, and $\ol{\betas}' =
  d(\betas')$. Then we have a commutative rectangle
  \[
  \xymatrix{ \HF^\circ(\S,\alphas,\betas) \ar[r]^{\Psi^\alphas_{\betas
        \to \betas'}} \ar[d]^{d_*}
    & \HF^\circ(\S,\alphas,\betas') \ar[d]^{d_*} \\
    \HF^\circ\left(\ol{\S},\ol{\alphas},\ol{\betas}\right)
    \ar[r]^{\Psi^{\ol{\alphas}}_{\ol{\betas} \to \ol{\betas}'}} &
    \HF^\circ\left(\ol{\S},\ol{\alphas},\ol{\betas}'\right).  }
  \]
  An analogous result holds for the maps $\Psi^{\alphas \to
    \alphas'}_\betas$.
\end{lemma}

\begin{proof}
  If we choose corresponding complex structures and perturbations
  for~$\S$ and~$\ol{\S}$, the statement becomes a tautology.  Indeed,
  $\Sym^k(d)$ is a diffeomorphism between $\Sym^k(\S)$ and
  $\Sym^k(\ol{\S})$ that takes the triple
  $(\Torus_\a,\Torus_\b,\Torus_{\b'})$ to the triple
  $(\Torus_{\ol{\a}}, \Torus_{\ol{\b}}, \Torus_{\ol{\b}'})$, and
  matches up the complex structures and perturbations. Hence the
  triangle maps $\Psi^\alphas_{\betas \to \betas'}$ and
  $\Psi^{\ol{\alphas}}_{\ol{\betas} \to \ol{\betas}'}$ are conjugate
  along $d_*$.
\end{proof}

It follows from Lemma~\ref{lem:diffeo-triangle} that the
diffeomorphism maps and the canonical isomorphisms $\Phi^{\alphas \to
  \alphas'}_{\betas \to \betas'}$ for admissible diagrams
$(\S,\alphas,\betas)$ and $(\S,\alphas',\betas')$ such that $\alphas
\sim \alphas'$ and $\betas \sim \betas'$ also commute, as
$\Phi^{\alphas \to \alphas'}_{\betas \to \betas'}$ can be written as a
composition of triangle maps. Hence, if $H$ and $H'$ are isotopy
diagrams and $d \colon H \to H'$ is a diffeomorphism, then $d$
descents to a map of direct limits
\[
d_* \colon \HF^\circ(H) \to \HF^\circ(H').
\]

Finally, we define maps induced by stabilizations. We proceed as
Ozsv\'ath and Szab\'o~\cite[Section~10]{OS04:HolomorphicDisks},
\cite[p.~346]{OS06:HolDiskFour}.  Suppose that
$\HD'=(\Sigma',\alphas',\betas')$ is a stabilization of the admissible
diagram $\HD=(\Sigma,\alphas,\betas)$.  Then, for suitable
almost-complex structures, there is an isomorphism of chain complexes
\[
\sigma_{\HD\to\HD'}\colon \CF^\circ(\Sigma,\alphas,\betas)\to
\CF^\circ(\Sigma',\alphas',\betas'),
\]
as defined by Ozsv\'ath and
Szab\'o~\cite[Theorems~10.1 and~10.2]{OS04:HolomorphicDisks}. If $\alphas' =
\alphas \cup \{\a\}$, $\betas' = \betas \cup \{\b\}$, and $\a \cap \b =
\{c\}$, then $\sigma_{\HD \to \HD'}$ maps the generator $\x \in
\Torus_\a \cap \Torus_\b$ to $\x \times \{c\} \in \Torus_{\a'} \cap
\Torus_{\b'}$.  This induces an isomorphism on homology.

Before stating the next lemma, we introduce some notation. If $\HD_1 =
(\S,\alphas_1,\betas_1)$ and $\HD_2 = (\S,\alphas_2,\betas_2)$ are
admissible diagrams such that $\alphas_1 \sim \alphas_2$ and $\betas_1
\sim \betas_2$, then we denote $\Phi^{\alphas_1 \to
  \alphas_2}_{\betas_1 \to \betas_2}$ by $\Phi_{\HD_1 \to \HD_2}$.

\begin{lemma}
  \label{lem:StabilizePhi}
  The stabilization maps $\sigma_{\HD \to \HD'}$ commute with the maps
  $\Phi_{\HD_1 \to \HD_2}$, in the following sense: Let
  $\HD_1=(\Sigma,\alphas_1,\betas_1)$ and
  $\HD_2=(\Sigma,\alphas_2,\betas_2)$ be two admissible Heegaard
  diagrams such that $\alphas_1 \sim \alphas_2$ and $\betas_1 \sim
  \betas_2$.  If $\HD_1' = (\S',\alphas_1',\betas_1')$ and $\HD_2' =
  (\S',\alphas_2',\betas_2')$ are stabilizations of $\HD_1$ and
  $\HD_2$, respectively, then $\alphas_1' \sim \alphas_2'$, $\betas_1'
  \sim \betas_2'$, and
  \[
  \sigma_{\HD_2\to\HD_2'} \circ \Phi_{\HD_1\to\HD_2} = \Phi_{\HD_1'
    \to \HD_2'} \circ \sigma_{\HD_1\to\HD_1'}.
  \]
\end{lemma}

\begin{proof}
  This is verified in~\cite[Lemma~2.15]{OS06:HolDiskFour}. Note that
  the continuation maps in that proof agree with our triangle maps by
  Lemma~\ref{lem:cont-triangle}.
\end{proof}

\begin{definition} \label{def:stab-iso} Given isotopy diagrams $H$ and
  $H'$ such that $H'$ is a stabilization of $H$, we define an
  isomorphism
  \[
  \sigma_{H \to H'}\colon \HF^\circ(H) \to \HF^\circ(H')
  \]
  as follows. By definition, there are diagrams $\HD$ and $\HD'$
  representing $H$ and $H'$, respectively, such that $\HD'$ is a
  stabilization of $\HD$.  There are canonical isomorphisms $i_{\HD}
  \colon \HF^\circ(\HD) \to \HF^\circ(H)$ and $i_{\HD'} \colon
  \HF^\circ(\HD') \to \HF^\circ(H')$ coming from the colimit
  construction.  We define $d_{H \to H'}$ as $i_{\HD'} \circ
  \sigma_{\HD \to \HD'} \circ i_{\HD}^{-1}$.  This is independent of
  the choice of $\HD$ and $\HD'$ by Lemma~\ref{lem:StabilizePhi},
  together with the observation that for any two diagrams $\HD_1$ and
  $\HD_2$ representing the same isotopy diagram, $i_{\HD_2}^{-1} \circ
  i_{\HD_1} = \Phi_{\HD_1 \to \HD_2}$.  If $H'$ is obtained from $H$
  by a destabilization, then we set $\sigma_{H \to H'} = (\sigma_{H'
    \to H})^{-1}$.
\end{definition}

Having constructed $\HF^\circ(H)$ for any isotopy diagram $H$ (in the
class of diagrams for which $\HF^\circ(H)$ is defined), and
isomorphisms induced by $\a$\hyp equivalences, $\b$\hyp equivalences,
diffeomorphisms, stabilizations, and destabilizations, we have proved
that $\HF^\circ$ is a weak Heegaard invariant.  This reproves
Theorem~\ref{thm:HF-weak}, Theorem~\ref{thm:HFL-weak}, and
Theorem~\ref{thm:SFH-weak}. However, note that we have already used the
invariance of Heegaard Floer homology up to isomorphism for the manifolds
$M(R_+,k)$ in the proof of Lemma~\ref{lem:triangle}, where we
constructed the element $\Theta_{\b,\g}$ for $\b \sim \g$. We could
have avoided this by imitating the invariance proof of Ozsv\'ath and
Szab\'o~\cite{OS04:HolomorphicDisks}, at the price of making the
discussion longer.

Recall that, at the end of Section~\ref{sec:main-theorems}, we
indicated the necessary checks for obtaining the
$\SpinC$-refinement. If $\HD$ is an admissible diagram of the balanced
sutured manifold $(M,\g)$ and $\spinc \in \SpinC(M,\g)$, then
$\CF^\circ(\S,\alphas,\betas,\spinc)$ is generated by those $\x \in
\Torus_\a \cap \Torus_\b$ for which $\spinc_{(M,\g)}(\x) = \spinc$.
It follows from the work of Ozsv\'ath and
Sza\-b\'o~\cite{OS04:HolomorphicDisks} that the $\SpinC$-grading is
preserved by the isomorphisms
$\Phi_{J_s \to J_s'}$, the triangle maps $\Psi_{\alphas}^{\betas \to
  \betas'}$ for $\betas \sim \betas'$ and $\Psi^{\alphas \to
  \alphas'}_\betas$ for $\alphas \sim \alphas'$ (cf.~Lemma~\ref{lem:grading}),
and the stabilization maps $\sigma_{\HD \to \HD'}$.
Furthermore, given a diagram $\HD = (\S,\alphas,\betas)$ of $(M,\g)$,
a diagram $\HD'$ of $(M',\g')$, and a diffeomorphism $d \colon (M,\g)
\to (M',\g')$ mapping $\HD$ to $\HD'$, it is straightforward to see
that $d_* \left(\spinc_{(M,g)}(\x) \right) = \spinc_{(M',\g')}(d(\x))$
for every $\x \in \Torus_\a \cap \Torus_\b$. In particular, if $(M,\g)
= (M',\g')$ and $d$ is isotopic to the identity in $(M,\g)$, then $d_*
\colon \SpinC(M,\g) \to \SpinC(M,\g)$ is the identity.  The existence
of a $\SpinC$-grading on $\HF^\circ(M,\g)$ follows
once we show that $\HF^\circ$ is a strong Heegaard invariant.

\subsection{Heegaard Floer homology as a strong Heegaard invariant} \label{sec:HFstrong}

In this section, we show that the invariant $\HF^\circ$ of isotopy
diagrams, together with the maps induced by $\a$-equivalences,
$\b$-equivalences, diffeomorphisms, and (de)stabilizations, satisfy
the axioms of strong Heegaard invariants listed in
Definition~\ref{def:strong-Heegaard}. We postpone the verification of
axiom~\eqref{item:strong-handleswap}, handleswap invariance, to the
following section.

We first prove that $\HF^\circ$ satisfies
axiom~\eqref{item:strong-funct}, functoriality. The $\a$-equivalence
and $\b$-equivalence maps $\Phi^{A \to A'}_B$ and $\Phi^A_{B \to B'}$
are functorial by equations~\eqref{eqn:phi-compose}
and~\eqref{eqn:id}. Functoriality of the diffeomorphism maps $d_*$
follows immediately from the definition.  If $H'$
is obtained from $H$ by a stabilization, then the destabilization map
$\sigma_{H' \to H} = (\sigma_{H \to H'})^{-1}$, by definition.

Next, we consider axiom~\eqref{item:strong-commute}, commutativity. In
Definition~\ref{def:distinguished-rect}, we defined five different
types of distinguished rectangles of the form
\[
\xymatrix{
  H_1 \ar[r]^e \ar[d]^f & H_2 \ar[d]^g \\
  H_3 \ar[r]^h & H_4, }
\]
where $H_i = (\S_i,[\alphas_i],[\betas_i])$.  For a rectangle of
type~\eqref{item:rect-alpha-beta}, commutativity follows from
equation~\eqref{eqn:phi-compose}. Lemma~\ref{lem:StabilizePhi} implies
commutativity along a rectangle of type~\eqref{item:rect-alpha-stab}.
Commutativity along a rectangle of type~\eqref{item:rect-alpha-diff}
follows from Lemma~\ref{lem:diffeo-triangle}.

Now consider a rectangle of type~\eqref{item:rect-stab-stab}.  Then
there are disjoint disks $D_1$, $D_2 \subset \S_1$ and punctured tori
$T_1$, $T_2 \subset \S_4$ such that $\S_1 \setminus (D_1 \cup D_2) =
\S_4 \setminus (T_1 \cup T_2)$. Let $\alphas_4 \cap \betas_4 \cap T_i
= \{c_i\}$ for $i \in \{1,2\}$. Then there are representatives $\HD_i
= (\S_i,\alphas_i,\betas_i)$ of the isotopy diagrams $H_i$ for $i \in
\{\, 1,\dots,4 \,\}$ such that $\alphas_2 \cap \betas_2 \cap T_1 =
\{c_1\}$ and $\alphas_3 \cap \betas_3 \cap T_2 = \{c_2\}$, and the
four diagrams coincide outside $T_1$ and $T_2$.  Given a generator $\x
\in \Torus_{\a_1} \cap \Torus_{\b_1}$,
\[
\sigma_{\HD_2 \to \HD_4} \circ \sigma_{\HD_1 \to \HD_2}(\x) = \x
\times \{c_1\} \times \{c_2\} = \x \times \{c_2\} \times \{c_1\} =
\sigma_{\HD_3 \to \HD_4} \circ \sigma_{\HD_1 \to \HD_3}(\x).
\]
So the commutativity already holds on the chain level for an
appropriate choice of complex structures. Indeed, if $J_s$ and $J_s'$
are almost complex structures corresponding to different relative
neck lengths at $\partial T_1$ and $\partial T_2$,
then the change of complex structures map $\Phi_{J_s \to J_s'}$
only counts constant curves; see the proof of \cite[Proposition~4.20]{ZemCFLTQFT}.

Finally, for a rectangle of type~\eqref{item:rect-stab-diff}, we can
choose representatives $\HD_i = (\S_i,\alphas_i,\betas_i)$ of $H_i$
such that $\HD_2$ is a stabilization of $\HD_1$ and $\HD_4$ is a
stabilization of $\HD_3$; furthermore, $f(\HD_1) = \HD_3$ and
$g(\HD_2) = \HD_4$.  This is possible since for the stabilization
disks $D \subset \S_1$ and $D' \subset \S_3$ and punctured tori $T
\subset \S_2$ and $T' \subset \S_4$, the diffeomorphisms satisfy $f(D)
= D'$, $g(T) = T'$, and $f|_{\S_1 \setminus D} = g|_{\S_2 \setminus
  T}$. In particular, if $\alphas_2 \cap \betas_2 \cap T = \{c\}$ and
$\alphas_4 \cap \betas_4 \cap T' = \{c'\}$, then $g(c) = c'$. With
these choices, for $\x \in \Torus_{\a_1} \cap \Torus_{\b_1}$, we have
\[
g_* \circ \sigma_{\HD_1 \to \HD_2}(\x) = g(\x \times \{c\}) = g(\x)
\times \{g(c)\} = f(\x) \times \{c'\} = \sigma_{\HD_3 \to \HD_4} \circ
f_*.
\]
So we have commutativity on the chain level for an appropriate choice of
complex structures (for a sufficiently long neck,
as in the proof of~\cite[Theorem~10.2]{OS04:HolomorphicDisks}).

Finally, we verify axiom~\eqref{item:strong-cont}, continuity. This
follows from the following result.

\begin{proposition} \label{prop:continuity} Let $(\S,\alphas,\betas)$
  be an admissible diagram. Suppose that $d \colon \S \to \S$ is a
  diffeomorphism isotopic to $\text{Id}_\S$, and let $\alphas' =
  d(\alphas)$ and $\betas' = d(\betas)$. Then
  \begin{equation}\label{eqn:continuity}
  d_* = \Phi^{\alphas \to \alphas'}_{\betas \to \betas'} \colon
  \HF^\circ(\S,\alphas,\betas) \to \HF^\circ(\S,\alphas',\betas').
  \end{equation}
\end{proposition}

\begin{proof}
  Since $d \colon \S \to \S$ is isotopic to the identity, there are diagrams
  $\HD_i = (\S,\alphas_i,\betas_i)$ for $i \in \{\,0,\dots,n\,\}$ and
  diffeomorphisms $d_i \colon \HD_{i-1} \to \HD_i$ for $i \in
  \{\, 1,\dots,n \,\}$, such that
  \begin{itemize}
  \item $\HD_0 = (\S,\alphas,\betas)$ and $\HD_n =
    (\S,\alphas',\betas')$,
  \item every $d_i$ is Hamiltonian isotopic to $\text{Id}_\S$ for some symplectic form~$\omega_i$ on~$\S$,
  \item $d = d_n \circ \dots \circ d_1$,
  \item $|\a \cap d_i(\a)| = 2$ for every $\a \in \alphas_{i-1}$, and
  \item $|\b \cap d_i(\b)| = 2$ for every $\b \in \betas_{i-1}$.
  \end{itemize}
  \begin{figure}
    \centering
    \includegraphics{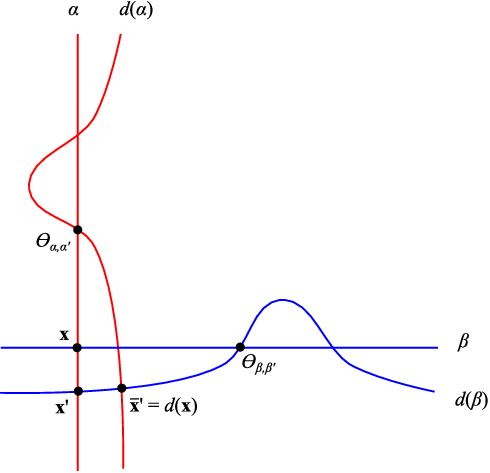}
    \caption{A schematic picture illustrating the diffeomorphism~$d_i$.}
    \label{fig:continuity}
  \end{figure}
  See Figure~\ref{fig:continuity} for a schematic picture of~$d_i$.

  By equation~\eqref{eqn:phi-compose} and the functoriality of the
  diffeomorphism maps, it suffices to prove the statement for each~$d_i$.
  So suppose that $d \colon (\S,\alphas,\betas) \to
  (\S,\alphas',\betas')$ is a diffeomorphism that is Hamiltonian isotopic to $\text{Id}_\S$
  for some symplectic form~$\omega$ on~$\S$, and such that
  $|\a \cap d(\a)| = 2$ and $|\b \cap d(\b)| = 2$ for every $\a \in \alphas$ and $\b \in \betas$.
  Let~$\{\, d_t \,\colon\, t \in \RR \,\}$ be a Hamiltonian isotopy supported
  in~$I = [0,1]$, where $d_0 = \text{Id}_\S$ and $d_1 = d$.

  Let $\mathfrak{j}$ be an almost complex structure on~$\S$, and $J_s$ a perturbation of $\text{Sym}^k(\mathfrak{j})$,
  where $k = |\alphas| = |\betas|$. For every $t \in I$, let $\mathfrak{j}_t = (d_t)_*(\mathfrak{j})$
  and $J_{s,t} = \text{Sym}^k(d_t)_*(J_s)$. We write $\mathfrak{j}' = \mathfrak{j}_1$ and $J_s' = J_{s,1}$.
  As in Lemma~\ref{lem:complex} and \cite[p.~1080]{OS04:HolomorphicDisks},
  the family~$J_{s,t}$ induces an isomorphism $\Phi_{J_s \to J_s'}$
  by counting Maslov index zero Whitney disks $u \colon I \times \RR \to \text{Sym}^k(\S)$ satisfying
  $du/ds + J_{s,t} (du/dt) = 0$, where~$s$ is the coordinate on $I$ and $t$ is the coordinate on~$\RR$.

  For every $t \in \RR$, let $\alphas_t = d_t(\alphas)$ and $\betas_t = d_t(\betas)$.
  As defined by Ozsv\'ath and Szab\'o~\cite[p.~344]{OS06:HolDiskFour} (also see \cite[p.~1086]{OS04:HolomorphicDisks}),
  the Hamiltonian isotopy induces an isomorphism
  \[
  \Gamma^{\alphas \to \alphas'}_{\betas \to \betas'} \colon \HF^\circ_{J_s}(\S,\alphas,\betas) \to 
  \HF^\circ_{J_s}(\S,\alphas',\betas')
  \]
  by counting index zero $J_s$-holomorphic disks with boundary condition $u(1+it) \in \Torus_{\a_t}$
  and $u(0+it) \in \Torus_{\b_t}$ for every $t \in \RR$
  and connecting some
  $\x \in \Torus_\a \cap \Torus_\b$ and $\y \in \Torus_{\a'} \cap \Torus_{\b'}$.
  As in the proof of~\cite[Lemma~9.6]{OS04:HolomorphicDisks}, the quadruple diagram
  $(\S,\alphas,\alphas',\betas,\betas')$ is admissible.
  Hence, by~\cite[Lemma~2.12]{OS06:HolDiskFour} and Lemma~\ref{lem:cont-triangle}, we have
  \[
  \Gamma^{\alphas \to \alphas'}_{\betas \to \betas'} =
  \Gamma^{\alphas'}_{\betas \to \betas'} \circ \Gamma^{\alphas \to \alphas'}_{\betas} =
  \Psi^{\alphas'}_{\betas \to \betas'} \circ \Psi^{\alphas \to \alphas'}_{\betas} =
  \Phi^{\alphas \to \alphas'}_{\betas \to \betas'}.
  \]
  So equation~\eqref{eqn:continuity} follows once we show that
  $d_* = \Gamma^{\alphas \to \alphas'}_{\betas \to \betas'}$.
  By Definition~\ref{def:diffeo}, we have $d_* = \Phi_{J_s' \to J_s} \circ d_{J_s, J_s'}$.
  As $\Phi_{J_s' \to J_s}^{-1} = \Phi_{J_s \to J_s'}$, it suffices to show that
  \[
  d_{J_s, J_s'} = \Phi_{J_s \to J_s'} \circ \Gamma^{\alphas \to \alphas'}_{\betas \to \betas'}.
  \]

  We now express $d_{J_s, J_s'}$ as a continuation map as well. We defined
  \[
  \Gamma_{d_t} \colon \HF^\circ_{J_s}(\S,\alphas,\betas) \to \HF^\circ_{J_s'}(\S,\alphas',\betas')
  \]
  by counting index zero maps $u \colon I \times \RR \to \text{Sym}^k(\S)$ connecting some
  $\x \in \Torus_\a \cap \Torus_\b$ and $\y \in \Torus_{\a'} \cap \Torus_{\b'}$
  that satisfy $du/ds + J_{s,t} (du/dt) = 0$ and the boundary conditions
  $u(1+it) \in \Torus_{\a_t}$ and $u(0+it) \in \Torus_{\b_t}$ for every $t \in \RR$.
  For every such $u$, let $v(s,t) = \text{Sym}^k(d_t)^{-1}(u(s,t))$.
  Then $v$ is an index zero $J_s$-holomorphic disk by the definition of~$J_{s,t}$,
  and satisfies the boundary conditions
  $v(1+it) \in \Torus_\a$ and $v(0+it) \in \Torus_\b$ for every $t \in \RR$,
  and is hence a constant disk with image some $\x \in \Torus_\a \cap \Torus_\b$.
  This implies that, on the chain level, $\Gamma_{d_t}$ maps every $\x \in \Torus_\a \cap \Torus_\b$
  to~$d(\x)$, and hence agrees with~$d_{J_s,J_s'}$.

  The last step is showing that
  \begin{equation}\label{eqn:homotopy}
  \Gamma_{d_t} = \Phi_{J_s \to J_s'} \circ \Gamma^{\alphas \to \alphas'}_{\betas \to \betas'}.
  \end{equation}
  For this, we use techniques analogous to those developed by Ozsv\'ath and Szab\'o~\cite[pp.~1081, 1089]{OS04:HolomorphicDisks}
  to prove that the maps $\Phi_{J_s \to J_s'}$ and $\Gamma^{\alphas \to \alphas'}_{\betas}$ are invertible.
  First, we give some motivation. Consider the product fibration
  \[
  \pi \colon (I\times I) \times \text{Sym}^k(\S) \to I \times I.
  \]
  For every $(t_1,t_2) \in I \times I$, we consider the totally real submanifolds $\Torus_{\a_{t_1}}$ and $\Torus_{\b_{t_1}}$
  of the fiber $\pi^{-1}(t_1,t_2)$ endowed with the almost complex structure~$\text{Sym}^k(\mathfrak{j}_{t_2})$
  and perturbation~$J_{s,t_2}$. Then $\Gamma_{d_t}$ is the continuation map along the path $h_0(t) = (t,t)$ in $I \times I$
  from~$(0,0)$ to~$(1,1)$, while $\Phi_{J_s \to J_s'} \circ \Gamma^{\alphas \to \alphas'}_{\betas \to \betas'}$
  is the continuation map along the path $h_1(t)$ that is $(2t,0)$ for $t \in [0,1/2]$ and is $h_1(t) = (1,2t)$
  for $t \in [1/2,1]$. Since $h_0$ and $h_1$ are homotopic in $I \times I$ fixing the endpoints, they induce
  chain homotopic chain maps, as we will now show.

  Let $h_\tau(t) = (h_{\tau,1}(t), h_{\tau,2}(t))$ for $\tau \in I$
  be a homotopy from $h_0$ to $h_1$ in $I \times I$ fixing the endpoints.
  Then $h_\tau(t)$ induces a chain homotopy between the two sides of equation~\eqref{eqn:homotopy}, as follows.
  For $(t,\tau) \in I \times I$, we write $\alphas_{t,\tau} = \alphas_{h_{\tau,1}(t)}$
  and $\betas_{t,\tau} = \betas_{h_{\tau,1}(t)}$.
  Fix $\tau \in I$, and let $\x \in \Torus_\a \cap \Torus_\b$ and $\y \in \Torus_{\a'} \cap \Torus_{\b'}$.
  Then $\pi_2^\tau(\x,\y)$ is the set of homotopy classes of maps $u \colon I \times \RR \to \text{Sym}^k(\S)$
  with boundary conditions $u(1+it) \in \Torus_{\a_{t,\tau}}$
  and $u(0+it) \in \Torus_{\b_{t,\tau}}$ for every $t \in \RR$,
  and such that $\lim_{t \to -\infty} u(s+it) = \x$ and $\lim_{t \to \infty} u(s+it) = \y$.

  For $\phi \in \pi_2^\tau(\x,\y)$, let $\mathcal{M}_\tau(\phi)$ be the moduli space of
  maps~$u$ representing~$\phi$ such that
  \[
  \frac{du}{ds} + J_{s,h_{\tau,2}(t)}\frac{du}{dt} = 0.
  \]
  For every $\tau \in I$, the homotopy $h$ induces a canonical identification between
  $\pi_2^0(\x,\y)$ and $\pi_2^\tau(\x,\y)$. Using these identifications,
  for every $\phi \in \pi_2^0(\x,\y)$, we define the moduli space
  \[
  \mathcal{M}^h(\phi) = \bigcup_{\tau \in I} \mathcal{M}_\tau(\phi) \times \{\tau\} =
  \left\{\, (u,\tau) \in C^\infty \left(I \times \RR, \text{Sym}^k(\S) \right) \times I \,\colon\, u \in \mathcal{M}_\tau(\phi)  \,\right\}.
  \]
  Then $\dim\left(\mathcal{M}^h(\phi)\right) = \mu(\phi)+1$. For $\x \in \Torus_\a \cap \Torus_\b$, let
  \[
  H^h(\x) = \sum_{\y \in \Torus_{\a'} \cap \Torus_{\b'}} \, \sum_{\substack{\phi \in \pi_2^0(\x,\y) \\ \mu(\phi) = -1}}
  \left(\left|\mathcal{M}^h(\phi)\right| \mod 2\right)\,\y
  \]
  in case of $\HFa$ or $\SFH$, and replacing $\x$ with $[\x,i]$ and $\y$ with $[\y,i-n_z(\phi)]$
  in case of~$\HF^+$, $\HF^-$, and~$\HF^\infty$.

  To see that $H^h$ gives a chain homotopy between $\Gamma_{d_t}$
  and $\Phi_{J_s \to J_s'} \circ \Gamma^{\alphas \to \alphas'}_{\betas \to \betas'}$, we consider
  the ends of the moduli spaces $\mathcal{M}^h(\psi)$ for $\psi \in \pi_2^0(\x,\y)$ such that $\mu(\psi) = 0$.
  There are three types of ends:
  The ends at $\tau = 0$ contribute to the map $\Gamma_{d_t}$. The ends at~$\tau = 1$ contribute to
  $\Phi_{J_s \to J_s'} \circ \Gamma^{\alphas \to \alphas'}_{\betas \to \betas'}$; to see that continuation
  maps are natural under juxtaposition of the paths involved, the proof of \cite[Lemma~2.12]{OS06:HolDiskFour}
  generalizes readily. Finally, broken strips contribute to $H^h \circ \partial + \partial' \circ H^h$,
  where $\partial$ is the boundary maps for $\Torus_\a$, $\Torus_\b$, and $(\mathfrak{j}, J_s)$,
  while $\partial'$ is the boundary map for $\Torus_\a'$, $\Torus_\b'$, and $(\mathfrak{j}', J_s')$.
  Hence, on the chain level,
  \[
  \Gamma_{d_t} + \Phi_{J_s \to J_s'} \circ \Gamma^{\alphas \to \alphas'}_{\betas \to \betas'} \equiv
  H^h \circ \partial + \partial' \circ H^h \mod 2,
  \]
  and equation~\eqref{eqn:homotopy} follows.
\end{proof}

\subsection{Simple handleswap invariance of Heegaard Floer homology}
\label{subsec:Handleswaps}

In this section, we prove simple handleswap invariance of Heegaard Floer homology.
The technical details of our proof are modeled on the proof of stabilization invariance
due to Lipshitz~\cite{Lipshitz06:CylindricalHF},
and invariance of $\HF^+$ under adding extra basepoints due to Ozsv\'ath and Szab\'o~\cite{OS08:HFL}.
The main analytical input is due to Lipshitz~\cite{Lipshitz06:CylindricalHF}
and Lipshitz, Ozsv\'ath, and Thurston~\cite{LOT1}.

Let
\[
\xymatrix{H_1 \ar[rd]^e & \\ H_3 \ar[u]^g & H_2 \ar[l]^f}
\]
be a simple handleswap, as in Definition~\ref{def:simple-handleswap},
where $H_i = (\S \# \S_0,\alphas_i,\betas_i)$ are admissible diagrams for $i \in \{1,2,3\}$,
the edge~$e$ is an $\alpha$-equivalence, $f$ is a $\beta$-equivalence, and $g$ is a diffeomorphism.
Note that $\betas_1 = \betas_2$ and $\alphas_2 = \alphas_3$.
We choose representatives of the isotopy diagrams $H_i$
that violate condition~\eqref{it:identical} of Definition~\ref{def:simple-handleswap},
in that we require $\alphas_1 \setminus \{\a_1\}$ and $\alphas_2 \setminus \{\a_1'\}$,
and similarly, $\betas_2 \setminus \{\b_2\}$ and $\betas_3 \setminus \{\b_2'\}$
to be small Hamiltonian isotopic translates of each other, and adjust the
diffeomorphism~$g$ accordingly. This will allow us to compute the
isomorphisms induced by~$e$ and~$f$ via a single triangle map~$\Psi$,
as opposed to the maps~$\Phi$ that are compositions of two triangle maps.
Changing $g$ by an isotopy does not change the map $g_* \colon \HF^\circ(H_3) \to \HF^\circ(H_1)$
by the Continuity Axiom, Proposition~\ref{prop:continuity}.

The $\alpha$- and $\beta$-equivalence maps corresponding to the arrows $e$ and $f$
can be computed with the triangle maps $\Phi_e := \Psi^{\alphas_1 \to \alphas_2}_{\betas_1}$
and~$\Phi_f := \Psi^{\alphas_2}_{\betas_2 \to \betas_3}$.
The map corresponding to the arrow $g$ is the map induced by a diffeomorphism,
which we also denote by~$g$. Simple handleswap invariance amounts to proving the following:

\begin{theorem}\label{prop:handleswapinvariance}
The induced maps $g_*$, $\Phi_e$, and $\Phi_f$ satisfy
\[
g_*\circ \Phi_f\circ \Phi_e = \Iden_{\HF^\circ(H_1)}.
\]
\end{theorem}

To prove this, we will explicitly count triangles by a neck stretching
argument. Let $\mathcal{T}_0 = (\Sigma_0,\alphas'_0,\alphas_0,\betas_0)$ be
the Heegaard triple shown in Figure~\ref{fig::2}, where $\S_0$ is a surface
of genus two, $p_0 \in \S_0 \setminus (\alphas'_0 \cup \alphas_0 \cup
\betas_0)$ is a distinguished point, $\alphas'_0 = \{\alpha'_1,\alpha'_2\}$,
$\alphas_0 = \{\a_1,\a_2\}$, and $\betas_0 = \{\b_1,\b_2\}$. Suppose that
$\mathcal{T}=(\Sigma, \alphas',\alphas,\betas)$ is a Heegaard triple with a
distinguished point $p\in \Sigma\setminus (\alphas'\cup\alphas\cup\betas)$.
We consider the Heegaard triple
\[
\mathcal{T} \# \mathcal{T}_0 = 
(\Sigma\#\Sigma_0, \alphas' \cup \alphas'_0, \alphas \cup \alphas_0, \betas \cup\betas_0),
\]
where the connected sum is taken at $p \in \S$ and $p_0 \in \S_0$.
Write $\mathbb{T}_{\a_0}\cap\mathbb{T}_{\b_0} =
\{\mathbf{a}\}$, $\mathbb{T}_{\a'_0} \cap \mathbb{T}_{\b_0} = \{\mathbf{b}\}$, and
\[
\mathbb{T}_{\a'_0} \cap \mathbb{T}_{\a_0} =
\{\,\theta_1^+\theta_2^+,\theta_1^+\theta_2^-,\theta_1^-\theta_2^+, \theta_1^-\theta_2^-\,\},
\]
where the intersection points $\theta_1^\pm \in \a_1'\cap \a_1 $
and $\theta_2^\pm \in \a_2'\cap \a_2$ are marked in Figure~\ref{fig::2}.
We will write $\Theta$ for $\theta_1^+\theta_2^+$.

\begin{figure}[ht!]
\centering
\input{fig2.pdf_tex}
\caption{The Heegaard triple $(\Sigma_0,\alphas'_0,\alphas_0,\betas_0,p_0)$
we will use for computing the $\alpha$-equivalence map $\Phi_e$.
Shaded in gray is a Maslov index zero triangle in $\pi_2(\Theta,\mathbf{a},\mathbf{b})$.
\label{fig::2}}
\end{figure}

The main technical result needed to prove Theorem \ref{prop:handleswapinvariance}
is the following count of holomorphic triangles:

\begin{proposition}\label{lem:handleswapinvariance}
Suppose that $\mathcal{T} = (\Sigma, \alphas',\alphas,\betas)$ is an
admissible Heegaard triple, and let $\mathcal{T} \# \mathcal{T}_0 = (\Sigma
\# \Sigma_0, \alphas' \cup \alphas'_0, \alphas \cup \alphas_0,\betas \cup
\betas_0)$ be as above. For a sufficiently stretched and appropriately
generic almost complex structure (specified precisely in
Subsection~\ref{sec:neck}), the triangle map $F^\circ_{\mathcal{T} \#
\mathcal{T}_0}$ associated to the triple~$\mathcal{T} \# \mathcal{T}_0$ is given by
\[
(\x\times \Theta) \otimes (\y\times \mathbf{a}) \mapsto 
F^\circ_{\mathcal{T}}(\x \otimes \y) \times \mathbf{b}
\]
for every $\x \in \mathbb{T}_{\alpha'} \cap \mathbb{T}_\a$ and 
$\y \in \mathbb{T}_\a \cap \mathbb{T}_\b$.
\end{proposition}

We will prove Proposition~\ref{lem:handleswapinvariance} in
Subsection~\ref{sec:neck}. The same proof can easily be adapted to show the
following. Let $\mathcal{T}'=(\Sigma,\alphas',\betas,\betas')$ denote a
Heegaard triple with a distinguished point $p \in \S \setminus (\alphas'\cup
\betas \cup \betas')$. Furthermore, let $\mathcal{T}'_0=(\Sigma_0,
\alphas_0',\betas_0,\betas_0')$ denote the Heegaard triple shown in
Figure~\ref{fig::6}, where $\S_0$ is a surface of genus two, $p_0 \in \S_0
\setminus (\alphas_0' \cup \betas_0 \cup \betas_0')$ is a distinguished
point, $\alphas_0' = \{\a_1',\a_2'\}$, $\betas_0 = \{\b_1,\b_2\}$, and
$\betas'_0 = \{\b'_1,\b'_2\}$. We write $\mathbb{T}_{\alpha_0'}\cap
\mathbb{T}_{\beta_0} = \{\mathbf{b}\}$, $\mathbb{T}_{\alpha_0'}\cap
\mathbb{T}_{\beta_0'} = \{\mathbf{c}\}$, and let $\Theta'$ denote the
top-graded intersection point of $\mathbb{T}_{\beta_0}\cap
\mathbb{T}_{\beta_0'}$. We consider the Heegaard triple
\[
\mathcal{T}' \# \mathcal{T}_0' = 
(\Sigma \# \Sigma_0, \alphas' \cup \alphas_0', \alphas \cup \alphas_0, \betas' \cup \betas_0'),
\]
where the connected sum is taken at $p \in \S$ and $p_0 \in \S_0$.

\begin{figure}[ht!]
\centering
\input{fig6.pdf_tex}
\caption{The Heegaard triple $(\Sigma_0,\alphas'_0,\betas_0,\betas_0',p_0)$
we will use for computing the $\beta$-equivalence map $\Phi_f$.
Shaded in gray is a Maslov index zero triangle in $\pi_2(\mathbf{b},\Theta',\mathbf{c})$.
\label{fig::6}}
\end{figure}

\begin{proposition}\label{lem:handleswapinvariance2}
Suppose that $\mathcal{T}'=(\Sigma, \alphas',\betas,\betas')$ is an
admissible Heegaard triple, and let $\mathcal{T}' \# \mathcal{T}_0' = (\Sigma
\# \Sigma_0, \alphas' \cup \alphas_0', \alphas \cup \alphas_0, \betas' \cup
\betas_0')$ be as above. For a sufficiently stretched and appropriately
generic almost complex structure (specified precisely in
Subsection~\ref{sec:neck}), the triangle map $F^\circ_{\mathcal{T}'}$
associated to the triple $\mathcal{T}' \# \mathcal{T}_0'$ is given by
\[
(\x\times \mathbf{b}) \otimes (\y\times \Theta') \mapsto
F^\circ_{\mathcal{T}'}(\x \otimes \y) \times \mathbf{c}
\]
for every $\x \in \mathbb{T}_{\alpha'} \cap \mathbb{T}_\beta$ and
$\y \in \mathbb{T}_\beta \cap \mathbb{T}_{\b'}$.
\end{proposition}

\begin{remark}
In Propositions~\ref{lem:handleswapinvariance}
and~\ref{lem:handleswapinvariance2}, our notation might suggest that $\alphas
\sim \alphas'$ and~$\betas \sim \betas'$. However, they hold for arbitrary
admissible triples~$\mathcal{T}$ and~$\mathcal{T}'$. We use this notation
since we will only need to apply Propositions~\ref{lem:handleswapinvariance}
and~\ref{lem:handleswapinvariance2} when the curves~$\alphas$ are isotopic
to~$\alphas'$ and the curves~$\betas$ are isotopic to~$\betas'$.
\end{remark}

We now use the above counts of holomorphic triangles to prove invariance
under simple handleswaps:

\begin{proof}[Proof of Theorem~\ref{prop:handleswapinvariance}]
Consider the diagram $H_1 = (\S \# \S_0, \alphas_1,\betas_1)$. Recall that
$\Sigma_0$ is a genus two surface. Let $\alphas_0= \alphas_1 \cap \S_0 =
\{\alpha_1,\alpha_2\}$ and $\betas_0 = \betas_1 \cap \S_0 =
\{\beta_1,\beta_2\}$. The diagram $H_0 = (\Sigma_0,\alphas_0,\betas_0)$ is
shown in Figures~\ref{fig:handleswap} and~\ref{fig::2}. Furthermore, let $H =
(\Sigma, \alphas,\betas)$, where $\alphas = \alphas_1 \cap \S$ and $\betas =
\betas_1 \cap \S$. By Definition~\ref{def:simple-handleswap}, we can write
\[
H_1 = H \# H_0 = (\S \# \S_0, \alphas \cup \alphas_0, \betas \cup \betas_0).
\]
Then $H$ is admissible since~$H_1$ is admissible.

The $\alpha$-equivalence $e$ corresponds to handlesliding $\alpha_1$ over
$\alpha_2$. Let $\alphas_0' = \{\alpha_1',\alpha_2'\}$ be as in
Figure~\ref{fig::2}. The curve $\alpha_2'$ is a small Hamiltonian translate
of $\alpha_2$, and $\alpha_1'$ is the result of handlesliding $\alpha_1$ over
$\alpha_2$. Furthermore, let $\alphas'$ be a small Hamiltonian translate
of~$\alphas$ on $\Sigma$. Then
\[
H_2 = (\S \# \S_0, \alphas' \cup \alphas_0', \betas \cup \betas_0),
\]
and $\Phi_e = \Psi_{\betas\cup \betas_0}^{\alphas\cup \alphas_0\to
\alphas'\cup \alphas_0'}$. As before, we write $\mathbb{T}_{\a_0}\cap
\mathbb{T}_{\b_0} = \{\mathbf{a}\}$ and $\mathbb{T}_{\a'_0}\cap
\mathbb{T}_{\b_0}=\{\mathbf{b}\}$. By Proposition~\ref{lem:handleswapinvariance},
for an appropriate choice of almost complex structure, we have
\begin{equation}
\begin{split}
\Psi_{\betas\cup \betas_0}^{\alphas\cup \alphas_0 \to \alphas'\cup \alphas_0'}(\y \times \mathbf{a}) &=
F_{\a'\cup \a_0', \a \cup \a_0,\b \cup \b_0}\left((\Theta_{\a',\a} \times \Theta) \otimes (\y \times \mathbf{a})\right) = \\
F_{\a',\a,\b}(\Theta_{\a',\a} \otimes \y) \times \mathbf{b} &=
\Psi_{\betas}^{\alphas\to\alphas'}(\y)\times \mathbf{b}.
\end{split}
\label{eq:Phiemap}
\end{equation}

The $\beta$-equivalence map $\Phi_f$ can be computed similarly. Let
$\betas_0'=\{\beta_1',\beta_2'\}$, where~$\beta_1'$ is a Hamiltonian
translate of $\beta_1$ and $\beta_2'$ is obtained by handlesliding $\beta_2$
over $\beta_1$; see Figure~\ref{fig::6}. We write $\mathbb{T}_{\a_0'} \cap
\mathbb{T}_{\b_0'}=\{\mathbf{c}\}$.
By Proposition~\ref{lem:handleswapinvariance2}, we have
\begin{equation}
\Psi_{\betas\cup \betas_0\to \betas'\cup \betas_0'}^{\alphas'\cup \alphas_0'}(\mathbf{x}\times \mathbf{b}) =
\Psi_{\betas\to \betas'}^{\alphas'}(\x)\times \mathbf{c}.\label{eq:Phifmap}
\end{equation}

Outside $\Sigma_0$, we can take $g$ to be a diffeomorphism isotopic to the identity
and mapping~$\alphas'$ to~$\alphas$ and~$\betas'$ to~$\betas$. By construction, we have that
\begin{equation}
g_*(\x\times \mathbf{c})=(g|_\S)_*(\x) \times \mathbf{a}.\label{eq:g*map}
\end{equation}

Combining Equations~\eqref{eq:Phiemap}, \eqref{eq:Phifmap}, and~\eqref{eq:g*map}, we see that
\[
\left(g_*\circ \Psi_{\betas \cup \betas_0\to \betas'\cup \betas_0'}^{\alphas' \cup \alphas_0'} \circ
\Psi_{\betas \cup \betas_0}^{\alphas \cup \alphas_0\to \alphas'\cup \alphas_0'}\right)(\x \times \mathbf{a}) =
\left((g|_\S)_*\circ \Psi_{\betas \to \betas'}^{\alphas'} \circ
\Psi_{\betas}^{\alphas \to \alphas'}\right)(\x) \times \mathbf{a}.
\]
By Proposition~\ref{prop:continuity}, we have
\[
(g|_\S)_* \circ \Psi_{\betas \to \betas'}^{\alphas'} \circ 
\Psi_{\betas}^{\alphas \to \alphas'} \simeq \Iden_{\CF(H)}.
\]
If $\partial_H$ denotes the differential on $\CF(H)$, then stabilization
invariance (see \cite[Section~10]{OS04:HolomorphicDisks} and
\cite[Section~12]{Lipshitz06:CylindricalHF}) implies that the differential on
$\CF(H_1)$ is $\partial_H \otimes \Iden_{\CF(H_0)}$ for a sufficiently
stretched almost complex structure. Hence
\[
g_*\circ \Phi_f\circ \Phi_e =
\left((g|_\S)_*\circ \Psi_{\betas\to \betas'}^{\alphas'} \circ 
\Psi_{\betas}^{\alphas\to \alphas'}\right)\otimes \Iden_{\CF(H_0)}
\simeq \Iden_{\CF(H)} \otimes \Iden_{\CF(H_0)},
\]
completing the proof.
\end{proof}

\subsubsection{Moduli spaces, compactness, transversality, and gluing results}
\label{sec:neck}

In order to prove Propositions~\ref{lem:handleswapinvariance}
and~\ref{lem:handleswapinvariance2}, we describe the limiting behavior of
pseudo-holomorphic curves as one stretches the connected sum tube between
$\Sigma$ and $\Sigma_0$. In principle, it is possible to describe the
behavior of holomorphic triangles appearing in
$\Sym^{g(\Sigma)+g(\Sigma_0)}(\Sigma \# \Sigma_0)$ as one stretches the
almost complex structure (see, for example, Ozsv\'ath and
Szab\'o~\cite[Section~10]{OS04:HolomorphicDisks}), but it is easier to use
the cylindrical reformulation of Heegaard Floer homology due to
Lipshitz~\cite{Lipshitz06:CylindricalHF}. Most of the material in this
section is an elaboration on Lipshitz's proof of stabilization invariance
in~\cite[Section~12]{Lipshitz06:CylindricalHF}.

Lipshitz showed that the Heegaard Floer chain complexes can also be defined
by counting holomorphic curves in $\Sigma \times [0,1] \times \R$.
We endow this manifold with the split symplectic form
$\omega = dA + ds \wedge dt$ for an area form~$dA$ on $\S$,
where $s$ is the coordinate on $[0,1]$ and $t$ is the coordinate on $\R$.
The holomorphic triangle maps described by Ozsv\'{a}th and Szab\'{o}
are replaced in his theory with counts of pseudo-holomorphic curves mapping
into $\Sigma \times \Delta$, where $\Delta$ is the subset of the complex plane shown
in Figure~\ref{fig::5}, viewed as having three cylindrical ends of the form
$[0,1] \times [0,\infty)$ or $[0,1] \times (-\infty, 0]$.

\begin{figure}[ht!]
\centering
\input{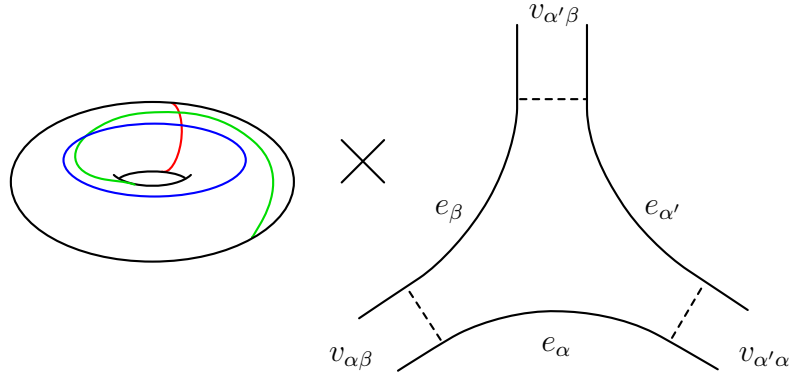}
\caption{An example of a 4-manifold $\Sigma \times \Delta$
used to define the triangle maps in the cylindrical formulation.
\label{fig::5}}
\end{figure}

To prove stabilization invariance,
Lipshitz~\cite[p.~959]{Lipshitz06:CylindricalHF} considers almost complex
structures $J$ on $\Sigma \times [0,1]\times \R$ that satisfy the following
five axioms:

\begin{enumerate}
\item[$(J1)$] $J$ is tamed by $\omega$.
\item[$(J2)$] $J$ is split in a cylindrical neighborhood of $P \times [0,1] \times \R$,
    where $P \subset \Sigma \setminus (\alphas\cup \betas)$ is a finite collection of points
    with at least one point in each component of $\Sigma\setminus (\alphas\cup \betas)$.
\item[$(J3)$] $J$ is translation invariant in the $\R$-factor.
\item[$(J4)$] $J(\partial/\partial s) = \partial/\partial t$.
\item[$(J5')$] There is a 2-plane distribution $\xi$ on $\Sigma \times [0,1] \times \{0\}$
    such that the restriction of $\omega$ to $\xi$ is non-degenerate, $J$ preserves $\xi$,
    and the restriction of $J$ to $\xi$ is compatible with $\omega$.
    We further assume that $\xi$ is tangent to $\Sigma$ near $(\alphas \cup \betas)\times [0,1] \times \{0\}$
    and near $\Sigma \times \{0,1\} \times \{0\}$.
\end{enumerate}

Lipshitz often considers an axiom~$(J5)$ that is more restrictive than
$(J5')$, which requires that the 2-planes $T(\Sigma\times \{(s,t)\})$ are
complex lines inside $\Sigma \times [0,1]\times \R$ for $(s,t) \in [0,1]
\times \R$. Axiom~$(J5')$ is used to ensure that certain curves achieve
transversality in the proof of stabilization invariance.

For the proof of Proposition~\ref{lem:handleswapinvariance}, we will consider
almost complex structures~$J$ on  $\Sigma \times \Delta$ that satisfy the
following:

\begin{enumerate}
\item[$(J'1')$] $J$ is tamed by the split symplectic form on $\Sigma \times \Delta$.
\item[$(J'2')$] There is a finite collection of points 
    $P \subset \Sigma \setminus (\alphas'\cup \alphas\cup \betas)$,
    with at least one point in each component of  $\Sigma \setminus (\alphas' \cup \alphas\cup \betas)$,
    such that the almost complex structure is split on a product neighborhood of $P \times \Delta$;
    i.e., $J=\mathfrak{j}_\Sigma \times \mathfrak{j}_\Delta$.
\item[$(J'3')$] Near  the cylindrical ends of $\Delta$,
    the almost complex structure~$J$ agrees with cylindrical almost complex structures
    on $\Sigma \times [0,1] \times \R$ satisfying condition $(J5')$ above.
\item[$(J'4')$] The 2-planes $T_d(\{p\}\times \Delta)$ are complex lines
    of $J$ for all $(p,d) \in \Sigma \times \Delta$.
\item[$(J'5')$] The 2-planes $T_p(\Sigma \times \{d\})$ are complex lines of $J$
    for $(p,d)$ near $(\alphas' \cup \alphas \cup \betas) \times \Delta$ and
    for $(p,d) \in \Sigma \times U$, where $U\subset \Delta$ is an open subset containing
    $\partial \Delta \setminus \{v_{\alpha'\alpha},v_{\alpha\beta},v_{\alpha'\beta}\}$.
\end{enumerate}

Lipshitz considers axioms $(J'1)$--$(J'4)$
in~\cite[Section~10.2]{Lipshitz06:CylindricalHF} for the triangle maps. Our
axioms $(J'1')$ and $(J'2')$ coincide with $(J'1)$ and $(J'2)$, respectively,
but $(J'3')$--$(J'5')$ together are more generic than the axioms $(J'3)$ and
$(J'4)$, and are analogous to the axiom $(J5')$ that Lipshitz uses for
holomorphic strips.

Note that the projections $\pi_\Delta \colon \S \times \Delta \to \Delta$ and
$\pi_\Sigma \colon \S \times \Delta \to \S$ are not necessarily holomorphic
maps for almost complex structures satisfying the above axioms. Nonetheless,
if $u \colon S \to \S \times \Delta$ is a $J$-holomorphic curve, then the map
$\pi_\Sigma \circ u$ is either constant, or an open map:

\begin{lemma}\label{lem:openmapping}
Let $J$ be an almost complex structure on $\Sigma \times \Delta$ that
satisfies axioms $(J'1')$--$(J'5')$. If $u \colon S \to \S \times \Delta$ is
$J$-holomorphic and $\pi_\Sigma \circ u$ is nonconstant on a component~$S_0$
of~$S$, then $\pi_\Sigma\circ u|_{S_0}$ is an open map. In fact, there are
coordinates near any critical point of $\pi_\Sigma\circ u|_{S_0}$ where
$\pi_\Sigma \circ u$ takes the form $z \mapsto z^k$ for some $k>0$.
\end{lemma}

\begin{proof}
By condition $(J'4')$, the fiber bundle $\pi_\Sigma \colon \Sigma \times
\Delta\to \Sigma$ has holomorphic fibers. The hypotheses of \cite[Lemma~3.1]
{Lipshitz06:CylindricalHF} are satisfied, immediately yielding the result.
\end{proof}

In the cylindrical setting, a (possibly nodal) holomorphic triangle in the triple diagram
$(\S,\alphas',\alphas,\betas)$ with $d = |\alphas'| = |\alphas| = |\betas|$ is a map
$u \colon S \to \Sigma \times \Delta$ that satisfies the following:
\begin{enumerate}
\item[$(M1)$] $(S,j)$ is a (possibly nodal) Riemann surface with boundary and $3d$ punctures on~$\partial S$.
\item[$(M2)$] $u$ is locally nonconstant and $(j,J)$-holomorphic.
\item[$(M3)$] $u(\partial S)\subset (\alphas'\times e_{\alpha'})\cup  (\alphas\times e_\alpha)\cup (\betas\times e_\beta) $.
\item[$(M4)$] $u$ has finite energy.
\item[$(M5)$] For each $i \in \{\,1, \dots, d\,\}$ and $\sigma \in \{\,\alpha',\alpha,\beta\,\}$, the preimage
$u^{-1}(\sigma_i \times e_\sigma)$ consists of exactly one component of the punctured boundary of~$S$.
\item[$(M6)$] As one approaches the punctures along $\partial S$, the map $u$ converges to a collection of Reeb chords
     (i.e., intersection points on the Heegaard triple) in the cylindrical ends of $\Sigma \times \Delta$.
\end{enumerate}
For the holomorphic triangle maps, we additionally impose the following axioms:
\begin{enumerate}
\item[$(M7)$] $\pi_\Delta\circ u$ is nonconstant on each component of $S$.
\item[$(M8)$] $S$ is smooth (i.e., not nodal), and the map
    $u \colon S \to \Sigma \times\Delta$ is an embedding.
\end{enumerate}
If $\psi$ is a homology class of triangles on a Heegaard triple
$\mathcal{T}$, we write $\mathcal{M}(\psi)$ for the moduli space of
holomorphic maps $u \colon S \to \Sigma \times \Delta$ that satisfy
$(M1)$--$(M6)$. If $S$ is a decorated surface (possibly with nodes), then we
write $\mathcal{M}(\psi,S)$ for the subset of~$\mathcal{M}(\psi)$ consisting
of holomorphic curves with underlying source~$S$. We will see that if $\psi$
is a Maslov index~0 homology class of triangles, then, for a generic choice
of almost complex structure $J$ satisfying $(J'1')$--$(J'5')$, any
$J$-holomorphic map $u \colon S \to \Sigma \times \Delta$ representing $\psi$,
and satisfying $(M1)$--$(M6)$ also satisfies $(M7)$ and $(M8)$; see
Equation~\eqref{eq:expecteddimension} below and the surrounding discussion.

Maps $u \colon S \to \Sigma \times [0,1] \times \R$ that
satisfy the obvious analogs of $(M1)$--$(M6)$ are called holomorphic strips.
We will need to consider compactifications of spaces of holomorphic triangles.
To this end, we make the following definition:

\begin{definition}\label{def:brokenholtriangle}
If $\mathcal{T} = (\Sigma,\alphas',\alphas,\betas)$ is a triple diagram, a
\emph{broken holomorphic triangle} on $\mathcal{T}$ representing a homology
class $\psi$ of triangles is a collection of (possibly nodal, non-embedded)
$(j,J)$-holomorphic curves $(u_1,v_1,\dots, v_n,w_1,\dots, w_m)$ such that
the following hold:
\begin{enumerate}
\item[$(BT1)$] $u_1$ maps into $\Sigma \times \Delta$, and satisfies $(M1)$ and $(M3)$--$(M6)$.
\item[$(BT2)$] $v_1, \dots, v_n$ map into $\Sigma \times [0,1] \times \R$,
    satisfy the analogs of $(M1)$ and $(M3)$--$(M6)$, and represent homology classes of disks
    in the diagrams $(\Sigma,\alphas',\alphas)$, $(\Sigma,\alphas',\betas)$, and $(\Sigma,\alphas,\betas)$.
\item[$(BT3)$] $w_1,\dots, w_m$ map into $\Sigma \times \Delta$ and $\Sigma \times [0,1] \times \R$.
    For every $i \in \{1, \dots, m\}$, the curve $w_i$ has $d$ boundary components,
    each with a single puncture, and they all map onto one set of attaching curves (boundary degenerations).
\item[$(BT4)$] The total homology class of the curves $u_1, v_1, \dots, v_n, w_1, \dots, w_m$ is equal to~$\psi$.
\end{enumerate}
\end{definition}

The appropriate notion of convergence of holomorphic triangles to a broken
holomorphic triangle is somewhat involved; see~\cite{BEHWZ03}
and~\cite{Abbas:CompactSympField}. Roughly speaking, the source curve can
degenerate along arcs connecting boundary components, or simple closed curves
along the interior, as in the Deligne-Mumford compactification of stable
curves. Additionally, different parts of the source curve can be sent off
to~$\infty$ in the cylindrical ends, resulting in a holomorphic building,
with one story in $\Sigma \times \Delta$, and all other stories in one of the
three cylindrical ends of $\Sigma \times \Delta$. Finally, constant
components might appear.

In our definition of a broken holomorphic triangle, for notational
simplicity, we have omitted the relative ordering of different stories of the
resulting holomorphic building. For the purposes of this section, this
additional structure will not be important, as we will only encounter broken
triangles with one or two stories.

\begin{proposition}\label{prop:gromovcompactness}
Let $J$ be an almost complex structure on $\S \times \Delta$ satisfying $(J'1')$--$(J'5')$.
If $u_i$ is a sequence of $J$-holomorphic triangles satisfying $(M1)$--$(M6)$ in
$\mathcal{M}(\psi)$ for a homology class $\psi$, then there is a subsequence
that converges to a broken holomorphic triangle.
\end{proposition}

The above result follows by adapting the arguments of
Lipshitz~\cite[Section~7]{Lipshitz06:CylindricalHF} to holomorphic triangles,
and follows from standard results about compactness in symplectic field
theory; see Bourgeois, Eliashberg, Hofer, Wysocki, and Zehnder~\cite{BEHWZ03}
and Abbas~\cite{Abbas:CompactSympField}. For analogous results in a closely
related context, see Lipshitz, Ozsv\'ath, and
Thurston~\cite[Section~5.4]{LOT1}.

We need an additional compactness result that concerns the neck stretching
procedure. Suppose we are given almost complex structures~$J$ and~$J_0$
on~$\Sigma \times \Delta$ and~$\Sigma_0 \times \Delta$ that are each split on
sets $D \times \Delta$ and $D_0 \times \Delta$, for disks $D \subset \Sigma$
and $D_0 \subset \Sigma_0$ containing the points $p$ and $p_0$, respectively.
We can form an almost complex structure~$J(T)$ on~$(\Sigma\# \Sigma_0) \times
\Delta$ by inserting a connected sum tube of length~$T$ between~$\Sigma$
and~$\Sigma_0$. We need to understand how a sequence of $J(T)$-holomorphic
triangles behaves as~$T \to \infty$.

\begin{proposition}\label{prop:weaklimitsoftriangles}
If $u_{T_i}$ is a sequence of holomorphic triangles on $(\Sigma \# \Sigma_0)
\times \Delta$ for a sequence of almost complex structures $J(T_i)$ with $T_i
\to \infty$, then a subsequence can be extracted that converges to a triple
$(U, V, U_0)$, where $U$ is a broken holomorphic triangle in $\Sigma\times
\Delta$ and $U_0$ is a broken holomorphic triangle in $\Sigma_0\times
\Delta$. Furthermore, $V$ is a collection of holomorphic curves mapping into
the neck regions $S^1 \times \R \times \Delta$ or $S^1 \times \R \times [0,1]
\times \R$ that are asymptotic to possibly multiply covered Reeb orbits of the form $S^1\times \{d\}$
in $S^1 \times \Delta$ or $S^1 \times [0,1] \times \R$ for $d \in \Delta$ or
$d \in [0,1] \times \R$, respectively.
\end{proposition}

The splitting process is described in detail by
Lipshitz~\cite[Appendix~A]{Lipshitz06:CylindricalHF} and follows from
adapting \cite[Proposition~12.4  or Sublemma~A.12]{Lipshitz06:CylindricalHF}
to the setting of holomorphic triangles. See also Ozsv\'ath and
Szab\'o~\cite[Section~5.1]{OS08:HFL} and Lipshitz, Ozsv\'ath, and
Thurston~\cite[Proposition~5.24]{LOT1}. The process of splitting a symplectic
manifold along a hypersurface is considered in more generality in
\cite[Section~9]{BEHWZ03}.

We note that one could enhance the previous proposition to keep track of the
\emph{level splittings} (see \cite[p.~993]{Lipshitz06:CylindricalHF}) of the
curves appearing in the limit, in analogy with the notion of holomorphic twin
tower appearing in \cite{BEHWZ03} and
\cite[Proposition~12.4]{Lipshitz06:CylindricalHF}. We could also keep track
of how curves in $U$ and $U_0$ that are in the same level splitting match up
across the connected sum point, somewhat analogous to the notion of a
holomorphic comb appearing in \cite{LOT1}, or the more complicated limits of
holomorphic polygonal combs appearing in the setting of bordered Heegaard
Floer homology in \cite[Section~5.6]{Bordered-branched}. There is a simple
reason why we omit this data from our definition: In our proof of handleswap
invariance, we will reduce to the case when the broken triangles on $\Sigma
\times \Delta$ and $\Sigma_0 \times\Delta$ each consist of a single
holomorphic triangle satisfying $(M1)$--$(M8)$.

Consider the triple diagram $(\S,\alphas',\alphas,\betas)$ and fix a point $p
\in \Sigma \setminus (\alphas' \cup \alphas \cup \betas)$. Using
Lemma~\ref{lem:openmapping}, if $u \colon S \to \Sigma \times \Delta$ is a
holomorphic curve satisfying $(M1)$--$(M6)$, then condition $(J'4')$ ensures
that $\{p\} \times \Delta$ and the image of $u$ intersect at only finitely
many points, and all intersections are positive. We write $(\pi_\Sigma \circ
u)^{-1}(p) = (x_1, \dots, x_{n_p(u)}) \in \Sym^{n_p(u)}(S)$, where $x_i$
appears with multiplicity~$m$ if it is a branch point of $\pi_\S \circ u$ of
order~$m$, and define
\begin{equation} \label{eqn:divisor}
\rho^{p}(u)= \left(\pi_\Delta \circ u(x_1), \dots, \pi_\Delta \circ u(x_{n_p(u)}) \right)
\in \Sym^{n_p(u)}(\Delta).
\end{equation}
Note that this definition makes sense even if $u$ is nodal: If a node~$x$ of
the source is mapped to $\{p\} \times \Delta$, then $\pi_\Delta \circ u(x)$
appears with multiplicity $m_1 + m_2$ in $\rho^p(u)$, where $m_1$ and $m_2$
are the multiplicities of branching of $\pi_\Sigma \circ u$ along the two
sheets.

\begin{definition} \label{def:mathcing}
Suppose $\psi$ is a homology class of triangles and $n_p(\psi) = k$.
Given a subset $X \subset \Sym^k(\Delta)$, we define
$\mathcal{M}(\psi,X)$ to be the moduli space of holomorphic curves in the
homology class $\psi$ that satisfy $(M1)$--$(M6)$ and match $X$ at the marked
point~$p$:
\[
\mathcal{M}(\psi,X) = \{\,u \in \mathcal{M}(\psi) \,\colon\, \rho^p(u) \in X \,\}.
\]
Similarly, if $S$ is a decorated source curve (smooth or nodal),
then we define the moduli space
$\mathcal{M}(\psi, S, X)$ to be the moduli space of holomorphic curves
in the homology class~$\psi$ with source curve~$S$ that match $X$ at the marked point~$p$.
\end{definition}

\begin{definition}
The \emph{fat diagonal} $\Diag^k(\Delta)$ in $\Sym^k(\Delta)$ is the subset consisting
of elements of $\Sym^k(\Delta)$ with at least one repeated entry.
\end{definition}

\begin{definition}
We say a holomorphic curve $u : S \to \Sigma \times \Delta$ with connected
source~$S$ is \emph{somewhere injective} if there is a point $z \in S$ such
that $du(z)\neq 0$ and $u^{-1}(u(z)) = \{z\}$. Such a point $z \in S$
satisfying the above condition is an \emph{injective point}.
\end{definition}

\begin{lemma} \label{lem:notmultiplycovered}
Let $(\S,\alphas',\alphas,\betas)$ be a triple diagram, and fix a point $p
\in \Sigma \setminus (\alphas' \cup \alphas \cup \betas)$. Suppose that
$\mathbf{d} \in \Sym^k(\Delta)$ is an element that is not in the fat
diagonal. Let $u \colon S \to \S \times \Delta$ be a holomorphic curve
satisfying $(M1)$--$(M6)$ for an almost complex structure satisfying
$(J'1')$--$(J'5')$ on $\Sigma$ such that $\rho^{p}(u) = \mathbf{d}$. Then
each component of $u$ is somewhere injective.
\end{lemma}

\begin{proof}
A modification of \cite[Lemma~3.3]{Lipshitz06:CylindricalHF} shows that any
component of~$u$ that is asymptotic at a puncture to a Reeb chord in an end
of $\Sigma \times \Delta$ is somewhere injective. Thus the only components
left to consider are ones that have no boundary components; i.e., closed
components mapping into the interior of $\Sigma \times \Delta$. By
assumption, $(\pi_\Sigma \circ u)^{-1}(p)$ contains exactly $|\mathbf{d}|$
points. Using condition $(J'4')$, the fact that $u$ has no constant
components, positivity of intersections of pseudo-holomorphic curves, and the
assumption that the entries of~$\mathbf{d}$ are distinct, we see that
$(\pi_\Sigma\circ u)^{-1}(q)$ contains $|\mathbf{d}|$ points for any $q$
close to $p$. This implies that $du$ does not vanish at any of the points in
$(\pi_\Sigma\circ u)^{-1}(p)$, and each point $x \in (\pi_\Sigma\circ
u)^{-1}(p)$ has the property that $u(z)\neq u(x)$ for any $z \neq x$ in the
source. Since each closed component must have nontrivial intersection with
$\{p\}\times \Delta$ by Lemma~\ref{lem:openmapping}, we conclude that each
component of~$u$ is somewhere injective.
\end{proof}

We now briefly discuss the index of holomorphic curves in the cylindrical
setting. If $\psi$ is a homology class of triangles, then one can define the
Maslov index of $\psi$, denoted~$\mu(\psi)$, in analogy to the original
setting of Heegaard Floer homology. Similarly, if $u \colon S \to
\Sigma\times \Delta$ is a holomorphic triangle representing $\psi$, one can consider the Fredholm
index of the linearized $\bar{\partial}$ operator at $u$, for which we write
$\ind(\psi, S)$. The Fredholm index gives the expected dimension of the moduli
space $\mathcal{M}(\psi, S)$. According to
Lipshitz~\cite[Section~4]{Lipshitz06:CylindricalHF}, the Maslov index
$\mu(\psi)$ agrees with the Fredholm index $\ind(\psi, S)$ at a holomorphic disk
or triangle $u$ that has smooth source and is embedded. For a curve $u$ that
is immersed, adapting \cite[Proposition~5.69]{LOT1}, we instead have
\begin{equation}
\ind(\psi, S) = \mu(\psi) - 2 \Sing(u),
\label{eq:expecteddimension}
\end{equation}
where $\Sing(u)$ denotes the order of singularity of~$u$. The
quantity~$\Sing(u)$ is nonnegative for $u$ holomorphic, and is zero if and
only if $u$ is embedded. A double point along the interior of $S$ contributes
$+1$ to $\Sing(u)$. A double point along the boundary of $S$ contributes
$+\tfrac{1}{2}$ to $\Sing(u)$. See also \cite[Proposition~4.2']{Lipshitz06:CylindricalHFErrata}.

We state the following transversality result, and sketch the proof:

\begin{proposition} \label{prop:transversality}
Let $(\S,\alphas',\alphas,\betas)$ be a triple diagram, and fix a point $p
\in \Sigma \setminus (\alphas' \cup \alphas \cup \betas)$. Suppose $X \subset
\Sym^k(\Delta)$ for some $k \in \NN$  is a non-empty submanifold that does
not intersect the fat diagonal. Furthermore, suppose that, for every $x \in
X$, the $k$-tuple $x$ has no coordinate in the open set $U \subset \Delta$
from $(J'5')$ that contains $\partial \Delta \setminus \{v_{\alpha'\alpha},
v_{\alpha\beta},v_{\alpha'\beta}\}$. Then, for a generic choice of almost
complex structure~$J$, the set $\mathcal{M}(\psi,S,X)$ is a smooth manifold
of dimension
\[
\ind(\psi,S)-\codim(X),
\]
where $\ind(\psi,S)$ denotes the Fredholm index.

When $X = \Sym^k(\Delta)$, the same statement holds near any curve $u$ that
has no component $T$ on which $\pi_\Delta \circ u|_T$ is constant and has
image in~$U$, and such that all components of $u$ are somewhere injective.
\end{proposition}

\begin{proof}[Sketch of proof]
We first note that if $X$ does not intersect the fat diagonal in
$\Sym^k(\Delta)$, then, by Lemma~\ref{lem:notmultiplycovered}, each curve $u$
that satisfies $\rho^p(u) \in X$ for some $p \in \S \setminus (\alphas' \cup
\alphas \cup \betas)$ is automatically somewhere injective with no assumption
on the almost complex structure being generic.

Proposition~\ref{prop:transversality} follows by modifying
\cite[Proposition~3.7]{Lipshitz06:CylindricalHF} to take into account the map
$\rho^p \colon \mathcal{M}(\psi,S)\to \Sym^k(\Delta)$. The required changes
to the proof of \cite[Proposition~3.7]{Lipshitz06:CylindricalHF} can be
adapted from \cite[Proposition~3.4.2]{MS:JHolandSympTop}, as we briefly
explain. One considers the universal moduli space $\mathcal{B}^\ell(\psi,S)$
of triples $(u,j,J)$, where $j$ is a complex structure  on $S$, while $J$ is
a $C^\ell$ almost complex structure on $\Sigma\times \Delta$ satisfying
$(J'1')$--$(J'5')$, and $u$ is a $(j,J)$-holomorphic map from $S$ to $\Sigma
\times \Delta$ with each component of $u$ somewhere injective. By
Lemma~\ref{lem:openmapping}, around a critical point, $\pi_\Sigma \circ u$ is
of the form $z \mapsto z^n$ for some $n > 0$, hence the set of critical
points of $\pi_\Sigma \circ u$ is discrete. As the set of injective points
of~$u$ is open and dense, we can find an injective point of~$u$ in each
component of~$S$ where $d(\pi_\Sigma \circ u)$ is non-zero. Adapting the
proof of \cite[Proposition~3.7]{Lipshitz06:CylindricalHF}, one can use this
to show that $D\bar{\partial}$ is surjective at $(u,j,J)$, and hence that
$\mathcal{B}^\ell(\psi,S)$ is a Banach manifold.

The next step is to show that the universal evaluation map
\[
\rho^p \colon \mathcal{B}^\ell(\psi,S) \to \Sym^k(\Delta)
\]
is a submersion at triples $(u,j,J)$ where $\rho^p(u)$ is not in the fat
diagonal. The fact that the map $\rho^p$ is a submersion when
$\rho^p(u)$ is not in the fat diagonal can be proven by modifying
\cite[Proposition~3.4.2]{MS:JHolandSympTop}.
For almost complex structures satisfying $(J'1')$--$(J'5')$,
this fact can also be proven by hand, as we now describe.
Note that, since we are using the split symplectic form on $\Sigma \times
\Delta$, if $\phi_\Delta$ is a symplectomorphism of $\Delta$, then the map
$\text{Id}_\S \times \phi_\Delta$ will be a symplectomorphism of $\Sigma
\times \Delta$. Furthermore, $(\text{Id}_\S \times \phi_\Delta)^*(J)$
satisfies $(J'1')$--$(J'5')$ for any $J$ satisfying $(J'1')$--$(J'5')$.
Suppose that $\rho^p(u) = \mathbf{d}$, where $\mathbf{d}$ is not in the fat
diagonal. At each $d_i \in \mathbf{d}$, pick an arbitrary vector $v_i \in
T_{d_i} \Delta$. Pick any Hamiltonian function $H$ on $\Delta$ such that the
Hamiltonian vector field $X_H$ satisfies
\[
(X_H)_{d_i}=v_i.
\]
Multiplying $H$ by bump functions, we can assume that $H$ is supported in a
neighborhood of the points $d_i$. Let $\phi_t \colon \Delta \times [-1,1] \to
\Delta$ denote the flow of~$X_H$. Then $\phi_t$ is a symplectomorphism for
each $t$. We define the action of $\text{Id}_\S \times \phi_t$ on
$\mathcal{B}^{\ell}(\psi,S)$ by
\[
(\text{Id}_\S \times \phi_t)^*(u,j,J) := 
\left((\text{Id}_\S \times \phi_t)^{-1} \circ u, j, (\text{Id}_\S \times \phi_t)^*(J) \right).
\]
Then
\[
\frac{d}{dt}\bigg|_{t=0} \rho^p\left((\text{Id}_\S \times \phi_t)^*(u,j,J)\right) = (-v_1,\dots, -v_k),
\]
which shows that $\rho^p \colon \mathcal{B}^{\ell}(\psi,S) \to \Sym^k(\Delta)$ is a submersion at $(u,j,J)$.

We define the universal matched moduli space
\[
\mathcal{B}^{\ell}(\psi,S,X) = \{\,(u,j,J)\in \mathcal{B}^\ell(\psi,S) \,\colon\, \rho^p(u)\in X\}.
\]
Near any $(u,j,J)$ where $\rho^p$ is a submersion, this is a Banach manifold
of codimension $\codim(X)$ in $\mathcal{B}^\ell(\psi,S)$. We write
$\mathcal{J}^\ell$ for the space of $C^\ell$ almost complex structures on
$\Sigma \times \Delta$ that satisfy $(J'1')$--$(J'5')$. If we apply the
Sard-Smale theorem to the Fredholm map $\pi \colon
\mathcal{B}^{\ell}(\psi,S,X) \to \mathcal{J}^\ell$, we obtain that a generic
choice of $J \in \mathcal{J}^\ell$ is a regular value of $\pi$. The Fredholm
index of $\pi$ is $\ind(\psi,S)-\codim(X)$. Hence $\mathcal{M}(\psi,S,X) =
\pi^{-1}(J)$ is a smooth manifold of dimension $\ind(\psi,S)-\codim(X)$ for a
generic choice of~$J$. An approximation argument yields the statement for
$C^\infty$ almost complex structures.

When $X = \Sym^k(\Delta)$, the above proof works given the additional assumptions
in the statement of the proposition, and we do not need to use the universal evaluation map.
\end{proof}

\begin{remark}
The only curves that we cannot achieve transversality for are curves where
$\pi_\Delta \circ u$ is constant and takes a value in the set $U$ containing
$\partial \Delta \setminus \{v_{\alpha'\alpha}, v_{\alpha\beta},
v_{\alpha'\beta}\}$, as well as curves with closed components that are
multiply covered.
\end{remark}

We finally need to discuss gluing results for holomorphic triangles. If
$\psi$ is a homology class of triangles on the triple $\mathcal{T}$ and
$\psi_0$ is a homology class of triangles on $\mathcal{T}_0$, then we can
take the connected sum $\psi \# \psi_0$ of the two homology classes, and
consider the moduli space $\mathcal{M}_{J(T)}(\psi \# \psi_0)$ for large $T$.

Let $p$ and $p_0$ be the connected sum points on $\Sigma$ and $\Sigma_0$,
respectively. If $u\in \mathcal{M}(\psi)$ and $u_0\in \mathcal{M}(\psi_0)$,
we consider the divisors $\rho^p(u)\in \Sym^{n_p(u)}(\Delta)$ and
$\rho^{p_0}(u_0) \in \Sym^{n_{p_0}(u_0)}(\Delta)$ defined in
equation~\eqref{eqn:divisor}. Then we have the following gluing result:

\begin{proposition}\label{prop:fiberedproduct}
Let $u$ and $u_0$ be holomorphic triangles
representing homology classes $\psi$ and $\psi_0$ in  
$\Sigma \times \Delta$ and $\Sigma_0 \times \Delta$
of Maslov indices~0 and~$2k$, respectively. Suppose that
\[
\rho^p(u) = \rho^{p_0}(u_0) \in \Sym^k(\Delta) \setminus \Diag^k(\Delta),
\]
where $k = n_{p}(\psi) = n_{p_0}(\psi_0)$,
and that the moduli spaces $\mathcal{M}(\psi)$ and $\mathcal{M}(\psi_0,\rho^p(u))$
are transversely cut out near~$u$ and~$u_0$, respectively. Then there
is a homeomorphism~$h$ between a neighborhood of $(u, u_0)$ in the
compactified 1-dimensional moduli space
\[
\overline{\bigcup_T \mathcal{M}_{J(T)}(\psi \# \psi_0) }
\]
(cf.~Proposition~\ref{prop:weaklimitsoftriangles}) and $[0,1)$ such that $h(u,u_0) = \{0\}$.
\end{proposition}

The above proposition follows immediately from the work of
Lipshitz~\cite[Proposition~A.2]{Lipshitz06:CylindricalHF}, together with a
computation of the relevant Fredholm indices associated to
$\mathcal{M}(\psi)$ and $\mathcal{M}(\psi_0, \rho^p(u))$ from
Proposition~\ref{prop:transversality}. See also Ozsv\'ath and
Szab\'o~\cite[Theorem~5.1]{OS08:HFL} for a similar result. The conditions
that $\rho^p(u) \not\in \Diag^k(\Delta)$, and that~$u$ has Maslov index zero
ensure that the Reeb orbits involved are
embedded, and that the techniques of Lipshitz~\cite{Lipshitz06:CylindricalHF} apply.
This is the only case that we need. To relax these conditions, one would
need to adapt~\cite[Proposition~A.2]{Lipshitz06:CylindricalHF} to a setting
that allowed for multiply-covered Reeb orbits.

\subsubsection{Proof of Proposition \ref{lem:handleswapinvariance}}

In this section, we prove the main holomorphic triangle count that features in the proof of handleswap invariance.
We first prove several simpler computations that will be used in the proof of the main count.

\begin{lemma}\label{lem:Maslovindex}
Consider the triple diagram $\mathcal{T}_0 = (\S_0,\alphas_0',\alphas_0,\betas_0)$ in Figure~\ref{fig::2}.
If $x \in \mathbb{T}_{\a'_0} \cap \mathbb{T}_{\a_0}$ and $\psi_0 \in \pi_2(x,\mathbf{a},\mathbf{b})$, then
\begin{equation} \label{eqn:firstformula}
\mu(\psi_0) = 2n_{p_0}(\psi_0) + \mu(x,\Theta),
\end{equation}
where $\mu(x,\Theta)$ denotes the relative Maslov grading of $x$ and $\Theta$.
\end{lemma}

\begin{proof}
We first prove equation~\eqref{eqn:firstformula} when $x = \Theta$.
This holds when $\psi_0$ is the Maslov index zero triangle domain shown in Figure~\ref{fig::2}
as $\mu(\psi_0) = 0$ and $n_{p_0}(\psi_0) = 0$.
The formula respects adding a multiple of~$[\Sigma_0]$ as well as triply periodic domains.
Indeed, every triply periodic domain in~$\mathcal{T}_0$ is a sum of doubly periodic domains.
There are no doubly periodic domains in $(\Sigma_0,\alphas_0,\betas_0,p_0)$
or $(\Sigma_0,\alphas'_0,\betas_0,p_0)$,
and every doubly periodic domain in $(\Sigma_0,\alphas'_0,\alphas_0,p_0)$ has Maslov index zero.

For a general $x \in \mathbb{T}_{\a'_0} \cap \mathbb{T}_{\a_0}$, pick a domain $\phi \in \pi_2(x,\Theta)$.
Then $\psi_0 - \phi$ represents an element of $\pi_2(\Theta,\mathbf{a},\mathbf{b})$,
and hence by the above, $\mu(\psi_0) = \mu(\phi) + 2n_{p_0}(\psi_0)$. Finally, note that
$\mu(\phi) = \mu(x,\Theta)$; this is independent of the choice of $\phi$ since
every periodic domain in $(\Sigma_0,\alphas'_0,\alphas_0,p_0)$ has Maslov index zero.
\end{proof}

Before we consider the curves counted by the map $F_{\mathcal{T} \# \mathcal{T}_0}^\circ$ in full generality,
we prove two simple lemmas that will be useful later when we perform more complicated counts of holomorphic triangles.

\begin{lemma}\label{lem:differentialvanishes}
The differential on $\CFa(\Sigma_0,\alphas'_0,\alphas_0,p_0)$ vanishes.
\end{lemma}

\begin{proof}
Observe that $(\Sigma_0,\alphas'_0,\alphas_0)$ represents $(S^1 \times S^2) \# (S^1 \times S^2)$,
and hence
\[
\dim \HFa(\Sigma_0,\alphas'_0,\alphas_0,p_0) = 4
\]
over $\FF_2$.
But $\CFa(\Sigma_0,\alphas'_0,\alphas_0,p_0)$ is also 4-dimensional,
and hence the differential must vanish.
\end{proof}

\begin{lemma}\label{lem:handleswapinhat}
The map
\[
\Psi^{\alphas_0\to \alphas'_0}_{\betas_0} \colon \CFa(\Sigma_0,\alphas_0,\betas_0,p_0) \to
\CFa(\Sigma_0,\alphas'_0,\betas_0,p_0)
\]
satisfies $\Psi^{\alphas_0 \to \alphas'_0}_{\betas_0}(\mathbf{a}) = \mathbf{b}$.
\end{lemma}

\begin{proof}
By Proposition~\ref{prop:Compatibility1}, the map $\Psi^{\alphas_0 \to \alphas'_0}_{\betas_0}$ is
a quasi-isomorphism of 1-dimensional chain complexes over~$\FF_2$ with vanishing differentials,
and such a map is uniquely determined.
\end{proof}

We remark that Lemma~\ref{lem:handleswapinhat} implies handleswap invariance in the hat version
or in the case of sutured Floer homology,
regardless of the almost complex structure,
but only for handleswaps where the region containing~$p$ also contains a basepoint
or a boundary component of the sutured Heegaard diagram.
(As pointed out to us by Sungkyung Kang, general handleswaps 
can be reduced to these using the Commutativity Axiom.)

For the $+$, $-$, and $\infty$ versions, and for general simple handleswaps in
the hat version, we must consider how moduli spaces of holomorphic curves degenerate
as we degenerate the almost complex structure.

\begin{proof}[Proof of Proposition~\ref{lem:handleswapinvariance}]
If $\psi \in \pi_2(\x,\y,\z)$ is a homology class of triangles on 
$\mathcal{T}=(\Sigma, \alphas',\alphas,\betas)$
and $\psi_0 \in \pi_2(x, \mathbf{a},\mathbf{b})$ is a homology class on 
$\mathcal{T}_0 = (\S_0,\alphas_0',\alphas_0,\betas_0)$
with $n_p(\psi) = n_{p_0}(\psi_0)$, then
we can form the homology class $\psi \# \psi_0$. By Lemma~\ref{lem:Maslovindex}
and the Maslov index formula of Sarkar~\cite[Theorem~4.1]{Sarkar:Maslov}, we have
\begin{equation} \label{eqn:Maslov}
\mu(\psi\#\psi_0) = \mu(\psi) + \mu(\psi_0) - 2n_{p_0}(\psi_0) = \mu(\psi) + \mu(x,\Theta),
\end{equation}
since when we take the connected sum of the two homology classes,
we must remove two balls with multiplicity $n_{p_0}(\psi_0) = n_{p}(\psi)$,
and each ball has Euler measure one.

Suppose that
\[
\psi\# \psi_0\in \pi_2(\x\times \Theta, \y\times \mathbf{a},\z\times \mathbf{b})
\]
is a homology class of triangles that has holomorphic representatives
$u_{T_i}$ for a sequence of neck lengths~$T_i$ approaching~$\infty$. By
Proposition~\ref{prop:weaklimitsoftriangles}, we can extract a subsequence
that converges to a broken triangle $U$ on $(\Sigma,
\alphas',\alphas,\betas)$ representing~$\psi$, a broken triangle $U_0$ on
$(\Sigma_0,\alphas'_0,\alphas_0,\betas_0)$ representing~$\psi_0$, and a
collection of holomorphic curves $V$ mapping into the connected sum regions
$S^1\times \R\times \Delta$ or $S^1\times \R\times [0,1]\times \R$. Supposing
that $\mu(\psi\#\psi_0) = 0$, we know from equation~\eqref{eqn:Maslov} that
\begin{equation}
\mu(\psi\# \psi_0) = \mu(\psi)+\mu(\Theta,\Theta) = \mu(\psi),
\label{eq:Maslovindexconnectedsum}
\end{equation}
and hence $\mu([U]) = 0$.

We list the steps of the rest of the proof:
\begin{enumerate}
\item \label{it:step1} We show that $U$ consists of a single Maslov index zero holomorphic triangle~$u$
    satisfying $(M1)$--$(M8)$, as well as potentially some constant holomorphic curves.
\item \label{it:step2} We show that $U_0$ consists of a single Maslov index $2k$ triangle~$u_0$
    satisfying $(M1)$--$(M8)$ and $\rho^{p_0}(u_0) = \rho^p(u)$, as well as potentially
    some constant holomorphic curves.
\item \label{it:step3} We show that $V$ consists of only a collection of trivial cylinders,
    and that there are no constant holomorphic components in $U$ or $U_0$.
\item \label{it:step4} We use the fibered product description of the moduli spaces
    in Proposition~\ref{prop:fiberedproduct} to yield the relevant counts of curves
    on~$\mathcal{T} \# \mathcal{T}_0$ in terms of the counts on~$\mathcal{T}$.
\end{enumerate}

First, let us restrict our attention to the components of~$U$ and~$U_0$ that
are nonconstant. Write $U'$ and $U_0'$ for the collections obtained by
removing components of $U$ and $U_0$, respectively, that map to a single
point. Note that this does not change the homology classes of $U$ and $U'$,
hence $\mu([U]) = \mu([U'])$ and $\mu([U_0]) = \mu([U_0'])$.

\emph{A priori}, the broken triangle $U'$ could consist of curves
mapping into both $\Sigma \times \Delta$ and $\Sigma \times [0,1] \times \R$.
However, holomorphic curves mapping into $\Sigma \times [0,1] \times \R$ that
satisfy the analogs of $(M1)$--$(M6)$ and have nonzero domain on $\Sigma$
have Fredholm index at least one, by transversality, and since there is an
$\R$-action on the moduli space. As equation~\eqref{eq:expecteddimension} also
holds for curves mapping into $\S \times [0,1] \times \R$, we obtain that
these curves must also have Maslov index at least one. Similarly, if we
choose the almost complex structure $J$ on $\S \times \Delta$ generically,
the curves mapping into $\Sigma \times \Delta$ must have nonnegative Fredholm
index, and hence also Maslov index by equation~\eqref{eq:expecteddimension}.
Since $\mu([U']) = \mu([U]) = 0$, the limiting curve must map only into $\Sigma \times \Delta$.
Clearly, the limiting curve in $U'$ must satisfy $(M1)$--$(M4)$ and $(M6)$. Let us
write $u$ for the single holomorphic triangle in $U'$.

We now show that $u$ satisfies $(M5)$. Degeneration of a source curve along a
simple closed curve in the source preserves $(M5)$, though such a
degeneration is impossible since $u$ has Maslov index zero, and nodal curves
would have negative expected dimension by
equation~\eqref{eq:expecteddimension}. Collapsing along an arc connecting two
boundary components of different type (e.g., an $\alpha$ boundary component
and a $\beta$ boundary component) would result in strip breaking, which we
have already ruled out.

The final possibility is that the source collapses along arcs with endpoints
along boundary components of the same type (e.g., two $\alpha$ boundary
components). Since $u|_{\partial S}$ is monotone along each arc of the
punctured boundary for a holomorphic curve with $u|_{\partial S}$
nonconstant, using the maximum modulus principle, it is not hard to see that
some components of the resulting curve must be constant under the
map~$\pi_\Delta$ and map into $e_{\alpha'} \cup e_\alpha \cup e_\beta$. These
correspond to boundary degenerations. The remaining components of the limit
satisfy $(M5)$. However, a boundary degeneration increases the Maslov index
by at least two. For example, an $\alpha$ boundary degeneration has domain
equal to a nonnegative sum of components of $\Sigma \setminus \alphas$, and
upon examining the formula of
Lipshitz~\cite[Proposition~4.10]{Lipshitz06:CylindricalHF}, splicing in such
a class raises the Maslov index by twice the sum of multiplicities. Since
$\mu([u]) = 0$, and the remaining curves are transversely cut out, boundary
degenerations are thus prohibited.  It follows that axiom $(M5)$ holds, as it
is satisfied by all the curves~$u_{T_i}$, and we have ruled out a sequence of
curves collapsing along any arcs or simple closed curves in the source.

Similarly, $(M8)$ is satisfied because we have ruled out nodal source curves
in~$U'$. Condition $(M7)$ holds since holomorphic curves with $\pi_\Delta
\circ u$ constant are either constant curves (which we are excluding, at the
moment), nonconstant closed surfaces, or boundary degenerations, and our
previous discussion rules out the latter two, using Maslov index
considerations. This concludes Step~\eqref{it:step1}.

We now proceed to Step~\eqref{it:step2}.
If $v$ is a holomorphic triangle on $(\S, \alphas', \alphas, \betas)$,
and $\rho^p(v)$ contains a point of multiplicity greater than one, then
$(\pi_\Sigma \circ v)^{-1}(p)$ either contains a double point of $v$, or a branch
point of $\pi_\S \circ v$. For a curve whose projection to $\S$ is locally
non-constant, the set of critical points of $\pi_\S \circ v$ is discrete by
Lemma~\ref{lem:openmapping} (in fact, since $J$ is split in a neighborhood of $p$,
we could just use elementary complex analysis). In particular, by perturbing the connected
sum point $p$ slightly, we can assume that $p$ is not the image of a branch
point of $\pi_\S \circ v$ for any Maslov index zero triangle $v$. Since $u$ is also
embedded, we can assume that $\rho^p(u)$ is not contained in the fat diagonal.

Since $\rho^p(u) \not\in \Diag^k(\Delta)$, the curves appearing in~$V$
(mapping into $S^1 \times \R \times \Delta$ or $S^1 \times \R \times [0,1] \times \R$)
are asymptotic to embedded Reeb orbits of the form $S^1 \times \{d\}$ in
$S^1\times \Delta$ or $S^1\times [0,1] \times \R$ for some $d$ in $\Delta$ or $[0,1] \times \R$.
However, if $v$ is a curve in $V$, then $\pi_\Delta \circ v$ must be constant by the maximum
modulus principle. In particular, the asymptotics of~$V$ agree with the
asymptotics of the curve $u \in U'$. Also, since $U'$ consists of a single
holomorphic triangle, the curves in $V$ can only map into $S^1 \times \R
\times \Delta$.

By Lemma~\ref{lem:Maslovindex}, we have $\mu(\psi_0) = 2n_{p_0}(\psi_0)$. The
broken holomorphic triangle $U_0'$ could \emph{a priori} consist of a
genuinely broken holomorphic triangle, with nodal curves, boundary
degenerations, and  holomorphic strips. The curves in $U_0'$ are obtained by
completing the limiting curves in $(\Sigma_0 \setminus \{p_0\}) \times
\Delta$. By \cite[Section~10.2.3]{BEHWZ03}, curves in adjacent level
splittings must have asymptotics that agree. Since the asymptotics of the
curves in~$V$ agree with those of~$u$, by the previous paragraph, we conclude
that there must be curves in~$U_0'$ that satisfy a matching condition
with~$u$. The matching condition reads
\[
\rho^p(u) = \rho^{p_0}(u_0).
\]
Since $u$ satisfies $(M1)$--$(M8)$, the divisor $\rho^p(u)$ is a collection of $n_p(u)$
points (possibly with repetition) in the interior of~$\Delta$. Let
\[
\mathcal{D}:=
\{\, \rho^p(u) \,\colon\, u \text{ holomorphic on $\mathcal{T}$, } \mu(u) = 0 \,\}.
\]
Since $(\Sigma,\alphas', \alphas,\betas)$ is admissible and $J$ is generic, the set
$\mathcal{D}$ is finite.
As described above, by perturbing the connected sum point $p$ slightly,
we can assume that $\mathcal{D} \cap \Diag^k(\Delta) = \emptyset$.
By Lemma~\ref{lem:notmultiplycovered}, this ensures that if $u_0$ is any (not
necessarily embedded) holomorphic curve in $\Sigma_0\times \Delta$ that
satisfies
\[
\rho^{p_0}(u_0)\in \mathcal{D},
\]
then $u_0$ is somewhere injective on each component.

Thus, for a generic almost complex structure on $\Sigma_0 \times \Delta$,
by Proposition~\ref{prop:transversality},  the moduli space
\[
\mathcal{M}(\psi_0,\mathbf{d}) =
\{\,u_0 \in \mathcal{M}(\psi_0) \,\colon\, \rho^{p_0}(u_0) = \mathbf{d} \,\}
\]
is a transversely cut out manifold
of dimension
\[
\mu(\psi_0) - \codim(\{\mathbf{d}\}) = 2n_{p_0}(\psi_0) - 2n_p(u) = 0
\]
for every $\mathbf{d} \in \mathcal{D}$. Repeating the  argument that we made
previously for the curves appearing in $\Sigma \times \Delta$ shows that, for
a generic almost complex structure on $\Sigma_0 \times \Delta$, the broken
triangle $U_0'$ consists of exactly one curve $u_0 \in
\mathcal{M}(\psi_0,\mathbf{d})$ mapping into $\Sigma_0 \times \Delta$, and
that $u_0$ satisfies $(M1)$--$(M8)$. This concludes Step~\eqref{it:step2}.

We now argue that $V$ consists of only trivial cylinders, and that there are
no non-trivial constant holomorphic curves in $U$ and $U_0$. The argument
follows by adapting standard arguments (for example \cite[Lemma~5.57]{LOT1}).
For a nontrivial constant holomorphic curve to appear in the limit, it must
be stable; i.e., it cannot consist of a sphere with exactly one or two
punctures. The original sequence of curves $u_{T_i}$ had Maslov index zero
and satisfied $(M1)$--$(M8)$. Writing $\widehat{S}$ for the source curve of
$u_{T_i}$ (which we can assume are topologically the same for all $i$, by
passing to a subsequence), by \cite[Section~10.2]{Lipshitz06:CylindricalHF},
we have
\begin{equation}\label{eq:firstindexformulatriangle}
\ind\left(\psi \# \psi_0, \widehat{S}\right) = 
\frac{1}{2} (d+2) - \chi(\widehat{S}) + 2e(\psi\# \psi_0),
\end{equation}
where $d = |\alphas'| = |\alphas| = |\betas|$ is the number of attaching
curves of each type in $\mathcal{T}$. Notice too that none of the constant
curves in~$U$ or~$U_0$ can map to $\{p\} \times \Delta$ or $\{p_0\} \times
\Delta$, since the collections~$U$ and~$U_0$ were obtained by taking curves
in $(\Sigma \setminus \{p\}) \times \Delta$ and $(\Sigma_0 \setminus \{p_0\})
\times \Delta$ that were asymptotic to Reeb orbits in the ends corresponding
to the points~$p$ and~$p_0$, and then completing over these punctures to get
curves in~$\Sigma \times \Delta$ and~$\Sigma_0 \times \Delta$.

Hence, we can be reconstruct the surface $\widehat{S}$ from the curves in
$U$, $V$, and $U_0$ by gluing the source curves along their nodes. First,
glue the source curves of the constant holomorphic curves together along any
shared nodes. Assuming that this has been done, leaving only components $T_1,
\dots, T_n$ of the constant holomorphic curves with no shared nodes, we note
that there can only be one node between any~$T_i$ and~$S$ or~$S_0$, as~$u$
and~$u_0$ are embedded and have no nodal singularities. Since none of the
constant curves appearing in the limit are once or twice punctured spheres,
by stability of the limiting curves, we conclude that each~$T_i$ has genus at
least one, and has at most one node shared with~$S$ or~$S_0$. Hence, gluing
the~$T_i$ to~$S$ and~$S_0$ only raises the genus of~$S$ and~$S_0$. Similarly,
if $V$ contains curves that are not trivial cylinders, it is easy to see that
they must have positive genus. As a consequence, we have
\[
\chi(\widehat{S}) \le \chi(S) + \chi(S_0) - 2k,
\]
with equality if and only if there were no constant curves in~$U$ and~$U_0$,
and~$V$ consisted of only trivial cylinders (the $-2k$ comes from gluing $S$
and~$S_0$ together at the $k = n_p(\psi) = n_{p_0}(\psi_0)$ points where~$u$
and~$u_0$ intersect $\{p\} \times \Delta$ or $\{p_0\} \times \Delta$).
Furthermore, if $V$ consists of anything but trivial cylinders, or if there
are constant curves appearing in the limit, then
\[
\chi(\widehat{S}) \le \chi(S) + \chi(S_0) - 2k - 2.
\]
Applying the formula for the Fredholm index from equation~\eqref{eq:firstindexformulatriangle}
to~$u$ and~$u_0$ implies that
\[
\ind(\psi, S) + \ind(\psi_0, S_0) \le 2k + \ind(\psi \# \psi_0, \widehat{S}) - 2 = 2k-2.
\]
By assuming that $\mathcal{M}(\psi)$ and $\mathcal{M}(\psi_0,\mathbf{d})$ are
smoothly cut out and of the expected dimension for any $\mathbf{d} \in
\mathcal{D}$, if constant components appear in~$U$ or~$U_0$, or, if~$V$
contains nontrivial curves, then one of the expected dimensions
of~$\mathcal{M}(S,\psi)$ and~$\mathcal{M}(S_0,\psi_0,\mathbf{d})$ is
negative.  Hence, constant components of~$U$ and~$U_0$, as well as nontrivial
curves in~$V$, do not arise generically. This concludes
step~\eqref{it:step3}.

The above limiting argument shows that the sequence $u_{T_i}$
converges to a pair $(u,u_0)$, where $u$ has Maslov index zero, $u_0$ has
Maslov index $2k$, and $\rho^p(u) = \rho^{p_0}(u_0)$. On the other
hand, Proposition~\ref{prop:fiberedproduct} describes a neighborhood of $(u,u_0)$ in the
compactified 1-dimensional moduli space
$\overline{\bigcup_T M_{J(T)}(\psi \# \psi_0)}$, from which we conclude that
\begin{equation} \label{eq:fiberedproductdescription}
\# \mathcal{M}_{J(T)}(\psi\# \psi_0) = \# \{\,(u,u_0)\in
\mathcal{M}(\psi) \times \mathcal{M}(\psi_0) \,\colon\, \rho^p(u)=\rho^{p_0}(u_0)\,\}
\end{equation}
for $T$ sufficiently large.
Thus, to prove Proposition~\ref{lem:handleswapinvariance},
it is sufficient to count triangles $u_0$ on $(\Sigma_0,\alphas'_0,\alphas_0,\betas_0)$
that satisfy a matching condition with a generic divisor $\mathbf{d} \in \Sym^k(\Delta)$
for some $k \in \NN$. For $x \in \mathbb{T}_{\a'_0} \cap \mathbb{T}_{\a_0}$, we define
\[
\mathcal{M}_{(x,\mathbf{a},\mathbf{b})}(\mathbf{d}) =
\coprod_{\substack{\psi_0\in \pi_2(x,\mathbf{a},\mathbf{b}) \\ 
n_{p_0}(\psi_0) = k}} \mathcal{M}(\psi_0,\mathbf{d}).
\]
Note that, by using the expected dimension from
equation~\eqref{eq:expecteddimension}, for a $J$ achieving transversality,
any curve in $\mathcal{M}(\psi_0,\mathbf{d})$ satisfying $(M1)$--$(M6)$ also
satisfies~$(M7)$ and~$(M8)$.

Summarizing, we have shown that, for sufficiently large $T$, the only Maslov
index zero homology classes that have representatives are of the form $\psi
\# \psi_0$ with $\mu(\psi) = 0$ and $\mu(\psi_0) = 2n_{p_0}(\psi_0) =
2n_{p}(\psi)$. Using the fibered product description of the moduli spaces
$\mathcal{M}(\psi \# \psi_0)$ in
equation~\eqref{eq:fiberedproductdescription}, to prove
Proposition~\ref{lem:handleswapinvariance}, it is sufficient to show that for
each $u \in \mathcal{M}(\psi)$, we have
\[
\#\mathcal{M}_{(\Theta,\mathbf{a},\mathbf{b})}(\rho^p(u))\equiv 1 \pmod{2}.
\]
This follows from Lemma~\ref{lem:divisorcount} below.
\end{proof}

\begin{lemma}\label{lem:divisorcount}
For $\mathbf{d} \in \Sym^k(\Delta)$ not contained in the fat diagonal, the
moduli space $\mathcal{M}_{(\Theta,\mathbf{a},\mathbf{b})}(\mathbf{d})$ is a
smoothly cut out 0-manifold for a generic choice of almost complex
structure~$J$. Furthermore, for $J$ achieving transversality, we have
\[
\# \mathcal{M}_{(\Theta,\mathbf{a},\mathbf{b})}(\mathbf{d})\equiv 1 \pmod 2.
\]
\end{lemma}

\begin{proof}
Note that, for a fixed $\mathbf{d}\in \Sym^k(\Delta)$ that is not in the fat
diagonal in $\Sym^k(\Delta)$ and also includes no points in the
neighborhood~$U$ of $\partial \Delta \setminus
\{v_{\alpha'\alpha},v_{\alpha\beta},v_{\alpha'\beta}\}$ appearing in
axiom~$(J'5')$, transversality of the matched moduli space
$\mathcal{M}_{(\Theta,\mathbf{a},\mathbf{b})}(\mathbf{d})$ can be achieved by
Proposition~\ref{prop:transversality}. In fact, since the set of such $J$ is
of the second category by the Sard-Smale theorem, one can pick a~$J$
achieving transversality simultaneously for any countable collection of
submanifolds of $\Sym^k(\Delta)$ that do not intersect the fat diagonal and
$U \times \Sym^{k-1}(\Delta)$.

We first show that the quantity $\#
\mathcal{M}_{(\Theta,\mathbf{a},\mathbf{b})}(\mathbf{d})$ is independent
of~$\mathbf{d}$ for generic~$\mathbf{d}$. Let $\mathbf{p} \colon [0,1] \to
\Sym^k(\Delta)$ be an embedded arc starting at $\mathbf{d}_0$ and ending at
$\mathbf{d}_1$ that does not intersect the fat diagonal and $U \times
\Sym^{k-1}(\Delta)$. Consider the 1-dimensional moduli space
\[
\mathcal{M}_{(\Theta,\mathbf{a},\mathbf{b})}(\mathbf{p})= \bigcup_{t\in [0,1]} \mathcal{M}_{(\Theta,\mathbf{a},\mathbf{b})}(\mathbf{p}(t)).
\]
We will count the ends of $\mathcal{M}_{(\Theta,\mathbf{a},\mathbf{b})}(\mathbf{p})$.
First, note that there are ends corresponding
to~$\mathcal{M}_{(\Theta,\mathbf{a},\mathbf{b})}(\mathbf{d}_i)$ for $i \in \{0,1\}$.

There are additional ends of
$\mathcal{M}_{(\Theta,\mathbf{a},\mathbf{b})}(\mathbf{p})$ that correspond to
degenerations of holomorphic triangles into broken holomorphic triangles, in
the sense of Definition \ref{def:brokenholtriangle}. We will show that, by
picking $J$ generically, the only degenerations that occur correspond to a
sequence of Maslov index $2k$ holomorphic triangles breaking into a Maslov
index $2k-1$ triangle that matches a divisor $\mathbf{p}(t)$ and a Maslov
index~1 holomorphic disk that has multiplicity zero at~$p_0$.

Let $u_i \colon S_0 \to \S_0 \times \Delta$ for $i \in \NN$ be a sequence of
triangles in $\mathcal{M}_{(\Theta,\mathbf{a},\mathbf{b})}(\mathbf{p})$ with
domain~$\psi_0$. We show first that the sequence~$u_i$ cannot degenerate
along a closed curve~$c$ in the interior of~$S_0$ as $i \to \infty$ for a
generic ~$J$. Note that a closed surface bubbles off if $c$ is separating,
and a source curves with an interior node appears in the limit if $c$ is
non-separating. Since the path $\mathbf{p}$ misses the fat diagonal and $U
\times \Sym^{k-1}(\Delta)$, the moduli space
$\mathcal{M}(\psi_0,S_0,\mathbf{p}) :=
\mathcal{M}(\psi_0,S_0,\text{Im}(\mathbf{p}))$ (see
Definition~\ref{def:mathcing}) is a manifold of the expected dimension by
Proposition~\ref{prop:transversality}. If $S_0$ collapses along a simple
closed curve in the interior, we note that the limiting curves must map the
node corresponding to the collapsed curve to a point in $\Sigma_0 \times
\Delta$ or $\Sigma_0 \times [0,1] \times \R$ (as opposed to being asymptotic
to a Reeb orbit in the ends) since there are no Reeb orbits in the ends, only
Reeb chords. Thus, if $S_0$ collapses along a closed curve in the interior,
the resulting source curve must have a nodal singularity in the interior.
Suppose that the resulting node occurs in a curve mapping into $\Sigma_0
\times \Delta$ (the case that the curve maps into $\Sigma_0 \times [0,1]
\times \R$ is handled analogously). Let $S_0'$ denote the collection of all
source curves in the broken triangle mapping into $\Sigma_0 \times \Delta$,
and write~$\psi_0'$ for the associated homology class. Since~$S_0'$ has at
least one nodal singularity,
\[
\ind(\psi_0',S_0') \le \mu(\psi_0') - 2\le \mu(\psi_0) - 2 = 2k-2
\]
by equation~\eqref{eq:expecteddimension}.
Hence, we have
\[
\dim \mathcal{M}(\psi_0',S_0',\mathbf{p}) = \ind(\psi_0',S_0')-(2k-1) \le 
(2k-2)-(2k-1) = -1,
\]
so $\mathcal{M}(\psi_0',S_0',\mathbf{p})$ is empty for a generic~$J$. We
conclude that, for a generically chosen almost complex structure on
$\Sigma_0\times \Delta$, 
$\mathcal{M}(\psi_0,\mathbf{p})$ contains no curves with nodes along the
interior, the limit of any sequence of curves in
$\mathcal{M}(\psi_0,\mathbf{p})$ has no nodes along the interior. Note that,
in the above argument, we assumed transversality at the limiting curves. By
Proposition~\ref{prop:transversality}, the only possibility for curves to
arise that do not achieve transversality are curves with a component~$u_0$
such that $\pi_\Delta \circ u_0$ is constant with value in~$U$. Such curves
cannot appear as they would need to match a point in $\mathbf{p}(t)$ for some
$i \in [0,1]$, but $\mathbf{p}$ avoids $U \times \Sym^{k-1}(\Delta)$.

The formation of curves with nodes along the boundary is ruled out by a
reasoning similar to the proof of
\cite[Proposition~7.1]{Lipshitz06:CylindricalHF}. By the index calculation of
equation~\eqref{eq:expecteddimension}, since boundary nodes contribute
$+\tfrac{1}{2}$ to $\Sing(u)$, transversality alone does not prohibit them
from forming. Furthermore (and more troublesome), we will not be able to
achieve transversality at such curves, as will be clear from the following
discussion. To show that such curves cannot form, we instead use the matching
condition. Following the proof of~\cite[Proposition
7.1]{Lipshitz06:CylindricalHF}, by applying the maximum modulus principle
near the boundary of the curve, one can see that nodes forming along the
boundary of the source of a holomorphic curve result in a limiting curve that
is mapped entirely into a fiber $\S_0 \times \{x\}$ over a point $x \in
e_{\alpha'} \cup e_\alpha \cup e_\beta \subset \partial \Delta$, or over an
$x$ in the corresponding edges of $\partial([0,1] \times \R)$. As each of
$\alphas'_0$, $\alphas_0$, and $\betas_0$ is individually non-separating on
$\Sigma_0$, any such holomorphic curve must have multiplicity at least 1 on
all of $\Sigma_0$, but such curves cannot match any non-empty subset
of~$\mathbf{p}(t)$, as~$\mathbf{p}$ does not intersect the boundary
of~$\Delta$. The remaining curves cannot match any~$\mathbf{p}(t)$ since they
have combined multiplicity at most $k-1$ over $p_0$.

Finally, we note that the formation of nontrivial constant holomorphic curves
is prohibited by transversality, as in the proof of
Proposition~\ref{lem:handleswapinvariance}. More precisely, if nontrivial
closed components form, then the holomorphic curves obtained by deleting
these constant curves would achieve transversality. But, in light of the
index formula from equation~\eqref{eq:firstindexformulatriangle}, if the
constant curves are nontrivial (i.e., there are no spheres with exactly one
or two punctures), then the index of the remaining curves drops by at
least~2, so the remaining curves do not appear in any moduli space
matching~$\mathbf{p}$, as $\text{Im}(\mathbf{p}) \subset \Sym^k(\Delta)$ has
dimension~1.

The remaining types of degenerations correspond to arcs connecting boundary
components of source curves that map to two different types of attaching
curves (e.g., collapsing an arc that connects an $\alpha$ boundary component
and a $\beta$ boundary component). These correspond to holomorphic strips
breaking off.

Given a sequence of holomorphic triangles $u_i$ in
$\mathcal{M}(\psi_0,\mathbf{p}(t_i))$ that converges to a broken triangle
with a nontrivial holomorphic strip, we know that the holomorphic triangle
$u$ appearing in the limit must also satisfy the matching condition with
respect to a divisor $\mathbf{p}(t)$ for some $t \in (0,1)$. The curve $u$
represents a class in $\pi_2(x,\mathbf{a},\mathbf{b})$ for some $x \in
\mathbb{T}_{\alpha'} \cap \mathbb{T}_\alpha$. We now consider separately the
contributions to the ends of
$\mathcal{M}_{(\Theta,\mathbf{a},\mathbf{b})}(\mathbf{p})$ corresponding to
different possibilities for $x$.

First, suppose that $x = \Theta$. As the class of the whole degenerate curve is $\psi_0$, we have
\[
\mu(u) = 2|\mathbf{d}| = \mu(\psi_0)
\]
by Lemma~\ref{lem:Maslovindex}. The remaining curves have total Maslov index
zero and have multiplicity zero at~$p_0$, and hence and hence represent the constant class.
By transversality, the curve $u$ is contained in the interior of
$\mathcal{M}_{(\Theta,\mathbf{a},\mathbf{b})}(\mathbf{p})$ and so does not
contribute to the boundary of the moduli space.

We now consider the case that $u$ represents a class in either
$\pi_2(\theta_1^+\theta_2^-,\mathbf{a},\mathbf{b})$ or
$\pi_2(\theta_1^-\theta_2^+,\mathbf{a},\mathbf{b})$. By
Lemma~\ref{lem:Maslovindex}, the remaining curves have total Maslov index~1,
and have multiplicity zero over $p_0$. The only possibility is that the
remaining curve consists of a single Maslov index~1 holomorphic strip. So
$\mathcal{M}_{(\Theta,\mathbf{a},\mathbf{b})}(\mathbf{p})$ has additional
ends corresponding to
\[
\bigcup_{\substack{t\in [0,1]\\ x \in \{\theta_1^+\theta_2^-,\theta_1^-\theta_2^+\}}} 
\bigcup_{\substack{\phi \in \pi_2(\Theta,x)\\ 
n_{p_0}(\phi)=0}}\widehat{\mathcal{M}}(\phi)\times \mathcal{M}_{(x,\mathbf{a},\mathbf{b})}(\mathbf{p}(t)).
\]
The above space has an even number of points, since the differential on
the chain complex $\CFa(\Sigma_0,\alphas'_0,\alphas_0,p_0)$
vanishes by Lemma~\ref{lem:differentialvanishes}.

We now claim that, for a generic choice of almost complex structure, there
are no additional ends of
$\mathcal{M}_{(\Theta,\mathbf{a},\mathbf{b})}(\mathbf{p})$. To this end, we
now claim that, for generic $J$, the space $\mathcal{M}(\psi, \mathbf{p})$ is
empty for any $\psi \in \pi_2(\theta_1^-\theta_2^-,\mathbf{a},\mathbf{b})$
with $n_{p_0}(\psi) = k$. By Lemma~\ref{lem:Maslovindex}, $\mu(\psi) = 2k-2$.
Proposition~\ref{prop:transversality} implies that $\mathcal{M}(\psi, S,
\mathbf{p})$ is a smooth manifold of dimension
\[
\ind(\psi,S) - (2k-1) \le \mu(\psi)-(2k-1) = -1,
\]
for a source~$S$. In particular, for generic $J$, there are no holomorphic
triangles $u$ (embedded or nodal) representing a class in
$\pi_2(\theta_1^-\theta_2^-, \mathbf{a}, \mathbf{b})$ with $\rho^{p_0}(u) \in
\mathbf{p}(t)$, for any~$t$.

Thus, modulo two, the count of the ends of
$\mathcal{M}_{(\Theta,\mathbf{a},\mathbf{b})}(\mathbf{p})$ is just
\[
0= \# \mathcal{M}_{(\Theta,\mathbf{a},\mathbf{b})}(\mathbf{d}_0) -
\#\mathcal{M}_{(\Theta,\mathbf{a},\mathbf{b})}(\mathbf{d}_1),
\]
showing that the quantity $\# \mathcal{M}_{(\Theta,\mathbf{a},\mathbf{b})}(\mathbf{d})$
is independent of~$\mathbf{d}$ modulo two for generic~$J$.

The final step is to show that
$\#\mathcal{M}_{(\Theta,\mathbf{a},\mathbf{b})}(\mathbf{d}) \equiv 1 \pmod 2$
for some choice of $\mathbf{d}$. To do this, we adapt an argument of
Ozsv\'ath and Szab\'o~\cite[p.~653]{OS08:HFL} from holomorphic strips to
triangles. Let $\mathbf{p} \colon [1,\infty) \to \Sym^k(\Delta)$ be a smooth
embedded path not intersecting the fat diagonal (a codimension 2 subset)
starting at a divisor $\mathbf{d} \in \Sym^k(\Delta)$, such that the points
in $\mathbf{p}(T)$ are spaced at least distance~$T$ apart that smoothly
approach the vertex~$v_{\alpha\beta}$ in~$\Delta$ as $T \to \infty$. Here, we
endow $\Delta$ with a metric in which the ends look like strips, as in
Figure~\ref{fig::5}.  We further assume that the points in $\mathbf{p}(T)$
avoid the three sides of~$\Delta$. We consider the space
\[
\mathcal{M}_{(\Theta,\mathbf{a},\mathbf{b})}(\mathbf{p})=\bigcup_{T \in [1,\infty)}  \mathcal{M}_{(\Theta,\mathbf{a},\mathbf{b})}(\mathbf{p}(T)),
\]
which has ends corresponding to
$\mathcal{M}_{(\Theta,\mathbf{a},\mathbf{b})}(\mathbf{d})$,  ends
corresponding to degenerations of holomorphic curves, as well as ends
corresponding to whatever curves appear in the limit as $T \to \infty$. The
same argument as in the previous case shows that the only degenerations that
can occur at finite $T$ correspond to a Maslov index~1 strip breaking off.
The same argument we used in the previous paragraph shows that the ends
corresponding to Maslov index~1 strips breaking off at finite $T$ have total
count equal to zero, modulo two.

We now claim that, as $T \to \infty$, the limit is to a Maslov index zero
triangle~$\tau$ on $(\Sigma_0,\alphas'_0,\alphas_0,\betas_0)$, and~$k$ Maslov
index~$2$ curves on $(\Sigma_0,\alphas_0,\betas_0)$, each satisfying a
matching condition to a point $d_i \in \Sym^1([0,1]\times\R)=[0,1]\times \R$.
Note that the total Maslov index is~$2k$. The limit contains curves mapping
into $\Sigma_0\times \Delta$ and $\Sigma_0 \times [0,1] \times \R$, all of
which could potentially be nodal. Because this configuration is the limit of
curves that satisfied the matching condition with the
divisors~$\mathbf{p}(T)$, and since the points of $\mathbf{p}(T)$ are
separating and heading off to the vertex $v_{\a\b}$, in the limit, we must
have~$k$ curves on $(\S_0,\alphas_0,\betas_0)$ that each satisfy the matching
condition with points $d_i \in [0,1]\times \R$. The homology classes of
curves on $(\S_0,\alphas_0,\betas_0)$ are all of the form
$e_{\mathbf{a}}+s[\S_0]$, where $e_{\mathbf{a}}$ denotes the constant disk at
$\mathbf{a} \in \Torus_{\a_0} \cap \Torus_{\b_0}$.  The Maslov index of such
a curve is~$2s$.  Since there are~$k$ curves satisfying matching conditions
with~$d_i$ (implying $s \ge 1$), each has Maslov index~$2$. If $\tau$ denotes
the remaining curves, which map into $\Sigma_0\times \Delta$, the only way
for the total Maslov index to be~$2k$ is for $\mu(\tau) = 0$, and for
those~$k$ curves satisfying matching conditions to be all of the remaining
curves in the limit. Furthermore, the holomorphic triangle $\tau$ must
satisfy $n_{p_0}(\tau) = 0$, since there are $k = |\mathbf{d}|$ holomorphic
strips with multiplicity~1 in the limit, and axiom~$(J'4')$ implies
positivity of intersections with  set $\{p_0\} \times \Delta$. As before, we
can use equation~\eqref{eq:expecteddimension}, as well as
Proposition~\ref{prop:transversality} to show that all of the $k+1$ curves
must satisfy $(M1)$--$(M8)$ (or the appropriate analogs for curves mapping
into $\Sigma_0 \times [0,1] \times \R$).

Hence, by a gluing result of Lipshitz~\cite[Proposition~A.1]{Lipshitz06:CylindricalHF},
we conclude that
\begin{equation}
\# \mathcal{M}_{(\Theta,\mathbf{a},\mathbf{b})}(\mathbf{d}) \equiv
\left(\# \mathcal{M}_{(\mathbf{a},\mathbf{a})}(d)\right)^k \cdot
\sum_{\substack{\psi\in \pi_2(\Theta,\mathbf{a},\mathbf{b})\\
n_{p_0}(\psi)=0}} \# \mathcal{M}(\psi) \pmod 2,
\label{eq:countbysplittingdivisor}
\end{equation}
where $d \in [0,1] \times \R$ is a point and $
\mathcal{M}_{(\mathbf{a},\mathbf{a})}(d)$ denotes the moduli space of Maslov
index~2 holomorphic curves~$u$ on $(\S_0,\alphas_0,\betas_0)$ satisfying
$\rho^{p_0}(u)=d$, for a generic almost complex structure satisfying
$(J1)$--$(J4)$ and $(J5')$. Note again that our use of almost complex
structures satisfying $(J5')$ ensures that any curve appearing
in~$\mathcal{M}_{(\mathbf{a},\mathbf{a})}(d)$ achieves transversality, as
long as each component is somewhere injective. The condition that they match
$d \in [0,1] \times \R$ at $p_0$ immediately implies that they are not
multiply covered. Hence, using Proposition~\ref{prop:transversality} as well
as the expected dimension from equation~\eqref{eq:expecteddimension}, we see
the curves in $\mathcal{M}_{(\mathbf{a},\mathbf{a})}(d)$ satisfy
$(M1)$--$(M8)$ (and in particular are not nodal).

Hence, it is sufficient to count $\mathcal{M}_{(\mathbf{a}, \mathbf{a})}(d)$
for a generic $d \in [0,1] \times \R$. The argument from the proof of
stabilization invariance due to
Lipshitz~\cite[Appendix~A]{Lipshitz06:CylindricalHF} adapts to our present
situation. We consider the Heegaard diagram for $S^1 \times S^2$ shown in
Figure~\ref{fig::4}.

\begin{figure}[ht!]
\centering
\input{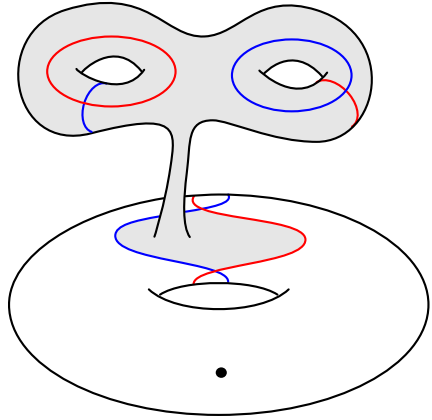}
\caption{A domain on a doubly stabilized diagram for $(S^1 \times S^2)$.
\label{fig::4}}
\end{figure}

There are exactly two homology classes of disks that contribute to the
differential of the hat complex. One is a bigon and has a unique holomorphic
representative. The other is a twice stabilized bigon. By invariance of
$\HFa$, we know the twice stabilized bigon must have one holomorphic
representative modulo~2 for any almost complex structure achieving
transversality. On the other hand, by stretching the neck, we also see that
the count for large neck length is equal to the count of strips in
$\mathcal{M}(e_{\mathbf{a}}+[\Sigma_0])$ that match the bigon at the
connected sum point. We conclude that
\[
\# \mathcal{M}_{(\mathbf{a},\mathbf{a})}(d) \equiv 1 \pmod 2.
\]
By Lemma~\ref{lem:handleswapinhat}, we know that
\[
\sum_{\substack{\psi\in \pi_2(\Theta,\mathbf{a},\mathbf{b})\\n_{p_0}(\psi)=0}} \# \mathcal{M}(\psi) \equiv 1 \pmod 2.
\]
Hence, from Equation~\eqref{eq:countbysplittingdivisor}, we see that
\[
\# \mathcal{M}_{(\Theta,\mathbf{a},\mathbf{b})}(\mathbf{d}) \equiv 1^k \cdot 1 = 1 \pmod 2,
\]
and the lemma follows.
\end{proof}

\subsection{Proof of Theorem~\ref{thm:strong}} \label{sec:strong}

We now combine the results that we have obtained and together imply
Theorem~\ref{thm:strong}. As explained at the beginning of
Section~\ref{sec:HeegaardFloer}, let $\HF^\circ$ denote one of the versions
of Heegaard Floer homology. In Subsection~\ref{sec:HeegaarFloerInvariant}, we
showed that $\HF^\circ$ is a weak Heegaard invariant in the sense of
Definition~\ref{def:weak-Heegaard}. In particular, in
Definition~\ref{def:HF-diagram}, we associated the $\FF_2$ vector space
$\HF^\circ(H)$ to an isotopy diagram~$H$. We associated the canonical
isomorphisms $\Phi^{A \to A'}_B$ and $\Phi^A_{B \to B'}$ to $\a$- and
$\b$-equivalences, respectively, in Definition~\ref{def:phi-ab}. For a
diffeomorphism~$d$, we constructed an induced isomorphism~$d_*$ in
Definition~\ref{def:diffeo}. Finally, to a stabilization $H \to H'$, we
assigned the isomorphism $\sigma_{H \to H'}$ in
Definition~\ref{def:stab-iso}, and to the destabilization $H' \to H$ its
inverse, $(\sigma_{H \to H'})^{-1}$.

The above isomorphisms satisfy the axioms for $\HF^\circ$ to be a strong
Heegaard invariant in the sense of Definition~\ref{def:strong-Heegaard}. At
the beginning of Subsection~\ref{sec:HFstrong}, we showed
axiom~\eqref{item:strong-funct}, functoriality, and
axiom~\eqref{item:strong-commute}, commutativity. We verified
axiom~\eqref{item:strong-cont}, continuity, in
Proposition~\ref{prop:continuity}. Finally,
axiom~\eqref{item:strong-handleswap}, simple handleswap invariance, was
proven in Subsection~\ref{subsec:Handleswaps}. \qed
\appendix
\section{The 2-complex of handleslides}
\label{app:handleslide}

In this appendix, we sketch a description of strong Heegaard invariants
for classical (i.e., not sutured) single pointed Heegaard diagrams
that is equivalent to Definition~\ref{def:strong-Heegaard}, and
instead of $\a$-equivalences and $\b$-equivalences, uses more
elementary moves: $\a$-isotopies, $\b$-isotopies, $\a$-handleslides,
and $\b$-handleslides. The tradeoff is that one has to check the
commutativity of the invariant~$F$ along a larger number of loops of
diagrams. But we do have to impose less on~$F$, and hence strengthen
Theorem~\ref{thm:iso}.  The main tool is a result of
Wajnryb~\cite{Waj98:MCG-Handlebody}, who constructed a
simply-connected 2-complex whose vertices consist of cut-systems, and
whose edges correspond to changing just one circle in a cut system.
We only sketch the proofs in this appendix.

We start off by looking at those moves that only involve $\a$-circles
or $\b$-circles.  For these, it is enough to consider only one of the
two handlebodies.  In particular, we show that any two cut-systems for
a handlebody can be connected by a sequence of handleslides.  This is
in fact a corollary of a result of
Wajnryb~\cite{Waj98:MCG-Handlebody}.  To state his result, let us
first recall some definitions.

\begin{definition}
  Let $B$ be a handlebody of genus $g$ and boundary $\S = \partial B$.
  A simple closed curve $\a \subset \S$ is a \emph{meridian curve} if
  it bounds a disk $D$ in $B$ such that $D \cap \S = \partial D = \a$.
  Then $D$ is called a \emph{meridian disk}.
  We also fix a finite
  number of disjoint distinguished disks on $\S$ and we shall assume
  that all isotopies of $\S$ are fixed on the distinguished disks.

  A \emph{cut-system} on $\S$ is an isotopy class of an unordered
  collection of $g$ disjoint meridian curves $\a_1, \dots, \a_g$ that
  are linearly independent in $H_1(\S)$
  and do not meet the distinguished disks.
  We denote the cut-system by $\langle \a_1,\dots,\a_g \rangle$.

  We say that two cut-systems are \emph{related by a simple move} if
  they have $g-1$ curves in common and the other two curves are
  disjoint.
\end{definition}

We construct a 2-dimensional complex $X_2(B)$. The vertices of $X$ are
the cut-systems on $\S$. Two cut-systems are connected by an edge if
they are related by a simple move; this gives the graph $X_1(B)$.  If
three vertices of $X$ have $g-1$ curves in common and the three
remaining curves, one from each cut-system, are pairwise disjoint,
then each pair of the vertices is connected by an edge in $X$ and the
vertices form a triangle. We glue a face to every triangle in $X_1(B)$
and get a 2-dimensional simplicial complex $X_2(B)$, called the
\emph{cut-system complex} of the handlebody $B$.

The following result is due to
Wajnryb~\cite[Theorem~1]{Waj98:MCG-Handlebody}.

\begin{citethm}\label{thm:handlebody-cmplx}
  The complex $X_2(B)$ is connected and simply-connected.
\end{citethm}

For compatibility with the other moves we consider, we work instead
with a 2-complex whose edges are elementary handleslides. To describe
the 2-cells, we need another definition.

\begin{definition} \label{def:handleslide-loop} A \emph{handleslide
    loop} is one of the following sequences of cut-systems connected
  by handleslides.
  \begin{enumerate}
  \item \label{item:slide-triang} A \emph{slide triangle}, formed by
    $\langle \a_1,\a_2, \vec{\alpha} \rangle$, $\langle \a_2,\a_3,
    \vec{\alpha} \rangle$, and $\langle \a_3, \a_1, \vec{\alpha}
    \rangle$, where $\a_1$, $\a_2$, and $\a_3$ bound a pair-of-pants.
  \item \label{item:simple-slide-square} A \emph{commuting slide
      square}, involving four distinct $\a$-curves, as in the link of
    a singularity of type~\ref{item:A1a}.
  \item \label{item:slide-square-over} A square formed by sliding
    $\a_1$ over $\a_2$ and/or $\a_4$, as in case~\ref{item:A1b}.
  \item A square formed by sliding $\a_1$ and/or $\a_3$ over $\a_2$,
    with $\a_1$ and $\a_3$ sliding over $\a_2$ from opposite sides, as
    in case~\ref{item:A1c}.
  \item A square formed by sliding $\a_1$ over $\a_2$ in two different
    ways, approaching $\a_2$ from opposite sides, as in case~\ref{item:A1d}.
  \item \label{item:slide-pentagon} A pentagon formed by sliding
    $\a_1$ over $\a_2$, which is itself sliding over $\a_3$, as in
    case~\ref{item:link-A2}; see Figure~\ref{fig:link-chain-flows}.
  \end{enumerate}
\end{definition}

Now suppose that there is exactly one distinguished disk on $\S
= \partial B$. Then let $Y_2(B)$ be the 2-complex whose vertices are
cut-systems on $B$, its edges correspond to handleslides avoiding the
distinguished disk, and its 2-cells correspond to the handleslide
loops of Definition~\ref{def:handleslide-loop}.

\begin{proposition}\label{prop:handlebody-cmplx-2-conn}
  The complex $Y_2(B)$ is connected.
\end{proposition}

\begin{proof}
  To prove connectivity, it suffices to show that the endpoints of
  each edge in $X_1(B)$ can be connected by a path lying in the
  1-skeleton $Y_1(B)$ of $Y_2(B)$.  Suppose we have an edge in
  $X_1(B)$ connecting $\langle \alpha_0, \vec{\alpha} \rangle$ and
  $\langle \alpha_1, \vec{\alpha} \rangle$.  Then $\alpha_0$ and
  $\alpha_1$ do not intersect. The combined set of circles $\langle
  \alpha_0, \alpha_1, \vec{\alpha} \rangle$ by hypothesis cuts $\bdy
  B$ into two components, exactly one of which does not contain the
  distinguished disk; call this component~$F$.  Both $\alpha_0$ and
  $\alpha_1$ necessarily appear in $\bdy F$.  We can get from $\langle
  \alpha_0, \vec{\alpha}\rangle$ to $\langle \alpha_1,
  \vec{\alpha}\rangle$ by sliding $\alpha_0$ over every component of
  $\bdy F \setminus (\alpha_0 \cup \alpha_1)$.
\end{proof}

\begin{proposition}\label{prop:handlebody-cmplx-2-sc}
  The complex $Y_2(B)$ is simply-connected.
\end{proposition}

\begin{proof}[Proof sketch]
  For simple connectivity, we first show that all the different ways
  of turning an edge of $X_1(B)$ into a path in $Y_1(B)$ are homotopic
  inside $Y_2(B)$. This can be done (with some work) using handleslide
  loops of
  type~\eqref{item:slide-square-over}. For a simple example, see
  Figure~\ref{fig:large-slide-edge}.%
\footnote{To do this properly, note that a minimal path in $Y_1(B)$
  corresponding to an edge in $X_1(B)$ gives a pants decomposition of
  a subsurface of $\partial B$. To show that two such paths are
  homotopic in $Y_2(B)$, it suffices to show connectivity of a
  suitable variant of the pants complex.}
  \begin{figure}
    \centering
    \includegraphics{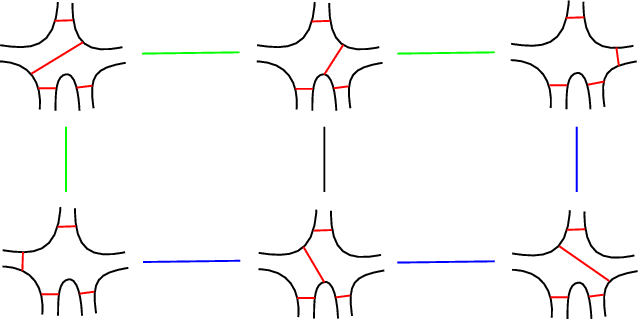}
    \caption{A simple example of a homotopy in $Y_2(B)$ connecting two different
      resolutions of an edge of $X_1(B)$.  The lower left and the
      upper right cut-systems are the vertices of the edge of $X_1(B)$
      we are resolving. One resolution is shown in green, the other
      one in blue.  We show half of the component~$F$ whose boundary
      contains $\alpha_0$, $\alpha_1$, and no basepoints.  In this
      case, there are $k=3$ other boundary components of~$F$.  The
      surfaces shown should be doubled along the black boundary to
      obtain~$F$; in this way the red arcs become red circles.}
    \label{fig:large-slide-edge}
  \end{figure}

  Next, we show that if we convert the edges $e_0$, $e_1$, and $e_2$
  of a triangle $\Delta$ in $X_2(B)$ into paths in $Y_1(B)$, we obtain
  a loop that is null-homotopic in $Y_2(B)$. Let $v_i$ be the vertex
  of $\Delta$ opposite the edge $e_i$.  We distinguish two cases:
  \begin{itemize}
  \item The same circle moves in all three edges of the triangle;
    i.e., the cut-system $v_i = \langle \a_i, \vec{\a} \rangle$ for $i
    \in \ZZ_3$.
  \item Two circles are involved; i.e., the cut-system $v_i = \langle
    \a_{i-1}, \a_{i+1},\vec{\a} \rangle$ for every $i \in \ZZ_3$,
    where $i-1$ and $i+1$ are to be considered modulo 3.
  \end{itemize}
  The first case is simple: we end up with a trivial loop even in
  $Y_1(B)$ for an appropriate choice of resolutions. Indeed, for $i
  \in \ZZ_3$, let $F_i$ be the component of the complement of $\langle
  \a_{i-1}, \a_{i+1}, \vec{\a} \rangle$ that does not contain the
  distinguished disk. Then $F_i = F_{i-1} \cup F_{i+1}$ for some $i
  \in \ZZ_3$.  We first convert $e_{i-1}$ and $e_{i+1}$ to paths
  $\g_{i-1}$ and $\g_{i+1}$ in $Y_1(B)$ using the procedure above,
  then we choose $\g_i$ to be $\g_{i+1}^{-1}\g_{i-1}^{-1}$. By the
  first step, any two choices for $\g_i$ are homotopic, so we can pick
  this particular one.

  In the second case, we get a component $F$ with boundary containing
  $\alpha_0$, $\alpha_1$, and~$\alpha_2$. A handleslide loop connects
  $\langle \alpha_0, \alpha_1, \vec{\alpha} \rangle$, $\langle
  \alpha_1, \alpha_2, \vec{\alpha} \rangle$, and $\langle \alpha_0,
  \alpha_2, \vec{\alpha} \rangle$.  If there are no other components
  of $\bdy F$, this is a slide triangle (a handleslide loop of
  type~\eqref{item:slide-triang}).
  Otherwise, if there are $k$ other
  boundary components of $\bdy F$, let $\alpha_0'$ be the curve
  obtained from $\alpha_0$ by sliding over one of the other $k$
  components. By induction, the triangle connecting
  $\langle \alpha_0',\alpha_1, \vec\alpha \rangle$,
  $\langle \alpha_1, \alpha_2, \vec\alpha \rangle$, and
  $\langle \alpha_0',\alpha_2, \vec\alpha \rangle$ can be decomposed
  into allowed two-cells. The remaining region (a quadrilateral with
  corners at
  $\langle \alpha_0',\alpha_1, \vec\alpha \rangle$,
  $\langle \alpha_0',\alpha_2, \vec\alpha \rangle$,
  $\langle \alpha_0,\alpha_1, \vec\alpha \rangle$, and
  $\langle \alpha_0,\alpha_2, \vec\alpha \rangle$) can be decomposed into
  $k-2$ commuting slide squares
  (type~\eqref{item:simple-slide-square}) and one slide pentagon
  (type~\eqref{item:slide-pentagon}). The entire large triangle is
  decomposed into one slide triangle, $k-1$ slide pentagons, and
  $\binom{k-1}{2}$ commuting slide squares.  An example of the end
  result is shown in Figure~\ref{fig:large-slide-triangle}.
\end{proof}

\begin{figure}
  \centering
  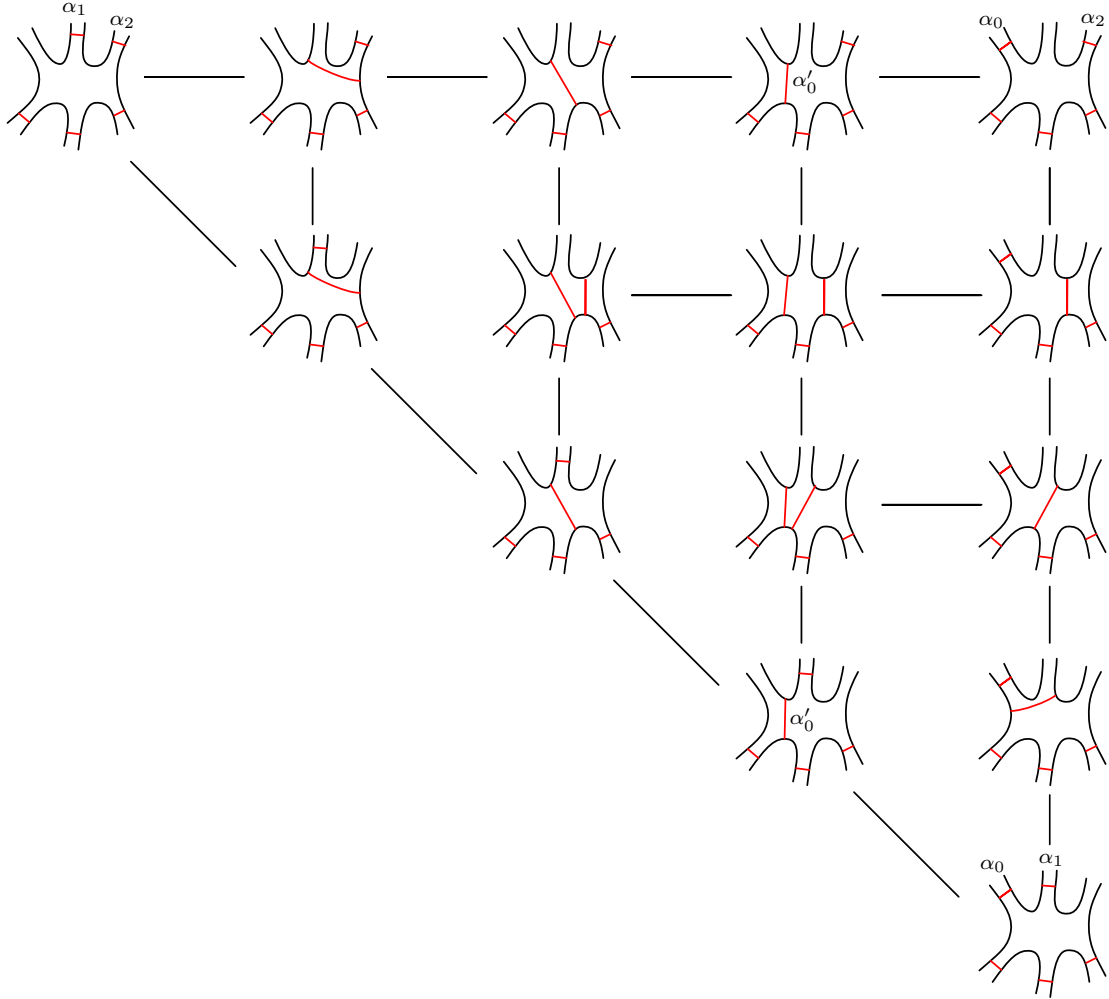
  \caption{Decomposing a large slide triangle.  The three vertices are
    the vertices of a triangle in the complex $X_2(B)$.  We show half
    of the component~$F$ whose boundary contains $\alpha_0$,
    $\alpha_1$, and $\alpha_2$, and no basepoints.  In this case,
    there are $k=3$ other boundary components of~$F$.}
  \label{fig:large-slide-triangle}
\end{figure}

Let $\G'$ be the graph defined just like in
Definition~$\ref{def:big-graph}$, but with the word
$\a$/$\b$-equivalence replaced by $\a$/$\b$-handleslide. So the
vertices of $\G'$ are isotopy diagrams, and its edges correspond to
handleslides, stabilizations, destabilizations, and
diffeomorphisms. Since every handleslide is an $\a$-equivalence or a
$\b$-equivalence, $\G'$ is a subgraph of $\G$.

Similarly, we can modify Definition~\ref{def:weak-Heegaard}.  If $\cS$
is a set of diffeomorphism types of sutured manifolds and $\C$ is a
category, then $\G'(\cS)$ is the full subgraph of $\G'$ spanned by
those isotopy diagrams $H$ for which $S(H) \in \cS$. The main result
of this appendix is the following.

\begin{theorem}
  Let $\cS = \cS_\man$ be the set of diffeomorphism types of sutured
  manifolds introduced in Definition~\ref{def:balanced}, and let $\C$
  be a category.  Then every morphism of graphs $F' \colon \G'(\cS)
  \to \C$ extends to a weak Heegaard invariant $F \colon \G(\cS) \to
  \C$.

  If, furthermore,
  \begin{itemize}
  \item $F'$ satisfies the second half of axiom~\eqref{item:strong-funct}
  and axioms~\eqref{item:strong-commute}--\eqref{item:strong-handleswap}
  of Definition~\ref{def:strong-Heegaard}, replacing
  ``$\a$/$\b$-equivalence'' with ``$\a$/$\b$-handleslide,''
  \item $F'$ commutes along each handleslide loop in Definition~\ref{def:handleslide-loop}, and
  \item $F'$ commutes along every stabilization slide (see Definition~\ref{def:stab-slide}),
  \end{itemize}
  then $F'$ uniquely extends to a strong Heegaard invariant $F \colon \G(\cS) \to \C$.
\end{theorem}

\begin{remark}
  Note that $\G'_\a(\cS)$ and $\G'_\b(\cS)$ are not sub-categories of
  $\G'(\cS)$, since the composition of two handleslides is in general not
  a handleslide. The functoriality of $F'$ restricted to the subgraphs
  $\G'_\a(\cS)$ and $\G'_\b(\cS)$ is replaced by the requirement that
  $F'$ commutes along handleslide loops.

  Also note that, in a stabilization slide, we subdivide the $\a$- or
  $\b$-equivalence into two handleslides, so we view this as a loop of
  length four.
\end{remark}

\begin{proof}[Proof sketch]
  To prove the first part, we only have to define $F(e)$ for the edges
  $e$ of $\G(\cS)$ that correspond to an $\a$-equivalence or a
  $\b$-equivalence. Without loss of generality, suppose that $e$ is an
  $\a$-equivalence between the isotopy diagrams $H =
  (\S,\alphas,\betas)$ and $H' = (\S, \alphas',\betas)$. Let
  $\ol{\S}$ be the surface obtained by attaching a disk $D$ to $\S$ along
  its boundary. This way we obtain two Heegaard diagrams $\ol{H}$ and
  $\ol{H}'$, containing a distinguished disk $D$.
  Let $Y$ be a 3-manifold containing both $\ol{H}$ and $\ol{H}'$
  as Heegaard diagrams, and let $B$ be the handlebody lying to the
  negative side of $\ol{\S}$.
  By Proposition~\ref{prop:handlebody-cmplx-2-conn}, the complex $Y_2(B)$
  is connected, so $\ol{H}$ and $\ol{H}'$ can be connected by a path
  of handleslides $\ol{h}_1, \dots, \ol{h}_k$ avoiding $D$.  This
  gives rise to a sequence of handleslides $h_1, \dots, h_k$
  connecting $H$ and $H'$.  Then the isomorphism $F(e)$ is defined to
  be the composite $F(h_k) \circ \dots \circ F(h_1)$.

  We now prove the second part. According to
  Proposition~\ref{prop:handlebody-cmplx-2-sc}, the complex
  $Y_2(B)$ is simply connected. Together with the fact that $F'$
  commutes along every handleslide loop (i.e., along the boundary of
  every face of $Y_2(B)$), we see that the extension of~$F'$ to an
  $\a$- or $\b$-equivalence edge $e$ is independent of the choice of
  path $h_1,\dots,h_k$. Functoriality of the restriction of $F$ to
  $\G_\a(\cS)$ and $\G_\b(\cS)$ is clear from the construction.

  What remains to show is that $F$ commutes along every distinguished
  rectangle of type~\eqref{item:rect-alpha-beta},
  \eqref{item:rect-alpha-stab}, and~\eqref{item:rect-alpha-diff}
  (see Definition~\ref{def:distinguished-rect}), with sides $e$, $f$,
  $g$, and $h$.  First, consider a rectangle of
  type~\eqref{item:rect-alpha-beta}.  Write the $\a$-equivalence $e$
  as a path of $\a$-handleslides $h_1, \dots, h_k$ and the
  $\b$-equivalence $f$ as a path of $\b$-handleslides
  $h_1',\dots,h_l'$. Then we can subdivide the big rectangle into a
  grid of smaller rectangles with sides $h_i$ and $h_j'$ for $i \in
  \{\, 1,\dots,k \,\}$ and $j \in \{\, 1,\dots,l \,\}$.

  Given a rectangle of type~\eqref{item:rect-alpha-diff}, let
  $h_1,\dots, h_k$ be the path of handleslides in the resolution of
  the $\a$- or $\b$-equivalence $e$, and let $d$ be the diffeomorphism
  corresponding to $f$ and $g$. Then we can subdivide the big
  rectangle into a row of smaller rectangles with sides $h_i$ and $d$
  for $i \in \{\, 1,\dots, k \,\}$.

  Finally, consider a rectangle of
  type~\eqref{item:rect-alpha-stab}. Then let $h_1,\dots,h_k$ be the
  resolution of the $\a$- or $\b$-equivalence $e$ on the destabilized
  side. Then, on the stabilized side $h$, we can choose the
  stabilizations $h_1',\dots,h_k'$ of the above handleslides. However,
  the endpoint of $h_k'$ might differ from $H_4$, the endpoint of $h$,
  by a sequence of handleslides over the new $\a$ or $\b$-curve
  appearing in the stabilization. We can correct this by attaching
  a row of stabilization slides to the row of rectangles with
  horizontal sides $h_i$ and $h_i'$.
\end{proof}

\bibliographystyle{amsplain}
\bibliography{heegaardfloer,topo}
\end{document}